\documentclass[11pt]{amsart}
\usepackage{amsfonts, amsthm, amssymb, amsmath, stmaryrd, color, enumerate}
\usepackage[pdftex]{graphicx}
\usepackage{mathrsfs,array}
\usepackage{xy}
\usepackage{hyperref}
\usepackage{tikz}
\input xy
\xyoption{all}

\def\from{\colon}

\DeclareMathOperator{\Spa}{Spa}

\DeclareMathOperator{\Spf}{Spf}

\DeclareMathOperator{\Spd}{Spd}

\DeclareMathOperator{\GL}{GL}
\DeclareMathOperator{\Gal}{Gal}
\DeclareMathOperator{\Hom}{Hom}
\DeclareMathOperator{\sHom}{\!{\mathscr{H}\! om}}
\DeclareMathOperator{\sCat}{\!{\mathscr{C}\! at}}

\DeclareMathOperator{\Map}{Map}

\DeclareMathOperator{\Spec}{Spec}

\DeclareMathOperator{\Perf}{Perf}
\DeclareMathOperator{\Perfd}{Perfd}

\DeclareMathOperator{\op}{op}

\DeclareMathOperator{\eq}{eq}
\DeclareMathOperator{\dimtr}{dim.tr}
\DeclareMathOperator{\dimtrg}{dim.trg}
\DeclareMathOperator{\trc}{tr.c}
\DeclareMathOperator{\cdim}{cd}
\DeclareMathOperator{\trdeg}{tr.deg}
\newcommand{\dotimes}{\otimes^{\mathbb L}}
\def\OO{\mathcal{O}}
\newcommand{\et}{{\mathrm{\acute{e}t}}}
\newcommand{\fet}{{\mathrm{f\acute{e}t}}}
\newcommand{\proet}{{\mathrm{pro\acute{e}t}}}
\newcommand{\qproet}{{\mathrm{qpro\acute{e}t}}}
\newcommand{\qcqs}{{\mathrm{qcqs}}}
\newcommand{\qc}{{\mathrm{qc}}}
\newcommand{\sep}{{\mathrm{sep}}}
\renewcommand{\min}{{\mathrm{min}}}
\newcommand{\aff}{\mathrm{aff}}
\newcommand{\Pro}{\mathrm{Pro}}
\renewcommand{\op}{\mathrm{op}}
\newcommand{\cons}{\mathrm{cons}}
\newcommand{\prop}{\mathrm{prop}}
\newcommand{\wl}{\mathrm{wl}}
\newcommand{\perf}{\mathrm{perf}}
\newcommand{\cart}{\mathrm{cart}}
\newcommand{\Cons}{\mathrm{Cons}}
\newcommand{\id}{\mathrm{id}}
\DeclareMathOperator{\varprojlimfi}{``\varprojlim_{\mathit{i}}''}
\def\isom{\cong}
\def\F{\mathbf{F}}
\def\Z{\mathbf{Z}}

\def\Q{\mathbf{Q}}

\def\Fl{{\mathbb F_\ell}}

\renewcommand*{\diamond}{\diamondsuit}
\renewcommand*{\hat}{\widehat}
\renewcommand*{\tilde}{\widetilde}

\newcommand{\abs}[1]{\left| #1 \right|}
\newcommand{\tatealgebra}[1]{\langle #1 \rangle}

\newcommand{\laurentseries}[1]{(\!( #1 )\!)}

\numberwithin{equation}{section}
\newtheorem{theorem}{Theorem}
\numberwithin{theorem}{section}
\newtheorem{lemma}[theorem]{Lemma}
\newtheorem{corollary}[theorem]{Corollary}
\newtheorem{proposition}[theorem]{Proposition}

\newtheorem{definition}[theorem]{Definition}
\newtheorem{defprop}[theorem]{Definition/Proposition}
\newtheorem{question}[theorem]{Question}

\theoremstyle{definition}
\newtheorem{remark}[theorem]{Remark}
\newtheorem{example}[theorem]{Example}

\newtheorem{convention}[theorem]{Convention}

\newenvironment{altenumerate}
   {\begin{list}
      {\textup{(\theenumi)} }
      {\usecounter{enumi}
       \setlength{\labelwidth}{0pt}
       \setlength{\labelsep}{2pt}
       \setlength{\leftmargin}{0pt}
       \setlength{\itemsep}{\the\smallskipamount}
       \renewcommand{\theenumi}{\roman{enumi}}
      }}
   {\end{list}}

\date{\today}

\title{\'Etale cohomology of diamonds}
\author{Peter Scholze}
\begin{document}

\begin{abstract}
Motivated by problems on the \'etale cohomology of Rapoport--Zink spaces and their generalizations, as well as Fargues's geometrization conjecture for the local Langlands correspondence, we develop a six functor formalism for the \'etale cohomology of diamonds, and more generally small v-stacks on the category of perfectoid spaces of characteristic $p$. Using a natural functor from analytic adic spaces over $\Z_p$ to diamonds which identifies \'etale sites, this induces a similar formalism in that setting, which in the noetherian setting recovers the formalism from Huber's book, \cite{Huber}.
\end{abstract}

\maketitle

\tableofcontents

\section{Introduction}

The aim of this manuscript is to lay foundations for the six operations in \'etale cohomology of adic spaces and diamonds, generalizing previous work of Huber, \cite{Huber}.

In this manuscript, we will deal with analytic adic spaces $X$ on which a fixed prime $p$ is topologically nilpotent. Associated with any such $X$, we have an \'etale site $X_\et$ defined (under noetherian hypotheses) by Huber, \cite{Huber}. For any ring $\Lambda$ such that $n\Lambda=0$ for some $n$ prime to $p$, we get a (left-completed) derived category $D_\et(X,\Lambda)$ of \'etale sheaves of $\Lambda$-modules on $X_\et$. It comes equipped with $-\otimes_\Lambda -$ and $R\sHom_\Lambda(-,-)$, and for a map $f: Y\to X$, one gets adjoint functors
\[
f^\ast: D_\et(X,\Lambda)\to D_\et(Y,\Lambda)
\]
and
\[
Rf_\ast: D_\et(Y,\Lambda)\to D_\et(X,\Lambda)\ .
\]

Moreover, under certain conditions, one can define the pushforward with proper supports
\[
Rf_!: D_\et(Y,\Lambda)\to D_\et(X,\Lambda)\ ,
\]
with a right adjoint
\[
Rf^!: D_\et(X,\Lambda)\to D_\et(Y,\Lambda)\ .
\]
If $f$ is smooth, then (a form of) Poincar\'e duality identifies $Rf^!$ as a twist of $f^\ast$. Under suitable hypotheses, as for example for the adic spaces corresponding to rigid spaces over a fixed complete nonarchimedean base field, these results all appear in \cite{Huber}.

Our aim here is to give a substantial generalization of this formalism. The motivation comes from some developments surrounding the theory of ``local Shimura varieties'' whose cohomology is expected to realize local Langlands correspondences for $p$-adic fields, and in particular Fargues' conjecture geometrizing the local Langlands correspondence through sheaves on the stack of $G$-bundles on the Fargues--Fontaine curve, \cite{RapoportViehmann, FarguesGeom, FarguesScholze}. The formalism of this paper has recently been used by Hansen--Kaletha--Weinstein for progress on the Kottwitz conjecture on the cohomology of local Shimura varieties, \cite{HansenKalethaWeinstein}.

In this context, it is necessary to work with certain objects that are not representable by adic spaces, but are merely quotients of adic spaces by ``pro-\'etale'' equivalence relations, akin to Artin's algebraic spaces. If one is willing to work with stacks, this brings objects like the classifying stack of $\GL_n(\mathbb Q_p)$ into the picture.

In fact, as any analytic adic space over $\Z_p$ is itself a quotient of a perfectoid space by a pro-\'etale equivalence relation, one can as well work with quotients of \emph{perfectoid} spaces by pro-\'etale equivalence relations. By tilting, one can even assume that these perfectoid spaces are of characteristic $p$. This leads to the notion of a diamond. Let us quickly review the definition.

\begin{definition} A map $f: Y=\Spa(S,S^+)\to X=\Spa(R,R^+)$ of affinoid perfectoid spaces is affinoid pro-\'etale if it can be written as a cofiltered limit of \'etale maps $Y_i=\Spa(S_i,S_i^+)\to X$ of affinoid perfectoid spaces. More generally, a map $f: Y\to X$ of perfectoid spaces is pro-\'etale if it is locally on the source and target affinoid pro-\'etale.
\end{definition}

Let $\Perf$ denote the category of perfectoid spaces of characteristic $p$.\footnote{We ignore small set-theoretic questions in this introduction. One solution is to carefully choose a cutoff cardinal $\kappa$ and only work with $\kappa$-small perfectoid spaces; we essentially adopt this solution, but take a large filtered colimit over all choices of $\kappa$ in the end, in order to have a canonical theory.} There are several different topologies one can consider on it, successively refining the previous.

\begin{altenumerate}
\item The analytic topology, where a cover $\{f_i: X_i\to X\}$ consists of open immersions $X_i\hookrightarrow X$ which jointly cover $X$.
\item The \'etale topology, where a cover $\{f_i: X_i\to X\}$ consists of \'etale maps $X_i\to X$ which jointly cover $X$.
\item The pro-\'etale topology, where a cover $\{f_i: X_i\to X\}$ consists of pro-\'etale maps $X_i\to X$ such that for any quasicompact open subset $U\subset X$, there are finitely many indices $i$ and quasicompact open subsets $U_i\subset X_i$ such that the $U_i$ jointly cover $U$.
\item The v-topology, where a cover $\{f_i: X_i\to X\}$ consists of any maps $X_i\to X$ such that for any quasicompact open subset $U\subset X$, there are finitely many indices $i$ and quasicompact open subsets $U_i\subset X_i$ such that the $U_i$ jointly cover $U$.
\end{altenumerate}

We note that the last part of the definition in cases (iii) and (iv) is the same condition that appears in the definition of the fpqc topology for schemes, and is automatic in cases (i) and (ii) as \'etale maps are open.

It turns out that all topologies are well-behaved; the v-topology is the finest, so for brevity we only state it in this case:

\begin{theorem}\label{thm:vdescent} The v-topology on $\Perf$ is subcanonical, and for any affinoid perfectoid space $X=\Spa(R,R^+)$, $H^0_v(X,\OO_X) = R$, $H^0_v(X,\OO_X^+)=R^+$ and for $i>0$, $H^i_v(X,\OO_X)=0$ and $H^i_v(X,\OO_X^+)$ is almost zero.
\end{theorem}

Analogous to Artin's theory of algebraic spaces, it is for the intended applications of the formalism necessary to allow quotients of perfectoid spaces by pro-\'etale equivalence relations. These are called diamonds:

\begin{definition} Let $Y$ be a pro-\'etale sheaf on $\Perf$. Then $Y$ is a diamond if $Y$ can be written as a quotient $Y=X/R$ of a perfectoid space $X$ by a pro-\'etale equivalence relation $R\subset X\times X$.
\end{definition}

It turns out that all diamonds are v-sheaves. Now, for a diamond $X$, one can define sites $X_v$ and $X_\et$. Actually, in the case of $X_\et$, one has to be careful, as in general there may not be enough \'etale maps into $X$. For this reason (and others), we often work with a restricted class of diamonds which turn out to be better-behaved. Here, we use the underlying topological space of a diamond $Y$, given as $|Y|=|X|/|R|$ in case $Y=X/R$ is a quotient of a perfectoid space $X$ by a pro-\'etale equivalence relation $R$. There is a bijection between open subspaces of $|Y|$ and open subfunctors of $Y$, for an obvious definition of the latter. If $Y$ is quasicompact as a v-sheaf, i.e.~any collection of collectively surjective maps $Y_i\to Y$ of v-sheaves has a finite subcover, then $|Y|$ is quasicompact, but the converse is not true in general.

\begin{definition} A diamond $Y$ is spatial if it is qcqs, and $|Y|$ admits a basis for the topology given by $|U|$, where $U\subset Y$ ranges over quasicompact open subdiamonds. More generally, $Y$ is locally spatial if it admits an open cover by spatial diamonds.
\end{definition}

In particular, any perfectoid space $X$ defines a locally spatial diamond, which is spatial precisely when $X$ is qcqs. If $X$ is a (locally) spatial diamond, then $|X|$ is a (locally) spectral topological space, and $X$ is quasicompact (resp.~quasiseparated) as a v-sheaf precisely when $|X|$ is quasicompact (resp.~quasiseparated) as a topological space. Now, if $X$ is a locally spatial diamond, we define $X_\et$ as consisting of (locally separated) \'etale maps $Y\to X$ from diamonds $Y$ (automatically locally spatial); it turns out that this is ``big enough''.

The following theorem shows that this generalizes the classical notions.

\begin{theorem} There is a natural functor from analytic adic spaces over $\Z_p$ to locally spatial diamonds, denoted $X\mapsto X^\diamond$, satisfying $X_\et\cong X^\diamond_\et$.
\end{theorem}

Here, the functor $X\mapsto X^\diamond$ is intuitively given as follows. Choose a pro-\'etale surjection $\tilde{X}\to X$ from a perfectoid space $\tilde{X}$, and let $R\subset \tilde{X}\times \tilde{X}$ be the induced equivalence relation. This is pro-\'etale over $\tilde{X}$, and thus perfectoid itself. Then
\[
X^\diamondsuit= \tilde{X}^\flat / R^\flat\ .
\]

\begin{theorem} Let $X$ be a locally spatial diamond. Then pullback induces a fully faithful functor
\[
\widehat{D}(X_\et,\Lambda)\to D(X_v,\Lambda)\ ,
\]
where $\widehat{D}$ denotes the left-completion of the derived category; $D(X_v,\Lambda)$ is already left-complete.

Moreover, if $A\in D(X_v,\Lambda)$ and $f: Y\to X$ is a map of locally spatial diamonds that is surjective as a map of v-sheaves, then if $f^\ast A$ lies in $\widehat{D}(Y_\et,\Lambda)$, then $A$ lies in $\widehat{D}(X_\et,\Lambda)$.
\end{theorem}

Much of the formalism actually extends to all v-sheaves, or even v-stacks, subject to a small set-theoretic assumption of being ``small''.

\begin{definition} Let $X$ be a small v-stack, and consider the site $X_v$ of all perfectoid spaces over $X$, with the v-topology. Define the full subcategory
\[
D_\et(X,\Lambda)\subset D(X_v,\Lambda)
\]
as consisting of all $A\in D(X_v,\Lambda)$ such that for all (equivalently, one surjective) map $f: Y\to X$ from a locally spatial diamond $Y$, $f^\ast A$ lies in $\widehat{D}(Y_\et,\Lambda)$.
\end{definition}

Now we can state that we have the following operations.

\begin{altenumerate}
\item A (derived) tensor product
\[
-\dotimes_\Lambda - : D_\et(X,\Lambda)\times D_\et(X,\Lambda)\to D_\et(X,\Lambda)\ .
\]
This is compatible with the inclusions into $D(X_v,\Lambda)$, and the usual derived tensor product on $X_v$.
\item An internal Hom
\[
R\sHom_\Lambda(-,-): D_\et(X,\Lambda)^\op\times D_\et(X,\Lambda)\to D_\et(X,\Lambda)
\]
characterized by the adjunction
\[
R\Hom_{D_\et(X,\Lambda)}(A,R\sHom_\Lambda(B,C)) = R\Hom_{D_\et(X,\Lambda)}(A\dotimes_\Lambda B,C)
\]
for all $A,B,C\in D_\et(X,\Lambda)$. In particular, for $A=\Lambda$,
\[
R\Gamma(X,R\sHom_\Lambda(B,C)) = R\Hom_{D_\et(X,\Lambda)}(B,C)\ .
\]
In general, the formation of $R\sHom_\Lambda$ does not commute with the inclusion $D_\et(X,\Lambda)\subset D(X_v,\Lambda)$.
\item For any map $f: Y\to X$ of small v-stacks, a pullback functor
\[
f^\ast: D_\et(X,\Lambda)\to D_\et(Y,\Lambda)\ .
\]
This is compatible with the inclusions into $D(X_v,\Lambda)$ resp.~$D(Y_v,\Lambda)$, and the pullback functor $D(X_v,\Lambda)\to D(Y_v,\Lambda)$.
\item For any map $f: Y\to X$ of small v-stacks, a pushforward functor
\[
Rf_\ast: D_\et(Y,\Lambda)\to D_\et(X,\Lambda)
\]
which is right adjoint to $f^\ast$. In general, formation of $Rf_\ast$ does not commute with the inclusions into $D(X_v,\Lambda)$ resp.~$D(Y_v,\Lambda)$, but this holds true if $f$ is qcqs and one starts with an object of $D^+$.
\item For any map $f: Y\to X$ of small v-stacks that is compactifiable (cf.~Definition~\ref{def:compactifiable}), representable in locally spatial diamonds (cf.~Definition~\ref{def:locallyspatialmorphism}) and with (locally) $\dimtrg f<\infty$ (cf.~Definition~\ref{def:dimtrgdiamond}), a functor
\[
Rf_!: D_\et(Y,\Lambda)\to D_\et(X,\Lambda)\ .
\]
\item For any map $f: Y\to X$ of small v-stacks that is compactifiable, representable in locally spatial diamonds and with (locally) $\dimtrg f<\infty$, a functor
\[
Rf^!: D_\et(X,\Lambda)\to D_\et(Y,\Lambda)
\]
that is right adjoint to $Rf_!$.
\end{altenumerate}

The first four operations are defined in Section~\ref{sec:fourfunctors}, and then $Rf_!$ is defined in Section~\ref{sec:properpushforward}, in particular Definition~\ref{def:Rfshriek}, and $Rf^!$ is defined in Section~\ref{sec:cohomsmooth}.

These operations satisfy the following formulas.

\begin{theorem} Let $f: Y\to X$ be a map of small v-stacks.
\begin{altenumerate}
\item For all $A,B\in D_\et(X,\Lambda)$, one has
\[
f^\ast A\dotimes_\Lambda f^\ast B\cong f^\ast(A\dotimes_\Lambda B)\ .
\]
\item For all $A\in D_\et(X,\Lambda)$, $B\in D_\et(Y,\Lambda)$, one has
\[
Rf_\ast R\sHom_\Lambda(f^\ast A,B)\cong R\sHom_\Lambda(A,Rf_\ast B)\ .
\]
\item Assume that $f$ is compactifiable and representable in locally spatial diamonds with locally $\dimtrg f <\infty$. For all $A\in D_\et(X,\Lambda)$, $B\in D_\et(Y,\Lambda)$, one has
\[
Rf_!(A\dotimes_\Lambda f^\ast B)\cong Rf_! A\dotimes_\Lambda B\ .
\]
\item Assume that $f$ is compactifiable and representable in locally spatial diamonds with locally $\dimtrg f <\infty$. For all $A\in D_\et(Y,\Lambda)$, $B\in D_\et(X,\Lambda)$, one has
\[
R\sHom_\Lambda(Rf_!A,B)\cong Rf_\ast R\sHom_\Lambda(A,Rf^! B)\ .
\]
\item Assume that $f$ is compactifiable and representable in locally spatial diamonds with locally $\dimtrg f <\infty$. For all $A, B\in D_\et(X,\Lambda)$, one has
\[
Rf^! R\sHom_\Lambda(A,B)\cong R\sHom_\Lambda(f^\ast A,Rf^! B)\ .
\]
\end{altenumerate}
\end{theorem}

Parts (i) and (ii) are proved in Section~\ref{sec:fourfunctors}. Part (iii) is Proposition~\ref{prop:projectionformula}, and parts (iv) and (v) are formal consequences of (iii) and adjunctions, cf.~Proposition~\ref{prop:localverdier}.

Moreover, there are the following base change results.

\begin{theorem} Let
\[\xymatrix{
Y^\prime\ar[r]^{\tilde{g}}\ar[d]^{f^\prime} &Y\ar[d]^f\\
X^\prime\ar[r]^g & X
}\]
be a cartesian diagram of small v-stacks.
\begin{altenumerate}
\item Assume that $f$ is qcqs, and $A\in D^+_\et(Y,\Lambda)$. Then
\[
g^\ast Rf_\ast A\cong Rf^\prime_\ast \tilde{g}^\ast A\ .
\]
\item Assume that $f$ is compactifiable and representable in locally spatial diamonds with locally $\dimtrg f<\infty$, and $A\in D_\et(Y,\Lambda)$. Then
\[
g^\ast Rf_! A\cong Rf^\prime_! \tilde{g}^\ast A\ .
\]
\item Assume that $g$ is compactifiable and representable in locally spatial diamonds with locally $\dimtrg g<\infty$, and $A\in D_\et(Y,\Lambda)$. Then
\[
Rg^! Rf_\ast A\cong Rf^\prime_\ast R\tilde{g}^! A\ .
\]
\end{altenumerate}
\end{theorem}

Part (i) is Proposition~\ref{prop:Rfastsimple}, part (ii) is Proposition~\ref{prop:Rfshriekbasechange}, and part (iii) is Proposition~\ref{prop:smoothbasechange}.

We also need a theory of smooth morphisms. From now on, fix a prime $\ell\neq p$, and assume that $\Lambda$ is killed by a power of $\ell$. In Definition~\ref{def:cohsmooth}, we define a notion of $\ell$-cohomologically smooth morphisms for any prime $\ell\neq p$ (which depends on the prime $\ell$), which includes \'etale maps.

\begin{theorem} Let $f: Y\to X$ be a map of small v-stacks that is separated and representable in locally spatial diamonds with locally $\dimtrg f<\infty$. The property of $f$ being $\ell$-cohomologically smooth is v-local on the target and $\ell$-cohomologically smooth local on the source; moreover, composites and base changes of $\ell$-cohomologically smooth maps are $\ell$-cohomologically smooth, and smooth maps of analytic adic spaces over $\Z_p$ give rise to $\ell$-cohomologically smooth morphisms under $X\mapsto X^\diamond$. Assume that $f$ is $\ell$-cohomologically smooth, and $\Lambda$ is $\ell$-power torsion.
\begin{altenumerate}
\item There is a natural equivalence of functors
\[
Rf^! A\cong f^\ast A\dotimes_\Lambda Rf^! \Lambda: D_\et(X,\Lambda)\to D_\et(Y,\Lambda)\ ,
\]
where $Rf^!\Lambda$ is v-locally isomorphic to $\Lambda[n]$ for some $n\in \mathbb Z$.
\item For $A,B\in D_\et(X,\Lambda)$, there is a natural equivalence
\[
f^\ast R\sHom_\Lambda(A,B)\cong R\sHom_\Lambda(f^\ast A,f^\ast B)\ .
\]
\item Let
\[\xymatrix{
Y^\prime\ar[r]^{\tilde{g}}\ar[d]^{f^\prime} &Y\ar[d]^f\\
X^\prime\ar[r]^g & X
}\]
be a cartesian diagram of small v-stacks, with $f$ being $\ell$-cohomologically smooth as before. Then for all $A\in D_\et(X^\prime,\Lambda)$,
\[
f^\ast Rg_\ast A\cong R\tilde{g}_\ast f^{\prime\ast} A\ .
\]
Moreover, for all $A\in D_\et(X,\Lambda)$,
\[
\tilde{g}^\ast Rf^! A\cong Rf^{\prime !} g^\ast A
\]
and
\[
f^{\prime\ast} Rg^! A\cong R\tilde{g}^! f^\ast A\ .
\]
\end{altenumerate}
\end{theorem}

We refer to Sections~\ref{sec:cohomsmooth} and~\ref{sec:geomsmooth} for these results.

Moreover, for $\ell$-cohomologically smooth morphisms, one can prove some finiteness results. For this, we restrict to the case $\Lambda=\Fl$. For any small v-stack $X$, one can define a full subcategory $D_\cons(X,\Fl)\subset D_\et(X,\Fl)$; containment can be checked v-locally, and on affinoid perfectoid spaces, those are bounded complexes that become locally constant and finite-dimensional in each degree over a \emph{constructible} stratification. Note that constructible here refers to the standard notion for spectral spaces, and excludes the stratification of a closed unit disc into a point and its complement (as the open complement is not quasicompact). We note that this condition can be checked on cohomology sheaves. The following theorem is also proved in Section~\ref{sec:cohomsmooth}.

\begin{theorem} Let $f: Y\to X$ be a quasicompact $\ell$-cohomologically smooth map (in particular, compactifiable, representable in locally spatial diamonds, and with $\dimtrg f<\infty$) of small v-stacks. For all $A\in D_\cons(Y,\Fl)$, one has $Rf_! A\in D_\cons(X,\Fl)$.
\end{theorem}

Another finiteness result is the following, cf.~Theorem~\ref{thm:constructiblereflexive}.

\begin{theorem} Let $C$ be a complete algebraically closed nonarchimedean field of characteristic $p$ with ring of integers $\OO_C$, and let $X$ be a locally spatial diamond that is separated and $\ell$-cohomologically smooth over $\Spa(C,\OO_C)$ for some $\ell\neq p$. Let $A\in D_\et(X,\Fl)$ be a bounded complex with constructible cohomology. Then the double (naive) duality map
\[
A\to R\sHom_\Fl(R\sHom_\Fl(A,\Fl),\Fl)
\]
is an equivalence; equivalently, as $Rf^! \Fl$ is invertible, the double Verdier duality map
\[
A\to R\sHom_\Fl(R\sHom_\Fl(A,Rf^! \Fl),Rf^! \Fl)
\]
is an equivalence. Moreover, if $X$ is quasicompact, then $H^i(X,A)$ is finite for all $i\in \mathbb Z$.
\end{theorem}

Finally, it is useful to have some ``invariance under change of algebraically closed base field'' statements. These come in different flavours, depending on whether the base field is discrete or nonarchimedean. The following is Theorem~\ref{thm:changebasefield}.

\begin{theorem} Let $X$ be a small v-stack.
\begin{altenumerate}
\item Assume that $X$ lives over $k$, where $k$ is a discrete algebraically closed field of characteristic $p$, and $k^\prime/k$ is an extension of discrete algebraically closed base fields, $X^\prime = X\times_k k^\prime$. Then the pullback functor
\[
D_\et(X,\Lambda)\to D_\et(X^\prime,\Lambda)
\]
is fully faithful.
\item Assume that $X$ lives over $k$, where $k$ is an algebraically closed discrete field of characteristic $p$. Let $C/k$ be an algebraically closed complete nonarchimedean field, and $X^\prime = X\times_k \Spa(C,C^+)$ for some open and bounded valuation subring $C^+\subset C$ containing $k$. Then the pullback functor
\[
D_\et(X,\Lambda)\to D_\et(X^\prime,\Lambda)
\]
is fully faithful.
\item Assume that $X$ lives over $\Spa(C,C^+)$, where $C$ is an algebraically closed complete nonarchimedean field with an open and bounded valuation subring $C^+\subset C$, $C^\prime/C$ is an extension of algebraically closed complete nonarchimedean fields, and $C^{\prime +}\subset C^\prime$ an open and bounded valuation subring containing $C^+$, such that $\Spa(C^\prime,C^{\prime +})\to \Spa(C,C^+)$ is surjective. Then for $X^\prime = X\times_{\Spa(C,C^+)} \Spa(C^\prime,C^{\prime +})$, the pullback functor
\[
D_\et(X,\Lambda)\to D_\et(X^\prime,\Lambda)
\]
is fully faithful.
\end{altenumerate}
\end{theorem}

Let us say some words about the proofs. We start with the proof of Theorem~\ref{thm:vdescent}. The proof proceeds in steps, proving first the analytic and \'etale cases (already in \cite{ScholzePerfectoidSpaces}, \cite{KedlayaLiu1}), deducing the pro-\'etale case by a limit argument, and then the v-case by making use of the notion of (strictly) totally disconnected spaces.

\begin{definition} Let $X$ be a perfectoid space. Then $X$ is totally disconnected (resp.~strictly totally disconnected) if $X$ is quasicompact, and every open cover of $X$ splits (resp.~and every \'etale cover of $X$ splits).
\end{definition}

It turns out that totally disconnected spaces have a very simple form.

\begin{proposition} Let $X$ be a perfectoid space. Then $X$ is totally disconnected (resp.~strictly totally disconnected) if and only if $X$ is affinoid, and every connected component of $X$ is of the form $\Spa(K,K^+)$, where $K$ is a perfectoid field with an open and bounded valuation subring $K^+\subset K$ (resp.~and $K$ is algebraically closed).
\end{proposition}

Thus, (strictly) totally disconnected are essentially profinite sets of (geometric) points. Their use is justified by the following two lemmas.

\begin{lemma} Let $X$ be any quasicompact perfectoid space. Then there is a pro-\'etale cover $\tilde{X}\to X$ such that $\tilde{X}$ is strictly totally disconnected.
\end{lemma}

\begin{lemma} Let $X=\Spa(R,R^+)$ be a totally disconnected perfectoid space, and $f: Y=\Spa(S,S^+)\to X$ be any map of perfectoid spaces. Then $R^+/\varpi\to S^+/\varpi$ is flat, for any pseudouniformizer $\varpi\in R$.
\end{lemma}

Using the first lemma and pro-\'etale descent, general v-descent reduces to the case of strictly totally disconnected spaces. Then the second lemma reduces this to faithfully flat descent.

Another important use of strictly totally disconnected spaces is a classification result for pro-\'etale maps. This uses the notion of separated morphisms of perfectoid spaces, for which we refer to Definition~\ref{def:separatedperf}.

\begin{lemma} Let $X$ be a strictly totally disconnected perfectoid spaces. Then a quasicompact separated map $f: Y\to X$ of perfectoid spaces is pro-\'etale if and only if for every rank-$1$-point $x=\Spa(C(x),\OO_{C(x)})\in X$, the fibre of $f$ over $x$ is isomorphic to $\underline{S_x}\times \Spa(C(x),\OO_{C(x)})$ for some profinite set $S_x$. In this case, $f$ is affinoid pro-\'etale.
\end{lemma}

Many arguments follow this pattern of reduction to (strictly) totally disconnected spaces, which are essentially just profinite sets. This is akin to covering compact Hausdorff spaces by profinite sets. Making these arguments work is often a question of delicate point-set topology, and we will need to use the properties of spectral spaces very intensely. One application of this is v-descent for separated \'etale maps:

\begin{theorem} The following prestacks are v-stacks.
\begin{altenumerate}
\item The prestack of affinoid perfectoid spaces over the category of totally disconnected perfectoid spaces.
\item The prestack of separated pro-\'etale perfectoid spaces over the category of strictly totally disconnected perfectoid spaces.
\item The prestack of separated \'etale maps over the category of all perfectoid spaces.
\end{altenumerate}
\end{theorem}

Using these descent theorems, the following definition becomes reasonable.

\begin{definition} A locally separated map $f: Y\to X$ of v-stacks is \'etale (resp.~quasi-pro-\'etale) if for any perfectoid space (resp.~any strictly totally disconnected perfectoid space) $Z$ mapping to $X$, the fibre product $Y\times_X Z$ is representable by a perfectoid space, and \'etale (resp.~pro-\'etale) over $Z$.
\end{definition}

This leads to the following characterization of diamonds.

\begin{proposition} A v-sheaf $Y$ is a diamond if and only if there is a surjective quasi-pro-\'etale map $X\to Y$ from a perfectoid space $X$.
\end{proposition}

A positive side of the v-topology is that any small v-stack can be accessed through the following steps: from perfectoid spaces to diamonds by quotients under a pro-\'etale equivalence relation; from diamonds to general small v-sheaves by quotients under a diamond equivalence relation; and from small v-sheaves to small v-stacks by quotients under a small v-sheaf equivalence relation. In the key step, one uses that any sub-v-sheaf of a diamond is automatically itself a diamond; this statement does not seem to have any classical analogue.

However, at some point, we have to prove some theorems, in particular invariance under change of algebraically closed base field, smooth and proper base change, and Poincar\'e duality. Invariance under change of algebraically closed base field is reduced by approximation to the noetherian cases handled by Huber (which in turn are reduced to the case of schemes via nearby cycles). Proper base change is reduced through a series of reductions to a statement about Zariski--Riemann spaces of (algebraically closed) fields, which follows from proper base change for schemes. For smooth base change and Poincar\'e duality, things are reduced by a series of reductions to the case of a closed unit disc, where we can use Huber's results. We note that our reductions reduce all appeals to Huber's book to the case of smooth curves over $\Spa(C,C^+)$, where $C$ is an algebraically closed complete nonarchimedean field, and $C^+\subset C$ is an open and bounded valuation subring; however, in that case, we need a solid theory of Poincar\'e duality.

We make a general warning to the reader that this manuscript contains essentially no examples, which probably makes it almost unreadable. We hope that forthcoming papers, including \cite{FarguesScholze} and the recent paper of Hansen--Kaletha--Weinstein, \cite{HansenKalethaWeinstein}, will give the reader the necessary examples and motivation to work through the long technical arguments of this paper.

{\bf Acknowledgments.} These ideas were formed during a long period after my Berkeley course in 2014. I want to thank Bhargav Bhatt, Dustin Clausen, Laurent Fargues, David Hansen, Kiran Kedlaya, Jared Weinstein and Bogdan Zavyalov for very helpful discussions about various parts. In the winter term 2016/17, the ARGOS seminar in Bonn was going through the manuscript, and I want to thank all participants heartily for their detailed comments. Moreover, I want to thank the organizers of the Hadamard Lectures at the IH\'ES for their invitation in March/April 2017, and in particular Ofer Gabber for his many helpful questions and suggestions. Special thanks go to Yutaro Mikami and the referee for very detailed feedback. Part of this work was done while the author was a Clay Research Fellow.

\section{Spectral Spaces}

We recall some basic properties of spectral spaces that are used throughout.

\begin{definition} A spectral space is a topological space $X$ such that $X$ is quasicompact, $X$ has a basis of quasicompact open subsets stable under finite intersection, and every irreducible closed subset has a unique generic point. A locally spectral space is a topological space $X$ that admits an open cover by spectral subspaces.

A morphism $f: X\to Y$ between spectral spaces is spectral if it is continuous and the preimage of any quasicompact open subset is quasicompact open. A morphism $f: X\to Y$ between locally spectral spaces is spectral if it is continuous and for any spectral open subspace $U\subset X$ mapping into some spectral open $V\subset Y$, the map $U\to V$ is spectral.
\end{definition}

\begin{theorem}[{Hochster, \cite{Hochster}}] Let $X$ be a topological space. The following conditions are equivalent.
\begin{altenumerate}
\item[{\rm (i)}] The space $X$ is spectral.
\item[{\rm (ii)}] There is a ring $A$ such that $X\cong \Spec A$.
\item[{\rm (iii)}] The space $X$ can be written as an inverse limit of finite $T_0$-spaces.
\end{altenumerate}

The category of spectral spaces with spectral maps is equivalent to the pro-category of finite $T_0$-spaces.
\end{theorem}

Recall that a subset $T\subset X$ of a spectral space $X$ is constructible if it lies in the boolean algebra generated by quasicompact open subsets. The constructible topology on $X$ is the topology generated by constructible subsets. Then $X$ equipped with its constructible topology is a profinite set. If $X$ is an inverse limit of finite $T_0$-spaces $X_i$, then $X$ with the constructible topology is the inverse limit of the $X_i$ equipped with the discrete topology.

A subset $T\subset X$ is pro-constructible if it is an intersection of constructible subsets. More generally, a subset $T\subset X$ of a locally spectral space $X$ is constructible (resp.~pro-constructible) if the intersection with any spectral open subspace is constructible (resp.~pro-constructible).

\begin{lemma} Let $X$ be a spectral space. A subset $S\subset X$ is pro-constructible if and only if $S$ is closed in the constructible topology. If $f: X\to Y$ is a spectral map of spectral spaces, then the image of any pro-constructible subset of $X$ is a pro-constructible subset of $Y$.
\end{lemma}

\begin{proof} Constructible subsets are open and closed in the constructible topology, so pro-constructible subsets are closed. Conversely, if a subset is closed in the constructible topology, then it is an intersection of open and closed subsets in the constructible topology. But open and closed subsets of the constructible topology are constructible, for example by writing $X$ as a cofiltered inverse limit of finite $T_0$-space $X_i$ (where any open and closed subset of $X$ for the constructible topology is a preimage of a subset of some $X_i$).

Now let $f: X\to Y$ be spectral map of spectral spaces, and $S\subset X$ a pro-constructible subset with image $T\subset Y$. If we equip $X$ and $Y$ with their constructible topology, then $f$ is a continuous map of compact Hausdorff spaces, and $S\subset X$ is a closed subset, which thus has closed image $T\subset Y$. Thus, $T$ is pro-constructible.
\end{proof}

\begin{lemma} Let $X$ be a spectral space, and $S\subset X$ a pro-constructible subset. Then the closure $\bar{S}\subset X$ of $S$ is the set of specializations of points in $S$.
\end{lemma}

\begin{proof} Clearly, any specialization of a point in $S$ lies in $\bar{S}$. Now let $x\in X$ be a point such that no generalization of $x$ lies in $S$. Let $X_x\subset X$ be the set of generalizations of $x$, which is the intersection of all quasicompact open subsets $U$ containing $x$; in particular, $X_x\subset X$ is a pro-constructible subset. Then
\[
\emptyset = X_x\cap S = \bigcap_{U\ni x} (U\cap S)\ .
\]
In other words, the $(X\setminus U)\cap S$ cover $S$. As $S$ is quasicompact in the constructible topology, and $X\setminus U$ is constructible, it follows that there is some quasicompact open neighborhood $U$ of $x$ such that $(X\setminus U)\cap S=S$, i.e.~$U\cap S=\emptyset$. Thus, $x$ has an open neighborhood which does not meet $S$, so does not lie in the closure of $S$, as desired.
\end{proof}

All maps of analytic adic spaces are generalizing. This notably implies that surjective maps are quotient maps:

\begin{lemma}\label{lem:quotientmap} Let $f: Y\to X$ be a surjective and generalizing spectral map of spectral spaces. Then $f$ is a quotient map.
\end{lemma}

\begin{proof} We have to show a subset $S\subset X$ is open if the preimage $T=f^{-1}(S)\subset Y$ is open, so assume that $T$ is open. Note that in particular $Y\setminus T\subset Y$ is closed, and thus pro-constructible; therefore, its image $X\setminus S\subset X$ is pro-constructible. But as $f$ is generalizing, this subset is also closed under specializations, and thus closed. Thus, $S$ is open, as desired.
\end{proof}

We also recall the following classical lemma.

\begin{lemma}\label{lem:quotmaphausdorff} Let $f: S\to T$ be a continuous surjective map from a quasicompact space $S$ to a compact Hausdorff space $T$. Then $f$ is a quotient map.
\end{lemma}

\begin{proof} Note that for any $x\in T$, the open neighborhoods are cofinal with the closed neighborhoods. Assume that $V\subset T$ is a subset whose preimage $U=f^{-1}(V)\subset S$ is open with closed complement $Z=S\setminus U$, and let $x\in V$. Then
\[
\emptyset = f^{-1}(x)\cap Z= \bigcap_{U_x\subset T} f^{-1}(\overline{U_x})\cap Z\ .
\]
Here, $U_x\subset T$ runs over open neighborhoods of $x$ in $T$. The intersection is over closed subsets of $S$; by quasicompactness, there is some open neighborhood $U_x$ of $x$ such that $f^{-1}(\overline{U_x})\cap Z=\emptyset$, i.e.~$f^{-1}(\overline{U_x})\subset U$. This implies that $f^{-1}(U_x)\subset U$, and thus $U_x\subset V$, as desired.
\end{proof}

Here is a lemma about equivalence relations on spectral spaces. Recall that a topological space is quasiseparated if the intersection of any two quasicompact open subsets is quasicompact.

\begin{lemma}\label{lem:equivrel} Let $X$ be a quasiseparated locally spectral space, and let $R\subset X\times X$ be a pro-constructible equivalence relation such that the maps $s,t: R\to X$ are quasicompact and generalizing. Then the quotient space $X/R$ is $T_0$. Moreover, for any quasicompact open $W\subset X$, there exists an open $R$-invariant subset $U\supset W$ such that $U\subset E$, where $E$ is some $R$-invariant subset that is an intersection of a nonempty family of quasicompact open subsets.
\end{lemma}

\begin{proof} We start with the second part, for which we follow the arguments from~\cite[Tag 0APA, 0APB]{StacksProject}. Let $s,t: R\to X$ denote the two projections. Let $W\subset X$ be a quasicompact open subset. As $s$ is quasicompact, $s^{-1}(W)\subset R$ is quasicompact open, and thus its image $E=t(s^{-1}(W))\subset X$ is quasicompact, pro-constructible, and generalizing. Note that by general nonsense about equivalence relations, $E=t(s^{-1}(W))$ is $R$-invariant. By quasicompactness, there is some quasicompact open subset $W^\prime\subset X$ containing $E$. Now $E$ is a pro-constructible and generalizing subset of the spectral space $W^\prime$; this implies that $E$ is an intersection of quasicompact open subsets.  Let $Z=X\setminus W^\prime$, which is a closed subset of $X$. Then $s^{-1}(Z)\subset R$ is closed, and in particular pro-constructible. As $t$ is quasicompact, it follows that $T=t(s^{-1}(Z))\subset X$ is pro-constructible. Thus, its closure $\bar{T}\subset X$ is the set of specializations of elements of $T$. We claim that $\bar{T}$ is $R$-invariant: Indeed, if $\xi\in \bar{T}$, there is some point $\xi^\prime\in T$ specializing to $\xi$. If $r\in R$ is a point mapping under $s$ to $\xi$, we can find $r^\prime\in R$ mapping under $s$ to $\xi^\prime$, as $s$ is generalizing. Then $t(r)$ is a specialization of $t(r^\prime)\in T$, so $t(r)\in \bar{T}$. Finally, $U=X\setminus \bar{T}$ is an open $R$-invariant subset of $X$, such that $W\subset U\subset W^\prime$. The inclusion $U\subset E$ does not hold with our definition of $E$, but we have $U\subset E^\prime$ where $E^\prime=t(s^{-1}(W^\prime))$.

To see that $X/R$ is $T_0$, start with two distinct points $\bar{x},\bar{y}\in X/R$, and lift them to $x,y\in X$, so that their $R$-orbits are $R\cdot x=t(s^{-1}(x))$ and $R\cdot y=t(s^{-1}(y))$, respectively, and $R\cdot x\cap R\cdot y=\emptyset$. Note that $R\cdot x$ and $R\cdot y$ are both pro-constructible quasicompact subsets of $X$, as $s$ is quasicompact and $t$ is spectral; in particular, $R\cdot x$, $R\cdot y$ and $R\cdot x\cup R\cdot y$ are spectral spaces. We claim that it cannot happen that there are $x_1,x_2\in R\cdot x$ and $y_1,y_2\in R\cdot y$ such that $x_1$ generalizes to $y_1$ and $x_2$ specializes to $y_2$. Indeed, assume this was the case. As $R\cdot x\cup R\cdot y\subset X$ is a pro-constructible subset, it is a spectral space, and thus admits a maximal point; possibly changing the roles of $x$ and $y$, we may assume that $y$ is a maximal point of $R\cdot x\cup R\cdot y$. As $t: R\to X$ is generalizing, $(y,y_2)\in R$, and $y_2$ generalizes to $x_2$, we can find some point $(z,x_2)\in R$, where $z$ is a generalization of $y$. Then $z\in R\cdot x\subset R\cdot x\cup R\cdot y$. As $y$ is maximal in that space, it follows that $z=y$, but then $(y,x_2)\in R$, so that $R\cdot y = R\cdot x_2=R\cdot x$, which is a contradiction.

Thus, we may assume that no point of $R\cdot x$ generalizes to a point of $R\cdot y$. In that case, we can find a quasicompact open subset $W\subset X$ containing $x$ such that $W\cap R\cdot y=\emptyset$: Indeed, the intersection of $W\cap R\cdot y$ over all quasicompact open neighborhoods $W$ of $x$ is empty, as $R\cdot y$ contains no generalization of $x$. But then by quasicompactness $R\cdot y$ for the constructible topology, $W\cap R\cdot y=\emptyset$ for some such $W$. Now we run the above argument to find some open $R$-invariant subset $U\supset W$ of $X$. We claim that we can do this so that $U\cap R\cdot y=\emptyset$: This would finish the proof, as $R\cdot x\subset U$, while $U\cap R\cdot y=\emptyset$, so that the image of $U$ in $X/R$ separates $x$ from $y$. To see that one can arrange $U\cap R\cdot y =\emptyset$, note that in the notation of the first paragraph, $E\cap R\cdot y=\emptyset$. As $E$ is generalizing, one can for any point $e\in E$ find a quasicompact open neighborhood $W^\prime_e$ of $e$ such that $W^\prime_e\cap R\cdot y=\emptyset$; in other words, by quasicompactness of $E$ we can choose $W^\prime$ so that $W^\prime\cap R\cdot y=\emptyset$. In that case, also $E^\prime\cap R\cdot y=\emptyset$, so as we have $U\subset E^\prime$ for the $R$-invariant open subset $U\subset X$, we see that $U\cap R\cdot y=\emptyset$, as desired.
\end{proof}

\begin{remark} Even if $X$ is spectral, it can happen under the hypothesis of the lemma that the quotient space $X/R$ is not spectral. In fact, using just profinite sets $X$ and $R$, one can get any compact Hausdorff space as the quotient $X/R$: Namely, given a compact Hausdorff space $Y$, let $X$ be the Stone-Cech compactification of $Y$ considered as a discrete set. Then $X$ is a profinite set which comes with a natural continuous surjective map to $Y$, and the induced equivalence relation $R\subset X\times X$ is closed, and thus profinite.
\end{remark}

\begin{lemma}\label{lem:spectralquotient} Let $X$ be a spectral space, and let $R\subset X\times X$ be a pro-constructible equivalence relation such that the maps $s,t: R\to X$ are generalizing. Assume that $X/R$ has a basis for the topology given by open subsets whose preimages in $X$ are quasicompact. Then $X/R$ is a spectral space, and $X\to X/R$ is a spectral and generalizing map.
\end{lemma}

\begin{proof} The set of open subsets of $X/R$ whose preimages in $X$ are quasicompact is stable under finite intersections. To see that $X/R$ is spectral, it remains to see that every irreducible closed subset $Z\subset X/R$ has a unique generic point. Uniqueness holds because $X/R$ is $T_0$ by Lemma~\ref{lem:equivrel}. For existence, look at the cofiltered intersection of all open subsets $U\subset X/R$ with $U\cap Z\neq\emptyset$ whose preimage in $X$ is quasicompact. We need to see that this intersection is nonempty. But the preimages of such $U$ are quasicompact open subsets, thus compact Hausdorff for the constructible topology. As they are all nonempty, their intersection is nonempty. Thus, the intersection of the $U$'s is nonempty as well, as desired.

As $X/R$ has a basis of opens for which the preimage in $X$ is quasicompact, it follows that $X\to X/R$ is spectral. It remains to see that $X\to X/R$ is generalizing. Assume that $\bar{x}$ generalizes to $\bar{x}^\prime\in X/R$. We claim first that there is some way to lift this to a generalization $x$ to $x^\prime$ in $X$. If not, then the final paragraph of the proof of Lemma~\ref{lem:equivrel} would produce an open $R$-invariant neighborhood of $x$ in $X$ which does not contain $x^\prime$, contradicting that $\bar{x}$ generalizes to $\bar{x}^\prime$. Now if $x_0\in X$ is any lift of $\bar{x}$, then $(x_0,x)\in R$. But $t: R\to X$ is generalizing, and $x^\prime$ is a generalization of $x$ in $X$, so we can find a generalization $(x_0^\prime,x^\prime)\in R$ of $(x_0,x)$. But then $x_0^\prime\in X$ is a generalization of $x_0$ mapping to $\bar{x}^\prime\in X/R$, as desired.
\end{proof}

\begin{lemma}\label{lem:spectralopenequivrel} In the situation of Lemma~\ref{lem:equivrel}, assume that $R\to X$ is open. Then the quotient space $X/R$ is locally spectral and quasiseparated, and $X\to X/R$ is an open spectral qcqs map.
\end{lemma}

\begin{proof} For any quasicompact open subset $U\subset X$, the set $t(s^{-1}(U))\subset X$ is a quasicompact $R$-invariant open subset under the assumptions. Note that $t(s^{-1}(U))\subset X$ is also the preimage of the image of $U$ in $X/R$, so that $X\to X/R$ is open. For the other assertions, replacing $X$ by $U$, we can assume that $X$ is quasicompact, thus spectral. We claim that the quasicompact open subsets $V_0\subset X/R$ whose preimages in $X$ are quasicompact, form a basis for the topology. Thus, let $W_0\subset X/R$ be open, with preimage $W\subset X$ an $R$-invariant open subset. Pick any $x_0\in W_0$, and lift it to $x\in W$. Then $x$ has a quasicompact open neighborhood $x\in V^\prime\subset W$, and $V=t(s^{-1}(V^\prime))\subset W$ is a quasicompact $R$-invariant open subset. Its image $V_0\subset W_0$ thus contains $x_0$, is quasicompact, with inverse image $V\subset X$ still quasicompact.

It remains to see that every irreducible closed subset $\bar{Z}\subset X/R$ has a generic point. Equivalently, the intersection of all quasicompact open subsets $V_0\subset X/R$ with $V_0\cap \bar{Z}\neq\emptyset$ is nonempty. But the corresponding intersection of the preimages in $X$ is nonempty by compactness of the constructible topology, so the result follows.
\end{proof}

We will need the following result about inverse limits.

\begin{lemma}\label{lem:spectralinvlimit} Let $X_i$, $i\in I$, be a cofiltered inverse system of spectral spaces along spectral maps, with inverse limit $X=\varprojlim_i X_i$. Then $X$ is a spectral space, the maps $X\to X_i$ are spectral, and a map $Y\to X$ from a spectral space is spectral if and only if all composites $Y\to X\to X_i$ are spectral. Moreover, if $Y$ is a spectral space with a spectral map $Y\to X$ such that the composite maps $Y\to X_i$ are generalizing, then $Y\to X$ is generalizing. In particular, if $Y\to X$ is in addition surjective, then by Lemma~\ref{lem:quotientmap} it is a quotient map.
\end{lemma}

\begin{proof} The first part follows for example from the interpretation of the category of spectral spaces with spectral maps as the pro-category of finite $T_0$-spaces. Now assume $Y\to X$ is a spectral map of spectral spaces such that $Y\to X_i$ is generalizing for all $i\in I$. Let $y\in Y$ map to $x\in X$, and let $\tilde{x}\in X$ be a generalization of $x\in X$. We look for a generalization $\tilde{y}$ of $Y$ mapping to $\tilde{x}$. For this, we may replace $Y$ by its localization $Y_y$ at $y$ (i.e.~the set of generalizations of $y$). Let $x_i\in X_i$ be the image of $x$, and $\tilde{x}_i\in X_i$ be the image of $\tilde{x}$. Then, for all $i\in I$, $\tilde{x}_i$ is a generalization of $x_i$. As $Y\to X_i$ is generalizing, the preimage of $\tilde{x}_i$ in $Y$ is a non-empty pro-constructible subset of $Y$. For varying $i$, these preimages form a cofiltered inverse system of nonempty pro-constructible subsets of $Y$; their intersection is thus nonempty (by Tychonoff). Any point in the intersection is a generalization of $y$ mapping to $\tilde{x}$, as desired.
\end{proof}

\section{Perfectoid Spaces}

Let $p$ be a fixed prime throughout. Recall that a topological ring $R$ is Tate if it contains and open and bounded subring $R_0\subset R$ and a topologically nilpotent unit $\varpi\in R$; such elements are called pseudo-uniformizers. Let $R^\circ\subset R$ denote the set of powerbounded elements. If $\varpi\in R$ is a pseudo-uniformizer, then necessarily $\varpi\in R^\circ$.  Furthermore, if $R^+\subset R$ is a ring of integral elements, i.e.~an open and integrally closed subring of $R^\circ$, then $\varpi\in R^+$.  Indeed, since $\varpi^n\to 0$ as $n\to\infty$ and since $R^+$ is open, we have $\varpi^n\in R^+$ for some $n\geq 1$.  Since $R^+$ is integrally closed, $\varpi\in R^+$.

The following definition is due to Fontaine, \cite{FontaineBourbaki}.

\begin{definition}\label{def:perfectoidring} A Tate ring $R$ is {\em perfectoid} if $R$ is complete, uniform, i.e.~$R^\circ\subset R$ is bounded, and there exists a pseudo-uniformizer $\varpi\in R$ such that $\varpi^p|p$ in $R^\circ$ and the Frobenius map 
\[
\Phi\from R^\circ/\varpi \to R^\circ/\varpi^p: x\mapsto x^p
\]
is an isomorphism. 
\end{definition}

Hereafter we use the following notational convention. If $R$ is a ring, and $I,J\subset R$ are ideals containing $p$ such that $I^p\subset J$, then $\Phi\from R/I\to R/J$ will refer to the ring homomorphism $x\mapsto x^p$.  

\begin{remark}\label{rem:defperfectoid} The condition that $\Phi: R^\circ/\varpi\to R^\circ/\varpi^p$ is an isomorphism is independent of the choice of a pseudo-uniformizer $\varpi\in R$ such that $\varpi^p|p$. In fact, we claim that it is equivalent to the condition that $\Phi: R^\circ/p\to R^\circ/p$ is surjective.

Indeed, for any complete Tate ring $R$ and pseudo-uniformizer $\varpi$ satisfying $\varpi^p|p$ in $R^\circ$, the Frobenius map $\Phi\from R^\circ/\varpi\to R^\circ/\varpi^p$ is necessarily injective. Indeed, if $x\in R^\circ$ satisfies $x^p=\varpi^p y$ for some $y\in R^\circ$ then the element $x/\varpi \in R$ lies in $R^\circ$ since its $p$-th power does.  Thus, the isomorphism condition on $\Phi$ in Definition~\ref{def:perfectoidring} is equivalent to surjectivity of $\Phi: R^\circ/\varpi\to R^\circ/\varpi^p$. Thus, if $\Phi: R^\circ/p\to R^\circ/p$ is surjective, then $R$ is perfectoid. Conversely, assume that $R$ is perfectoid. Let $x\in R^\circ$ be any element. By using that $\Phi: R^\circ/\varpi\to R^\circ/\varpi^p$ is surjective and successive $\varpi^p$-adic approximation, we can write
\[
x=x_0^p + \varpi^p x_1^p + \varpi^{2p} x_2^p+\ldots\ ,
\]
where $x_i\in R^\circ$. But then
\[
x\equiv (x_0+\varpi x_1 + \varpi^2 x_2 + \ldots)^p\mod p\ ,
\]
showing that $\Phi: R^\circ/p\to R^\circ/p$ is surjective, as desired.
\end{remark}

\begin{remark} The field $\Q_p$ is not perfectoid, even though the Frobenius map on $\F_p$ is an isomorphism. The issue is that there is no element $\varpi\in\Z_p$ whose $p$-th power divides $p$. More generally, a discretely valued non-archimedean field $K$ cannot be perfectoid. Indeed, if $\varpi$ is a pseudo-uniformizer as in Definition~\ref{def:perfectoidring}, then $\varpi$ is a non-zero element of the maximal ideal, so the quotients $K^\circ/\varpi$ and $K^\circ/\varpi^p$ are Artin local rings of different lengths and hence cannot be isomorphic.
\end{remark}

\begin{example}\label{ex:perfectoidrings}
\begin{altenumerate}
\item The cyclotomic field $\Q_p^{\text{cycl}}$, the completion of the cyclotomic extension $\Q_p(\mu_{p^\infty})$.
\item The $t$-adic completion of $\F_p\laurentseries{t}(t^{1/p^\infty})$, which we will write as $\F_p\laurentseries{t^{1/p^\infty}}$.
\item The algebra $\Q_p^{\text{cycl}}\tatealgebra{T^{1/p^\infty}}$.  This is defined as $A[1/p]$, where $A$ is the $p$-adic completion of $\Z_p^{\text{cycl}}[T^{1/p^\infty}]$.
\item As an example of a perfectoid Tate ring which does not live over a field, one may take
\[
R=\Z_p^{\text{cycl}}[[T^{1/p^\infty}]]\tatealgebra{(p/T)^{1/p^\infty}}[1/T]\ .
\]
Here we can take $\varpi=T^{1/p}$, because $\varpi^p=T$ divides $p$ in $R^\circ$.
\end{altenumerate}
\end{example}

\begin{proposition}\label{prop:perfectoidcharp} Let $R$ be a topological ring with $pR=0$.  The following are equivalent:
\begin{enumerate}
\item The topological ring $R$ is a perfectoid Tate ring.
\item The topological ring $R$ is a perfect complete Tate ring.
\end{enumerate}
\end{proposition}

Of course, perfect means that $\Phi\from R\to R$ is an isomorphism.  

\begin{proof} Let $R$ be a complete Tate ring.  If $R$ is perfect, then take $\varpi$ any pseudo-uniformizer.  The condition 
$\varpi^p|p=0$ is vacuous.  If $x\in R$ is powerbounded, then so is $x^p$, and vice versa, which means that $\Phi\from R^\circ\to R^\circ$ is an isomorphism.  This shows that $\Phi\from R^\circ/\varpi \to R^\circ/\varpi^p$ is surjective, and we had seen in Remark~\ref{rem:defperfectoid} that injectivity is automatic. It remains to see that $R$ is automatically uniform. For this, take any open and bounded subring $R_0\subset R$. As $\Phi: R\to R$ is an isomorphism, which is automatically open by Banach's open mapping theorem, there is some $n$ such that $\varpi^n R_0\subset \Phi(R_0)$. Equivalently, $\Phi^{-1}(R_0)\subset \varpi^{-n/p} R_0$. This implies that
\[
\Phi^{-2}(R_0)\subset \Phi^{-1}(\varpi^{-n/p} R_0)\subset \varpi^{-n/p-n/p^2} R_0
\]
and inductively
\[
\Phi^{-k}(R_0)\subset \varpi^{-n/p-n/p^2-\ldots-n/p^k} R_0\ .
\]
In particular, for all $k\geq 0$,
\[
\Phi^{-k}(R_0)\subset \varpi^{-n} R_0\ .
\]
On the other hand, $R^\circ\subset \bigcup_{k\geq 0} \Phi^{-k}(R_0)$, and so $R^\circ\subset \varpi^{-n} R_0$ is bounded, as desired.

Conversely, if $R$ is perfectoid, then $\Phi\from R^\circ/\varpi\to R^\circ/\varpi^p$ is an isomorphism, and therefore so is $R^\circ/\varpi^n\to R^\circ/\varpi^{np}$ by induction.  Taking inverse limits and using completeness, we find that $\Phi\from R^\circ\to R^\circ$ is an isomorphism. Inverting $\varpi$ shows that $R$ is perfect.
\end{proof}

\begin{definition}\label{def:perfectoidfield} A {\em perfectoid field} is a perfectoid Tate ring $R$ which is a nonarchimedean field.
\end{definition}

\begin{remark} It is not clear a priori that a perfectoid ring which is a field is a perfectoid field. However, this has recently been answered affirmatively by Kedlaya, \cite{KedlayaFields}.
\end{remark}

One easily checks the following equivalent characterization, which is the original definition, \cite{ScholzePerfectoidSpaces}, \cite{KedlayaLiu1}.

\begin{proposition} Let $K$ be a nonarchimedean field. Then $K$ is a perfectoid field if and only if the following conditions hold:
\begin{altenumerate}
\item The nonarchimedean field $K$ is not discretely valued,
\item the absolute value $\abs{p}<1$, and
\item the Frobenius $\Phi\from\OO_K/p\to\OO_K/p$ is surjective.
\end{altenumerate}$\hfill \Box$
\end{proposition}

Next, we recall the process of tilting.

\begin{definition}\label{def:tilting} Let $R$ be a perfectoid Tate ring.  The {\em tilt} of $R$ is the topological ring
\[
R^{\flat}=\varprojlim_{x\mapsto x^p} R\ ,
\]
with the inverse limit topology, the pointwise multiplication, and the addition given by
\[
(x^{(0)},x^{(1)},\dots)+(y^{(0)},y^{(1)},\dots)=(z^{(0)},z^{(1)},\dots)
\]
where 
\[
z^{(i)} = \lim_{n\to \infty} (x^{(i+n)}+y^{(i+n)})^{p^n} \in R\ .
\]
\end{definition}

\begin{lemma}\label{lem:tiltingexists} The limit $z^{(i)}$ above exists and defines a ring structure making $R^\flat$ a perfectoid $\F_p$-algebra.  The subset $R^{\flat\circ}$ of power-bounded elements is given by the topological ring isomorphism 
\[
R^{\flat\circ}=\varprojlim_{x\mapsto x^p} R^\circ\isom \varprojlim_{\Phi} R^\circ/\varpi\ ,
\]
where $\varpi\in R$ is a pseudo-uniformizer which divides $p$ in $R^\circ$.

Furthermore, there exists a pseudo-uniformizer $\varpi\in R$ with $\varpi^p\vert p$ in $R^\circ$ which admits a sequence of $p$-power roots $\varpi^{1/p^n}$, giving rise to an element $\varpi^\flat=(\varpi,\varpi^{1/p},\dots)\in R^{\flat\circ}$, which is a pseudo-uniformizer of $R^\flat$. Then $R^\flat=R^{\flat\circ}[1/\varpi^\flat]$.
\end{lemma}

\begin{proof} Let $\varpi_0$ be a pseudo-uniformizer of $R$ such that $\varpi_0|p$ in $R^\circ$. Let us check that the map
\[
\varprojlim_{x\mapsto x^p} R^\circ \to \varprojlim_\Phi R^\circ/\varpi_0
\]
is an isomorphism. We have to see that any sequence $(\overline{x}_0,\overline{x}_1,\dots)\in \varprojlim_\Phi R^{\circ}/\varpi_0$ lifts uniquely to a sequence
$(x_0,x_1,\dots)\in \varprojlim_\Phi R^\circ$. This lift is given by $x^{(i)}=\lim_{n\to\infty} x_{n+i}^{p^n}$, where $x_j\in R^\circ$ is any lift of $\overline{x}_j$.  (For the convergence of that limit, note that if $x\equiv y\pmod{\varpi_0^n}$, then $x^p\equiv y^p\pmod{\varpi_0^{n+1}}$.) This shows that we get a well-defined ring
\[
\varprojlim_{x\mapsto x^p} R^\circ\subset \varprojlim_{x\mapsto x^p} R = R^\flat\ ,
\]
which agrees with the powerbounded elements $R^{\flat\circ}\subset R^\flat$.

Next, we construct the element $\varpi^\flat$. For this, we assume that $\varpi_0=\varpi_1^p$ for some pseudouniformizer $\varpi_1\in R$ such that $\varpi_0=\varpi_1^p|p$.  Any preimage of $\varpi_1$ under $R^{\flat\circ}=\varprojlim_\Phi R^{\circ}/\varpi_0\to R^{\circ}/\varpi_0$ is an element $\varpi^\flat$ with the right properties. It is congruent to $\varpi_1$ modulo $\varpi_0$, and therefore it is also a topologically nilpotent, and invertible in $R^\flat$. Then $\varpi=\varpi^{\flat\sharp}$ is the desired pseudo-uniformizer of $R^\circ$.

Now one sees that $R^\flat = R^{\flat\circ}[1/\varpi^\flat]$ as multiplicative monoids, and in fact as rings. This implies that the addition in $R^\flat$ is always well-defined, and defines a ring of characteristic $p$, which is perfect by design. The rest of the lemma follows easily.
\end{proof}

We have a continuous, multiplicative (but not additive) map $R^\flat= \varprojlim_{x\mapsto x^p} R\to R$ by projecting onto the zeroth coordinate; call this $f\mapsto f^\sharp$.  This projection defines a ring isomorphism $R^{\flat\circ}/\varpi^\flat\isom R^\circ/\varpi$.  The open and integrally closed subrings of $R^{\flat\circ}$ and $R^\circ$ correspond exactly to the integrally closed subrings of their common quotients modulo $\varpi^\flat$ and $\varpi$. This defines an inclusion-preserving bijection between the sets of open and integrally closed subrings of $R^{\flat\circ}$ and $R^\circ$.  This correspondence can be made more explicit:

\begin{lemma}\label{RingsOfIntegralElements} The set of open and integrally closed subrings $R^+\subset R^\circ$ is in bijection with the set of open and integrally closed subrings $R^{\flat+}\subset R^{\flat\circ}$, via $R^{\flat+}=\varprojlim_{x\mapsto x^p} R^+$.  Also, $R^{\flat+}/\varpi^\flat=R^+/\varpi$.$\hfill \Box$
\end{lemma}

The following two theorems belong to a pattern of ``tilting equivalence''.  

\begin{theorem}[\cite{ScholzePerfectoidSpaces}, \cite{KedlayaLiu1}]\label{thm:tiltingtopspace} Let $R$ be a perfectoid Tate ring with an open and integrally closed subring $R^+\subset R$, with tilt $(R^\flat,R^{\flat+})$.

The map sending $x\in \Spa(R,R^+)$ to $x^\flat\in \Spa(R^\flat,R^{\flat +})$ defined by $|f(x^\flat)|=|f^\sharp(x)|$ defines a homeomorphism $\Spa(R,R^+)\isom \Spa(R^\flat,R^{\flat+})$. A subset $U\subset \Spa(R,R^+)$ is rational if and only if its image in $\Spa(R^\flat,R^{\flat +})$ is rational.
\end{theorem}

\begin{theorem}[\cite{ScholzePerfectoidSpaces}] \label{thm:tiltingequivalg} Let $R$ be a perfectoid Tate ring with tilt $R^\flat$. Then there is an equivalence of categories between perfectoid $R$-algebras and perfectoid $R^\flat$-algebras, via $S\mapsto S^\flat$.
\end{theorem}

\begin{proof} In \cite{ScholzePerfectoidSpaces}, these are only proved over a perfectoid field, but the proof works in general.
\end{proof}

Let us describe the inverse functor, along the lines of Fontaine's Bourbaki talk, \cite{FontaineBourbaki}.  In fact we will answer a more general question. Given a perfectoid Tate ring $R$ in characteristic $p$, what are all the untilts $R^\sharp$ of $R$?

\begin{lemma}\label{lem:fontainestheta} Let $R$ be a perfectoid Tate ring with a ring of integral elements $R^+\subset R$, and let $(R^\flat,R^{\flat +})$ be its tilt.
\begin{altenumerate}
\item There is a canonical surjective ring homomorphism 
\[\begin{aligned}
\theta: W(R^{\flat +})&\to R^+ \\
\sum_{n\geq 0} [r_n] p^n&\mapsto \sum_{n\geq 0}r_n^\sharp p^n 
\end{aligned}\]
\item The kernel of $\theta$ is generated by a nonzero-divisor $\xi$ of the form $\xi=p+[\varpi]\alpha$, where $\varpi\in R^{\flat +}$ is a pseudo-uniformizer, and $\alpha\in W(R^{\flat +})$. 
\end{altenumerate}
\end{lemma}

Related results are discussed in \cite{FontaineBourbaki}, \cite{FarguesFontaine} and \cite[Theorem 3.6.5]{KedlayaLiu1}.

\begin{definition} An ideal $I\subset W(R^{\flat +})$ is {\em primitive of degree 1} if $I$ is generated by an element of the form $\xi=p+[\varpi^\flat]\alpha$, with $\varpi^\flat\in R^{\flat +}$ a pseudo-uniformizer and $\alpha\in W(R^{\flat +})$.
\end{definition}

An element $\xi$ of this form is necessarily a non-zero-divisor:

\begin{lemma}  Any $\xi$ of the form $\xi=p+[\varpi^\flat]\alpha$, with $\varpi^\flat\in R^{\flat +}$ a pseudo-uniformizer and $\alpha\in W(R^+)$, is a non-zero-divisor.  
\end{lemma}

\begin{proof} Assume that $\xi\sum_{n\geq 0} [c_n]p^n=0$. Modulo $[\varpi^\flat]$, this reads $\sum_{n\geq 0}[c_n]p^{n+1}\equiv 0\pmod{[\varpi^\flat]}$, meaning that all $c_n\equiv 0\pmod{\varpi^\flat}$. We can then divide all $c_n$ by $\varpi^\flat$, and induct.
\end{proof}

\begin{proof}[Proof of Lemma~\ref{lem:fontainestheta}.]
For part (i), we first check that $\theta$ is a ring map. It is enough to check modulo $\varpi^m$ for any $m\geq 1$. For this, we use that the $m$-th ghost map
\[
W(R^+)\to R^+/\varpi^m\ :\ (x_0,x_1,\ldots)\mapsto \sum_{n=0}^m x_n^{p^{m-n}} p^n
\]
factors uniquely over $W(R^+/\varpi)$, by obvious congruences; the induced map $W(R^+/\varpi)\to R^+/\varpi^m$ must be a ring homomorphism. Now the composite
\[
W(R^{\flat +})\to W(R^+/\varpi)\to R^+/\varpi^m\ ,
\]
where the first map is given by the $m$-th component map $R^{\flat +}=\varprojlim_{x\mapsto x^p} R^+/\varpi\to R^+/\varpi$, is a ring map, which we claim is equal to $\theta$ modulo $\varpi^m$. This is a direct verification from the definitions.

For surjectivity of $\theta$, choose a pseudouniformizer $\varpi^\flat$ of $R^\flat$ such that the $p$-th power of $\varpi=(\varpi^\flat)^\sharp$ divides $p$. We know that $R^{\flat +}\to R^+/\varpi$ is surjective, which shows that $\theta$ mod $[\varpi^\flat]$ is surjective. As everything is $[\varpi^\flat]$-adically complete, this implies that $\theta$ is surjective.

For part (ii), we claim first that there exists $f\in \varpi^\flat R^{\flat +}$ such that $f^\sharp\equiv p\pmod{p\varpi R^+}$. Indeed, consider $\alpha=p/\varpi\in R^+$. There exists $\beta\in R^{\flat +}$ such that $\beta^\sharp\equiv \alpha\pmod{pR^+}$.  Then $(\varpi^\flat \beta)^\sharp=\varpi \alpha\equiv p\pmod{p\varpi R^+}$, and we can take $f=\varpi^\flat\beta$.

Thus we can write $p=f^\sharp+p\varpi^\sharp\sum_{n\geq 0}r_n^\sharp p^n$, with $r_n\in R^{\flat +}$. We can now define $\xi=p-[f]-[\varpi^\flat]\sum_{n\geq 0} [r_n]p^{n+1}$, which is of the desired form, and which lies in the kernel of $\theta$. Finally we need to show that $\xi$ generates $\ker(\theta)$. For this, note that $\theta$ induces a surjective map $f: W(R^{\flat +})/\xi\to R^+$. It is enough to show that $f$ is an isomorphism modulo $[\varpi^\flat]$, because $W(R^{\flat +})/\xi$ is $[\varpi^\flat]$-torsion free and $[\varpi^\flat]$-adically complete. But
\[
W(R^{\flat +})/(\xi,[\varpi^\flat])=W(R^{\flat +})/(p,[\varpi^\flat])=R^{\flat +}/\varpi^\flat=R^+/\varpi\ ,
\]
as desired.
\end{proof}

It is now straightforward to check the following theorem.

\begin{theorem}[\cite{KedlayaLiu1}, \cite{FontaineBourbaki}] There is an equivalence of categories between:
\begin{altenumerate}
\item Pairs $(S,S^+)$ of a perfectoid Tate ring $S$ and an open and integrally closed subring $S^+\subset S^\circ$, and
\item Triples $(R,R^+,\mathcal{J})$, where $R$ is a perfectoid Tate ring of characteristic $p$, $R^+\subset R^\circ$ is an open and integrally closed subring, and $\mathcal{J}\subset W(R^+)$ is an ideal which is primitive of degree $1$.
\end{altenumerate}
The functors are given by $(S,S^+)\mapsto (S^\flat,S^{\flat+},\ker\theta)$ and $(R,R^+,\mathcal{J})\mapsto (W(R^+)[[\varpi]^{-1}]/\mathcal{J},W(R^+)/\mathcal{J})$, where $\varpi\in R$ is any pseudouniformizer.$\hfill \Box$
\end{theorem}

Now we can define perfectoid spaces.

\begin{theorem}[{\cite[Thm. 6.3]{ScholzePerfectoidSpaces}},{\cite[Thm. 3.6.14]{KedlayaLiu1}}] \label{thm:perfectoidringssheafy} Let $R$ be a perfectoid Tate ring with an open and integrally closed subring $R^+\subset R^\circ$, and let $X=\Spa(R,R^+)$. Then $\OO_X$ is a sheaf, and for all rational subsets $U\subset X=\Spa(R,R^+)$, $\OO_X(U)$ is again perfectoid.

If $(R^\flat,R^{\flat +})$ is the tilt of $(R,R^+)$ and $X^\flat=\Spa(R^\flat,R^{\flat +})$, then recall that by Theorem~\ref{thm:tiltingtopspace}, $|X|\cong |X^\flat|$, identifying rational subsets. Moreover, for any rational subset $U\subset X$ with image $U^\flat\subset X^\flat$, $\OO_X(U)$ is perfectoid with tilt $\OO_{X^\flat}(U^\flat)$.
\end{theorem}

\begin{definition}\label{def:perfectoidspace} A {\em perfectoid space} is an adic space covered by open subspaces which are isomorphic to $\Spa(R,R^+)$, where $R$ is a perfectoid Tate ring, and $R^+\subset R^\circ$ is an open and integrally closed subring.
\end{definition}

The tilting process glues to give a functor $X\mapsto X^\flat$. Theorem~\ref{thm:tiltingequivalg} globalizes to the following result.

\begin{corollary}\label{cor:tiltingequiv} Let $X$ be a perfectoid space with tilt $X^\flat$. Then the functor $X^\prime\mapsto X^{\prime\flat}$ from perfectoid spaces over $X$ to perfectoid spaces over $X^\flat$ is an equivalence of categories.
\end{corollary}

Finally, let us recall some almost mathematics. Let $R$ be a perfectoid Tate ring.  

\begin{definition} An $R^\circ$-module $M$ is {\em almost zero} if $\varpi M=0$ for all pseudo-uniformizers $\varpi$.  Equivalently, if $\varpi$ is a fixed pseudo-uniformizer admitting $p$-power roots, then $M$ is almost zero if and only if $\varpi^{1/p^n}M=0$ for all $n$.
\end{definition}

A similar definition applies to $R^+$-modules; in fact, the condition on $M$ only uses the $R^+$-module structure.

\begin{example}
\begin{altenumerate}
\item If $K$ is a perfectoid field, then the residue field $k=\OO_K/\mathfrak{m}_K$ is almost zero as $\OO_K$-module, where $\mathfrak m_K\subset \OO_K$ is the maximal ideal. Conversely, any almost zero module over $\OO_K$ is a $k$-vector space, and thus a direct sum of copies of $k$.
\item If $R$ is perfectoid and $R^+\subset R^\circ$ is any ring of integral elements, then $R^\circ/R^+$ is almost zero.  Indeed, if $\varpi$ is a pseudo-uniformizer, and $x\in R^\circ$, then $\varpi x$ is topologically nilpotent.  Since $R^+$ is open, there exists $n$ with $(\varpi x)^n\in R^+$, so that $\varpi x \in R^+$ by integral closedness.  
\end{altenumerate}
\end{example}

Extensions of almost zero modules are almost zero. Thus the category of almost zero modules is a thick Serre subcategory of the category of all modules, and one can take the quotient.

\begin{definition} The category of {\em almost $R^\circ$-modules}, written $R^{\circ a}\text{-mod}$, is the quotient of the category of $R^\circ$-modules by the subcategory of almost zero modules.
\end{definition}

One can also define $R^{+a}\text{-mod}$, and the natural forgetful map $R^{\circ a}\text{-mod}\to R^{+a}\text{-mod}$ is an equivalence.

\begin{theorem}[\cite{ScholzePerfectoidSpaces}, \cite{KedlayaLiu1}] Let $R$ be a perfectoid Tate ring with an open and integrally closed subring $R^+\subset R^\circ$, and let $X=\Spa(R,R^+)$. Then the $R^+$-module $H^i(X,\OO_X^+)$ is almost zero for $i>0$, and $H^0(X,\OO_X^+)=R^+$.
\end{theorem}

\section{Set-theoretic bounds}

In this paper, we will have to deal with many ``large'' constructions. However, we want to avoid set-theoretic issues, and in particular want to avoid the use of universes. For this reason, we will sometimes consider cutoff cardinals, but in the end take a limit over all possible cutoffs.

\begin{lemma}\label{lem:choosekappa} There is a cofinal class of uncountable cardinals $\kappa$ with the following properties.
\begin{altenumerate}
\item For all cardinals $\lambda<\kappa$, one has $2^\lambda<\kappa$, i.e.~$\kappa$ is a strong limit cardinal.
\item For all countable sequences $\lambda_1,\lambda_2,\ldots<\kappa$ of cardinals less than $\kappa$, the supremum of all $\lambda_i$ is less than $\kappa$, i.e.~the cofinality of $\kappa$ is larger than $\omega$.
\item For all cardinals $\lambda<\kappa$, there is a strong limit cardinal $\kappa_\lambda<\kappa$ such that the cofinality of $\kappa_\lambda$ is larger than $\lambda$.
\end{altenumerate}
\end{lemma}

Note that in (iii), the set of such $\kappa_\lambda$ is automatically cofinal in $\kappa$ (as for any $\lambda\leq \lambda^\prime<\kappa$, one has $\kappa_{\lambda^\prime}\geq \lambda^\prime$, and $\kappa_{\lambda^\prime}$ has cofinality larger than $\lambda$). Intuitively, cutting off at $\kappa$ means the following. We are always allowed to take power sets, and countable unions. Moreover, we are allowed to do any construction that takes less than $\kappa$ many steps, as long as the cardinality of the objects we are manipulating is uniformly bounded above by a cardinal strictly less than $\kappa$. Indeed, if for example we want to take a union of $\lambda<\kappa$ many objects of size less than $\kappa^\prime<\kappa$, then we can assume that $\kappa^\prime=\kappa_\lambda$ is of cofinality larger than $\lambda$, in which case the union is still of size less than $\kappa_\lambda<\kappa$. Moreover, if at each of the $\lambda$ intermediate steps we are doing some auxiliary large construction that however respects sets that are less than $\kappa_\lambda$, then this is still admissible.

\begin{proof} First, we show that for any cardinal $\lambda$, one can find a strong limit cardinal $\kappa$ such that the cofinality of $\kappa$ is larger than $\lambda$. For this, one starts with the countable cardinal $\beth_0=\aleph_0$ and then defines by transfinite induction $\beth_\mu$ for any ordinal $\mu$ as the power set of $\beth_{\mu^-}$ if $\mu$ is the successor of $\mu^-$, or the union of all previous $\beth_{\mu^\prime}$ if $\mu$ is a limit ordinal. Then $\beth_\mu$ is a strong limit cardinal with cofinality larger than $\lambda$, where $\mu$ is any ordinal with cofinality larger than $\lambda$ (for example, the smallest ordinal of cardinality larger than $\lambda$).

Now build a sequence of cardinals $\kappa(0),\kappa(1),\ldots$ indexed by ordinals $\mu$, where $\kappa(0)$ is a strong limit cardinal of cofinality $>\omega$, and for a successor ordinal $\mu$ (of $\mu^-$), $\kappa(\mu)$ is a strong limit cardinal of cofinality $>\kappa(\mu^-)$. For limit ordinals, $\kappa(\mu)$ is the union of all previous $\kappa(\mu^\prime)$. Now let $\mu$ be any ordinal of uncountable cofinality. We claim that $\kappa(\mu)$ has the desired properties. These $\kappa(\mu)$ are clearly cofinal in all cardinals.

Indeed, any $\lambda<\kappa(\mu)$ is also less than $\kappa(\mu')$ for some $\mu'<\mu$, which we may assume to be a successor ordinal, and so $2^\lambda<\kappa(\mu')<\kappa(\mu)$; thus, $\kappa(\mu)$ is a strong limit cardinal. Also, as $\mu$ has uncountable cofinality, it follows that $\kappa(\mu)$ has uncountable cofinality. Finally, for all $\lambda<\kappa(\mu)$, one has $\lambda<\kappa(\mu')$ for some $\mu'<\mu$, and then for all successor ordinals $\mu''>\mu'$ less than $\mu$, the cardinal $\kappa(\mu'')<\kappa(\mu)$ is by construction a strong limit cardinal with cofinality greater than $\lambda$.
\end{proof}

\begin{definition}\label{def:kappasmall} A perfectoid space $X$ is $\kappa$-small if the cardinality of $|X|$ is less than $\kappa$, and for all open affinoid subspaces $U=\Spa(A,A^+)\subset X$, the cardinality of $A$ is less than $\kappa$.
\end{definition}

For the results of this section, it is actually enough to assume that $\kappa$ is an uncountable strong limit cardinal.

\begin{remark} If $A$ has a dense subalgebra $A_0\subset A$ of cardinality less than $\kappa$, then $A$ has cardinality less than $\kappa$. Indeed, identifying elements of $A$ with convergent Cauchy sequences in $A_0$, one sees that if $\lambda$ is the cardinality of $A_0$, then the cardinality of $A$ is at most
\[
\lambda^\omega\leq (2^\lambda)^\omega = 2^{\lambda\times \omega} = 2^\lambda<\kappa\ .
\]
\end{remark}

\begin{proposition}\label{prop:kappasmall} A perfectoid space $X$ is $\kappa$-small if and only if $X$ admits a cover by less than $\kappa$ many open affinoid subspaces $U=\Spa(A,A^+)$ for which the cardinality of $A$ is less than $\kappa$.
\end{proposition}

In particular, the category of $\kappa$-small perfectoid spaces is stable under finite products.

\begin{proof} Assume that $X$ is $\kappa$-small. In that case, the set of all subsets of $X$ is of cardinality less than $\kappa$, and in particular $X$ is covered by less than $\kappa$ many open subspaces of the form $U=\Spa(A,A^+)$, where the cardinality of $A$ is less than $\kappa$.

Conversely, for any open affinoid subspace $U=\Spa(A,A^+)\subset X$, the space $|U|$ is a subset of the space of binary relations on $A$ (where $x\in U$ corresponds to the relation $|a(x)|\leq |b(x)|$ for $a,b\in A$). Thus, $U$ embeds into the power set of $A\times A$, which is of cardinality $2^{|A|^2}=2^{|A|}<\kappa$. If $X$ is covered by less than $\kappa$ many such $U$, then also the cardinality of $|X|$ is less than $\kappa$. Also, for any open affinoid subspace $U=\Spa(A,A^+)$ of $X$, one can (by quasicompacity) embed $U$ into a union of finitely many rational subsets of $V=\Spa(B,B^+)$ for which $B$ has cardinality less than $\kappa$. Each of the rational subsets is of the form $\Spa(B^\prime,B^{\prime +})$, where $B^\prime$ has a dense subalgebra which is finitely generated over $B$; in particular, the cardinality of $B^\prime$ is also less than $\kappa$. As then $A$ embeds into the finite product of the $B^\prime$'s, one also sees that the cardinality of $A$ is less than $\kappa$.
\end{proof}

\section{Morphisms of perfectoid spaces}

In this section, we define some basic classes of maps between perfectoid spaces.

\begin{definition}\label{def:injection} A map $f: Y\to X$ of perfectoid spaces is an \emph{injection} if for all perfectoid spaces $Z$, the map $f_\ast: \Hom(Z,Y)\to \Hom(Z,X)$ is injective.
\end{definition}

It is enough to check the condition for affinoid perfectoid $Z$.

\begin{example}\label{ex:injpoint} Let $X$ be a perfectoid space, and $x\in X$ a point, giving rise to the map
\[
i_x: \Spa(K(x),K(x)^+)\to X\ ,
\]
where $K(x)$ is the completed residue field at $x$, and $K(x)^+\subset K(x)$ the corresponding open and bounded valuation subring. Then the map $i_x$ is an injection. To check this, we may assume that $X$ is affinoid by replacing $X$ by an affinoid neighborhood of $x$. Now
\[
|\Spa(K(x),K(x)^+)| = \bigcap_{U\ni x} U\subset X\ ,
\]
where the intersection runs over all rational neighborhoods $U$ of $x$ in $X$. In fact, $\Spa(K(x),K(x)^+)$ is the inverse limit of $U$ (considered as a perfectoid space) in the category of perfectoid spaces, as
\[
K(x)^+/\varpi = \varinjlim_{U\ni x} \OO_X^+(U)/\varpi\ ,
\]
and so $K(x)^+$ is the $\varpi$-adic completion of $\varinjlim_{U\ni x} \OO_X^+(U)$. As each $U\to X$ is an injection, it follows that $i_x$ is an injection.

In particular, note that if $X$ is qcqs and has a unique closed point $x\in X$, then $X=\Spa(K(x),K(x)^+)$, as in this case $X$ is the only quasicompact open subset of $X$ containing $x$.
\end{example}

We have the following characterizations of injections.

\begin{proposition}\label{prop:charinjection} Let $f: Y\to X$ be a map of perfectoid spaces. The following conditions are equivalent.
\begin{altenumerate}
\item The map $f$ is an injection.
\item For all perfectoid fields $K$ with an open and bounded valuation subring $K^+\subset K$, the map $f_\ast: Y(K,K^+)\to X(K,K^+)$ is injective.
\item The map $|f|: |Y|\to |X|$ is injective, and for all rank-$1$-points $y\in Y$ with image $x=f(y)\in X$, the map of completed residue fields $K(x)\to K(y)$ is an isomorphism.
\item The map $|f|: |Y|\to |X|$ is injective, and the map $f: Y\to X$ is final in the category of maps $g: Z\to X$ for which $|g|: |Z|\to |X|$ factors over a continuous map $|Z|\to |Y|$.
\end{altenumerate}
\end{proposition}

In particular, by (iv), given $X$, injections $Y\hookrightarrow X$ of perfectoid spaces are determined by the injective map of topological spaces $|Y|\to |X|$. We warn the reader that in general $|Y|$ may not have the subspace topology from $|X|$.

\begin{proof} Clearly, (i) implies (ii). Next, we check that (ii) implies (iii). To see that $|f|$ is injective, assume that two points $y_1,y_2\in Y$ map to $x\in X$. We may replace $X$ by $\Spa(K,K^+)$ via base change along $i_x$ as in Example~\ref{ex:injpoint}. There exists an extension $(K,K^+)\subset (L,L^+)$ of affinoid perfectoid fields such that there are $(L,L^+)$-valued points $\tilde{y}_1,\tilde{y}_2\in Y(L,L^+)$ which send the closed point of $\Spa(L,L^+)$ to $y_1$ resp.~$y_2$, and with same image in $X(L,L^+)$ given by the map $(K,K^+)\to (L,L^+)$. As $f$ satisfies (ii), this implies $\tilde{y}_1=\tilde{y}_2$, and thus $y_1=y_2$.

Now given a rank-$1$-point $y\in Y$ with image $x\in X$, we get a commutative diagram
\[\xymatrix{
\Spa(K(y),\OO_{K(y)})\ar^{i_y}[rr]\ar[d] && Y\ar^f[d]\\
\Spa(K(x),\OO_{K(x)})\ar^{i_x}[rr] && X\ .
}\]
It follows that $g: \Spa(K(y),\OO_{K(y)})\to \Spa(K(x),\OO_{K(x)})$ is functorially injective on $(K,K^+)$-valued points for all $(K,K^+)$. We claim that the two maps
\[
K(y)\to K(y)\hat{\otimes}_{K(x)} K(y)
\]
agree. Indeed, they agree after composition with $K(y)\hat{\otimes}_{K(x)} K(y)\to L$ for any perfectoid field $L$ by assumption; but as $K(y)\hat{\otimes}_{K(x)} K(y)$ is perfectoid (in particular, uniform), no element lies in the kernel of all maps to perfectoid fields. This implies that the surjection $K(y)\hat{\otimes}_{K(x)} K(y)\to K(y)$ is an isomorphism. It follows that also the base change $K(y)\to K(y)\hat{\otimes}_{K(x)} K(y)$ of $K(x)\to K(y)$ is an isomorphism. Let $V = K(y)/K(x)$, which is a $K(x)$-Banach space; then (by exactness of $\hat{\otimes}_{K(x)} K(y)$), we have $V\hat{\otimes}_{K(x)} K(y) = 0$, which implies that $V=0$ (indeed, $-\hat{\otimes}_{K(x)} V$ is exact, so $V$ injects into $K(y)\hat{\otimes}_{K(x)} V$). Thus, $K(y)=K(x)$, as desired, finishing the proof that (ii) implies (iii).

To check that (iii) implies (iv), we first check that for all perfectoid fields $K$ with an open and bounded valuation subring $K^+\subset K$, the map $Y(K,K^+)\to X(K,K^+)$ is injective, with image those maps $\Spa(K,K^+)\to X$ that factor over $|Y|$. Thus, assume given a map $\Spa(K,K^+)\to X$. The image of the closed point is a point $x\in X$, and $\Spa(K,K^+)$ factors uniquely over the injection $i_x: \Spa(K(x),K(x)^+)\to X$. Replacing $X$ by $\Spa(K(x),K(x)^+)$ and $Y$ by $Y\times_X \Spa(K(x),K(x)^+)$, we may assume that $X=\Spa(K(x),K(x)^+)$. Now $|Y|\subset |X|$ is a subset stable under generalizations; in particular, any quasicompact open subset of $Y$ has a unique closed point, so $Y$ is an increasing union of open subfunctors of the form $\Spa(K(x),(K(x)^+)^\prime)\subset \Spa(K(x),K(x)^+)$, for varying $(K(x)^+)^\prime$. But it is easy to see that for any such $(K(x)^+)^\prime$,
\[
\Spa(K(x),(K(x)^+)^\prime)(K,K^+)\to \Spa(K(x),K(x)^+)(K,K^+)
\]
is injective, with image those maps that factor through $|\Spa(K(x),(K(x)^+)^\prime)|\subset |\Spa(K(x),K(x)^+)|$; thus, the same holds for $Y$.

Now take any map $g: Z\to X$ such that $|g|$ factors continuously over $|Y|\subset |X|$. We need to check that $g$ factors uniquely over a map $h: Z\to Y$. This can be done locally on $Z$. Given a point $z\in Z$, we can replace $Z$ by a small open affinoid neighborhood of $z$ so that $|Z|\to |X|$ factors over an affinoid open subset of $X$, and also the map $|Z|\to |Y|$ factors over an affinoid open subset of $Y$. Thus, we can assume that $X$, $Y$ and $Z$ are all affinoid. Now consider the qcqs map $Y\times_X Z\to Z$. By Lemma~\ref{lem:bijectiveiso} below, it is enough to check that $(Y\times_X Z)(K,K^+)\to Z(K,K^+)$ is bijective for all perfectoid fields $K$ with an open and bounded valuation subring $K^+\subset K$. But this follows from the description of the map $Y(K,K^+)\to X(K,K^+)$ above.

Finally, (iv) implies (i): Given any map $Z\to X$, we need to see that there is at most one way to factor it over $Y$. If $|Z|\to |X|$ does not factor continuously over $|Y|$, there is no such factorization. However, if $|Z|\to |X|$ factors continuously over $|Y|$, then there is a unique factorization over $Y$, by the universal property of $Y\to X$ in (iv).
\end{proof}

\begin{lemma}\label{lem:bijectiveiso} Let $f: Y\to X$ be a qcqs map of perfectoid spaces. The following conditions are equivalent.
\begin{altenumerate}
\item The map $f$ is an isomorphism.
\item The map $|f|: |Y|\to |X|$ is bijective, and for all rank-$1$-points $y\in Y$ with image $x\in X$, the induced map of residue fields $K(x)\to K(y)$ is an isomorphism.
\item For all perfectoid fields $K$ with open and bounded valuation subring $K^+\subset K$, the map $Y(K,K^+)\to X(K,K^+)$ is a bijection.
\item For all algebraically closed perfectoid fields $C$ with an open and bounded valuation subring $C^+\subset C$, the map $Y(C,C^+)\to X(C,C^+)$ is a bijection.
\end{altenumerate}
\end{lemma}

\begin{proof} Clearly (i) implies (ii) and (iii). First, we check that (ii) and (iii) are equivalent. To see that (ii) implies (iii), take any map $\Spa(K,K^+)\to X$. Let $x\in X$ be the image of the closed point; then the map factors over the immersion $i_x: \Spa(K(x),K(x)^+)\to X$. We can replace $X$ by $\Spa(K(x),K(x)^+)$. Now $Y$ has a unique closed point $y\in Y$, so that $Y=\Spa(K(y),K(y)^+)$. By the assumption in (ii), we have $K(x)=K(y)$, and then bijectivity forces $K(y)^+=K(x)^+$. Now it is clear that (iii) holds.

For the converse (iii) implies (ii), assume that (iii) holds. To show that this implies (ii), it suffices to show that for any $x\in X$, the map $Y\times_X \Spa(K(x),K(x)^+)\to \Spa(K(x),K(x)^+)$ is an isomorphism. We may thus assume that $X=\Spa(K(x),K(x)^+)$. The condition in (iii) gives us a unique section $X\to Y$. Let $y\in Y$ be the image of the closed point $x\in X$. Note that the immersion $i_y: Y^\prime=\Spa(K(y),K(y)^+)\to Y$ has the property that $Y^\prime\to X$ still satisfies the conditions of (iii): Indeed, $Y^\prime(K,K^+)\subset Y(K,K^+)=X(K,K^+)$ is certainly injective, but there is a section $X\to Y^\prime$, so that it is also surjective. In particular, any point of $Y$ lies in $Y^\prime$, so that in fact $Y=Y^\prime$. As $Y\to X$ has a section, we get maps $K(x)\to K(y)\to K(x)$ of fields whose composite is the identity; as maps of fields are injective, this implies $K(y)=K(x)$, and then also $K(y)^+=K(x)^+$ as $K(x)^+\subset K(y)^+\subset K(x)^+$. Thus, $Y\cong X$, as desired.

Now we show that (ii) and (iii) together imply (i). We can assume that $X$ is affinoid; then, by assumption on $f$, $Y$ is qcqs. Then $|f|: |Y|\to |X|$ is a bijective, generalizing and spectral map of spectral spaces. By Lemma~\ref{lem:quotientmap}, it is a quotient map, and thus a homeomorphism. It remains to see that $\OO_X^+\to \OO_Y^+$ is an isomorphism of sheaves, for which it is enough to show that $\OO_X^+/\varpi\to \OO_Y^+/\varpi$ is an isomorphism of sheaves, where $\varpi$ is a pseudouniformizer on $X$. We can check this on stalks. But the stalk of $\OO_X^+/\varpi$ at $x\in X$ is given by $K(x)^+/\varpi$, and similarly for $\OO_Y^+/\varpi$. As $K(x)^+/\varpi=K(y)^+/\varpi$ (by the proof that (iii) implies (ii)), the result follows.

Clearly, (iii) implies (iv), so it remains to prove that (iv) implies (iii). Note that injectivity of $Y(K,K^+)\to X(K,K^+)$ is automatic by embedding $(K,K^+)$ into an algebraic closure $(C,C^+)$. But in fact, given any map $\Spa(K,K^+)\to X$, we can lift to a map $\Spa(C,C^+)\to Y$, which factors over some affinoid $\Spa(S,S^+)\subset Y$, and is given by a map $(S,S^+)\to (C,C^+)$. But this map has to be Galois-equivariant, and thus factors over $(K,K^+)$. This gives the desired lift $\Spa(K,K^+)\to Y$.
\end{proof}

Moreover, injections have a simple behaviour on pullbacks.

\begin{corollary}\label{cor:immpullback}
\begin{altenumerate}
\item Let $f: Y\to X$ be an injection of perfectoid spaces, and $X^\prime\to X$ any map of perfectoid spaces. Then the pullback $f^\prime: Y^\prime=X^\prime\times_X Y\to X^\prime$ is an injection, and the map
\[
|Y^\prime|\to |X^\prime|\times_{|X|} |Y|
\]
is a homeomorphism.
\item Let $f: Y\to X$ be a map of perfectoid spaces. Then $f$ is an injection if and only if $f$ is universally injective, i.e.~for all maps $X^\prime\to X$ with pullback $f^\prime: Y^\prime=X^\prime\times_X Y\to X^\prime$, the map $|f^\prime|: |Y^\prime|\to |X^\prime|$ is injective.
\end{altenumerate}
\end{corollary}

\begin{proof} In part (i), it is clear from the definition that $f^\prime$ is an injection. Using that $|f|$ and $|f^\prime|$ are injective by Proposition~\ref{prop:charinjection}, it follows that $|Y^\prime|\to |X^\prime|\times_{|X|} |Y|$ is injective. Moreover, all maps of perfectoid spaces are generalizing, so both $|Y^\prime|$ and $|X^\prime|\times_{|X|} |Y|$ are generalizing subsets of $|Y|$, and thus the map $|Y^\prime|\to |X^\prime|\times_{|X|} |Y|$ is generalizing. But this map is always surjective and spectral. By Lemma~\ref{lem:quotientmap}, it is a homeomorphism.

For part (ii), we have seen that if $f$ is an injection, then it is universally injective. For the converse, assume that $f$ is universally injective. Given any map $g: Z\to X$, we need to see that there is at most one way to factor it over $f: Y\to X$. Pulling back by $g$, we may assume that $Z=X$, and we have to show that there is at most one section of $f$. Assume that there is a section $h: X\to Y$. Being a section, $h$ is an injection; but then $|X|\to |Y|\to |X|$ is a factorization of the identity as a composition of injective maps. Thus, $|Y|=|X|$, and the map is a homeomorphism; then, by Proposition~\ref{prop:charinjection}, it follows that $Y=X$. In particular, $h$ is the only section, as desired.
\end{proof}

As injections are determined by their behaviour on topological spaces, the following definition is reasonable.

\begin{definition}\label{def:immersion} A map $f: Y\to X$ of perfectoid spaces is an immersion if $f$ is an injection and $|f|: |Y|\to |X|$ is a locally closed immersion. If $|f|$ in addition is open (resp.~closed), then $f$ is an open (resp.~closed) immersion.
\end{definition}

One cannot define closed immersions of schemes or adic spaces this way because of possible non-reduced structures. This is not a concern for perfectoid spaces, which makes this general definition possible.

Zariski closed embeddings as defined in \cite[Section II.2]{ScholzeTorsion} are a special class of closed embeddings.

\begin{definition}[{\cite[Definition II.2.1, Definition II.2.6]{ScholzeTorsion}}] Let $f: Z\to X$ be a map of perfectoid spaces, where $X=\Spa(R,R^+)$ is affinoid.
\begin{altenumerate}
\item The map $f$ is Zariski closed if $f$ is a closed immersion and there is an ideal $I\subset R$ such that $|Z|\subset |X|$ is the locus where $|f|=0$ for all $f\in I$.
\item The map $f$ is strongly Zariski closed if $Z=\Spa(S,S^+)$ is affinoid, $R\to S$ is surjective, and $S^+$ is the integral closure of $R^+$ in $S$.
\end{altenumerate}
\end{definition}

By \cite[Proposition 3.6.9 (c)]{KedlayaLiu1}, if $f$ is strongly Zariski closed, then $R^+\to S^+$ is almost surjective, so that the definition agrees with \cite[Definition II.2.6]{ScholzeTorsion}. It is easy to see that if $f$ is strongly Zariski closed, then $f$ is Zariski closed, cf.~\cite[Section II.2]{ScholzeTorsion}.

It was claimed in \cite[Section II.2]{ScholzeTorsion} that strongly Zariski closed is strictly stronger than Zariski closed. The example turned out to be erroneous, and in fact the two conditions are equivalent.

\begin{theorem}\label{thm:stronglyzariskiclosed} Let $f: Z\to X=\Spa(R,R^+)$ be a Zariski closed immersion. Then $f$ is strongly Zariski closed.
\end{theorem}

\begin{proof} This follows from \cite[Theorem 7.4, Remark 7.5]{BhattScholzePrism}.
\end{proof}

We note that \cite[Theorem 7.4]{BhattScholzePrism} also implies that $S^+\to R^+$ is almost surjective, independent of \cite{KedlayaLiu1}.

A general class of immersions is given by diagonal morphisms.

\begin{proposition}\label{prop:diagimmersion} Let $f: Y\to X$ be any map of perfectoid spaces. Then the diagonal morphism
\[
\Delta_f: Y\to Y\times_X Y
\]
is an immersion.
\end{proposition}

\begin{proof} Clearly, $\Delta_f$ is an injection. Thus, it is enough to show that $|\Delta_f|$ identifies $|Y|$ with a locally closed subspace of $|Y\times_X Y|$. This statement can be checked locally on $X$ and $Y$; thus, we can assume that $X=\Spa(R,R^+)$ and $Y=\Spa(S,S^+)$ are affinoid. But then $\Delta_f$ is strongly Zariski closed, as $S\hat{\otimes}_R S\to S$ is surjective. It follows that $\Delta_f$ is a closed immersion.
\end{proof}

\begin{definition}\label{def:separatedperf} A map $f: Y\to X$ of perfectoid spaces is \emph{separated} if $\Delta_f$ is a closed immersion.
\end{definition}

\begin{proposition}\label{prop:seperatedperf} Let $f: Y\to X$ be a map of perfectoid spaces. The following conditions are equivalent.
\begin{altenumerate}
\item The map $f$ is separated.
\item The map $|\Delta_f|: |Y|\to |Y\times_X Y|$ is a closed immersion.
\item The map $f$ is quasiseparated, and for all perfectoid fields $K$ with ring of integers $\OO_K$ and an open and bounded valuation subring $K^+\subset K$ and any commutative diagram
\[\xymatrix{
\Spa(K,\OO_K)\ar[d]\ar[r] & Y\ar^f[d] \\
\Spa(K,K^+)\ar[r]\ar@{-->}[ur] & X\ ,
}\]
there exists at most one dotted arrow making the diagram commute.
\end{altenumerate}
\end{proposition}

In (iii), recall that by definition $f: Y\to X$ is quasiseparated if the diagonal $\Delta_f: Y\to Y\times_X Y$ is quasicompact.

\begin{proof} The equivalence of (i) and (ii) follows from Proposition~\ref{prop:diagimmersion}. Next, we check that (ii) implies (iii). If $|\Delta_f|$ is a closed immersion, then it is in particular quasicompact; this implies that $f$ is quasiseparated. Now assume that there are two dotted arrows making the diagram commute. These define a point of $z\in (Y\times_X Y)(K,K^+)$ such that $z|_{(K,\OO_K)}\in \Delta_f(Y)(K,\OO_K)$. But if $|\Delta_f|: |Y|\to |Y\times_X Y|$ is a closed immersion, then the map $z: \Spa(K,K^+)\to Y\times_X Y$ maps into $|\Delta_f|(|Y|)$ if this is true for $z|_{(K,\OO_K)}$, as $\Spa(K,\OO_K)\subset \Spa(K,K^+)$ is dense. As $\Delta_f$ is an immersion, this implies that $z$ maps $\Spa(K,K^+)$ into $\Delta_f(Y)$, so that the dotted arrows agree.

Finally, (iii) implies (ii): The conditions in (iii) imply that the map $|\Delta_f|: |Y|\to |Y\times_X Y|$ is a quasicompact locally closed immersion of locally spectral spaces, which is moreover specializing: Any specialization in $|Y\times_X Y|$ is witnessed by an injection $\Spa(K,K^+)\to Y\times_X Y$, and if the generic point $\Spa(K,\OO_K)$ factors over $Y$, then by (iii) the whole map factors over $Y$. We claim that this implies that $|\Delta_f|$ is a closed immersion. By quasicompactness, the image of $|\Delta_f|$ is pro-constructible, thus the closure is given by the set of specializations; but the image is already closed under specializations.
\end{proof}

\section{\'Etale morphisms}

In this section, we recall briefly some facts about \'etale morphisms of perfectoid spaces. We start with the almost purity theorem.

\begin{theorem}[{\cite{FaltingsAlmostEtaleExtensions},\cite[Theorem 3.6.21, Theorem 5.5.9]{KedlayaLiu1},\cite[Theorem 5.25, Theorem 7.9 (iii)]{ScholzePerfectoidSpaces}}]\label{thm:almostpurity}  Let $R$ be a perfectoid Tate ring with tilt $R^\flat$.  
\begin{altenumerate}
\item For any finite \'etale $R$-algebra $S$, $S$ is perfectoid.  
\item Tilting induces an equivalence 
\begin{eqnarray*}
\{\text{Finite \'etale $R$-algebras}\}&\to& \{\text{Finite \'etale $R^{\flat}$-algebras}\} \\
S&\mapsto& S^\flat
\end{eqnarray*}
\item For any finite \'etale $R$-algebras $S$, $S^\circ$ is almost finite \'etale over $R^\circ$.
\end{altenumerate}
\end{theorem}

\begin{definition}\label{def:etale} Let $f: Y\to X$ be a morphism of perfectoid spaces.
\begin{altenumerate}
\item[{\rm (i)}] The morphism $f$ is finite \'etale if for all open affinoid perfectoid $\Spa(R,R^+)=U\subset X$, the preimage $V=f^{-1}(U)=\Spa(S,S^+)\subset Y$ is affinoid perfectoid, $R\to S$ is a finite \'etale morphism of rings, and $S^+$ is the integral closure of $R^+$ in $S$.
\item[{\rm (ii)}] The morphism $f$ is \'etale if for all $y\in Y$, there are open subsets $V\subset Y$, $U\subset X$ such that $y\in Y$, $f(V)\subset U$, and there is a factorization
\[\xymatrix{
V\ar^{f|_V}[dr]\ar@{^(->}[r] & W\ar[d]\\
&U,
}\]
where $V\hookrightarrow W$ is an open immersion, and $W\to U$ is finite \'etale.
\end{altenumerate}
\end{definition}

We refer to \cite[Section 7]{ScholzePerfectoidSpaces} for an extensive discussion. In particular, it is proved there that a morphism is finite \'etale if and only if there is a cover by open affinoid perfectoid $U=\Spa(R,R^+)\subset X$ such that their preimages $V=\Spa(S,S^+)\subset Y$ are affinoid perfectoid, $R\to S$ is finite \'etale, and $S^+$ is the integral closure of $R^+$ in $S$. Moreover, composites, base changes, and maps between (finite) \'etale maps are (finite) \'etale.

By $X_\fet$, resp. $X_\et$, we denote the categories of finite \'etale, resp. \'etale, perfectoid spaces over $X$. If $X=\Spa(R,R^+)$ is affinoid, then $X_\fet\cong R_\fet^\op$.

\begin{theorem}[\cite{ScholzePerfectoidSpaces}, \cite{KedlayaLiu1}]\label{thm:tiltingetalesite} Let $X$ be a perfectoid space with tilt $X^\flat$. Then tilting induces an equivalence $X_\et\cong X^\flat_\et$.

If $X=\Spa(R,R^+)$ is affinoid perfectoid, then the $R^+$-module $H^i(X_\et,\OO_X^+)$ is almost zero for $i>0$, and $H^0(X_\et,\OO_X^+)=R^+$.
\end{theorem}

We will need a result about \'etale morphisms to an inverse limit of spaces.

\begin{proposition}\label{prop:etmaptolim} Let $X_i=\Spa(R_i,R_i^+)$, $i\in I$, be a cofiltered inverse system of affinoid perfectoid spaces for some small index category $I$. Consider the inverse limit $X=\Spa(R,R^+)$, where $R^+$ is the $\varpi$-adic completion of $\varinjlim_i R_i^+$, and $R=R^+[\frac 1{\varpi}]$. Here, $\varpi\in R_i^+$ denotes any compatible choice of pseudouniformizers for large $i$. One has the following results.
\begin{altenumerate}
\item[{\rm (o)}] The map
\[
|X|\to \varprojlim_i |X_i|
\]
is a homeomorphism of spectral spaces.
\item[{\rm (i)}] The base change functors $(X_i)_\fet\to X_\fet$ induce an equivalence of categories
\[
\text{2-}\varinjlim_i (X_i)_\fet\to X_\fet\ .
\]
\item[{\rm (ii)}] Let $X_{\et,\qcqs}\subset X_\et$ denote the full subcategory of qcqs \'etale perfectoid spaces over $X$ (and similarly for $X_i$). Then the base change functors $(X_i)_{\et,\qcqs}\to X_{\et,\qcqs}$ induce an equivalence of categories
\[
\text{2-}\varinjlim_i (X_i)_{\et,\qcqs}\to X_{\et,\qcqs}\ .
\]
\item[{\rm (iii)}] Let $X_{\et,\qc,\sep}\subset X_{\et,\qcqs}$ be the full subcategory of quasicompact separated \'etale perfectoid spaces over $X$ (and similarly for $X_i$). Then the base change functors $(X_i)_{\et,\qc,\sep}\to X_{\et,\qc,\sep}$ induce an equivalence of categories
\[
\text{2-}\varinjlim_i (X_i)_{\et,\qc,\sep}\to X_{\et,\qc,\sep}\ .
\]
\item[{\rm (iv)}] Let $X_{\et,\aff}\subset X_{\et,\qc,\sep}$ be the full subcategory of affinoid perfectoid spaces which are \'etale over $X$ (and similarly for $X_i$). Then the base change functors $(X_i)_{\et,\aff}\to X_{\et,\aff}$ induce an equivalence of categories
\[
\text{2-}\varinjlim_i (X_i)_{\et,\aff}\to X_{\et,\aff}\ .
\]
\end{altenumerate}

Moreover, if $\kappa$ is as in Lemma~\ref{lem:choosekappa} and $I$ is $\lambda$-small for $\lambda<\kappa$ and all $X_i$ are $\kappa'$-small for some strong limit cardinal $\kappa'<\kappa$, then $X$ is $\kappa$-small.
\end{proposition}

\begin{proof} For the set-theoretic part, we can by Lemma~\ref{lem:choosekappa} assume that $\kappa^\prime = \kappa_\lambda$ has cofinality larger than $\lambda$, in which case the cardinality of the direct limit of the $R_i^+$ is still less than $\kappa_\lambda$, and so $R$ is $\kappa$-small.

Part (o) follows from \cite[Proposition 3.9]{HuberContinuousValuations}, cf.~\cite[Lemma 6.13 (ii)]{ScholzePerfectoidSpaces}. The statement in part (i) is equivalent to
\[
\text{2-}\varinjlim_i (R_i)_\fet\to R_\fet
\]
being an equivalence of categories. This is a consequence \cite[Lemma 7.5 (i)]{ScholzePerfectoidSpaces} of a theorem of Elkik, \cite{Elkik}, (in the noetherian case) and Gabber--Ramero, \cite[Proposition 5.4.53]{GabberRamero}, in general.

For part (ii), we first check fully faithfulness. Thus, let $Y_i, Y_i^\prime\to X_i$ be two qcqs \'etale perfectoid spaces over $X_i$, and denote by $Y_j,Y_j^\prime\to X_j$ their pullbacks to $X_j$ for $j\geq i$, and $Y,Y^\prime\to X$ their pullbacks to $X$. Assume that two morphisms $f_i,g_i: Y_i\to Y_i^\prime$ become equal after pullback to $X$, $f=g: Y\to Y^\prime$. Let $\Gamma_{f_i},\Gamma_{g_i}: Y_i\hookrightarrow Y_i\times_X Y_i^\prime=: Z_i$ be the graphs of $f_i$ and $g_i$. Then $\Gamma_{f_i}$ and $\Gamma_{g_i}$ are injections, and thus are determined by the quasicompact open image of $|Y_i|$ in $|Z_i|$. We know that after pullback along $|Z|\to |Z_i|$, $Z=Z_i\times_{X_i} X$, the two quasicompact open images $\Gamma_{f_i}(|Y_i|)$, $\Gamma_{g_i}(|Y_i|)$ agree. As $|Z|=\varprojlim_j |Z_j|$ is an inverse limit of spectral spaces along spectral maps, these two quasicompact open images agree after pullback to $|Z_j|$ for large $j$, which means that $\Gamma_{f_j}=\Gamma_{g_j}$, as desired. This proves faithfulness.

As an intermediate step to fullness, we check that if $f_i: Y_i\to Y_i^\prime$ is a morphism over $X$ whose pullback $f: Y\to Y^\prime$ to $X$ is an isomorphism, then the pullback $f_j: Y_j\to Y_j^\prime$ to $X_j$ is an isomorphism for large enough $j$. Note that the image of $f_i$ is a quasicompact open subset, over which $Y^\prime\to Y_i^\prime$ factors; a standard quasicompactness argument implies that $Y_j^\prime\to Y_i^\prime$ factors over this open subset for large enough $j$, so that $f_j$ is surjective for large enough $j$. Applying the same reasoning to $\Delta_{f_i}: Y_i\to Y_i\times_{Y_i^\prime} Y_i$ shows that $\Delta_{f_j}$ is surjective for large enough $j$. As $\Delta_{f_j}$ is always an injection, this implies that $\Delta_{f_j}$ is an isomorphism, which then implies that $f_j$ is an isomorphism.

Now, let $f: Y\to Y^\prime$ be a morphism over $X$; we need to see that this is a base change of a morphism $f_j: Y_j\to Y_j^\prime$ over $X_j$ for large enough $j$. The graph $\Gamma_f: Y\hookrightarrow Y\times_X Y^\prime$ is a quasicompact open embedding. By approximation of quasicompact open subsets, we can find a quasicompact open immersion $V_j\subset Y_j\times_{X_j} Y_j^\prime$ whose pullback agrees with the image of $\Gamma_f$. In particular, the projection $V_j\to Y_j$ of qcqs \'etale perfectoid spaces over $X_j$ becomes an isomorphism after pullback to $X$. By the intermediate step, we see that $V_j\to Y_j$ becomes an isomorphism for large enough $j$; composing the inverse with the map $V_j\hookrightarrow Y_j\times_{X_j} Y_j^\prime\to Y_j^\prime$ gives the desired map $Y_j\to Y_j^\prime$.

Finally, for essential surjectivity, take any qcqs \'etale $f: Y\to X$. It is enough to cover $Y$ be quasicompact open subsets which descend to $X_i$ for large $i$, as by fully faithfulness (and the assumption that $Y$ is qcqs), the pieces will automatically glue for large enough $i$. Thus, we can assume that $f$ is a composite of a rational open subset, a finite \'etale map, and a rational open subset. Each of these maps descends to $X_i$ for large enough $i$ (using (i)), finishing the proof of (ii).

For (iii), fully faithfulness follows from (ii). It remains to see that if $f_i: Y_i\to X_i$ is a qcqs \'etale map whose base change $f: Y\to X$ to $X$ is separated, then the base change $f_j: Y_j\to X_j$ to $X_j$ is separated for large enough $j$. The map $f_j: Y_j\to X_j$ is separated if and only if the quasicompact open immersion $\Delta_{f_j}: Y_j\to Y_j\times_{X_j} Y_j$ is also a closed immersion. But in general, if $S_i$, $i\in I$, is a cofiltered inverse system of spectral spaces (with spectral maps) and inverse limit $S$, then if $U_i\subset S_i$ is a quasicompact open subset whose inverse image $U\subset S$ is open and closed, then the inverse image $U_j\subset S_j$ is open and closed for large $j$. Indeed, the open and closed decomposition given by $U$ descends to $S_j$ for large enough $j$, and for large enough $j$, one of the two quasicompact open subsets in this decomposition is equal to $U_j$. Applied with $S_i=|Y_i\times_{X_i} Y_i|$, $S=|Y\times_X Y|$, we see that if $\Delta_f$ is an open and closed immersion, then $\Delta_{f_j}$ is an open and closed immersion for large enough $j$.

In (iv), fully faithfulness again follows from (ii). For essential surjectivity, we can work in characteristic $p$, where it follows from \cite[Proposition 1.7.1]{Huber}, which holds for affinoid perfectoid spaces of characteristic $p$ by using the theory of pseudocoherent sheaves of Kedlaya--Liu, \cite{KedlayaLiu2}. More precisely, for any \'etale map $f: Y\to X=\Spa(A,A^+)$ of affinoid perfectoid spaces of characteristic $p$, one can find a Zariski closed immersion of sousperfectoid adic spaces $Y\hookrightarrow \mathbb B^n_X$, whose ideal $\mathcal I$ is necessarily generated by $n$ functions $f_1,\ldots,f_n\in A\langle T_1,\ldots,T_n\rangle$ in a neighborhood of $Y$ (as $\Omega^1_{Y/X}=0$ and so $\mathcal I/\mathcal I^2\cong \Omega^1_{\mathbb B^n_X/X}/\mathcal I$), see \cite[Proposition IV.4.19]{FarguesScholze}. Slightly perturbing the $f_i$'s will still give an affinoid sousperfectoid space \'etale over $X$ (and thus perfectoid), by the same result. For close enough approximations, this will in fact be isomorphic to $Y$, by the approximation results above, and this is enough to deduce the essential surjectivity.
\end{proof}

Let us also recall the following result about inverse limits of affinoid perfectoid spaces (the first part of this was already implicitly used in the formulation of the last proposition).

\begin{proposition}\label{prop:limitaffperfd} Let $X_i=\Spa(R_i,R_i^+)$, $i\in I$, be a cofiltered diagram of affinoid perfectoid spaces. Then the inverse $X=\varprojlim_i X_i$ in the category of perfectoid spaces exists, is affinoid perfectoid $X=\Spa(R,R^+)$, where $R^+$ is the $\varpi$-adic completion of $\varinjlim_i R_i^+$ (for some compatible choice of pseudouniformizer $\varpi$ for all large enough $i$), $R=R^+[\varpi^{-1}]$.

If the index category $I$ is $\omega_1$-cofiltered, the maps
\[
\varinjlim_i R_i\to R, \varinjlim_i R_i^+\to R^+
\]
are bijective.

Moreover, any affinoid perfectoid space $X=\Spa(R,R^+)$ can be written as an $\omega_1$-cofiltered limit of affinoid perfectoid spaces $X_i=\Spa(R_i,R_i^+)$ where $R_i$ and $R_i^+$ are topologically countably generated (equivalently, have countable dense subsets).
\end{proposition}

Recall that the index category $I$ is $\omega_1$-cofiltered if for any countable category $J$ with a map $J\to I$, one can find an extension $J^{\triangleleft}\to I$ (where $J^{\triangleleft}$ consists of $J$ with one extra object $\ast$ with a unique map to each object of $J$).

\begin{proof} The first part follows from uniformity of perfectoid spaces, which means that whenever $Y$ maps compatibly to all $X_i$, the map $\varinjlim_i R_i^+\to \mathcal O^+(Y)$ extends uniquely to the $\varpi$-adic completion $R^+$, thus inducing a map $(R,R^+)\to (\mathcal O(Y),\mathcal O^+(Y))$. As $(R,R^+)$ is perfectoid, it follows that $X=\Spa(R,R^+)$ is the limit in perfectoid spaces.

When $I$ is $\omega_1$-cofiltered, the colimit $\varinjlim_i R_i^+$ is already complete, as any Cauchy sequence is already living in some $R_i^+$.

For the final statement, write $(R,R^+)$ as the (automatically $\omega_1$-filtered) colimit of all perfectoid $(R_i,R_i^+)\subset (R,R^+)$ that admit a countable dense subset.
\end{proof}

\section{Totally disconnected spaces}

In this section, we introduce a notion of (strictly) totally disconnected spaces. These are analogues of profinite sets in our setting.

\begin{definition} A perfectoid space $X$ is totally disconnected if $X$ is qcqs and every open cover of $X$ splits, i.e. if $\{U_i\subset X\}$ is an open cover of $X$, then $\bigsqcup U_i\to X$ has a splitting.
\end{definition}

The following characterization of spectral spaces for which every open cover splits is due to L.~Fargues:

\begin{lemma}\label{lem:fargues} Let $X$ be a spectral space. The following conditions are equivalent.
\begin{altenumerate}
\item Every open cover of $X$ splits.
\item Every connected component of $X$ has a unique closed point.
\item The functor $\mathcal F\mapsto \Gamma(X,\mathcal F)$ on sheaves on $X$ is exact, i.e.~commutes with all finite colimits.
\item The functor $\mathcal F\mapsto \Gamma(X,\mathcal F)$ on sheaves of abelian groups on $X$ is exact.
\item For all sheaves of abelian groups $\mathcal F$ on $X$, one has $H^i(X,\mathcal F)=0$ for all $i>0$.
\item For all sheaves of abelian groups $\mathcal F$ on $X$, one has $H^1(X,\mathcal F)=0$.
\end{altenumerate}
\end{lemma}

Recall here that any nonempty spectral space has at least one closed point: Indeed, minimal closed subspaces are points (as is easy to see using Zorn, and the fact that cofiltered limits of nonempty spectral spaces are nonempty).

\begin{proof} For (i) implies (ii), note that the condition that any open cover splits passes to closed subsets (by adding the open complement to the open cover). Thus, we may assume that $X$ is connected. But then if $x, y\in X$ are two distinct closed points, then $(X\setminus \{x\})\sqcup (X\setminus \{y\})\to X$ does not have a splitting. For the converse, let $\{U_i\subset X\}$ be an open cover of $X$. Let $c\in \pi_0 X$, corresponding to a connected component $X_c\subset X$. Take some $U_i$ which contains the unique closed point of $X_c$; then $X_c\subset U_i$. It follows that there is some open and closed neighborhood $U_c\subset X$ such that $U_c\subset U_i$. Thus, we can find a cover of $\pi_0 X$ by open and closed subsets such that $\bigsqcup U_i\to X$ splits on the preimage. As $\pi_0 X$ is profinite, we can find a disjoint cover of $\pi_0 X$ by such subsets, and then assemble the local splittings into a splitting of $\bigsqcup U_i\to X$.

Moreover, (i) implies (iii): We have to see that $\Gamma(X,-)$ commutes with coequalizers. Let $\mathcal H=\mathrm{coeq}(\mathcal F\rightrightarrows \mathcal G)$ be a coequalizer diagram. First, we check that
\[
\mathrm{coeq}(\Gamma(X,\mathcal F)\rightrightarrows \Gamma(X,\mathcal G))\to \Gamma(X,\mathcal{H})
\]
is injective. If $t,t^\prime\in \Gamma(X,\mathcal G)$ are sections whose images in $\Gamma(X,\mathcal H)$ agree, then after a cover $\{U_i\to X\}$, one can find $s_i\in \mathcal F(U_i)$ mapping to $t|_{U_i}$ and $t^\prime|_{U_i}$. But under (i), we can split $\bigsqcup U_i\to X$, giving us such a section $s\in \Gamma(X,\mathcal F)$. Now assume $s\in \Gamma(X,\mathcal H)$ is a section. Then after a cover $\{U_i\to X\}$, we can lift it to sections $s_i\in \mathcal G(U_i)$. But under (i), we can split $\bigsqcup U_i\to X$, giving us a lift of $s$ to $\Gamma(X,\mathcal G)$, proving the desired surjectivity.

It is clear that (iii)$\Rightarrow$(iv)$\Rightarrow$(v)$\Rightarrow$(vi)$\Rightarrow$(iv). To see that (iv) implies (i), let $\{U_1,\ldots,U_n\}$ be a finite quasicompact open cover of $X$, let $j_i: U_i\to X$ be the open immersions, and consider the surjection
\[
\mathcal F=\bigoplus_i j_{i!} \mathbb Z\to \mathbb Z
\]
of sheaves on $X$. Under condition (iv), this is surjective on global sections. This implies that there are locally constant functions $f_i: X\to \mathbb Z$ with $V_i=\mathrm{supp}\ f_i\subset U_i$ such that $\sum_i f_i = 1$. In particular, all $V_i$ are open and closed, and
\[
X=\bigsqcup_i (V_i\setminus \bigcup_{j<i} V_j)\to \bigsqcup_i U_i\to X
\]
gives a section, as desired.
\end{proof}

For perfectoid spaces, one can give a finer structural analysis.

\begin{lemma}\label{lem:basicwlocal} Let $X$ be a totally disconnected perfectoid space.
\begin{enumerate}
\item[{\rm (i)}] There is a continuous projection $\pi: X\to \pi_0(X)$ to the profinite set $\pi_0(X)$ of connected components.
\item[{\rm (ii)}] All fibres of $\pi$ are of the form $\Spa(K,K^+)$ for some perfectoid field $K$ with an open and bounded valuation subring $K^+\subset K$.
\end{enumerate}
Conversely, if $X$ is a qcqs perfectoid space all of whose connected components are of the form $\Spa(K,K^+)$ for some perfectoid field $K$ with an open and bounded valuation subring $K^+\subset K$, then $X$ is totally disconnected.
\end{lemma}

\begin{proof} Part (i) is a consequence of the assumption that the underlying topological space of $X$ is a spectral space. For part (ii), note that in a spectral space, any connected component is an intersection of open and closed subsets; this implies that connected components of qcqs perfectoid spaces are again perfectoid spaces. Thus, for (ii), we have to classify the connected totally disconnected perfectoid spaces. By Lemma~\ref{lem:fargues}, $X$ has a unique closed point. By Example~\ref{ex:injpoint}, it follows that if $x\in X$ is the unique closed point of a perfectoid space $X$, then $X=\Spa(K(x),K(x)^+)$, as desired.

The converse follows from Lemma~\ref{lem:fargues}, as then all connected components have a unique closed point.
\end{proof}

In \cite{BhattScholze}, we worked with a stronger notion of w-local spaces; this includes the extra condition that the subset of closed points is closed. This stronger notion will have little relevance for us.

\begin{definition}\label{def:wlocal} A perfectoid space $X$ is w-local if its underlying topological space is a w-local spectral space in the sense of~\cite[Definition 2.1.1]{BhattScholze}. Equivalently, $X$ is qcqs, for any open cover $X=\bigcup_{i\in I} U_i$, the map $\bigsqcup_{i\in I} U_i\to X$ splits, and the subset $X^c\subset X$ of closed points is closed.
\end{definition}

In particular, w-local spaces are totally disconnected.

\begin{lemma}\label{lem:totallydisconnectedaffinoid} Let $X$ be a totally disconnected perfectoid space. Then $X$ is affinoid.
\end{lemma}

\begin{proof} Let $c\in \pi_0(X)$ be any point, and let $x_c\in X$ be the unique closed point mapping to $c$. There is an open affinoid neighborhood $U\subset X$ of $x_c$; in particular, $U$ contains the connected component $X_c\subset X$, as all points of $X_c$ specialize to $x_c$. We claim that there is an open compact neighborhood $V$ of $c$ in $\pi_0(X)$ such that $\pi^{-1}(V)\subset U$. Indeed, the intersection
\[
\bigcap_{c\in V} (X\setminus U)\cap \pi^{-1}(V)=\emptyset\ ,
\]
where the intersection runs over open and closed subsets $V\subset \pi_0(X)$ containing $c$. By quasicompactness of the constructible topology, it follows that there is some $V$ such that $(X\setminus U)\cap \pi^{-1}(V)=\emptyset$, i.e.~$\pi^{-1}(V)\subset U$. In this case, $U$ has a closed and open decomposition $U=\pi^{-1}(V)\sqcup \pi^{-1}(\pi_0(X)\setminus V)$, so that $\pi^{-1}(V)$ is also affinoid. Thus, $\pi_0(X)$ is covered by open and closed subsets $V\subset \pi_0(X)$ such that $\pi^{-1}(V)\subset X$ is affinoid. By quasicompactness, finitely many such $V$ cover $\pi_0(X)$, and we may moreover assume that they form a disjoint cover $\pi_0(X)=V_1\sqcup \ldots \sqcup V_n$. Then $X$ is the disjoint union $\pi^{-1}(V_1)\sqcup \ldots \sqcup \pi^{-1}(V_n)$, and thus affinoid.
\end{proof}

\begin{lemma}\label{lem:subsetwlocal} Let $X$ be a totally disconnected perfectoid space. Let $U\subset X$ be a pro-constructible generalizing subset. Then $U$ is an intersection of subsets of the form $\{|f|\leq 1\}$ for varying $f\in H^0(X,\OO_X)$. In particular, $U$ is an affinoid perfectoid space, which is moreover totally disconnected.
\end{lemma}

Recall here that a finite intersection of subsets of the form $\{|f|\leq 1\}$ inside some affinoid perfectoid space $\Spa(R,R^+)$ always defines a rational subset, and in particular an affinoid perfectoid space; and that for infinite such intersections, one gets a cofiltered limit of affinoid perfectoid spaces, which is itself affinoid perfectoid.

\begin{remark}\label{rem:qcopenwlocal} Note that the lemma implies in particular that any quasicompact open subset of $X$ is affinoid.
\end{remark}

\begin{proof} Let $x\in X\setminus U$ be a point, lying in a connected component $c=\pi(x)\in \pi_0(X)$. The fibre $X_c$ is of the form $\Spa(K_c,K_c^+)$ for some perfectoid field $K_c$ and open and bounded valuation subring $K_c^+\subset K_c$, and similarly the fibre $U_c$ is of the form $\Spa(K_c,(K_c^+)^\prime)$ for some open and bounded valuation subring $(K_c^+)^\prime\subset K_c$ containing $K_c^+$. Now one can find some $f_c\in K_c$ such that $|f_c|\leq 1$ on $U_c = U\cap X_c$, but $|f_c(x)|>1$, by taking any element of $(K_c^+)^\prime\setminus K_c^+$. Modulo $K_c^+$, we can lift $f_c$ to a function $f_V$ on $\pi^{-1}(V)\subset X$ for some compact open neighborhood $V\subset \pi_0(X)$ of $c$. Then by assumption
\[
\bigcap_{c\in V^\prime\subset V} (U\cap \{|f_V|>1\})\cap \pi^{-1}(V^\prime)=\emptyset\ .
\]
Here, all terms are pro-constructible subsets of $X$, and thus compact in the constructible topology; by Tychonoff, it follows that for some open compact neighborhood $V^\prime$ of $v$, $(U\cap \{|f_V|>1\})\cap \pi^{-1}(V^\prime)=\emptyset$; in other words, after replacing $V$ by $V^\prime$, we can assume that $U\cap \pi^{-1}(V)\subset \{|f_V|\leq 1\}$. Extending $f_V$ by $0$ on $\pi^{-1}(\pi_0(X)\setminus V)$, we get a function $f\in H^0(X,\OO_X)$ with $|f|\leq 1$ on $U$, but $|f(x)|>1$. Taking the intersection over all such $f$ for varying $x\in X\setminus U$ gives $U$.

Thus, $U$ is affinoid. Every connected component of $U$ is an intersection of a connected component $\Spa(K_c,K_c^+)$ of $X$ with $U$, which is necessarily of the form $\Spa(K_c,(K_c^+)^\prime)$ for another open and bounded valuation subring $(K_c^+)^\prime\subset K_c$.
\end{proof}

The notions of w-local and totally disconnected spaces are useful as on the one hand, they are rather simple, but on the other hand, there are many examples. More precisely, any qcqs perfectoid space $X$ admits a universal map $X^\wl\to X$ from a w-local space $X^\wl$, and the map $X^\wl\to X$ is pro-\'etale. Before we formulate this result, we recall the definition of pro-\'etale morphisms.

\begin{definition}\label{def:proetale} Let $f: Y\to X$ be a map of perfectoid spaces.
\begin{altenumerate}
\item The map $f$ is \emph{affinoid pro-\'etale} if $Y=\Spa(S,S^+)$ and $X=\Spa(R,R^+)$ are affinoid, and one can write $Y=\varprojlim Y_i\to X$ as an inverse limit of \'etale maps $Y_i\to X$ from affinoid perfectoid spaces $Y_i=\Spa(S_i,S_i^+)$ along a small cofiltered index category $I$. Any such presentation $Y=\varprojlim Y_i\to X$ is called a \emph{pro-\'etale presentation}.
\item The map $f$ is \emph{pro-\'etale} if for all $y\in Y$, there is an open neighborhood $V\subset Y$ of $y$ and an open subset $U\subset X$ such that $f(V)\subset U$, and the restriction $f|_V: V\to U$ is affinoid pro-\'etale.
\end{altenumerate}
\end{definition}

For example, if $X=\Spa(C,\OO_C)$ is a geometric point, then any affinoid pro-\'etale map $Y\to X$ is of the form $Y=X\times \underline{S}$ for some profinite set $S$, where if $S=\varprojlim_i S_i$ is a limit of finite sets, we set
\[
X\times \underline{S} = \varprojlim_i X\times S_i\ .
\]

\begin{remark}\label{rem:diagproetale} A slightly curious example of pro-\'etale morphisms are Zariski closed immersions: If $f: Z\hookrightarrow X$ is a Zariski closed immersion of affinoid perfectoid spaces, then $f$ is affinoid pro-\'etale. Indeed, if $Z=\{x\in X\mid \forall f\in I: f(x)=0\}\subset X$ for an ideal $I\subset R=\OO_X(X)$, then $Z$ can be written as the intersection of the rational subsets
\[
U_{f_1,\ldots,f_n} = \{|f_1|,\ldots,|f_n|\leq 1\}
\]
over varying $n$ and $f_1,\ldots,f_n\in I$. Thus, $Z=\varprojlim U_{f_1,\ldots,f_n}\to X$, which gives a pro-\'etale presentation of $Z$.

In particular, if $f: Y\to X$ is any morphism of perfectoid spaces, then $\Delta_f: Y\to Y\times_X Y$ is pro-\'etale. Indeed, this can be checked locally, so we can assume that $X$ and $Y$ are affinoid; then $\Delta_f$ is a Zariski closed immersion, and thus affinoid pro-\'etale.
\end{remark}

\begin{proposition}\label{prop:affproetale} Fix an affinoid perfectoid space $X$, and let $X_\et^\aff$ be the category of \'etale maps $f:Y\to X$ from affinoid perfectoid spaces $Y$, and $X_\proet^\aff$ the category of affinoid pro-\'etale maps $f: Y\to X$. Then the functor
\[
\Pro(X_\et^\aff)\to X_\proet^\aff\ :\ \varprojlimfi Y_i\mapsto \varprojlim_i Y_i
\]
is an equivalence of categories. If $X$ is $\kappa$-small and one restricts to $\kappa$-small objets in $X_\proet^\aff$, then this subcategory is equivalent to $\Pro_\kappa(X_\et^\aff)$.
\end{proposition}

Here, $\Pro_\kappa$ denotes the category of pro-systems whose index category is bounded by $\kappa$.

\begin{proof} By definition, the functor is essentially surjective, so we need to see that it is fully faithful. For this, write $X=\Spa(R,R^+)$, and take $Y=\Spa(S,S^+)$, $Z=\Spa(T,T^+)$ two objects of $X_\proet^\aff$ that are written as limits of $Y_i=\Spa(S_i,S_i^+)\in X_\et^\aff$, $Z_j=\Spa(T_j,T_j^+)\in X_\et^\aff$, for cofiltered index categories $I$, $J$. We need to show that
\[
\Hom_X(Y,Z)=\varprojlim_j \varinjlim_i \Hom_X(Y_i,Z_j)\ .
\]
Without loss of generality $J$ is a singleton, so $Z\to X$ is \'etale. Now have to check that
\[
\Hom_X(Y,Z)=\varinjlim_i \Hom_X(Y_i,Z)\ .
\]
Recall that $Y_{\et,\qcqs} =\text{2-}\varinjlim_i (Y_i)_{\et,\qcqs}$ by Proposition~\ref{prop:etmaptolim}~(ii). Thus,
\[
\Hom_X(Y,Z) = \Hom_Y(Y,Y\times_X Z) = \varinjlim_i \Hom_{Y_i}(Y_i,Y_i\times_X Z) = \varinjlim_i \Hom_X(Y_i,Z)\ ,
\]
as desired.
\end{proof}

We need some basic properties of pro-\'etale maps.

\begin{lemma}\label{lem:proetalecompbc}$ $
\begin{altenumerate}
\item Let $g: Z\to Y$ and $f: Y\to X$ be pro-\'etale (resp.~affinoid pro-\'etale) morphisms. Then $f\circ g: Z\to X$ is pro-\'etale (resp.~affinoid pro-\'etale).
\item Let $f: Y\to X$ be pro-\'etale (resp.~affinoid pro-\'etale), and let $g: X^\prime\to X$ be a map from a perfectoid space (resp.~affinoid perfectoid space). Then $f^\prime: Y^\prime=X^\prime\times_X Y\to X^\prime$ is pro-\'etale (resp.~affinoid pro-\'etale).
\item Let $f: Y\to X$ and $f^\prime: Y^\prime\to X$ be pro-\'etale (resp.~affinoid pro-\'etale). Then any map $g: Y\to Y^\prime$ over $X$ is pro-\'etale (resp.~affinoid pro-\'etale).
\item For any affinoid perfectoid space $X$, the category $X^\aff_\proet$ has all small limits.
\end{altenumerate}
\end{lemma}

\begin{proof} It is enough to prove the assertions in the affinoid pro-\'etale case. For part (i), we use Proposition~\ref{prop:affproetale}, Proposition~\ref{prop:etmaptolim}~(iv) and the natural functor $\Pro(\Pro(\mathcal C))\to \Pro(\mathcal C)$, which exists for any category $\mathcal C$. Part (ii) follows directly from the definition. For part (iii), we can factor $g$ as the composite of a section of the affinoid pro-\'etale map $Y\times_X Y^\prime\to Y$ (given by the graph), and the affinoid pro-\'etale map $Y\times_X Y^\prime\to Y^\prime$; thus, it is enough to show that sections of affinoid pro-\'etale maps are affinoid pro-\'etale. But if $Y=\varprojlim Y_i\to X$ is a pro-\'etale presentation, and $s: X\to Y$ is a section, then $s$ is given by a compatible system of sections $s_i: X\to Y_i$, and $X=\varprojlim_i (X\times_{Y_i} Y)\to Y$ is a pro-\'etale presentation.

Finally, for part (iv), as $X^\aff_\proet = \Pro(X^\aff_\et)$ by Proposition~\ref{prop:affproetale}, it is enough to prove that $X^\aff_\et$ has all finite limits. For this, it is enough to check that it has a final object and fibre products, which is clear.
\end{proof}

\begin{proposition}\label{prop:wlocalisation} The inclusion from the category of w-local perfectoid spaces with w-local maps, to the category of qcqs perfectoid spaces, admits a left adjoint $X\mapsto X^\wl$. The adjunction map $X^\wl\to X$ is pro-\'etale. If $X$ is affinoid, the map $X^\wl\to X$ is an inverse limit of surjective maps of the form $\bigsqcup_{i\in I} U_i\to X$, where $I$ is a finite set, and $U_i\subset X$ are rational subsets; in particular, it is affinoid pro-\'etale.
\end{proposition}

\begin{proof} By~\cite[Lemma 2.1.10]{BhattScholze}, there is a similar w-localization $|X|^\wl$ of the underlying topological space $|X|$. Consider the topological space $|X|^\wl$ endowed with the presheaves $\OO^+_{X^\wl}$, $\OO_{X^\wl}$ of complete topological rings given by $\OO^+_{X^\wl} = (p^\ast \OO_X^+)_\varpi^\wedge$, $\OO_{X^\wl} = \OO^+_{X^\wl}[\varpi^{-1}]$, where $p: |X|^\wl\to |X|$ denotes the projection, and $\varpi$ is a (local choice of) pseudouniformizer. We claim that this defines a perfectoid space $X^\wl$; it is then clear that it is final in the category of w-local perfectoid spaces over $X$ (with w-local maps).

We handle the affinoid case first. We use the following general presentation of the w-localization.

\begin{lemma}\label{lem:computewlocalizationspectral} Let $X$ be a spectral topological space, and let $B$ be a basis for the topology consisting of quasicompact open subsets, such that for $U,V\in B$, also $U\cap V\in B$. Let $X^\wl\to X$ be the w-localization of $X$, and let $C$ be the category of factorizations $X^\wl\to T\to X$, where $T$ is a finite disjoint union of elements of $B$. Then $C$ is cofiltered, and the natural map
\[
f: X^\wl\to \varprojlim_C T
\]
is a homeomorphism.
\end{lemma}

\begin{proof} Note that the category of $C$ is cofiltered: Given any two factorizations $X^\wl\to T_1\to X$ and $X^\wl\to T_2\to X$, one gets another one $X^\wl\to T_1\times_{X} T_2\to X$, and for any two given maps $f,g: T_1\to T_2$ between two such factorizations, the locus $T\subset T_1$ where $f=g$ is open and closed (more precisely, if $T_1=\bigsqcup_{i\in I} U_i$, then $T=\bigsqcup_{i\in J} U_i$ for some subset $J\subset I$), and $X^\wl\to T_1$ factors over $T$. In particular, we can assume that $X\in B$, as the two resulting categories are cofinal.

By Lemma~\ref{lem:quotientmap}, it is enough to see that $f: X^\wl\to \varprojlim_C T$ is bijective and generalizing. Recall that $\pi_0(X^\wl)=X^\cons$, where $X^\cons$ is $X$ endowed with the constructible topology, and the connected component corresponding to $x\in X$ is given by the localization $X_x$, i.e.~the set of points specializing to $x$. In particular, points of $X^\wl$ are in bijection with pairs $(x,y)$ of points in $X$, where $y\in X_x$. The map $X^\wl\to X$ is given by $(x,y)\mapsto y$, which is generalizing; more precisely, for any $(x,y)\in X^\wl$, the localization $X^\wl_{(x,y)}$ is identified with the localization $X_y$. As $T\to X$ is a finite disjoint union of open subsets, it also identifies localizations, and thus the map $f: X^\wl\to \varprojlim_C T$ identifies localizations, and in particular is generalizing.

For injectivity of $f: X^\wl\to \varprojlim_C T$, it is enough to show that given points $(x,y),(x^\prime,y^\prime)\in X^\wl$ with $x\neq x^\prime$, $f(x,y)\neq f(x^\prime,y^\prime)$ (as the $y$-coordinate is seen by the projection to $X$). For this, choose some $U\in B$ such that $U$ contains exactly one of $x$ and $x^\prime$, say $x$, and look at $T=U\sqcup X\to X$. We have a factorization
\[
X^\wl = X^\wl\times_{X^\cons} U^\cons \sqcup X^\wl\times_{X^\cons} (X\setminus U)^\cons\to U\sqcup X=T\to X
\]
which sends $(x,y)$ into $U\subset T$, and $(x^\prime,y^\prime)$ into $X\subset T$. Thus, we get injectivity.

For surjectivity, we claim that it is enough to show that if $X^\wl\to T\to X$ is a factorizaton, and $t\in T$ is a point that is not in the image of $X^\wl\to T$, then one can find another factorization $X^\wl\to \tilde{T}\to T\to X$ such that $t$ is not in the image of $\tilde{T}\to T$. Indeed, this implies that the image of $X^\wl\to T$ agrees with the image $T_0\subset T$ of $\varprojlim_C T\to T$. But then $\varprojlim_C T_0=\varprojlim_C T$, and we have a family of surjective maps $X^\wl\to T_0$; by a usual quasicompactness argument, this implies that $X^\wl\to \varprojlim_C T_0=\varprojlim_C T$ is surjective (the fibre over any point is a cofiltered limit of nonempty spectral spaces along spectral maps, thus nonempty by Tychonoff applied to the constructible topology).

Thus, assume $t\in T$ is not in the image of $X^\wl\to T=\bigsqcup_{i\in I} U_i$. Assume $t\in U_{i_0}$, and choose finitely many $V_j\subset U_{i_0}$, $j\in J$, $V_j\in B$ with $t\not\in V_j$, such that $X^\wl\to T$ factors over
\[
\bigcup_{j\in J} V_j \sqcup \bigsqcup_{i\in I, i\neq i_0} U_i\ .
\]
This is possible, as the map $X^\wl\to T$ is generalizing, so if $t$ is not in the image, the closure of $t$ is not in the image as well. We have an open cover
\[
\tilde{T} := \bigsqcup_{j\in J} V_j\sqcup \bigsqcup_{i\in I, i\neq i_0} U_i\to \bigcup_{j\in J} V_j \sqcup \bigsqcup_{i\in I, i\neq i_0} U_i\ .
\]
After pullback to $X^\wl$, this open cover splits by the definition of w-local spaces. In particular, the map
\[
X^\wl\to \bigcup_{j\in J} V_j \sqcup \bigsqcup_{i\in I, i\neq i_0} U_i
\]
lifts to a map $X^\wl\to \tilde{T}$, giving the desired factorization $X^\wl\to\tilde{T}\to T\to X$, with $t$ not in the image of $\tilde{T}\to T$.
\end{proof}

We apply this in the case where $B$ is the basis of rational subsets of $X$. Clearly, any $T$ as in the lemma admits a canonical structure as an affinoid perfectoid space, for which $T\to X$ is \'etale. Then $\varprojlim_C T = X^\wl$, on topological spaces by Lemma~\ref{lem:computewlocalizationspectral}, and on $\OO$ and $\OO^+$ by a direct verification; this finishes the argument.

In general, we can reduce to the affinoid case by the following lemma.

\begin{lemma}\label{lem:wlocalicationopenspectralsubset} Let $X$ be a spectral topological space, and $U\subset X$ a quasicompact open subset. Then $U^\wl\to X^\wl$ is a quasicompact open embedding. If $X=\cup_{i\in I} U_i$ is a covering of $X$ by quasicompact open subsets $U_i\subset X$, then the induced map $\sqcup U_i^\wl\to X^\wl$ is an open covering by quasicompact open subsets.
\end{lemma}

\begin{proof} In general $X^\wl$ comes with functorial maps to $X$ (by adjunction) and $X^\cons$, which is $X$ endowed with the constructible topology, via the identification $\pi_0(X^\wl)=X^\cons$ from \cite[Lemma 2.1.10]{BhattScholze}. In particular, there is a natural map
\[
U^\wl\to X^\wl\times_{X^\cons} U^\cons\to X^\wl\ .
\]
Here, $U^\cons\subset X^\cons$ is a quasicompact open subset; thus, it suffices to see that the map $U^\wl\to X^\wl\times_{X^\cons} U^\cons$ is a homeomorphism. By Lemma~\ref{lem:quotientmap}, it is enough to show that it is bijective and generalizing. Recall that $\pi_0(X^\wl) = X$, and for $x\in X$, the corresponding connected component of $X^\wl$ is the set $X_x\subset X$ of all generalizations of $x$, i.e.~the localization of $X$ at $x$. If $x\in U$, it follows that $U_x=X_x$; this shows bijectivity. Moreover, it shows that the map is generalizing, as desired.

The final statement follows from $X^\cons = \cup_{i\in I} U_i^\cons$.
\end{proof}
\end{proof}

Moreover, we have the following variant.

\begin{definition} A perfectoid space $X$ is strictly totally disconnected if it is qcqs and every \'etale cover splits.
\end{definition}

\begin{proposition}\label{prop:etcoverssplittotdisconnected} Let $X$ be a qcqs perfectoid space. Then $X$ is strictly totally disconnected if and only if every connected component is of the form $\Spa(C,C^+)$, where $C$ is algebraically closed.
\end{proposition}

\begin{proof} Assume first that $X$ is strictly totally disconnected. This condition passes to closed subsets, so we can assume that $X$ is connected. In that case, as $X$ is in particular totally disconnected, $X=\Spa(K,K^+)$ for some perfectoid field $K$ with an open and bounded valuation subring $K^+\subset K$. If $K$ is not algebraically closed, then for any finite extension $L$ of $K$ with integral closure $L^+\subset L$ of $K^+$, the map $\Spa(L,L^+)\to \Spa(K,K^+)$ is a nonsplit finite \'etale cover. Thus, necessarily $K$ is algebraically closed.

Conversely, assume that all connected components of $X$ are of the form $\Spa(C,C^+)$. By Lemma~\ref{lem:basicwlocal}, $X$ is totally disconnected. To see that it is strictly totally disconnected, it is enough to see that any \'etale cover is locally split (as then the \'etale cover admits a refinement by an open cover, which is split as $X$ is totally disconnected). We may assume that the \'etale cover is qcqs. But any \'etale cover of any connected component $\Spa(C,C^+)$ splits, and then the splitting extends to a small neighborhood by Proposition~\ref{prop:etmaptolim}.
\end{proof}

Again, one could consider the following strengthening.

\begin{definition} A w-strictly local perfectoid space is a w-local perfectoid space $X$ such that for all $x\in X$, the completed residue field $K(x)$ is algebraically closed.
\end{definition}

In particular, a w-strictly local space is strictly totally disconnected.

For any $X$, one can find strictly totally disconnected spaces over $X$, as the following lemma shows. The construction of this lemma is slightly different than the w-localization functor, and in particular also ensures that $\tilde{X}\to X$ is universally open, which is not true for $X^\wl\to X$.

\begin{lemma}\label{lem:strtotdisccover} Let $X$ be an affinoid perfectoid space. Then there is an affinoid perfectoid space $\tilde{X}$ with an affinoid pro-\'etale surjective and universally open map $\tilde{X}\to X$, such that $\tilde{X}$ is strictly totally disconnected.

If $\kappa$ is as in Lemma~\ref{lem:choosekappa} and $X$ is $\kappa$-small, one can choose $\tilde{X}$ to also be $\kappa$-small.
\end{lemma}

\begin{proof} We want to take the cofiltered limit of ``all'' affinoid \'etale covers over $X$. For this, fix a set of representatives $\{X_i\to X\}_{i\in I}$ of all affinoid \'etale and surjective spaces over $X$. We note that this is of cardinality less than $\kappa$: The cardinality of $\OO_{X_i}(X_i)$ is bounded independently of $i$, say by $\lambda$ (which is essentially the cardinality of $\OO_X(X)$). The set of all affinoid perfectoid spaces $Y$ for which $\OO_Y(Y)$ has cardinality bounded by $\lambda$ is also of cardinality less than $\kappa$ (by thinking about all possible ways of endowing a set of cardinality $\leq \lambda$ with the structure of a perfectoid algebra, and possible rings of integral elements). For each such $Y$, the set of maps from $Y$ to $X$ has cardinality less than $\kappa$. In total, one sees that there less than $\kappa$ such \'etale maps $Y\to X$. For each finite subset $J\subset I$, let $X_J$ be the product of all $X_i$, $i\in J$, over $X$, which is still affinoid. Then $X_J\to X$ is \'etale and surjective, and the transition maps $X_J\to X_{J^\prime}$ for $J^\prime\subset J$ are also \'etale and surjective. In particular, if we set $X_\infty = \varprojlim_J X_J$, then $X_\infty\to X$ is affinoid pro-\'etale, and universally open. Indeed, it is open as any quasicompact open subset $U_\infty\subset X_\infty$ comes as the preimage of some quasicompact open $U_J\subset X_J$, and then $U_J$ is also equal to the image of $U_\infty$ in $X_J$, as $X_\infty\to X_J$ is surjective. Thus, the image of $U_\infty$ in $X$ agrees with the image of $U_J$ in $X$, which is open, as \'etale maps are open. The same argument applies after any base change, so $X_\infty\to X$ is universally open. Repeating the construction $X\mapsto X_\infty$ countably often (here, we use our assumption that countable unions of cardinals less than $\kappa$ are less than $\kappa$), we can find an affinoid pro-\'etale $\tilde{X}\to X$ such that all \'etale covers of $\tilde{X}$ split. Indeed, any \'etale cover can be refined by an affinoid \'etale cover. Such a cover comes from some finite level by Proposition~\ref{prop:etmaptolim}~(iv), and becomes split at the next.
\end{proof}

In the following, we will analyze perfectoid spaces which are pro-\'etale over a strictly totally disconnected space.

\begin{lemma}\label{lem:proetaleoverwlocal} Let $X$ be a strictly totally disconnected perfectoid space. Let $f: Y\to X$ be a quasicompact separated map of perfectoid spaces. Then $f$ is pro-\'etale if and only if for all rank-$1$-points $x=\Spa(C,\OO_C)\in X$, the fibre $Y_x = Y\times_X x$ is isomorphic to $x\times \underline{S_x}$ for some profinite set $S_x$. In this case, $Y$ is strictly totally disconnected, and $Y\to X$ is affinoid pro-\'etale.
\end{lemma}

Note that the lemma applies in particular if $Y$ is affinoid, as any map of affinoid perfectoid spaces is separated.

\begin{proof} The condition is clearly necessary. Thus assume that $f: Y\to X$ has the property that for all rank-$1$-points $x=\Spa(C,\OO_C)\in X$, the fibre $Y_x = Y\times_X x$ is isomorphic to $x\times \underline{S_x}$ for some profinite set $S_x$. We want to show that $Y$ can be written as an inverse limit $Y=\varprojlim_i Y_i\to X$ of \'etale maps $Y_i\to X$, where all $Y_i$ are affinoid.

First, note that $f: Y\to X$ factors as $Y\to X\times_{\pi_0(X)} \pi_0(Y)\to X$, where $X\times_{\pi_0(X)} \pi_0(Y)\to X$ is affinoid pro-\'etale, and $X\times_{\pi_0(X)} \pi_0(Y)$ is strictly totally disconnected. Thus, replacing $X$ by $X\times_{\pi_0(X)} \pi_0(Y)$, we can reduce to the case that $\pi_0(f): \pi_0(Y)\to \pi_0(X)$ is a homeomorphism.

We claim that in this situation, $f$ is an injection. To check this, we can check on connected components, so we may assume that $X=\Spa(C,C^+)$, and then by assumption also $Y$ is connected. To check that $f$ is injective, it is by Proposition~\ref{prop:charinjection} enough to check injectivity on $(K,K^+)$-valued points, and then as $f$ is separated, it is enough to check on $(K,\OO_K)$-valued points. Thus, assume two rank-$1$-points $y_1,y_2\in Y$ map to the same point $x\in X$. Now $x$ is the unique rank-$1$-point of $X$, and the set of preimages of $x$ forms a profinite set $Y_x\subset Y$ by assumption. By our assumption, $Y_x$ contains at least two points, so we can find a closed and open decomposition $Y_x=U_{1,x}\sqcup U_{2,x}$, for some quasicompact open subsets $U_1,U_2\subset Y$. Let $V_1, V_2\subset Y$ be the closures of $U_{1,x}$ and $U_{2,x}$. As $U_{1,x}$ and $U_{2,x}$ are pro-constructible subsets, their closures are precisely the subsets of specializations of points in $U_{1,x}$ resp.~$U_{2,x}$. As any point of $Y$ generalizes to a unique point of $Y_x$, $Y=V_1\sqcup V_2$. As $V_1$ and $V_2$ are both closed, this gives a contradiction to our assumption that $Y$ is connected, finishing the proof that $f$ is an injection.

Now, since $|f|: |Y|\to |X|$ is injective, the image of $|f|$ is a pro-constructible generalizing subset of $X$, and thus by Lemma~\ref{lem:subsetwlocal} an intersection of subsets of the form $\{|g|\leq 1\}$ for varying $g\in H^0(X,\OO_X)$. As such, the image is affinoid pro-\'etale in $X$, and itself (strictly) totally disconnected. Thus, replacing $X$ by the image of $f$, we can assume that $|f|$ is a bijection. But then $f$ is an isomorphism by Lemma~\ref{lem:bijectiveiso}.
\end{proof}

Lemma~\ref{lem:bijectiveiso} allows us to give a purely topological description of perfectoid spaces which are pro-\'etale over a strictly totally disconnected space.

\begin{definition} Let $T$ be a spectral space such that every connected component is a totally ordered chain of specializations. A spectral map $S\to T$ of spectral spaces is called affinoid pro-\'etale if the induced map $S\to T\times_{\pi_0(T)} \pi_0(S)$ is a pro-constructible and generalizing embedding. A spectral map $S\to T$ from a locally spectral space $S$ is called pro-\'etale if $S$ is covered by spectral subsets $S^\prime\subset S$ for which the restriction $S^\prime\to T$ is affinoid pro-\'etale.
\end{definition}

\begin{remark} If $S\to T$ is pro-\'etale and $S$ is spectral, it is not automatic that $S\to T$ is affinoid pro-\'etale; the analogue of the separatedness condition can fail.
\end{remark}

\begin{corollary}\label{cor:proetaleoverwlocaltop} Let $X$ be a strictly totally disconnected perfectoid space. Then the functor sending a map $f: Y\to X$ to $|f|: |Y|\to |X|$ induces equivalences between:
\begin{enumerate}
\item[{\rm (i)}] the category of ($\kappa$-small) affinoid pro-\'etale perfectoid spaces $Y\to X$, and the category of affinoid pro-\'etale $S\to |X|$ (of cardinality less than $\kappa$);
\item[{\rm (ii)}] the category of ($\kappa$-small) pro-\'etale perfectoid spaces $Y\to X$, and the category of pro-\'etale $S\to |X|$ (of cardinality less than $\kappa$).
\end{enumerate}
\end{corollary}

\begin{proof} Lemma~\ref{lem:proetaleoverwlocal} implies that if $f: Y\to X$ is (affinoid) pro-\'etale, then $|f|: |Y|\to |X|$ is (affinoid) pro-\'etale. We construct an inverse functor. For this, given a pro-\'etale map $g: S\to |X|$, set $\OO_S^+ := (g^\ast \OO_X^+)^\wedge_\varpi$ and $\OO_S = \OO_S^+[\varpi^{-1}]$, where $\varpi$ is a pseudouniformizer on the affinoid space $X$. For any point of $S$, one gets an induced valuation on $\OO_S$ by pullback from $X$. One checks using Lemma~\ref{lem:subsetwlocal} that if $g: S\to |X|$ is affinoid pro-\'etale, then this defines an affinoid perfectoid space $Y$; Lemma~\ref{lem:proetaleoverwlocal} implies that the corresponding map $f: Y\to X$ is affinoid pro-\'etale. In general, $S\to |X|$ is locally affinoid pro-\'etale, so again one gets a perfectoid space $Y$, pro-\'etale over $X$.
\end{proof}

Another important property of totally disconnected spaces is an automatic flatness assertion.

\begin{proposition}\label{prop:wlocalflat} Let $X=\Spa(R,R^+)$ be a totally disconnected perfectoid space, and $f: Y=\Spa(S,S^+)\to X$ any map from an affinoid perfectoid space $Y$. Then $S^+/\varpi$ is flat over $R^+/\varpi$, for any pseudouniformizer $\varpi\in R$.

Moreover, if $|f|: |Y|\to |X|$ is surjective, then the map is faithfully flat.
\end{proposition}

\begin{proof} Let $g: |X|\to \pi_0(X)$ be the projection. First, we check that $((g\circ f)_\ast \OO_Y^+)/\varpi$ is flat over $(g_\ast \OO_X^+)/\varpi$, as sheaves on $\pi_0(X)$. This can be checked on stalks. But the stalk of $(g_\ast \OO_X^+)/\varpi$ at $c\in \pi_0(X)$ is given by $K_c^+/\varpi$, where $\Spa(K_c,K_c^+)\subset X$ is the connected component corresponding to $c$. Similarly, the stalk of $((g\circ f)_\ast \OO_Y^+)/\varpi$ is given by $S_c^+/\varpi$, where $\Spa(S_c,S_c^+)=Y\times_X \Spa(K_c,K_c^+)$. But $S_c^+$ is $\varpi$-torsionfree, which implies that it is flat over $K_c^+$. By base change, $S_c^+/\varpi$ is flat over $K_c^+/\varpi$, as desired.

As flatness can be checked locally, it follows that $S^+/\varpi$ is flat over $R^+/\varpi$. More precisely, flatness can be checked on connected components of $\Spec(R^+/\varpi)$, but $\pi_0(\Spec(R^+/\varpi)) = \pi_0(X)$, and on each connected component, we have verified the desired flatness.

To check faithful flatness in case $|f|$ is surjective, we have to see that $\Spec(S^+/\varpi)\to \Spec(R^+/\varpi)$ is surjective. This can be checked on connected components, so we can assume that $X=\Spa(K,K^+)$ for some perfectoid field $K$ with an open and bounded valuation subring $K^+\subset K$. In that case, as $|f|$ is surjective, we can find a map $\Spa(L,L^+)\to Y$ with $L$ a perfectoid field with an open and bounded valuation subring $L^+\subset L$, such that $\Spa(L,L^+)\to \Spa(K,K^+)$ is surjective. But one can identify $\Spa(L,L^+)=\Spec(L^+/\varpi)$ and $\Spa(K,K^+)=\Spec(K^+/\varpi)$, so $\Spec(L^+/\varpi)\to \Spec(S^+/\varpi)\to \Spec(K^+/\varpi)$ is surjective, as desired.
\end{proof}

\section{The pro-\'etale and v-topology}

In this section, we define two Grothendieck topologies on the category of perfectoid spaces.

\begin{definition} Let $\Perfd$ be the category of perfectoid spaces and $\Perfd_\kappa\subset \Perfd$ the subcategory of $\kappa$-small perfectoid spaces.
\begin{altenumerate}
\item The big pro-\'etale site is the Grothendieck topology on $\Perfd$ (resp.~$\Perfd_\kappa$) for which a collection $\{f_i: Y_i\to X\}_{i\in I}$ of morphisms is a covering if all $f_i$ are pro-\'etale, and for each quasicompact open subset $U\subset X$, there exists a finite subset $J\subset I$ and quasicompact open subsets $V_i\subset Y_i$ for $i\in J$, such that $U=\bigcup_{i\in J} f_i(V_i)$.
\item Let $X$ be a perfectoid space (resp.~$\kappa$-small perfectoid space). The pro-\'etale site $X_\proet$ (resp.~$X_{\proet,\kappa}$) of $X$ is the Grothendieck topology on the category of ($\kappa$-small) perfectoid spaces $f: Y\to X$ pro-\'etale over $X$, with covers the same as in the big pro-\'etale site.
\item The v-site is the Grothendieck topology on $\Perfd$ (resp.~$\Perfd_\kappa$) for which a collection $\{f_i: Y_i\to X\}_{i\in I}$ of morphisms is a covering if for each quasicompact open subset $U\subset X$, there exists a finite subset $J\subset I$ and quasicompact open subsets $V_i\subset Y_i$ for $i\in J$, such that $U=\bigcup_{i\in J} f_i(V_i)$.
\end{altenumerate}
\end{definition}

We note that the slightly technical condition on quasicompact subsets is the same condition that appears in the definition of the fpqc topology for schemes. Also, there is no ``small'' v-site of a perfectoid space $X$, as this would by design include all perfectoid spaces over $X$.

Regarding the set-theoretic bounds, we have the following very general result.

\begin{proposition}\label{prop:changekappa} For any $\kappa,\kappa'$ as in Lemma~\ref{lem:choosekappa}, and any $\kappa$-small perfectoid space $X$, the pullback functor $c_{\kappa,\kappa'}^\ast$ from sheaves on $X_{\proet,\kappa}$ to $X_{\proet,\kappa'}$ is fully faithful and preserves cohomology; i.e.~for any sheaf of sets (resp.~of groups, resp.~of abelian groups) $\mathcal F$ on $X_{\proet,\kappa}$, the map
\[
\mathcal F\to c_{\kappa,\kappa',\ast} c_{\kappa,\kappa'}^\ast \mathcal F
\]
is an isomorphism of sheaves (resp.~and
\[
R^1c_{\kappa,\kappa',\ast} c_{\kappa,\kappa'}^\ast \mathcal F=0,
\]
resp.~and
\[
R^i c_{\kappa,\kappa',\ast} c_{\kappa,\kappa'}^\ast \mathcal F=0
\]
for all $i>0$).

The similar result holds for pro-\'etale or v-cohomology on $\Perfd_\kappa$, $\Perfd_{\kappa'}$, and their slices.
\end{proposition}

\begin{proof} We may restrict to affinoid $X$. Then any $\kappa'$-small affinoid $Y\to X$ can be written as a $\kappa$-cofiltered limit of affinoid $Y_j\to X$, by writing its ring of functions as a $\kappa$-filtered colimit of $\kappa$-small perfectoid rings. If $Y\to X$ is affinoid pro-\'etale, one can assume all $Y_j\to X$ to be affinoid pro-\'etale. Accordingly, the pullback $c_{\kappa,\kappa'}^\ast \mathcal F$ is the sheafification of $Y\mapsto \varinjlim_j \mathcal F(Y_j)$. We claim that this is already a sheaf, and that similarly $H^i(Y,\mathcal F)=\varinjlim_j H^i(Y_j,\mathcal F)$ for all relevant $i$ (i.e.~$i=0$ if $\mathcal F$ is a sheaf of sets, $i=0,1$ if $\mathcal F$ is a sheaf of groups, all $i\geq 0$ if $\mathcal F$ is a sheaf of abelian groups). To check this, write any cover $Y'\to Y$ as a $\kappa$-cofiltered limit of covers of $\kappa$-small perfectoid spaces (in the relevant site), and then pass to $\kappa$-filtered colimits of the corresponding \v{C}ech spectral sequences, implying the desired result.
\end{proof}

Thus, the choice of $\kappa$ is immaterial to cohomology. Passing to a large filtered colimit over all $\kappa$, we define the category of small sheaves on $X_\proet$ as
\[
X_\proet^\sim := \varinjlim_\kappa X_{\proet,\kappa}^\sim
\]
where for any small site $S$, we let $S^\sim$ be its category of sheaves. Then $X_\proet^\sim$ is generated under small colimits by the sheaves represented by the objects of $X_\proet$. (This would not be the case if one considers all sheaves on $X_\proet$ as $X_\proet$ is a large category.) Similar remarks apply to the other large sites in use.

We note that if $\mathcal F$ is a sheaf that commutes with $\omega_1$-filtered colimits, i.e.~for any $\omega_1$-filtered diagram of perfectoid $(R_i,R_i^+)$ with colimit $(R,R^+)$, the map
\[
\varinjlim_i \mathcal F(\Spa(R_i,R_i^+))\to \mathcal F(\Spa(R,R^+))
\]
is an isomorphism, then $\mathcal F$ is small. Indeed, the proof of Proposition~\ref{prop:changekappa} shows that $\mathcal F$ arises via pullback from the $\kappa$-small site for any uncountable $\kappa$. In practice, all sheaves $\mathcal F$ of actual interest have this property.

Our goal in this section is prove that both the pro-\'etale and the v-site are well-behaved. We start with some general remarks about quasicompact and quasiseparated objects. Let us quickly recall their definition in any topos $T$:

\begin{altenumerate}
\item An object $X\in T$ is quasicompact if for any collection of objects $X_i\in T$, $i\in I$, with maps $f_i: X_i\to X$ such that $\bigsqcup f_i: \bigsqcup X_i\to X$ is surjective, there is a finite subset $J\subset I$ such that $\bigsqcup_{i\in J} X_i\to X$ is surjective, cf.~SGA 4 VI D\'efinition 1.1.
\item An object $X\in T$ is quasiseparated if for all quasicompact objects $Y,Z\in T$, the fibre product $Y\times_X Z\in T$ is quasicompact, cf.~cf.~SGA 4 VI D\'efinition 1.13.
\item A map $f: Y\to X$ is quasicompact if for all quasicompact $Z\in T$, the fibre product $Z\times_X Y$ is quasicompact, cf.~SGA 4 VI D\'efinition 1.7.
\item A map $f: Y\to X$ is quasiseparated if $\Delta_f: Y\to Y\times_X Y$ is quasicompact, cf.~SGA 4 VI D\'efinition 1.7.
\end{altenumerate}

The combination of quasicompact and quasiseparated is abbreviated to ``qcqs''. The same definitions also apply for the categories $X_\proet^\sim$ that we associated to large sites (and which are not quite topoi).

In the following, we assume that $T$ is an \emph{algebraic} topos, cf.~SGA 4 VI D\'efinition 2.3. This means that there is a generating full subcategory $C\subset T$ consisting of qcqs objects that is stable under fibre products, and moreover for $X\in C$, the map $X\to\ast$ is quasiseparated. This assumption will be satisfied in all our examples. Under this assumption, it suffices to check these various conditions after pullback to objects of $C$; also, it suffices to check after one cover by objects of $C$. Thus, for example $X\in T$ is quasiseparated if and only if for one cover of $X$ by objects $X_i\in C$, all $X_i\times_X X_j$ are quasicompact, cf.~SGA 4 VI Corollaire 1.17. Also, $f: Y\to X$ is quasiseparated if and only if for all $Z\in C$ with a map $Z\to X$, the object $Y\times_X Z$ is quasiseparated, cf.~SGA 4 VI Corollaire 2.6. Another equivalent characterization is that for all quasiseparated $Z\in T$ with a map $Z\to X$, the object $Y\times_X Z$ is quasiseparated, cf.~SGA 4 VI Corollaire 2.8.

We warn the reader that if the final object of $T$ is not quasiseparated (as in the case for (sheaves on) $\Perfd$), it is not equivalent to say that $X$ is quasicompact (or quasiseparated) and to say that the map $X\to\ast$ is quasicompact (or quasiseparated). It is true that $X$ quasiseparated implies $X\to\ast$ is quasiseparated, and that $X\to\ast$ being quasicompact implies that $X$ is quasicompact. Indeed, the former follows from our assumption that $T$ is algebraic by SGA 4 VI Proposition 2.2, and the latter follows from SGA 4 VI Proposition 1.3, by choosing any quasicompact object $Y\in T$, and the corresponding cover $X\times Y\to X$, where $X\times Y$ is quasicompact if $X\to \ast$ is.

\begin{proposition} The categories of small sheaves on $\Perfd$ for either the (big) pro-\'etale or v-topology, and the category of small sheaves on $X_\proet$ for a perfectoid space $X$, are algebraic. A basis of qcqs objects stable under fibre products is in all cases given by affinoid perfectoid spaces. Moreover, a perfectoid space $X$ is quasicompact resp.~quasiseparated in any of these settings if and only if $|X|$ is quasicompact resp.~quasiseparated.
\end{proposition}

\begin{proof} Left to the reader.
\end{proof}

\begin{convention}\label{conv:stackquasisep} Further below, we will sometimes ask that a map $f: Y^\prime\to Y$ of stacks on some topos $T$ is quasiseparated. Following the convention of \cite[Tag 04YW]{StacksProject}, we define this to mean that $\Delta_f: Y^\prime\to Y^\prime\times_Y Y^\prime$ is quasicompact \emph{and quasiseparated}, where the latter condition is not automatic for stacks (as the diagonal need not be injective).
\end{convention}

\begin{proposition}\label{prop:compareetproet} Let $X$ be a perfectoid space, and $\nu: X_\proet\to X_\et$ the natural map of sites.
\begin{altenumerate}
\item For any sheaf $\mathcal{F}$ on $X_\et$, the natural adjunction map $\mathcal{F}\to \nu_\ast \nu^\ast \mathcal{F}$ is an equivalence. If $\mathcal{F}$ is a sheaf of abelian groups, then $R^i\nu_\ast \nu^\ast \mathcal{F}=0$ for all $i\geq 1$.
\item Let $Y=\varprojlim_i Y_i\to X_0\subset X$ be a pro-\'etale presentation of an affinoid pro-\'etale map $Y\to X_0$ to an affinoid open subset $X_0\subset X$, and let $\mathcal{F}$ be a sheaf on $X_\et$. The natural map
\[
\varinjlim_i \mathcal{F}(Y_i)\to (\nu^\ast \mathcal{F})(Y)
\]
is an isomorphism.
\item The presheaves $\OO$, $\OO^+$ on $X_\proet$ sending $Y\in X_\proet$ to $\OO_Y(Y)$ resp.~$\OO_Y^+(Y)$, are small sheaves. If $X$ is affinoid perfectoid, then $H^i(X_\proet,\OO)=0$ for $i>0$, and $H^i(X_\proet,\OO^+)$ is almost zero for $i>0$.
\end{altenumerate}
\end{proposition}

\begin{proof} All assertions reduce to the case that $X$ is affinoid. In this case, the sites $X^\aff_\proet$ and $X_\proet$ define equivalent topoi, so we can work with $X^\aff_\proet$ instead. Let $\mathcal{F}$ be a sheaf on $X_\et$. It follows from Proposition~\ref{prop:affproetale} that $\nu^\ast \mathcal{F}$ is the sheafification of the presheaf $\nu^+\mathcal F$ sending a pro-\'etale presentation $Y=\varprojlim_i Y_i\to X$ to
\[
(\nu^+\mathcal F)(Y)=\varinjlim_i \mathcal{F}(Y_i)\ ;
\]
in fact, this presheaf is the pullback on the level of presheaves. To prove (ii), one argues as in \cite[Lemma 5.1.1]{BhattScholze}. More precisely, we want to show that $\nu^+\mathcal F$ is a sheaf. Thus, for any surjective map $\tilde{Y}\to Y$ in $X^\aff_\proet$, we need to see that
\[
(\nu^+\mathcal F)(Y)\to \eq((\nu^+\mathcal F)(\tilde{Y})\rightrightarrows (\nu^+\mathcal F)(\tilde{Y}\times_Y \tilde{Y}))
\]
is an isomorphism. But we can write $\tilde{Y}=\varprojlim_i \tilde{Y}_i\to Y$ as a cofiltered limit of affinoid \'etale surjective $\tilde{Y}_i\to Y$, and
\[
(\nu^+\mathcal F)(\tilde{Y}) = \varinjlim_i (\nu^+\mathcal F)(\tilde{Y}_i)\ ,
\]
and similarly for $(\nu^+\mathcal F)(\tilde{Y}\times_Y \tilde{Y})$, which reduces us to the case that $\tilde{Y}\to Y$ is \'etale. In that case, writing $Y=\varprojlim_j Y_j$ as a cofiltered limit of affinoid \'etale $Y_j\to X$, the map $\tilde{Y}\to Y$ comes as the pullback of some affinoid \'etale $\tilde{Y}_j\to Y_j$ for $j$ large enough by Proposition~\ref{prop:etmaptolim}~(iv). Then
\[
\eq((\nu^+\mathcal F)(\tilde{Y}^\prime)\rightrightarrows (\nu^+\mathcal F)(\tilde{Y}^\prime\times_Y \tilde{Y}^\prime)) = \varinjlim_j \eq(\mathcal F(\tilde{Y}_j^\prime)\rightrightarrows \mathcal F(\tilde{Y}_j^\prime\times_{Y_j} \tilde{Y}_j^\prime)) = \varinjlim_j \mathcal F(Y_j) = \mathcal F(Y)\ ,
\]
as desired.

Part (i) is proved similarly, following \cite[Corollary 5.1.6]{BhattScholze}. Finally, for part (iii), choose a pseudouniformizer $\varpi\in \OO_X^+(X)$. Then by \cite[Proposition 7.13]{ScholzePerfectoidSpaces}, we have a sheaf of almost $\OO_X^+(X)/\varpi$-modules on $X^\aff_\et$, sending $Y\in X^\aff_\et$ to $(\OO_Y^+(Y)/\varpi)^a$. By part (ii), the pullback of this sheaf to the pro-\'etale site sends $Y=\varprojlim_i Y_i\to X$ to
\[
(\varinjlim_i \OO_{Y_i}^+(Y_i)/\varpi)^a = (\OO_Y^+(Y)/\varpi)^a\ ,
\]
so sending $Y\in X^\aff_\proet$ to $(\OO_Y^+(Y)/\varpi)^a$ is a sheaf of almost $\OO_X^+(X)/\varpi$-modules. By induction, one gets sheaves $(\OO_Y^+(Y)/\varpi^n)^a$ for all $n\geq 1$, and then passing to the inverse limit, a sheaf $Y\mapsto \OO_Y^+(Y)^a$; inverting $\varpi$ gives a sheaf $\OO: Y\mapsto \OO_Y(Y)$. All of these sheaves have vanishing higher cohomology groups on $X$ by part (i) and \cite[Proposition 7.13]{ScholzePerfectoidSpaces}. Finally, $\OO^+\subset \OO$ is the subpresheaf of those functions which are everywhere of absolute $\leq 1$, and so is a sheaf as $\OO$ is.

One also sees that $\OO$ and $\OO^+$ are small as they commute with $\omega_1$-filtered colimits of affinoid perfectoid rings $(R,R^+)$, see Proposition~\ref{prop:limitaffperfd}.
\end{proof}

\begin{corollary}\label{cor:proetalesub} The presheaves $\OO: X\mapsto \OO_X(X)$ and $\OO^+: X\mapsto \OO_X^+(X)$ on the big pro-\'etale site are small sheaves. Moreover, the big pro-\'etale site is subcanonical, i.e.~for every perfectoid space $X$, the functor $Y\mapsto \Hom(Y,X)$ is a small sheaf for the big pro-\'etale site.
\end{corollary}

\begin{proof} As any cover $\{f_i: Y_i\to X\}$ in the big pro-\'etale site can also be regarded as a cover in the small pro-\'etale site of $X$, these questions reduce to the small pro-\'etale site of $X$. Then the first part follows from Proposition~\ref{prop:compareetproet}~(iii). For the second part, assume that $\{f_i: Y_i\to Y\}$ is a pro-\'etale cover of $Y$, and one has maps $g_i: Y_i\to X$ which agree on overlaps $Y_i\times_Y Y_j$. Note first that by Lemma~\ref{lem:quotientmap}, the maps $|Y_i|\to |X|$ glue to a continuous map $|Y|\to |X|$; in particular, the problem can be considered locally on $X$, so we can assume that $X=\Spa(R,R^+)$ is affinoid. Then we are given maps $(R,R^+)\to (\OO(Y_i),\OO^+(Y_i))$ which agree on overlaps; as $\OO$ and $\OO^+$ are sheaves for the pro-\'etale site, these maps glue to a map $(R,R^+)\to (\OO(Y),\OO^+(Y))$, which gives a map $Y\to X$, as desired.

Smallness is clear, as representable sheaves are preserved by pullback, so if $X$ is $\kappa$-small it comes via pullback from the $\kappa$-small site.
\end{proof}

Next, we want to extend these results to the v-site. Here, the strategy is to combine pro-\'etale localization with the automatic faithful flatness results over totally disconnceted bases, Proposition~\ref{prop:wlocalflat}.

\begin{theorem}\label{thm:vsub} The presheaves $\OO: X\mapsto \OO_X(X)$ and $\OO^+: X\mapsto \OO_X^+(X)$ on the v-site are small sheaves. Moreover, the v-site is subcanonical.
\end{theorem}

\begin{proof} As in Corollary~\ref{cor:proetalesub}, it is enough to show that $\OO$ and $\OO^+$ are sheaves. As $\OO^+\subset \OO$ is the subpresheaf of functions which are everywhere of absolute value $\leq 1$ (which can be checked after v-covers), it is enough to prove that $\OO$ is a sheaf. As for every perfectoid space $X$, $\OO(X)$ injects into $\prod_{x\in X} K(x)$, it is clear that $\OO$ is separated. It remains to see that given a v-cover $\{f_i: Y_i\to X\}$ of a perfectoid space and functions $g_i\in \OO(Y_i)$ which agree on overlaps $Y_i\times_X Y_j$, they glue uniquely to a function $g\in \OO(X)$. This can be checked locally on $X$, so we can assume that $X$ is affinoid. In that case, the quasicompactness assumptions on a v-cover ensure that we can refine the given cover into a single $f: Y\to X$, where $Y$ is affinoid. Let $\tilde{X}\to X$ be a totally disconnected cover of $X$ along an affinoid pro-\'etale map (for example the w-localization). We can also replace $f$ by the composite $Y\times_X \tilde{X}\to \tilde{X}\to X$. By Corollary~\ref{cor:proetalesub}, we have descent for $\tilde{X}\to X$, so it is enough to handle the map $Y\times_X \tilde{X}\to \tilde{X}$; in other words, we can assume that $X$ is totally disconnected, so let $f: Y=\Spa(S,S^+)\to X=\Spa(R,R^+)$ be a map of affinoid perfectoid spaces, where $X$ is totally disconnected.

But now $S^+/\varpi$ is flat over $R^+/\varpi$ by Proposition~\ref{prop:wlocalflat}. In fact, as $f$ is surjective, it is faithfully flat, so $R^+/\varpi$ is the equalizer of the two maps from $S^+/\varpi$ to $S^+/\varpi\otimes_{R^+/\varpi} S^+/\varpi$. But if $Y\times_X Y=\Spa(T,T^+)$, then the map $S^+/\varpi\otimes_{R^+/\varpi} S^+/\varpi\to T^+/\varpi$ is almost an isomorphism by~\cite[proof of Proposition 6.18]{ScholzePerfectoidSpaces}. Thus, $R^+/\varpi$ is almost the equalizer of the two maps from $S^+/\varpi$ to $T^+/\varpi$; passing to the limit over the similar statement for $\varpi^n$ and inverting $\varpi$ shows that $R$ is the equalizer of the two maps from $S$ to $T$, as desired.
\end{proof}

In fact, the vanishing of cohomology also extends to the v-site.

\begin{proposition}\label{prop:vcohomO} Let $X$ be an affinoid perfectoid space. Then $H^i_v(X,\OO)=0$ for $i>0$, and $H^i_v(X,\OO^+)$ is almost zero for $i>0$.
\end{proposition}

\begin{proof} Assume that $X$ is totally disconnected. First, we check vanishing of Cech cohomology, i.e.~for any v-cover $f: Y=\Spa(S,S^+)\to X=\Spa(R,R^+)$, we show that the complex
\[
0\to R^+\to S^+\to \OO^+(Y\times_X Y)\to \ldots
\]
is almost exact. As everything is $\varpi$-adically complete and $\varpi$-torsionfree, where $\varpi\in R^+$ is a pseudouniformizer, we can check this modulo $\varpi$. But then the complex is almost the same as
\[
0\to R^+/\varpi\to S^+/\varpi\to S^+/\varpi\otimes_{R^+/\varpi} S^+/\varpi\to \ldots\ ,
\]
which is exact by Proposition~\ref{prop:wlocalflat}. Similarly, for general affinoid $X$, the Cech complex for the affinoid pro-\'etale cover $X^\wl\to X$ is almost exact by Proposition~\ref{prop:compareetproet}~(iii).

We claim that this implies that $H^i_v(X,\OO^+)$ is almost zero for $i>0$ for affinoid $X$. Indeed, choose $i$ minimal for which $H^i_v(X,\OO^+)$ is not almost zero for all affinoid $X$, and choose $X$ and a class $\alpha\in H^i_v(X,\OO^+)$ that is not almost zero. Let $f: \tilde{X}\to X$ be a totally disconnected cover; if $f^\ast(\alpha)$ is almost zero in $H^i_v(\tilde{X},\OO^+)$, a Cech-to-sheaf cohomology spectral sequence gives that $\check{H}^i(\tilde{X}/X,\OO^+)$ is not almost zero, which is a contradiction. Thus, we can replace $X$ by $\tilde{X}$ and assume that $X$ is totally disconnected. But there is some v-cover, without loss of generality by an affinoid $f: Y\to X$, such that $f^\ast(\alpha)=0$. A Cech-to-sheaf cohomology spectral sequence would then give that $\check{H}^i(Y/X,\OO^+)$ is not almost zero, which contradicts the first paragraph.
\end{proof}

\section{Descent}

In this section, we establish some descent results for perfectoid spaces. As a piece of general notation, we define the category of descent data as follows.

\begin{definition} Let $\mathcal C$ be a site with fibre products, and let $F$ be a prestack on $\mathcal C$, i.e.~a functor from $\mathcal C^\op$ to groupoids. Let $f: Y\to X$ be a map in $\mathcal C$, $p_1,p_2: Y\times_X Y\to Y$ the two projections, and $p_{12},p_{13},p_{23}: Y\times_X Y\times_X Y\to Y\times_X Y$ the three projections. Then
\[
F(Y/X) := \{(s,\alpha)\mid s\in F(Y), \alpha: p_1^\ast(s)\cong p_2^\ast(s)\in F(Y\times_X Y), p_{23}^\ast \alpha\circ p_{12}^\ast \alpha  = p_{13}^\ast \alpha\}\ ,
\]
which comes with a natural map $F(X)\to F(Y/X)$, sending $t\in F(X)$ to $(s,\alpha)$ with $s=f^\ast(t)$ and $\alpha$ the natural identification $p_1^\ast(s) = (p_1\circ f)^\ast(t) = (p_2\circ f)^\ast(t) = p_2^\ast(s)$.
\end{definition}

This is the category of objects in $F(Y)$ equipped with a descent datum to $X$. Note that if $f$ is a covering and $F$ is a stack, then the natural map $F(X)\to F(Y/X)$ is an equivalence.

In this section, we will be interested in several examples of prestacks $F$ on the pro-\'etale or v-site of perfectoid spaces.

\begin{lemma}\label{lem:descentfullyfaithful} Let $F$ be the prestack on the category of perfectoid spaces sending $X$ to the groupoid of perfectoid spaces over $X$. Let $X$ be a perfectoid space, and $Y\to X$ a v-cover. The functor $F(X)\to F(Y/X)$ is fully faithful.
\end{lemma}

\begin{proof} Given two perfectoid spaces $X_1,X_2\to X$, we have to see that we can glue morphisms over v-covers. This follows formally from the fact that the v-site is subcanonical.
\end{proof}

In the following, we want to prove that in some instances, the map is an equivalence.

\begin{proposition}\label{prop:affinoidwlocaldescent} Let $F$ be the prestack on the category of affinoid perfectoid spaces, sending $X$ to the groupoid of affinoid perfectoid spaces over $X$. Let $X$ be a totally disconnected perfectoid space, and $Y\to X$ a v-cover, where $Y$ is an affinoid perfectoid space. The functor $F(X)\to F(Y/X)$ is an equivalence of categories.
\end{proposition}

\begin{proof} Let $X=\Spa(R,R^+)$, $Y=\Spa(S,S^+)$, and $\tilde{Y}=\Spa(\tilde{S},\tilde{S}^+)$ an affinoid perfectoid space over $Y$ with a descent datum to $X$. Choose a pseudouniformizer $\varpi\in R^+$ dividing $p$. By Proposition~\ref{prop:wlocalflat}, $S^+/\varpi$ is faithfully flat over $R^+/\varpi$; in particular, $(S^+/\varpi)^a$ is faithfully flat over $(R^+/\varpi)^a$. One can then use faithfully flat descent in the almost world, cf.~\cite[Section 3.4.1]{GabberRamero}, to see that the $(S^+/\varpi)^a$-algebra $(\tilde{S}^+/\varpi)^a$ descends to an $(R^+/\varpi)^a$-algebra, which is perfectoid in the sense of~\cite[Definition 5.1 (iii)]{ScholzePerfectoidSpaces}, and thus, by~\cite[Theorem 5.2]{ScholzePerfectoidSpaces}, is of the form $(\tilde{R}^\circ/\varpi)^a$, where $\tilde{R}$ is a perfectoid $R$-algebra. In particular, $\tilde{R}\hat{\otimes}_R S=\tilde{S}$.

It remains to see that one can find the correct $\tilde{R}^+\subset \tilde{R}$. Let $\tilde{R}^+_\min\subset \tilde{R}$ be the integral closure of $R^+ + \tilde{R}^{\circ\circ}\subset \tilde{R}^\circ$, and $\tilde{X}^\prime = \Spa(\tilde{R},\tilde{R}^+_\min)$. We get a map $\tilde{Y}\to \tilde{X}^\prime$ such that the induced map $\tilde{Y}\to \tilde{Y}^\prime := \tilde{X}^\prime\times_X Y$ is a pro-(open immersion): It is cut out by the conditions $|f|\leq 1$ for all $f\in \tilde{S}^+$. Let $|\tilde{X}|\subset \tilde{X}^\prime$ denote the image of $\tilde{Y}\to \tilde{X}^\prime$; then $\tilde{Y}\subset \tilde{Y}^\prime$ is precisely the preimage of $\tilde{X}\subset \tilde{X}^\prime$ (as $|\tilde{X}^\prime|$ is the quotient of $|\tilde{Y}^\prime|$ by the equivalence relation which is the image of
\[
|\tilde{Y}^\prime\times_{\tilde{X}^\prime} \tilde{Y}^\prime|\to |\tilde{Y}^\prime|\times |\tilde{Y}^\prime|\ ).
\]
We need to see that $|\tilde{X}|\subset \tilde{X}^\prime$ is cut out by conditions $|f|\leq 1$ for elements $f\in \tilde{R}$. This is the content of the following lemma.
\end{proof}

\begin{lemma} Let $X=\Spa(R,R^+)$ be a totally disconnected affinoid perfectoid space, let $\tilde{X}=\Spa(\tilde{R},\tilde{R}^+)$ be an affinoid perfectoid space over $X$, and let $A\subset |\tilde{X}|$ be a subset. Assume that there is a surjective map $Y=\Spa(S,S^+)\to X$ such that the preimage $B$ of $A$ in $|\tilde{Y}|$, where $\tilde{Y}=\tilde{X}\times_X Y=\Spa(\tilde{S},\tilde{S}^+)$, is an intersection of subsets of the form $|g|\leq 1$ for elements $g\in \tilde{S}$. Then $A$ is an intersection of subsets of the form $|f|\leq 1$ for elements $f\in \tilde{R}$.
\end{lemma}

\begin{proof} We may assume that $X$ is of characteristic $p$, as the subsets of the form $|f|\leq 1$ are unchanged under tilting. Note that $B$ is pro-constructible, and thus the image $A$ of $B$ is also pro-constructible.

As a first reduction step, we reduce to the case that $X$ is connected. Let $\pi: X\to \pi_0(X)$ be the projection. Let $\tilde{x}\in \tilde{X}\setminus A$ be any point, and let $c\in \pi_0(X)$ be the connected component $X_c=\pi^{-1}(c)\subset X$ whose preimage contains $\tilde{x}$. If the result is true in the connected case, it holds true for the intersection $A_c$ of $A$ with $\tilde{X}_c := \tilde{X}\times_X X_c$. Thus, there is a function $f_c\in H^0(\tilde{X}_c,\OO_{\tilde{X}_c})$ such that $|f_c(\tilde{x})|>1$, but $|f_c|\leq 1$ on $A_c$. Modulo functions in $H^0(\tilde{X}_c,\OO^+_{\tilde{X}_c})$, $f_c$ can be lifted to a function $f_U\in H^0(\tilde{X}_U,\OO_{\tilde{X}_U})$ for some compact open neighborhood $U$ of $c$ in $\pi_0(X)$, where $\tilde{X}_U\subset\tilde{X}$ is the preimage of $U\subset \pi_0(X)$. Then $f_U$ still satisfies $|f_U(\tilde{x})|>1$, and the pro-constructible subset $\{|f_U|>1\}\cap A_U\subset \tilde{X}_U$ does not meet the fibre over $c$. As $\{|f_U|>1\}\cap A_U\subset \tilde{X}_U$ is a pro-constructible subset, it is a spectral space, and hence a quasicompactness argument implies that after shrinking $U$ to a smaller compact open neighborhood of $c$, $\{|f_U|>1\}\cap A_U=\emptyset$, so that $A_U\subset \{|f_U|\leq 1\}$. Now, extending $f_U$ to a function $f$ on all of $\tilde{X}$ by setting it to $0$ on the preimage of $\pi_0(X)\setminus U$, we get a function $f$ on $\tilde{X}$ with $|f(\tilde{x})|>1$, but $|f|\leq 1$ on $A$. The intersection of $\{|f|\leq 1\}$ over all such functions thus gives $A$.

Thus, we can assume that $R=K$ is a perfectoid field, and $R^+=K^+$ is an open and bounded valuation subring; fix a pseudouniformizer $\varpi\in K$. As a second reduction step, we reduce to the case that $K$ is algebraically closed. Indeed, let $C$ be a completed algebraic closure of $K$, and $C^+\subset C$ the completion of the integral closure of $K^+$ in $C$. If the result is known in the algebraically closed case, we can assume that $Y=\Spa(C,C^+)$. Fix a point $\tilde{x}\in \tilde{X}\setminus A$, and a lift $\tilde{y}\in \tilde{Y}\setminus B$ of $\tilde{x}$. By assumption, there is some function $g_{\tilde{y}}\in \tilde{S}$ with $|g_{\tilde{y}}(\tilde{y})|>1$, but $|g_{\tilde{y}}|\leq 1$ on $B$. Now $Y$ is an inverse limit of $\Spa(L,L^+)$ over all finite extensions $L\subset C$ of $K$ with integral closure $L^+\subset L$ of $K^+$ in $L$. Approximating the function $g_{\tilde{y}}$ modulo $\tilde{S}^+$, we can assume that $g_{\tilde{y}}$ is a function on $\tilde{X}\times_X \Spa(L,L^+)=:\Spa(\tilde{R}_L,\tilde{R}_L^+)$, for some big enough finite extension $L$ of $K$. In this case, $\tilde{R}_L = \tilde{R}\otimes_K L$ is a finite free $\tilde{R}$-algebra; consider the characteristic polynomial $P_{g_{\tilde{y}}}(X) = X^d + f_1 X^{d-1} + \ldots + f_d$ of $g_{\tilde{y}}$ acting on $\tilde{R}_L$, so all $f_i\in \tilde{R}$. As $|g_{\tilde{y}}|\leq 1$ on the preimage of $A$, it follows that $|f_i|\leq 1$ on $A$ for all $i=1,\ldots,d$. On the other hand, if we had $|f_i(\tilde{x})|\leq 1$ for all $i=1,\ldots,d$, then the equation
\[
0=P_{g_{\tilde{y}}}(g_{\tilde{y}}) = g_{\tilde{y}}^d + f_1 g_{\tilde{y}}^{d-1} + \ldots + f_d
\]
would imply that $|g_{\tilde{y}}(\tilde{y})|\leq 1$, which is a contradiction. Thus, at least one $f_i$ satisfies $|f_i|\leq 1$ on $A$ while $|f_i(\tilde{x})|>1$. Intersecting all such $f_i$ for varying $\tilde{x}$ thus writes $A$ as an intersection of sets of the form $\{|f|\leq 1\}$ for varying $f\in \tilde{R}$.

Finally, assume that $R=C$ is a complete algebraically closed nonarchimedean field of characteristic $p$, and $R^+=C^+$ is an open and bounded valuation subring. We may assume that similarly $S=L$ is a perfectoid field, and $S^+=L^+$ is an open and bounded valuation subring. Fix a point $\tilde{x}\in \tilde{X}\setminus A$, with preimage $\tilde{Y}_{\tilde{x}}\subset \tilde{Y}$, a pro-constructible subset. Let $x\in X$ be the image of $\tilde{x}$. By assumption, for any $\tilde{y}\in \tilde{Y}_{\tilde{x}}$, there is some function $g_{\tilde{y}}\in \tilde{S}$ with $|g_{\tilde{y}}(\tilde{y})|>1$, but $|g_{\tilde{y}}|\leq 1$ on $B$. By quasicompactness, we can find finitely many $g_1,\ldots,g_n$ such that $|g_i|\leq 1$ on $B$, but
\[
\tilde{Y}_{\tilde{x}}\subset \bigcup_{i=1}^n \{|g_i|>1\}\ .
\]
Up to an element of $\tilde{S}^+$, we can approximate $g_i\in \tilde{S} = S\hat{\otimes}_R \tilde{R}$ as a sum $\varpi^{-n_i}\sum_{j=1}^{m_i} s_{ij}\otimes \tilde{r}_{ij}$ with $s_{ij}\in S$, $\tilde{r}_{ij}\in \tilde{R}$. Let $Y^\prime = \Spa(S^\prime,S^{\prime +})$, with $S^\prime=C\langle (T_{ij}^{1/p^\infty})_{i=1,\ldots,n,j=1,\ldots,m_i}\rangle$, $S^{\prime +} = C^+\langle (T_{ij}^{1/p^\infty})_{i=1,\ldots,n,j=1,\ldots,m_i}\rangle$, and consider the map $Y\to Y^\prime$ given by $T_{ij}\mapsto s_{ij}$. Then there are functions $g_i^\prime=\varpi^{-n_i}\sum_{j=1}^{m_i} T_{ij}\otimes \tilde{r}_{ij}\in \tilde{S}^\prime$, where $\tilde{Y}^\prime= \tilde{X}\times_X Y^\prime = \Spa(\tilde{S}^\prime,\tilde{S}^{\prime +})$. These functions pull back to $g_i$, and in particular they satisfy $|g_i^\prime|\leq 1$ on $B$. Let $B^\prime\subset |\tilde{Y}^\prime|$ be the preimage of $A$. Let $f: \tilde{Y}^\prime\to Y^\prime$ be the projection. Let $W\subset Y^\prime$ be the image of $Y\to Y^\prime$; thus $W=\Spa(L^\prime,L^{\prime +})$ is the set of generalizations of a point $y^\prime\in Y^\prime$ (given by the image of the closed point of $Y$), and $W$ can thus be written as the intersection of all rational subsets $V\subset Y^\prime$ containing $W$. Then $B^\prime\cap f^{-1}(W)\subset \{|g_i^\prime|\leq 1\}$, as $B$ surjects onto $B^\prime\cap f^{-1}(W)$. By a quasicompactness argument, this implies that there is some rational neighborhood $V\subset Y^\prime$ of $W$ such that $B^\prime\cap f^{-1}(V)\subset \{|g_i^\prime|\leq 1\}$. Similarly, the fibre of $\tilde{Y}^\prime_{\tilde{x}}$ over $W$ is contained in $\bigcup_{i=1}^n \{|g_i^\prime|>1\}$; said differently, the spectral space
\[
M=\{\tilde{y}\in \tilde{Y}^\prime_{\tilde{x}}\mid \forall i=1,\ldots,n: |g_i^\prime(\tilde{y})|\leq 1\}
\]
has empty fiber over $W\subset Y^\prime$. By a quasicompactness argument, this implies that replacing $V$ by a smaller rational subset of $Y^\prime$ containing $W$, the fiber of $M$ over $V$ is empty. Thus, we can choose $V$ with the properties
\begin{equation}\label{eq:inclusion1}
B^\prime\cap f^{-1}(V)\subset \bigcap_{i=1}^n \{|g_i^\prime|\leq 1\}\ ,
\end{equation}
\begin{equation}\label{eq:inclusion2}
\tilde{Y}^\prime_{\tilde{x}}\cap f^{-1}(V)\subset \bigcup_{i=1}^n \{|g_i^\prime|>1\}\ .
\end{equation}
Now $V$ is a rational subset of
\[
\Spa(C\langle (T_{ij}^{1/p^\infty})_{i=1,\ldots,n,j=1,\ldots,m_i}\rangle, C^+\langle (T_{ij}^{1/p^\infty})_{i=1,\ldots,n,j=1,\ldots,m_i}\rangle)
\]
mapping surjectively to $\Spa(C,C^+)$. By Lemma~\ref{lem:secoveralgclosed} this implies that there is a $(C,C^+)$-valued point $z:\Spa(C,C^+)\to V$. Let $f_i\in \tilde{R}$ be the evaluation of $g_i^\prime$ at $z$. Then $|f_i|\leq 1$ on $A$ by~\eqref{eq:inclusion1}, but $\tilde{x}\in \bigcup_{i=1}^n \{|f_i|>1\}$ by~\eqref{eq:inclusion2}. This finishes the proof, as now some $f=f_i$ satisfies $|f|\leq 1$ on $A$, but $|f(\tilde{x})|>1$, and intersecting $\{|f|\leq 1\}$ over all such $f$ for varying $\tilde{x}$ gives $A$.
\end{proof}

\begin{lemma}\label{lem:secoveralgclosed} Let $C$ be an algebraically closed nonarchimedean field with open and bounded valuation subring $C^+\subset C$. Let $V$ be an open subset of
\[
\Spa(C\langle T_1,\ldots,T_d\rangle,C^+\langle T_1,\ldots,T_d\rangle)
\]
mapping surjectively to $\Spa(C,C^+)$. Then there is a section $\Spa(C,C^+)\to V$.
\end{lemma}

\begin{proof} We may assume that $V$ is rational, and in particular affinoid. Then $\mathfrak X=\Spf \OO^+(V)$ is a formal scheme which is flat and topologically of finite type (thus, of finite presentation) over $\Spf C^+$; moreover, $\mathfrak X\to \Spf C^+$ is surjective. We need to see that $\mathfrak X\to \Spf C^+$ has a section. First, $R=C^+/C^{\circ\circ}$ is a valuation ring with algebraically closed fraction field $K$, and $X=\mathfrak X\times_{\Spf C^+} \Spec R$ is a scheme of finite presentation and faithfully flat over $R$. This implies that $X(R)$ is nonempty: Any finitely presented faithfully flat cover of a strictly henselian local ring is refined by a finite flat locally free cover, cf.~\cite[after Lemma 11.2]{GrothBrauer3}, and finite flat covers of $\Spec R$ split. But $\mathfrak{X}(C^+)=\mathfrak{X}(\OO_C)\times_{X(K)} X(R)$, and the reduction map $\mathfrak{X}(\OO_C)\to X(K)$ is surjective.
\end{proof}

We need another result for morphisms which are not necessarily affinoid. We will however need to make a separatedness assumption.

\begin{proposition}\label{prop:sepproetaledescent} Let $F$ be the prestack on the category of perfectoid spaces sending a perfectoid space $X$ to the groupoid of separated pro-\'etale perfectoid spaces over $X$. Let $X$ be a strictly totally disconnected perfectoid space, and $Y\to X$ a v-cover. Then $F(X)\to F(Y/X)$ is an equivalence of categories.
\end{proposition}

\begin{proof} By Lemma~\ref{lem:descentfullyfaithful}, we only have to prove that all descent data are effective; for this, we may assume that $Y$ is strictly totally disconnected as well.

Let $R=\tilde{Y}\times_X Y\subset \tilde{Y}\times \tilde{Y}$ be the equivalence relation on $\tilde{Y}$ corresponding to the descent datum. We claim that the image $R^\prime\subset |\tilde{Y}|\times |\tilde{Y}|$ satisfies the hypothesis of Lemma~\ref{lem:equivrel}. First, $|\tilde{Y}|$ is a quasiseparated locally spectral space. As all maps of perfectoid spaces are generalizing, it follows that $R^\prime\to \tilde{Y}$ is generalizing; also, it is quasicompact (as it is a quotient of the map $|R|=|\tilde{Y}\times_X Y|\to |\tilde{Y}|$, and $Y\to X$ is quasicompact). Finally, $R^\prime$ is pro-constructible: For any quasicompact open subset $W\subset \tilde{Y}$, the intersection of $R^\prime$ with $|W|\times |W|$ is the image of the spectral space $W\times_X Y\cap Y\times_X W$ under the spectral map to $|W|\times |W|$.

Thus, by Lemma~\ref{lem:equivrel}, for any quasicompact open $W\subset \tilde{Y}$ we can find $R$-invariant subsets $U\subset E\subset \tilde{Y}$ such that $U$ is open in $\tilde{Y}$ and contains $W$, and $E$ is an intersection of a nonempty family of quasicompact open subsets. Thus, $E\subset \tilde{Y}$ corresponds by Lemma~\ref{lem:proetaleoverwlocal} to an affinoid pro-\'etale perfectoid space, which we will still denote by $E$. As $E$ is $R$-invariant, it comes with an induced descent datum, which is effective by Proposition~\ref{prop:affinoidwlocaldescent}. We find some affinoid perfectoid space $E_X\to X$ whose pullback to $Y$ is $E$; it follows (from Lemma~\ref{lem:proetaleoverwlocal}) that $E_X\to X$ is affinoid pro-\'etale. Now the open $R$-invariant subset $U\subset E$ descends to some open subset $U_X\subset E_X$, which is thus pro-\'etale over $X$. For varying $W$, the $U$ will cover $\tilde{Y}$, and thus the $U_X$ will glue to give the desired space $\tilde{X}\to X$ pro-\'etale over $X$.
\end{proof}

Finally, we need some stronger descent results for \'etale morphisms.

\begin{proposition}\label{prop:etalesepdescent} The prestack on the category of perfectoid spaces sending any perfectoid space $X$ to the groupoid of separated \'etale perfectoid spaces over $X$ is a stack in the v-topology.

Similarly, the prestack on the category of perfectoid spaces sending any perfectoid space $X$ to the groupoid of finite \'etale perfectoid spaces over $X$ is a stack in the v-topology.
\end{proposition}

\begin{proof} We give the proof in the more difficult case of separated \'etale maps; it is easy to see that the property of being finite \'etale passes through all the arguments.

Let $F$ be the prestack on the category of perfectoid spaces sending any perfectoid space $X$ to the groupoid of separated \'etale perfectoid spaces over $X$. Let $f: Y\to X$ be any v-cover; we need to prove that $F(X)\to F(Y/X)$ is an equivalence. From Lemma~\ref{lem:descentfullyfaithful}, we know that the functor is fully faithful, so we need to show effectivity of descent. Assume first that $X$ is strictly totally disconnected; we may also assume that $Y$ is strictly totally disconnected. Then Proposition~\ref{prop:sepproetaledescent} implies that any separated \'etale $\tilde{Y}\to Y$ with descent data descends to a separated pro-\'etale $\tilde{X}\to X$. It remains to see that $\tilde{X}\to X$ is \'etale. For this, we may assume that $\tilde{X}$ is affinoid. Under the equivalence of Corollary~\ref{cor:proetaleoverwlocaltop}, we have to see that $|\tilde{X}|\to |X|$ is a local isomorphism. But $|\tilde{X}\times_X Y|\to |\tilde{X}|\times_{|X|} |Y|$ is a homeomorphism, and so we know that $|\tilde{X}|\to |X|$ becomes a local isomorphism after pullback along the surjection $|Y|\to |X|$. By Lemma~\ref{lem:localisomvlocal} below, we see that $|\tilde{X}|\to |X|$ is a local isomorphism, as desired.

This finishes the case that $X$ is strictly totally disconnected. In general, we can assume that $X$ is affinoid, and by Lemma~\ref{lem:strtotdisccover} find an affinoid pro-\'etale surjective and universally open map $Y^\prime\to X$ with $Y^\prime$ strictly totally disconnected. Then we have a v-cover $Y\times_X Y^\prime\to X$ refining $Y\to X$, and we already know effectivity of descent for $Y\times_X Y^\prime\to Y^\prime$. In other words, we may replace $Y$ by $Y^\prime$, and assume that $Y\to X$ is affinoid pro-\'etale and universally open. Write $Y=\varprojlim_i Y_i\to X$ as a cofiltered inverse limit of \'etale maps $Y_i\to X$, where all $Y_i$ are affinoid.

Now let $\tilde{Y}\to Y$ be a separated \'etale map equipped with a descent datum to $X$, which induces an equivalence relation $\tilde{R}\subset \tilde{Y}\times \tilde{Y}$. The map $\tilde{R}\to \tilde{Y}$ can be identified with the projection $\tilde{Y}\times_X Y\to \tilde{Y}$, which is open (as $Y\to X$ is universally open). From Lemma~\ref{lem:spectralopenequivrel}, we see that $|\tilde{Y}|/|\tilde{R}|$ is a quasiseparated locally spectral space. Covering it by quasicompact open subsets, we can reduce to the case that $\tilde{Y}$ is quasicompact.

In this situation, Proposition~\ref{prop:etmaptolim}~(iii) says that
\[
Y_{\et,\qc,\sep} = \text{2-}\varinjlim_i (Y_i)_{\et,\qc,\sep}\ ,
\]
so any quasicompact separated \'etale $Y$-space $\tilde{Y}$ comes via pullback from some quasicompact separated \'etale $Y_i$-space $\tilde{Y}_i$ for $i$ large enough. Similarly, any descent datum is defined at some finite level, so we can replace $Y$ by $Y_i$ for $i$ sufficiently large, and reduce to the case that $Y\to X$ is \'etale. By the variant \cite[Proposition 8.2.20]{KedlayaLiu1} of an argument of de Jong and van der Put, \cite[Proposition 3.2.2]{deJongvanderPut}, one further reduces to the case that $Y\to X$ is finite \'etale. We can also assume that it is Galois, with Galois group $G$.

Finally, take any point $x\in X$. As $(X_x)_{\et,\qc,\sep} = \text{2-}\varinjlim_{U\ni x} U_{\et,\qc,\sep}$ (and similarly for pullbacks to $Y$ and $Y\times_X Y$), it is enough to show effectivity of descent after base change to $X_x$; thus we may assume that $X=\Spa(K,K^+)$ for some perfectoid field $K$ with open and bounded valuation subring $K^+\subset K$. It remains to see that in this case, if $\tilde{Y}\in X_{\et,\qc,\sep}$ has a free $G$-action for some finite group $G$, then the quotient $\tilde{Y}/G$ exists (with $\tilde{Y}\times_{\tilde{Y}/G} \tilde{Y} = \tilde{Y}\times G$). For this, we use Lemma~\ref{lem:compactificationetalemap} below to reduce to the case that $\tilde{Y}\to X$ is finite \'etale, which reduces us to the algebraic case of finite \'etale $K$-algebras.
\end{proof}

The following two lemmas were used in the proof.

\begin{lemma}\label{lem:localisomvlocal} Let $f: S^\prime\to S$ be a map of spectral spaces, and let $p: \tilde{S}\to S$ be a surjective, spectral and generalizing map of spectral spaces, with pullback $\tilde{f}: \tilde{S}^\prime = S^\prime \times_S \tilde{S}\to \tilde{S}$. If $\tilde{f}$ is a local isomorphism, then $f$ is a local isomorphism.
\end{lemma}

\begin{proof} Let $p^\prime: \tilde{S}^\prime\to S^\prime$ denote the projection. Let $s^\prime\in S^\prime$ be any point, with image $s\in S$. We claim that there is a quasicompact open neighborhood $U^\prime\subset S^\prime$ of $s^\prime$ such that $f|_{U^\prime}: U^\prime\to S$ is injective. For this, note that $\Delta_{\tilde{f}}: \tilde{S}^\prime\to \tilde{S}^\prime\times_{\tilde{S}} \tilde{S}^\prime$ is an open immersion (as $\tilde{f}$ is a local isomorphism); this implies by Lemma~\ref{lem:quotientmap} that also $\Delta_f: S^\prime\to S^\prime\times_S S^\prime$ is an open immersion. In particular, there is a quasicompact open neighborhood $U^\prime$ of $s^\prime$ in $S^\prime$ such that $U^\prime\times_S U^\prime\subset \Delta_f(S^\prime)$, i.e.~$U^\prime\to S$ is injective.

Replacing $S^\prime$ by $U^\prime$, we may assume that $f$ is injective. In that case $\tilde{f}$ is injective and a local isomorphism. This implies that $\tilde{f}$ is an open immersion (as the image of $\tilde{f}$ is open, and the map to the image is a continuous open bijective map). By Lemma~\ref{lem:quotientmap}, it follows that $f$ is an open immersion.
\end{proof}

\begin{lemma}\label{lem:compactificationetalemap} Let $X=\Spa(K,K^+)$, where $K$ is a perfectoid field, and $K^+\subset K$ is an open and bounded integrally closed subring (not necessarily a valuation subring). Let $Y\to X$ be a quasicompact separated \'etale map. Then there is a functorial factorization $Y\hookrightarrow \bar Y\to X$, where $Y\hookrightarrow \bar Y$ is a quasicompact open immersion, and $\bar Y\to X$ is finite \'etale.
\end{lemma}

\begin{remark} This compactification is a special case of the canonical compactifications introduced in Proposition~\ref{prop:cancomp} below.
\end{remark}

\begin{proof} Let $X^\circ = \Spa(K,\OO_K)$, and $Y^\circ = Y\times_X X^\circ$. Then $Y^\circ$ is finite \'etale over $X^\circ$; in particular, $|Y^\circ|$ is a finite set of maximal points. Each of them is a pro-constructible subset, and their closures (which are given by their sets of specializations) do not intersect, as any point generalizes to a unique maximal point. Thus, replacing $Y$ by an open and closed subset, we may assume that $Y^\circ=\Spa(L,\OO_L)$ is a point, where $L$ is a finite extension of $K$. Then $Y^\circ\subset Y$ is the unique maximal point, so $Y$ is connected, and $\OO_Y$ is the constant sheaf $L$. Let $\bar Y=\Spa(L,L^+)$, where $L^+$ is the integral closure of $K^+$ in $L$. Then $\bar Y\to X$ is finite \'etale, and there is a natural map $Y\to \bar Y$, given by the map $(L,L^+)\to (\OO_Y(Y),\OO_Y^+(Y))$ of algebras. Then $Y\to \bar Y$ is a separated \'etale map of perfectoid spaces, which after pullback along $X^\circ\subset X$ is an isomorphism. In particular, it is an isomorphism on rank-$1$-valued points, so by the valuative criterion for separatedness (Proposition~\ref{prop:seperatedperf}), it is an injection (cf.~Proposition~\ref{prop:charinjection}~(ii)). Thus, $Y$ is determined by the image of the injective map $|Y|\to |\bar Y|$, which is a quasicompact open subspace of $\bar Y$, as $Y\to \bar Y$ is \'etale. Therefore, $Y\to \bar Y$ is a quasicompact open immersion, as desired.
\end{proof}

\begin{corollary}\label{cor:etproetpropdescent} Let $f: Y\to X$ be a map of perfectoid spaces, and let $\tilde{X}\to X$ be a v-cover, with pullback $\tilde{f}: \tilde{Y}=Y\times_X \tilde{X}\to \tilde{X}$.
\begin{altenumerate}
\item[{\rm (i)}] Assume that $\tilde{f}$ is pro-\'etale, and $X$ is strictly totally disconnected. Then $f$ is pro-\'etale.
\item[{\rm (ii)}] Assume that $\tilde{f}$ is \'etale. Then $f$ is \'etale.
\item[{\rm (iii)}] Assume that $\tilde{f}$ is finite \'etale. Then $f$ is finite \'etale.
\end{altenumerate}
\end{corollary}

\begin{proof} We can assume that $X$ is affinoid. In cases (i) and (ii), the conclusion is local on $Y$, so we can in these cases assume that $Y$ is affinoid; in particular, $f$ is separated. Now (i) follows from Proposition~\ref{prop:sepproetaledescent}, and parts (ii) and (iii) from Proposition~\ref{prop:etalesepdescent}.
\end{proof}

\section{Morphisms of v-stacks}

The descent results of the previous section make it possible to define good notions of \'etale and (quasi-)pro-\'etale morphisms between pro-\'etale stacks on $\Perfd$. In the following, we will often confuse a perfectoid space and the sheaf that it represents.

\begin{definition}\label{def:etquasiproet} Let $f: Y^\prime\to Y$ be a map of pro-\'etale stacks on the category of perfectoid spaces $\Perfd$.
\begin{altenumerate}
\item[{\rm (i)}] Assume that $f$ is locally separated, i.e.~there is an open cover of $Y^\prime$ over which $f$ becomes separated. The map $f$ is quasi-pro-\'etale if for any strictly totally disconnected $X$ with a map $X\to Y$ (i.e., a section in $Y(X)$), the pullback $Y^\prime\times_Y X$ is representable and $Y^\prime\times_Y X\to X$ is pro-\'etale.
\item[{\rm (ii)}] Assume that $f$ is locally separated. The map $f$ is \'etale if for any perfectoid space $X$ with a map $X\to Y$, the pullback $Y^\prime\times_Y X$ is representable, and $Y^\prime\times_Y X\to X$ is \'etale.
\item[{\rm (iii)}] The map $f$ is finite \'etale if for any perfectoid space $X$ with a map $X\to Y$, the pullback $Y^\prime\times_Y X\to X$ is representable, and $Y^\prime\times_Y X\to X$ is finite \'etale.
\end{altenumerate}
\end{definition}

In parts (i) and (ii), we restrict the definition to the case that $f$ is locally separated, as otherwise one may want to give a slightly different definition. Note that $f$ is always locally separated if $Y^\prime$ is a perfectoid space, and $Y$ is a v-sheaf. We make the following important convention.

\begin{convention}\label{conv:etlocsep} In the remainder of this text, we take ``locally separated'' as part of the definition of \'etale or quasi-pro-\'etale maps of pro-\'etale stacks.
\end{convention}

Any pro-\'etale morphism of perfectoid spaces is quasi-pro-\'etale, but the converse need not hold. This is similar to the difference between pro-\'etale and weakly \'etale morphisms of schemes, cf.~\cite{BhattScholze}.

The following proposition is immediate from the definitions.

\begin{proposition}\label{prop:comparenotionsfunctorrepr} Let $f: X^\prime\to X$ be a map of perfectoid spaces. Then $f$ is \'etale, resp.~finite \'etale, if and only if the corresponding map of pro-\'etale sheaves on $\Perfd$ is \'etale, resp.~finite \'etale. If $X$ is strictly totally disconnected, then $f$ is pro-\'etale if and only if the corresponding map of pro-\'etale sheaves on $\Perfd$ is quasi-pro-\'etale.$\hfill \Box$
\end{proposition}

\begin{proposition}\label{prop:quasiproetalefirstprop} Let $f: Y_1\to Y_2$, $g: Y_2\to Y_3$ be morphisms of pro-\'etale stacks on $\Perfd$, with composite $h=g\circ f: Y_1\to Y_3$.
\begin{altenumerate}
\item[{\rm (i)}] If $f$ and $g$ are quasi-pro-\'etale (resp.~\'etale, finite \'etale), then so is $h$.
\item[{\rm (ii)}] If $g$ and $h$ are quasi-pro-\'etale (resp.~\'etale, finite \'etale), then so is $f$.
\item[{\rm (iii)}] Any pullback of a quasi-pro-\'etale (resp.~\'etale, finite \'etale) map is quasi-pro-\'etale (resp.~\'etale, finite \'etale).
\end{altenumerate}
\end{proposition}

In some situations, one can also deduce that $g$ is quasi-pro-\'etale (resp.~\'etale, finite \'etale) from knowing that $f$ and $h$ are quasi-pro-\'etale (resp.~\'etale, finite \'etale), cf.~Proposition~\ref{prop:qproetthirdmap} below.

\begin{proof} These are all straightforward. Let us check the most subtle part, which is (i) in the quasi-pro-\'etale case. We may assume that $X_3=Y_3$ is a strictly totally disconnected perfectoid space. In this case, $Y_2$ is representable and pro-\'etale over $X_3$. As such, $Y_2$ admits an open cover by strictly totally disconnected perfectoid spaces $X_2\subset Y_2$. The preimage of any such $X_2\subset Y_2$ in $Y_1$ is representable and pro-\'etale over $X_2$. It follows that $Y_1$ has an open cover by representable functors, so $Y_1$ is itself representable. Also $Y_1\to Y_2$ is pro-\'etale locally on $Y_2$, thus pro-\'etale, and then $Y_1\to Y_3$ is pro-\'etale as a composite of pro-\'etale maps.
\end{proof}

We will also need to extend other notions like injections, open and closed immersions, and separated maps to morphisms of sheaves. These notions work better for v-sheaves, because of the following result.

\begin{proposition}\label{prop:vinjindrepr} Let $X$ be a totally disconnected perfectoid space, and let $Y\subset X$ be a sub-v-sheaf. Then $Y$ is ind-representable. More precisely, one can write $Y$ as a filtered colimit of $Y_i\subset Y\subset X$ which are pro-constructible and generalizing subsets of $X$, in particular affinoid pro-\'etale over $X$.
\end{proposition}

\begin{proof} Let $Z$ be any affinoid perfectoid space with a map $Z\to Y$. Then the image of the composition $Z\to Y\to X$ on topological spaces is a pro-constructible and generalizing subset $\bar Z\subset X$. (Indeed, any map of spectral spaces has pro-constructible image, and maps of analytic adic spaces are generalizing, so have generalizing image.) By Lemma~\ref{lem:subsetwlocal}, the subset $\bar Z$ is naturally an affinoid perfectoid space, and $Z\to \bar Z$ is a v-cover. As $Y$ is sub-v-sheaf of $X$, it follows that one gets a map $\bar Z\to Y$. Note that the category of pro-constructible generalizing subsets $Y_i\subset X$ for which the inclusion $Y_i\to X$ factors over $Y$ is filtered: This follows from the sheaf property of $Y$. This implies that $Y$ is the filtered colimit of such $Y_i$, as desired.
\end{proof}

\begin{corollary}\label{cor:qcvinjqproet} Let $f: Y^\prime\to Y$ be a quasicompact injection of v-stacks. Then $f$ is quasi-pro-\'etale, and for every totally disconnected perfectoid space $X$ mapping to $Y$, the fibre product $Y^\prime\times_Y X$ is representable by a pro-constructible and generalizing subset of $X$.
\end{corollary}

\begin{proof} We may assume that $Y=X$ is a totally disconnected perfectoid space. The previous proposition shows that $Y^\prime\subset X$ is ind-representable. By quasicompactness of $Y^\prime$, this implies that $Y^\prime$ is representable; applying the proposition again, we see that $Y^\prime\subset X$ is affinoid pro-\'etale, as desired.
\end{proof}

\begin{definition} Let $f: Y^\prime\to Y$ be a map of pro-\'etale stacks.
\begin{altenumerate}
\item[{\rm (i)}] The map $f$ is an open immersion if for every perfectoid space $X$ mapping to $Y$, the pullback $Y^\prime\times_Y X\to X$ is representable by an open immersion.
\item[{\rm (ii)}] The map $f$ is a closed immersion if for every totally disconnected perfectoid space $X$ mapping to $Y$, the pullback $Y^\prime\times_Y X\subset X$ is representable by a closed immersion.
\item[{\rm (iii)}] The map $f$ is separated if $\Delta_f: Y^\prime\to Y^\prime\times_Y Y^\prime$ is a closed immersion.
\item[{\rm (iv)}] The map $f$ is $0$-truncated if for all $X\in \Perfd$, the map of groupoids $f(X): Y^\prime(X)\to Y(X)$ is faithful.
\end{altenumerate}

Moreover, the pro-\'etale stack $Y$ is separated if the map $Y\to \ast$ is separated.
\end{definition}

\begin{remark} We will only occasionally use the last definition, with which one has to be a bit careful. Namely, it can happen that $Y$ is separated, but $Y$ is not quasiseparated: The issue is that the notions of $Y$ being quasiseparated and of $Y\to \ast$ being quasiseparated do not coincide. (An example is given by $Y=X/\varphi^{\mathbb Z}$, where $X$ is some perfectoid space of characteristic $p$, and $\varphi$ is its absolute Frobenius: This defines a diamond $Y$ that is separated, but not quasiseparated.)
\end{remark}

In part (iv), it is equivalent to ask that for all affinoid perfectoid spaces (or more generally, all pro-\'etale sheaves) $X$ mapping to $Y$, the pro-\'etale stack $Y^\prime\times_Y X$ is discrete, i.e.~a pro-\'etale sheaf. It is also equivalent to ask that $\Delta_f: Y^\prime\to Y^\prime\times_Y Y^\prime$ is an injection; in particular, separated maps are automatically $0$-truncated. (This differs from the convention in \cite[Tag 04YW]{StacksProject}.)

For a map of perfectoid spaces $f: X^\prime\to X$, these notions are equivalent to the previously defined notions. Moreover, there is again a valuative criterion for separatedness.

\begin{proposition}\label{prop:valcritsepsheaf} Let $f: Y^\prime\to Y$ be a map of v-stacks. Then $f$ is separated if and only if $f$ is $0$-truncated, quasiseparated, and for every perfectoid field $K$ with ring of integers $\OO_K$ and an open and bounded valuation subring $K^+$, and any diagram
\[\xymatrix{
\Spa(K,\OO_K)\ar[d]\ar[r] & Y^\prime\ar^f[d] \\
\Spa(K,K^+)\ar[r]\ar@{-->}[ur] & Y\ ,
}\]
there exists at most one dotted arrow making the diagram commute.
\end{proposition}

Note that if $f$ is $0$-truncated, the groupoid of dotted arrows is automatically discrete, i.e.~there are no automorphisms.

\begin{proof} If $f$ is separated, then $\Delta_f$ is a closed immersion, which implies that $\Delta_f$ is quasicompact, so that $f$ is quasiseparated. Also, closed immersions are specializing, which shows that there is at most one such dotted arrow. Conversely, if $f$ is quasiseparated, then $\Delta_f$ is quasicompact. Thus, by Corollary~\ref{cor:qcvinjqproet}, the diagonal $\Delta_f$ is quasi-pro-\'etale. It remains to check that for any totally disconnected perfectoid space $X$ mapping to $Y^\prime \times_Y Y^\prime$, the pro-constructible generalizing subset
\[
Y^\prime\times_{Y^\prime\times_Y Y^\prime} X\subset X
\]
is closed. For this, it is enough to check that it is closed under specializations, which is exactly the valuative criterion.
\end{proof}

Sometimes it is useful to know that a version of the valuative criterion holds true for general pairs $(R,R^+)$.

\begin{proposition}\label{prop:genvalcritvsheaf} Let $f: Y^\prime\to Y$ be a separated map of v-stacks, and let $R$ be a perfectoid Tate ring with an open and integrally closed subring $R^+\subset R^\circ$. Then, for any diagram
\[\xymatrix{
\Spa(R,R^\circ)\ar[d]\ar[r] & Y^\prime\ar^f[d] \\
\Spa(R,R^+)\ar[r]\ar@{-->}[ur] & Y\ ,
}\]
there exists at most one dotted arrow making the diagram commute.
\end{proposition}

\begin{proof} Assume that $a,b: \Spa(R,R^+)\to Y^\prime$ are two maps projecting to the same map $\bar{a}=\bar{b}: \Spa(R,R^+)\to Y$, and with same restriction $a^\circ=b^\circ: \Spa(R,R^\circ)\to Y^\prime$. Consider the sub-v-sheaf
\[
Z=\Spa(R,R^+)\times_{Y^\prime\times_Y Y^\prime} Y^\prime\subset \Spa(R,R^+)\ ,
\]
where the two projections $\Spa(R,R^+)\to Y^\prime$ are given by $a$ and $b$. As $f$ is separated, $\Delta_f$ is a closed immersion, so $Z\subset \Spa(R,R^+)$ is a closed immersion. To check that it is an equality, we can assume that $R$ is totally disconnected; then $Z$ is a totally disconnected perfectoid space as well. Now for any perfectoid field $(K,K^+)$, we know that $Z(K,K^+)=\Spa(R,R^+)(K,K^+)$ by the valuative criterion, Proposition~\ref{prop:valcritsepsheaf}. Thus, $|Z|=|\Spa(R,R^+)|$, and then $Z=\Spa(R,R^+)$.
\end{proof}

The next proposition shows that many of these conditions can be checked v-locally on the target.

\begin{proposition}\label{prop:checksheafpropvloc} Let $f: Y^\prime\to Y$ be a map of v-stacks, and let $g: \tilde{Y}\to Y$ be a surjective map of v-stacks. Let $\tilde{f}: \tilde{Y}^\prime = \tilde{Y}\times_Y Y^\prime\to \tilde{Y}$ be the pullback of $f$.
\begin{altenumerate}
\item[{\rm (o)}] If $\tilde{f}$ is quasicompact (resp.~quasiseparated), then $f$ is quasicompact (resp.~quasiseparated).
\item[{\rm (i)}] If $\tilde{f}$ is an open (resp.~closed) immersion, then $f$ is an open (resp.~closed) immersion.
\item[{\rm (ii)}] If $\tilde{f}$ is separated, then $f$ is separated.
\item[{\rm (iii)}] If $\tilde{f}$ is finite \'etale, then $f$ is finite \'etale.
\item[{\rm (iv)}] If $\tilde{f}$ is separated and \'etale, then $f$ is separated and \'etale.
\item[{\rm (v)}] If $\tilde{f}$ is separated and quasi-pro-\'etale, then $f$ is separated and quasi-pro-\'etale.
\end{altenumerate}
\end{proposition}

\begin{proof} For (o), the quasiseparated case reduces to the quasicompact case by passing to the diagonal. One can also replace $Y$ by an affinoid perfectoid space $X$, and $\tilde{Y}$ by a v-cover $\tilde{X}\to X$, where $\tilde{X}$ may be assumed to be an affinoid perfectoid space. Now let $Y^\prime_i$, $i\in I$, be a set of v-sheaves with a surjection $\bigsqcup_{i\in I} Y^\prime_i\to Y^\prime$. The pullback $\bigsqcup_{i\in I} \tilde{Y}^\prime_i\to \tilde{Y}^\prime$ has a finite subcover by quasicompactness of $\tilde{Y}^\prime$, so assume $J\subset I$ is a finite subset such that $\bigsqcup_{i\in J} \tilde{Y}^\prime_i\to \tilde{Y}^\prime$ is surjective. In particular, the composite $\bigsqcup_{i\in J} \tilde{Y}^\prime_i\to \tilde{Y}^\prime\to Y^\prime$ is surjective, which implies that $\bigsqcup_{i\in J} Y^\prime_i\to Y^\prime$ is surjective.

Part (ii) reduces to (i) by passing to the diagonal. Now for (i), (iii), (iv) and (v), we need to check that something holds after any pullback $X\to Y$ for an affinoid perfectoid space $X$. As $\tilde{Y}\to Y$ is surjective, we can find a v-cover $\tilde{X}\to X$ such that $X\to Y$ lifts to $\tilde{X}\to \tilde{Y}$. We may then replace $Y$ by $X$ and $\tilde{Y}$ by $\tilde{X}$.

Now (i) for open immersions follows because $|\tilde{X}|\to |X|$ is a quotient map by Lemma~\ref{lem:quotientmap}. For closed immersions, one uses in addition that if $X$ is totally disconnected, then $f$ is representable by Corollary~\ref{cor:qcvinjqproet}, as $\tilde{f}$ is a quasicompact injection, and thus $f$ is a quasicompact injection.

Finally, parts (iii) and (iv) follow from Proposition~\ref{prop:etalesepdescent}, and part (v) (in which case one assumes that $X$ is strictly totally disconnected) from Proposition~\ref{prop:sepproetaledescent}.
\end{proof}

As an application, let us discuss $\underline{G}$-torsors for locally profinite groups $G$. Here, for any topological space $T$, we denote by $\underline{T}$ the associated v-sheaf given by
\[
\underline{T}(X) = C^0(|X|,T)\ .
\]

\begin{definition}\label{def:gtorsor} Let $G$ be a locally profinite group. A $\underline{G}$-torsor is map $f: \tilde{X}\to X$ of v-stacks with an action of $\underline{G}$ on $\tilde{X}$ over $X$, such that v-locally on $X$, there is a $\underline{G}$-equivariant isomorphism $\tilde{X}\cong \underline{G}\times X$.
\end{definition}

\begin{lemma}\label{lem:gtorsoroverperfectoid} Let $G$ be a locally profinite group, $X$ a perfectoid space, and $f: \tilde{X}\to X$ a $\underline{G}$-torsor. Then the v-sheaf $\tilde{X}$ is representable by a perfectoid space, and $\tilde{X}\to X$ is pro-\'etale, universally open and a v-cover. More precisely, for any open subgroup $K\subset G$, pushout along the discrete $G$-set $G/K$ defines a separated \'etale map
\[
\tilde{X}_K := \tilde{X}\times_{\underline{G}} \underline{G/K}\to X\ ,
\]
the transition map $\tilde{X}_{K^\prime}\to \tilde{X}_K$ is finite \'etale if $K^\prime\subset K$ is of finite index, and
\[
\tilde{X} = \varprojlim_K \tilde{X}_K\to X\ .
\]
\end{lemma}

Related results had been obtained by Hansen, \cite{Hansen}.

\begin{proof} In case $\tilde{X} = \underline{G}\times X$, the results are clear, as then $\tilde{X}_K = \underline{G/K}\times X$ is a disjoint union of copies of $X$. Thus, Proposition~\ref{prop:checksheafpropvloc}~(iv) implies that $\tilde{X}_K\to X$ is representable by a separated \'etale map in general, and Proposition~\ref{prop:checksheafpropvloc}~(iii) implies that $\tilde{X}_{K^\prime}\to \tilde{X}_K$ is finite \'etale in case $K^\prime\subset K$ is of finite index. Also, $\tilde{X}=\varprojlim_K \tilde{X}_K$, as this can be checked v-locally. This presentation shows that $\tilde{X}\to X$ is pro-\'etale, as desired. It is also a v-cover, as this can be checked v-locally, where it admits a section. Finally, the map $\tilde{X}\to X$ is universally open, as any open subspace comes via pullback from $\tilde{X}_K$ for sufficiently small $K$, the map $\tilde{X}\to \tilde{X}_K$ is surjective, and $\tilde{X}_K\to X$ is open (as it is \'etale).
\end{proof}

\section{Diamonds}

From now on, we change the setup, and work with the full subcategory $\Perf\subset \Perfd$ of perfectoid spaces of characteristic $p$.

\begin{definition}\label{def:diamond} A diamond is a sheaf $Y$ for the pro-\'etale topology of $\Perf$ such that $Y$ can be written as a quotient $X/R$, where $X$ is representable by a perfectoid space, and $R\subset X\times X$ is a representable equivalence relation for which the projections $s,t: R\to X$ are pro-\'etale.
\end{definition}

In particular, any diamond is a small sheaf. For convenience, we introduce the following name for the equivalence relations.

\begin{definition}\label{def:proetaleequiv} Let $X$ be a perfectoid space. A pro-\'etale equivalence relation on $X$ is a representable equivalence relation $R\subset X\times X$ such that the projections $s,t: R\to X$ are pro-\'etale.
\end{definition}

Note that we do not make any assumptions like representability of the diagonal; the issue is that there is no good notion of morphisms that are relatively representable (in perfectoid spaces).

\begin{proposition}\label{prop:diamondfirstprop} Let $X\in \Perf$, and $R\subset X\times X$ a pro-\'etale equivalence relation.
\begin{enumerate}
\item[{\rm (i)}] The quotient sheaf $Y=X/R$ is a diamond.
\item[{\rm (ii)}] The natural map $R\to X\times_Y X$ of sheaves on $\Perf$ is an isomorphism.
\item[{\rm (iii)}] Let $\tilde{X}\to X$ be a pro-\'etale cover by a perfectoid space $\tilde{X}$, and
\[
\tilde{R} = R\times_{X\times X} (\tilde{X}\times \tilde{X})\subset \tilde{X}\times \tilde{X}
\]
the induced equivalence relation. Then $\tilde{R}$ is a pro-\'etale equivalence relation on $\tilde{X}$, and the natural map $\tilde{Y}:=\tilde{X}/\tilde{R}\to Y=X/R$ is an isomorphism.
\item[{\rm (iv)}] The map $X\to Y$ is quasi-pro-\'etale.
\end{enumerate}
\end{proposition}

Recall that quasi-pro-\'etale morphisms are defined in Definition~\ref{def:etquasiproet}~(i).

\begin{proof} Part (i) is clear by definition. For part (ii), note that both are subsheaves of $X\times X$, so $R\to X\times_Y X$ is injective. On the other hand, let $Z$ be any perfectoid space with a map $Z\to X\times_Y X$. By definition, this means that we have two maps $a,b: Z\to X$ such that the two induced maps $Z\to Y$ agree. This condition means that after some pro-\'etale cover $\tilde{Z}\to Z$, the composite map $\tilde{Z}\to Z\to X\times X$ factors over the equivalence relation $R$. We get a map $\tilde{Z}\to R$, which descends to a map $Z\to R$ as the two maps $\tilde{Z}\times_Z \tilde{Z}\to R$ coming from projection on either factor agree: Namely, their composites with $R\hookrightarrow X\times X$ agree. We have produced a map $Z\to R$, which factors the map $Z\to X\times X$, as this holds true after composition with the pro-\'etale cover $\tilde{Z}\to Z$. Therefore, $R\to X\times_Y X$ is surjective, as desired.

In part (iii), first note that $\tilde{R}$ is representable as a fibre product of representable objects, and the two projections $\tilde{s},\tilde{t}:\tilde{R}\to \tilde{X}$ are pro-\'etale: For example, one can write $\tilde{s}$ as the composite of
\[
(R\times_X \tilde{X})\times_X(\tilde{X}\to X): \tilde{R} = R\times_{X\times X} (\tilde{X}\times \tilde{X})\to R\times_{X\times X} (\tilde{X}\times X) = R\times_X \tilde{X}
\]
and
\[
s\times_X \tilde{X}: R\times_X \tilde{X}\to \tilde{X}\ ,
\]
both of which are a base change of a pro-\'etale map, and thus pro-\'etale.

Now the map $\tilde{Y}\to Y$ of pro-\'etale sheaves on $\Perf$ is surjective, as the composite $\tilde{X}\to \tilde{Y}\to Y$ is, as this can also be written as the composite $\tilde{X}\to X\to Y$. To check whether $\tilde{Y}\to Y$ is injective, let $Z$ be a perfectoid space with two maps $a,b: Z\to \tilde{Y}$ which agree after composing with $\tilde{Y}\to Y$. Replacing $Z$ by a pro-\'etale cover, we may assume that $a,b$ factor over $\tilde{a},\tilde{b}: Z\to \tilde{X}$. The associated map $Z\to \tilde{X}\times \tilde{X}\to X\times X$ factors over $R$ by assumption and part (ii), so we get a map $Z\to R\times_{X\times X}(\tilde{X}\times \tilde{X}) = \tilde{R}$. This means that $\tilde{a},\tilde{b}: Z\to \tilde{X}$ induce the same map $Z\to \tilde{Y}$, proving injectivity of $\tilde{Y}\to Y$, thereby finishing the proof of part (iii).

For part (iv), we can replace $X$ by $\tilde{X}=\bigsqcup_i U_i\to X$, where the $U_i$ form an affinoid cover of $X$ and $R$ by the induced equivalence relation $\tilde{R}\subset \tilde{X}\times \tilde{X}$: Indeed, these spaces still satisfy the same assumptions by part (iii), and to check whether $X^\prime\times_Y X$ is representable and pro-\'etale over $X^\prime$, it is enough to check for the open subspaces $X^\prime\times_Y U_i$. In particular, after this replacement, $s,t: R\to X$ are separated: Indeed, both maps $R\subset X\times X$ and $X\times X\to X$ are separated.

As by definition $X\to Y$ is surjective in the pro-\'etale topology, we can find a pro-\'etale cover $\widetilde{X^\prime}\to X^\prime$ and a map $\widetilde{X^\prime}\to X$ lying over $X^\prime\to Y$. We can assume that $\widetilde{X^\prime}$ is affinoid. Let $W=X^\prime\times_Y X\to X^\prime$ be the fibre product. Then
\[
\widetilde{X^\prime}\times_{X^\prime} W = \widetilde{X^\prime}\times_X (X\times_Y X) = \widetilde{X^\prime}\times_X R\ ,
\]
which is representable and pro-\'etale over $\widetilde{X^\prime}$. Moreover, as $s,t: R\to X$ are separated, also the base change $\widetilde{X^\prime}\times_{X^\prime} W=\widetilde{X^\prime}\times_X R\to \widetilde{X^\prime}$ is separated. Now applying Proposition~\ref{prop:sepproetaledescent}, we see that $W\to X^\prime$ is representable, pro-\'etale and separated over $X^\prime$.
\end{proof}

Another basic fact is the following.

\begin{proposition}\label{prop:diamondfibreproduct} Products and fibre products exist in the category of diamonds.
\end{proposition}

\begin{proof} Let $Y_1\to Y_3\leftarrow Y_2$ be a diagram of diamonds. Choose presentations $Y_i=X_i/R_i$ as in the definition of a diamond, for $i=1,2,3$. After replacing $X_i$ for $i=1,2$ by a pro-\'etale cover, we can assume that there are maps $X_i\to X_3$ lying over $Y_i\to Y_3$. We can then further replace $X_i$ by $X_i\times_{Y_3} X_3=X_i\times_{X_3} R_3$ to assume that the maps $X_i\to Y_1\times_{Y_3} X_3$ are surjective in the pro-\'etale topology.

In this case, the map $X_4:=X_1\times_{X_3} X_2\to Y_1\times_{Y_3} Y_2=:Y_4$ is surjective in the pro-\'etale topology. The induced equivalence relation $R_4 = X_4\times_{Y_4} X_4$ can be calculated as $R_4 = R_1\times_{R_3} R_2$, which is representable. It remains to see that $R_4\to X_4$ is pro-\'etale. But $R_1\to X_1$ and $R_2\to X_2$ are pro-\'etale, so $R_1\times R_2\to X_1\times X_2$ and its base change $R_1\times_{X_3} R_2\to X_1\times_{X_3} X_2=X_4$ are pro-\'etale. It remains to see that $R_4=R_1\times_{R_3} R_2\to R_1\times_{X_3} R_2$ is pro-\'etale; this is a base change of $R_3\to R_3\times_{X_3} R_3$. But the diagonal of any representable map is pro-\'etale by Remark~\ref{rem:diagproetale}.

The similar argument works for products of diamonds, once one has checked that a product of two perfectoid spaces is again a perfectoid space. By choosing locally maps to $\Spa \mathbb F_p((t^{1/p^\infty}))$, one reduces to that base case, in which case one computes the fibre product to be the perfectoid punctured open unit disc over $\Spa \mathbb F_p((t^{1/p^\infty}))$.
\end{proof}

Another characterization of diamonds is given as follows.

\begin{proposition}\label{prop:diamondqproetsurj} Let $Y$ be a pro-\'etale sheaf on $\Perf$. Then $Y$ is a diamond if and only if there is a surjective quasi-pro-\'etale morphism $X\to Y$ from a perfectoid space $X$. If $X$ is a disjoint union of strictly totally disconnected spaces, then $R=X\times_Y X\subset X\times X$ is a pro-\'etale equivalence relation with $Y=X/R$.
\end{proposition}

\begin{proof} We may assume that $X$ is a disjoint union of strictly totally disconnected spaces. In this case, by definition of quasi-pro-\'etale morphisms (Definition~\ref{def:etquasiproet}~(i)), $R=X\times_Y X\to X$ is pro-\'etale. As in the proof of Proposition~\ref{prop:diamondfirstprop}~(iii), one finds that $Y=X/R$, so $Y$ is a diamond.
\end{proof}

One can use this to obtain some stability properties for diamonds.

\begin{proposition}\label{prop:quotientsbydiamondequivrel} Let $Y$ be a pro-\'etale sheaf on $\Perf$, and assume that there is a surjective quasi-pro-\'etale map $Y^\prime\to Y$, where $Y^\prime$ is a diamond. Then $Y$ is a diamond.
\end{proposition}

\begin{proof} Choose a surjective quasi-pro-\'etale map $X\to Y^\prime$. Then $X\to Y^\prime\to Y$ is a surjective quasi-pro-\'etale map, so $Y$ is a diamond by Proposition~\ref{prop:diamondqproetsurj}.
\end{proof}

\begin{proposition}\label{prop:quasiproetaleoverdiamond} Let $f: Y^\prime\to Y$ be a quasi-pro-\'etale map of pro-\'etale sheaves on $\Perf$, and assume that $Y$ is a diamond. Then $Y^\prime$ is a diamond.
\end{proposition}

\begin{proof} Choose a surjective quasi-pro-\'etale map $X\to Y$ where $X$ is a disjoint union of strictly totally disconnected spaces. Then $X^\prime=Y^\prime\times_Y X$ is representable (and pro-\'etale over $X$), and $X^\prime\to Y^\prime$ is a surjective quasi-pro-\'etale map. Now Proposition~\ref{prop:diamondqproetsurj} implies that $Y^\prime$ is a diamond.
\end{proof}

\begin{proposition}\label{prop:quotientbydiamondequivrel} Let $X$ be a diamond, and let $R\subset X\times X$ be an equivalence relation such that $s,t: R\to X$ are quasi-pro-\'etale. Then the quotient $Y=X/R$ is a diamond.
\end{proposition}

In the statement, $R$ is just assumed to be a pro-\'etale sheaf, but note that as $s: R\to X$ is quasi-pro-\'etale, it follows from Proposition~\ref{prop:quasiproetaleoverdiamond} that $R$ is automatically a diamond.

\begin{proof} We may assume that $X$ is a separated perfectoid space. In that case, the map $X\to Y$ is a surjective separated map which is quasi-pro-\'etale: Indeed, its base change $R=X\times_Y X\to X$ is quasi-pro-\'etale, and $X\to Y$ is a v-cover, so we can apply Proposition~\ref{prop:checksheafpropvloc}~(v). Thus, the result follows from Proposition~\ref{prop:quotientsbydiamondequivrel}.
\end{proof}

A nice property of diamonds is that they are always sheaves for the v-topology. This resembles a result of Gabber, \cite[Tag 0APL]{StacksProject}, that all algebraic spaces are sheaves for the fpqc topology, and our proof will follow his arguments.

\begin{proposition}\label{prop:diamondvsheaf} Let $Y$ be a diamond. Then $Y$ is a sheaf for the v-topology.
\end{proposition}

\begin{proof} Choose a presentation $Y=X/R$ of $Y$ as a quotient of a perfectoid space $X$ by a pro-\'etale equivalence relation $R\subset X\times X$. By Proposition~\ref{prop:diamondfirstprop}~(iii), we may assume that $X$ is a disjoint union of totally disconnected (in particular, affinoid) perfectoid spaces.

Let $Z$ be a perfectoid space with a v-cover $\tilde{Z}\to Z$. First, we prove injectivity of $Y(Z)\to Y(\tilde{Z})$. Thus, assume that $a,b: Z\to Y$ are two maps such that the composite maps $\tilde{a},\tilde{b}: \tilde{Z}\to Z\to Y$ agree. After replacing $Z$ by a pro-\'etale cover (and using that $Y$ is a pro-\'etale sheaf), we can assume that $a,b$ lift to maps $a_X,b_X: Z\to X$; then the composite maps $\tilde{a}_X,\tilde{b}_X: \tilde{Z}\to Z\to X$ have the property that $(\tilde{a}_X,\tilde{b}_X): \tilde{Z}\to X\times X$ factors over $R\subset X\times X$. In other words, the map $(a_X,b_X): Z\to X\times X$ factors over $R$ after precomposition with $\tilde{Z}\to Z$. As $R$ is a v-sheaf, this implies that $(a_X,b_X): Z\to X\times X$ factors over $R$, showing that $a=b$, as desired.

It remains to prove surjectivity of
\[
Y(Z)\to \eq(Y(\tilde{Z})\rightrightarrows Y(\tilde{Z}\times_Z \tilde{Z}))\ .
\]
By standard reductions (using that $Y$ is a pro-\'etale sheaf and separated for the v-topology), we may assume that $Z$ and $\tilde{Z}$ are strictly totally disconnected. Let $\tilde{a}: \tilde{Z}\to Y$ be a map such that the two induced maps $\tilde{Z}\times_Z \tilde{Z}\to Y$ agree. By Proposition~\ref{prop:diamondfirstprop}~(iv), the fibre product $\tilde{Z}\times_Y X\to \tilde{Z}$ is representable and pro-\'etale; moreover, by our assumption on $X$, it is also separated over $\tilde{Z}$. Note that $\tilde{W}=\tilde{Z}\times_Y X\to \tilde{Z}$ comes with a descent datum relative to $\tilde{Z}/Z$: Indeed, the fibre product
\[
\tilde{W}\times_Z \tilde{Z} = (\tilde{Z}\times_Z \tilde{Z})\times_Y X
\]
is independent of which of the two induced maps $\tilde{Z}\times_Z \tilde{Z}\to Y$ is chosen, as they agree by assumption. By Proposition~\ref{prop:sepproetaledescent}, it follows that $\tilde{W}\to \tilde{Z}$ descends to a surjective separated pro-\'etale map $W\to Z$. In particular, $\tilde{W} = W\times_Z \tilde{Z}\to W$ is a v-cover, and the map $\tilde{W}\to X$ descends to $W\to X$ by construction of the descent datum, and using that $X$ is a v-sheaf. We have constructed a pro-\'etale cover $W\to Z$ with a map $W\to X$. The induced map $W\times_Z W\to X\times X$ factors over $R\subset X\times X$, as this is true for the v-cover
\[
(W\times_Z W)\times_Z \tilde{Z}=\tilde{W}\times_{\tilde{Z}} \tilde{W}=\tilde{Z}\times_Y R
\]
of $W\times_Z W$. Thus, the map composite map $W\to X\to Y$ factors over a map $Z\to Y$, which pulls back to the given map $\tilde{Z}\to Y$.
\end{proof}

Another interesting stability property is the following.

\begin{proposition}\label{prop:vinjdiamond} Let $f: Y^\prime\to Y$ be an injection of v-sheaves, where $Y$ is a diamond. Then $Y^\prime$ is a diamond.
\end{proposition}

\begin{proof} We have to find a surjective quasi-pro-\'etale morphism $X^\prime\to Y^\prime$ from a perfectoid space $X'$. Pulling back via a surjective quasi-pro-\'etale morphism $X\to Y$, we may assume that $Y$ is a disjoint union of totally disconnected spaces; we may replace $Y$ by a totally disconnected space $X$. In that case, Proposition~\ref{prop:vinjindrepr} says that $Y^\prime$ is a filtered colimit of pro-constructible generalizing subsets $X_i\subset X$. Let $X^\prime = \bigsqcup_i X_i$, with its natural morphism to $Y^\prime$. Then $X^\prime\to Y^\prime$ is surjective, and quasi-pro-\'etale (as each map $X_i\to Y^\prime$ is quasi-pro-\'etale).
\end{proof}

A convenient property of diamonds is that it is easy to check whether a map is an isomorphism. For an even more general statement, see Lemma~\ref{lem:isomqcqsvsheaves} below.

\begin{lemma}\label{lem:isomspatialdiamondspoints} Let $f: Y\to X$ be a qcqs map of diamonds. Then $f$ is an isomorphism if and only if for every algebraically closed perfectoid field $K$ with open and bounded valuation subring $K^+\subset K$, the map $f(K,K^+): Y(K,K^+)\to X(K,K^+)$ is a bijection.
\end{lemma}

\begin{proof} Pulling back by some presentation of $X$, we may replace $X$ by a perfectoid space, which we can then assume to be affinoid. In that case, $Y$ is a qcqs diamond. Write $Y=\tilde{Y}/R$ as the quotient of a qcqs perfectoid space $\tilde{Y}$ by a pro-\'etale equivalence relation $R$, which is still qcqs. The map $\tilde{Y}\to X$ is surjective in the v-topology, as it is a map of qcqs perfectoid spaces which by assumption is surjective on points, i.e.~a v-cover. As both $Y$ and $X$ are v-sheaves by Proposition~\ref{prop:diamondvsheaf}, it suffices to see that the map $R\to \tilde{Y}\times_X \tilde{Y}$ is an isomorphism. Now both $R$ and $\tilde{Y}\times_X \tilde{Y}$ are qcqs perfectoid spaces. Moreover, the map $R(K,K^+)\to (\tilde{Y}\times_X \tilde{Y})(K,K^+)$ is a bijection for all $(K,K^+)$ as in the statement of the lemma. Thus, the result follows from Lemma~\ref{lem:bijectiveiso}.
\end{proof}

Let us give a slightly weird example of diamonds.

\begin{example}\label{ex:compacthausdorff} Fix a perfectoid field $K$ of characteristic $p$. We claim that there is a fully faithful functor from the category of compact Hausdorff spaces to the category of diamonds over $\Spa(K,\OO_K)$. Indeed, for any compact Hausdorff space $T$, one has the functor $\underline{T}$ on $\Perf$ which takes a perfectoid space $X$ to $C^0(|X|,T)$, where $C^0$ denotes the continuous functions. We claim that the functor
\[
T\mapsto \underline{T}\times \Spa(K,\OO_K)
\]
is a fully faithful functor from the category of compact Hausdorff spaces to the category of diamonds over $\Spa(K,\OO_K)$.

Let us check first that it takes values in diamonds. If $S$ is a profinite set, then we can write $S$ as an inverse limit of finite sets $S_i$, and if we also denote by $S_i$ the corresponding constant sheaf on $\Perf$, then $\underline{S} = \varprojlim S_i$, and $\underline{S}\times \Spa(K,\OO_K) = \varprojlim (S_i\times \Spa(K,\OO_K))$ is an affinoid perfectoid space, given by
\[
\underline{S}\times \Spa(K,\OO_K) = \Spa(C^0(S,K),C^0(S,\OO_K))\ .
\]
Thus, profinite sets are mapped to affinoid perfectoid spaces. Now recall that any compact Hausdorff space $T$ admits a surjection $S\to T$ from a profinite set $S$: One can take for $S$ the Stone-Cech compactification of $T$ considered as a discrete set. The induced equivalence relation $R = S\times_T S\subset S\times S$ is a closed subspace of the profinite set $S\times S$, and thus a profinite set itself. Now one can form the equivalence relation
\[
\underline{R}\times \Spa(K,\OO_K)\subset (\underline{S}\times \Spa(K,\OO_K))\times_{\Spa(K,\OO_K)} (\underline{S}\times \Spa(K,\OO_K))\ ,
\]
which is representable and pro-\'etale over $\underline{S}\times \Spa(K,\OO_K)$. Let $Y$ be the quotient $\underline{S}\times \Spa(K,\OO_K)/\underline{R}\times \Spa(K,\OO_K)$. We claim that $Y=\underline{T}\times \Spa(K,\OO_K)$. Certainly, there is a natural map $Y\to \underline{T}\times \Spa(K,\OO_K)$. It is injective: If we have two maps $a,b: Z\to \underline{S}\times \Spa(K,\OO_K)$ whose composites to $\underline{T}\times \Spa(K,\OO_K)$ agree, then we have two continuous maps $|Z|\to S$ whose composite to $T$ agree, i.e.~a continuous map $|Z|\to R=S\times_T S$, or in other words a map $Z\to \underline{R}\times \Spa(K,\OO_K)$. For surjectivity, let $Z$ be any perfectoid space with a map $Z\to \underline{T}\times \Spa(K,\OO_K)$. We want to find a pro-\'etale cover $\tilde{Z}\to Z$ and a lift $\tilde{Z}\to \underline{S}\times \Spa(K,\OO_K)$. We may assume that $Z$ is affinoid. In that case any continuous map $|Z|\to T$ to a compact Hausdorff space $T$ factors through the profinite set $\pi_0(Z)$ of connected components. We get an induced map $\pi_0(Z)\to T$, and a cover by the profinite set $\pi_0(Z)\times_T S\subset \pi_0(Z)\times S\to \pi_0(Z)$. It follows that
\[
\tilde{Z} = Z\times_{\pi_0(Z)} (\pi_0(Z)\times_T S)
\]
is affinoid perfectoid, pro-\'etale over $Z$, and admits a lift $|\tilde{Z}|\to S$ of $|Z|\to T$, as desired.

Finally, to check that $T\mapsto \underline{T}\times \Spa(K,\OO_K)$ is fully faithful, we need to prove that the map
\[
C^0(T_1,T_2)\to \Hom_{\Spa(K,\OO_K)}(\underline{T_1}\times \Spa(K,\OO_K),\underline{T_2}\times \Spa(K,\OO_K))
\]
is bijective for all compact Hausdorff sets $T_1, T_2$. For injectivity, we can assume that $T_1$ is a point. Now any map $\Spa(K,\OO_K)\to \underline{T_2}\times \Spa(K,\OO_K)$ is given by a map $|\Spa(K,\OO_K)|\to T_2$ whose image recovers the given point of $T_2$.

For surjectivity, assume first that $T_1$ is profinite. In that case, $\underline{T_1}\times \Spa(K,\OO_K)$ is representable, and so
\[
\Hom_{\Spa(K,\OO_K)}(\underline{T_1}\times \Spa(K,\OO_K),\underline{T_2}\times \Spa(K,\OO_K)) = \underline{T_2}(\underline{T_1}\times \Spa(K,\OO_K)) = C^0(T_1,T_2)
\]
by definition, using that $|\underline{T_1}\times \Spa(K,\OO_K)| = T_1$.

Now choose a surjection $S_1\to T_1$ from a profinite set. Given any map $\underline{T_1}\times \Spa(K,\OO_K)\to \underline{T_2}\times \Spa(K,\OO_K)$, we get a composite map $\underline{S_1}\times \Spa(K,\OO_K)\to \underline{T_2}\times \Spa(K,\OO_K)$, which comes from a map $S_1\to T_2$ by what we have just shown. The two induced maps $S_1\times_{T_1} S_1\to T_2$ agree by faithfulness. Thus, we get a map from the quotient $T_1\to T_2$ (recalling that any surjective map of compact Hausdorff spaces is a quotient map), which induces the given map $\underline{T_1}\times \Spa(K,\OO_K)\to \underline{T_2}\times \Spa(K,\OO_K)$ (as this is true after precomposing with the pro-\'etale cover $\underline{S_1}\times \Spa(K,\OO_K)\to \underline{T_1}\times \Spa(K,\OO_K)$).
\end{example}

Although compact Hausdorff spaces may have some appeal, we will generally work with spaces whose behaviour is closer to that of schemes. They are essentially characterized by having a spectral underlying topological space.

\begin{proposition}\label{prop:topspaceindep} Let $Y$ be a diamond, and let $Y=X/R$ be a presentation as a quotient of a perfectoid space $X$ by a pro-\'etale equivalence relation $R\subset X\times X$. There is a canonical bijection between $|X|/|R|$ and the set
\[
|Y|=\{\Spa(K,K^+)\to Y\}/\sim\ ,
\]
where $(K,K^+)$ runs over all pairs of a perfectoid field $K$ with an open and bounded valuation subring $K^+\subset K$, and two such maps $\Spa(K_i,K_i^+)\to Y$, $i=1,2$, are equivalent, if there is a third pair $(K_3,K_3^+)$, and a commutative diagram
\[\xymatrix{
&\Spa(K_1,K_1^+)\ar[dr] & \\
\Spa(K_3,K_3^+)\ar[ur]\ar[rr]\ar[dr] && Y\ ,\\
& \Spa(K_2,K_2^+)\ar[ur] &
}\]
where the maps $\Spa(K_3,K_3^+)\to \Spa(K_i,K_i^+)$, $i=1,2$, are surjective.

Moreover, the quotient topology on $|Y|$ induced by the surjection $|X|\to |Y|$ is independent of the choice of presentation $Y=X/R$.
\end{proposition}

\begin{proof} First, note that the given relation $\sim$ on $\{\Spa(K,K^+)\to Y\}$ is an equivalence relation. This follows from the observation that for a diagram $(K_1,K_1^+)\leftarrow (K_0,K_0^+)\to (K_2,K_2^+)$ of perfectoid fields with open and bounded valuation subrings for which $\Spa(K_i,K_i^+)\to \Spa(K_0,K_0^+)$ is surjective for $i=1,2$, one can find another such pair $(K_3,K_3^+)$ sitting in a diagram
\[\xymatrix{
&\Spa(K_1,K_1^+)\ar[dr] & \\
\Spa(K_3,K_3^+)\ar[ur]\ar[rr]\ar[dr] && \Spa(K_0,K_0^+)\ ,\\
& \Spa(K_2,K_2^+)\ar[ur] &
}\]
with all transition maps surjective. Indeed, this follows from surjectivity of
\[
|\Spa(K_1,K_1^+)\times_{\Spa(K_0,K_0^+)} \Spa(K_2,K_2^+)|\to |\Spa(K_1,K_1^+)|\times_{|\Spa(K_0,K_0^+)|} |\Spa(K_2,K_2^+)|\ .
\]

There is a canonical map $|X|\to |Y|$ by first interpreting $|X|$ as the similar set of equivalence classes of maps $\Spa(K,K^+)\to X$, which is a standard fact. This map is surjective. Indeed, given a map $\Spa(K,K^+)\to Y$, one can lift it to $X$ after a pro-\'etale cover. Thus, after increasing $K$, we can lift it to a map $\Spa(K,K^+)\to X$, as desired. One checks directly that this map $|X|\to |Y|$ factors over $|X|/|R|$.

Now assume that two maps $\Spa(K_i,K_i^+)\to X$, $i=1,2$, project to the same point of $|Y|$. By definition of $|Y|$, this means that up to enlarging $K_1$ and $K_2$, we can assume that $K=K_1=K_2$, $K^+=K_1^+=K_2^+$, and the composite maps $\Spa(K,K^+)\rightrightarrows X\to Y$ agree. In other words, we get a map $\Spa(K,K^+)\to R=X\times_Y X$, proving that $|X|/|R|\to |Y|$ is injective, finishing the proof that $|X|/|R|$ maps bijectively to $|Y|$.

Finally, to see that the quotient topology on $|Y|$ is independent of the presentation, choose another presentation $Y=X^\prime/R^\prime$. Then $X\times_Y X^\prime$ is a diamond, so we can find a pro-\'etale surjection $X^{\prime\prime}\to X\times_Y X^\prime$ from a perfectoid space $X^{\prime\prime}$. Now $|X^{\prime\prime}|\to |X|$ and $|X^{\prime\prime}|\to |X^\prime|$ are quotient maps by Lemma~\ref{lem:quotientmap}, so the quotient topologies on $|Y|$ induced by the surjection from $|X|$, $|X^{\prime\prime}|$ and $|X^\prime|$ are equivalent.
\end{proof}

In particular, the following definition makes sense.

\begin{definition}\label{def:topspace} Let $Y$ be a diamond, and write $Y=X/R$ as the quotient of a perfectoid space $X$ by a pro-\'etale equivalence relation $R\subset X\times X$. The underlying topological space of $|Y|$ is the quotient space $|X|/|R|$.
\end{definition}

It is easy to see that this construction is functorial in $Y$. One can read off the open subspaces of a diamond $Y$ on the topological space $|Y|$.

\begin{proposition}\label{prop:topspaceopensubsets} Let $Y$ be a diamond with underlying topological space $|Y|$. Any open subfunctor of $Y$ is a diamond. Moreover, there is bijective correspondence between open immersions $U\subset Y$ and open subsets $V\subset |Y|$, given by sending $U$ to $V=|U|$, and $V$ to the subfunctor $U\subset Y$ of those maps $X\to Y$ for which $|X|\to |Y|$ factors over $V\subset |Y|$.

Moreover, if $f: Y^\prime\to Y$ is a map of diamonds which is a surjection of v-sheaves, then $|f|: |Y^\prime|\to |Y|$ is a quotient map.
\end{proposition}

\begin{proof} Write $Y=X/R$ as usual. Given an open immersion $U\subset Y$, one gets an open subspace $U_X\subset X$ which is stable under the equivalence relation $R$; then $U=U_X/(R\cap (U_X\times U_X))$, where $R\cap (U_X\times U_X)\subset U_X\times U_X$ is a pro-\'etale equivalence relation, so that $U$ is a diamond.

In particular, in this situation $|U_X|\subset |X|$ is an open $|R|$-invariant subspace of $|X|$, giving rise to the open subspace $V=|U|\subset |X|$. Conversely, if $V\subset |Y|$ is open, then the preimage $V_X\subset |X|$ is an open $|R|$-invariant subset, giving rise an open $R$-invariant subspace $U_X\subset X$, whose quotient $U_X/R\subset Y$ defines an open immersion. One checks easily that the two processes are inverse.

The final statement follows from the fact that the prestack sending a v-sheaf $Y$ to its open subsheaves is a v-stack, which follows from Proposition~\ref{prop:checksheafpropvloc}~(i).
\end{proof}

\begin{example} If $Y$ is a perfectoid space, then $|Y|$ is the usual underlying topological space of $Y$. In the context of Example~\ref{ex:compacthausdorff}, if $T$ is a compact Hausdorff space and $K$ is a perfectoid field, then
\[
|\underline{T}\times \Spa(K,\OO_K)| = T\ .
\]
\end{example}

In particular, the example shows that in general $|Y|$ can be quite far from a spectral space.

\begin{definition}\label{def:spatial} Let $Y$ be a diamond. Then $Y$ is \emph{spatial} if $Y$ is quasicompact and quasiseparated, and $|Y|$ admits a basis of open subsets given by $|U|$ for quasicompact open immersions $U\subset Y$. More generally, $Y$ is \emph{locally spatial} if $Y$ admits an open cover by spatial diamonds.
\end{definition}

We note that any perfectoid space is locally spatial, and it is spatial precisely when it is qcqs.

\begin{proposition}\label{prop:spatialfirstprop} Let $Y$ be a spatial diamond.
\begin{altenumerate}
\item[{\rm (i)}] The underlying topological space $|Y|$ is spectral.
\item[{\rm (ii)}] Any quasicompact open subfunctor $U\subset Y$ is spatial.
\item[{\rm (iii)}] For any perfectoid space $X^\prime$ with a map $X^\prime\to Y$, the induced map of locally spectral spaces $|X^\prime|\to |Y|$ is spectral and generalizing.
\end{altenumerate}
\end{proposition}

\begin{proof} For part (i), first note that $|Y|$ is quasicompact, as by Proposition~\ref{prop:topspaceopensubsets}, any open cover of $|Y|$ determines an open cover of $Y$. Note that the set of quasicompact open immersions $U\subset Y$ is stable under finite intersections (as $Y$ is quasiseparated); it follows that the set of $|U|\subset |Y|$ for $U$ running over quasicompact open immersions forms a basis of quasicompact open subsets stable under finite intersections. It remains to check that $|Y|$ is sober, i.e.~every irreducible closed subset has a unique generic point. Uniqueness holds because $|Y|$ is $T_0$ by Lemma~\ref{lem:equivrel}. For existence of generic points, write $Y=X/R$, where $X$ is a qcqs perfectoid space and $R$ is a pro-\'etale equivalence relation, so that $|X|\to |Y|$ is a surjective spectral map of spectral spaces. Let $Z\subset |Y|$ be an irreducible closed subset, with preimage $W\subset |X|$. It is enough to see that the intersection of all nonempty quasicompact open subsets of $Z$ is nonempty, as any point in the intersection is a generic point for $Z$. But the preimages of quasicompact open subsets of $Z$ are quasicompact open subsets of $W$, and by quasicompactness of the constructible topology on $W$, their intersection is nonempty; thus, the intersection is nonempty in $Z$, as desired.

In (ii), any quasicompact open subfunctor $U\subset Y$ is given by a quasicompact open subset $|U|\subset |Y|$. As $Y$ is quasiseparated, $U$ is quasiseparated. Moreover, if $V\subset Y$ is a quasicompact open subfunctor, then $V\times_Y U\subset U$ is a quasicompact open subfunctor. Now the $|V\times_Y U|\subset |U|$ form a basis of open subsets, verifying that $U$ is spatial.

For part (iii), we may assume that $X^\prime$ is qcqs, in which case $|X^\prime|$ and $|Y|$ are spectral. We need to check that the preimage of a quasicompact open is quasicompact. If $V\subset |Y|$ is a quasicompact open, then it is covered by finitely many $|U_i|\subset |Y|$ for quasicompact open $U_i\subset Y$, so it remains to check that $X^\prime\times_Y U_i\subset X^\prime$ is quasicompact. But this follows from the assumption that $Y$ is quasiseparated. To see that $|X^\prime|\to |Y|$ is generalizing, write $Y=X/R$ as usual, with $X$ strictly totally disconnected. Then we can replace $X^\prime$ by $X^\prime\times_Y X$, and assume that $X^\prime\to Y$ lifts to $X^\prime\to X$. Now $|X^\prime|\to |X|$ is generalizing, so it remains to see that $|X|\to |Y|$ is generalizing. But this follows from Lemma~\ref{lem:spectralquotient}.
\end{proof}

Let us note that some of these properties extend to locally spatial diamonds.

\begin{proposition}\label{prop:locallyspatial} Let $Y$ be a locally spatial diamond.
\begin{altenumerate}
\item[{\rm (i)}] The underlying topological space $|Y|$ is locally spectral.
\item[{\rm (ii)}] Any open subfunctor $U\subset Y$ is locally spatial.
\item[{\rm (iii)}] The functor $Y$ is quasicompact (resp.~quasiseparated) if and only if $|Y|$ is quasicompact (resp.~quasiseparated).
\item[{\rm (iv)}] For any locally spatial diamond $Y^\prime$ with a map $Y^\prime\to Y$, the induced map of locally spectral spaces $|Y^\prime|\to |Y|$ is spectral and generalizing.
\end{altenumerate}
In particular, for any $y\in |Y|$, its set of generalizations is totally ordered.
\end{proposition}

\begin{proof} Left as an exercise to the reader. For the final statement, choose a surjection from a perfectoid space and use (iv) and the corresponding result for perfectoid spaces.
\end{proof}

We need some permanence properties of the class of (locally) spatial diamonds.

\begin{proposition}\label{prop:qcvinjspatdiamond} Let $f: Y^\prime\to Y$ be a quasicompact injection of v-sheaves, where $Y$ is a locally spatial diamond. Then $Y^\prime$ is a locally spatial diamond, $|Y^\prime|\subset |Y|$ is a pro-constructible and generalizing subset (with the subspace topology), and the map
\[
Y^\prime\to Y\times_{\underline{|Y|}} \underline{|Y^\prime|}
\]
of v-sheaves is an isomorphism.
\end{proposition}

\begin{proof} We may assume that $Y$ is spatial. Let $X\to Y$ be a surjective quasi-pro-\'etale map where $X$ is a strictly totally disconnected perfectoid space. By Corollary~\ref{cor:qcvinjqproet}, $X^\prime=X \times_Y Y^\prime$ is a pro-constructible and generalizing subspace of $X$; in particular $X^\prime = X\times_{\underline{|X|}} \underline{|X^\prime|}$. The spectral map $|X|\to |Y|$ of spectral spaces is generalizing, so the image $V\subset |Y|$ of $|X^\prime|$ in $|Y|$ is a pro-constructible and generalizing subset, with $|X^\prime| = |X|\times_{|Y|} V$. Also, if $R=X\times_Y X$ is the equivalence relation, which is affinoid pro-\'etale over $X$, then $R^\prime = R\times_Y Y^\prime$ is a pro-constructible and generalizing subset of $R$. One has a homeomorphism $|X^\prime|/|R^\prime| = V$: Indeed, the map $|X^\prime|/|R^\prime|\to V$ is a continuous bijective map, and $|X^\prime|\to V$ is a quotient map by Lemma~\ref{lem:quotientmap}. Thus, $|Y^\prime|=V$, and any quasicompact open subset of $V$ comes via pullback from a quasicompact open subset of $|Y|$, which corresponds to a quasicompact open subspace of $Y$, and pulls back to a quasicompact open subspace of $Y^\prime$: Thus, $Y^\prime$ is spatial. The natural map $Y^\prime\to Y\times_{\underline{|Y|}} \underline{|Y^\prime|}$ becomes an isomorphism after pullback along $X\to Y$. As it is a map of v-sheaves, we see that $Y^\prime = Y\times_{\underline{|Y|}} \underline{|Y^\prime|}$, as desired.
\end{proof}

\begin{lemma}\label{lem:finetoverspatial} Let $Y$ be a (locally) spatial diamond, and let $Y^\prime\to Y$ be a finite \'etale map of pro-\'etale sheaves. Then $Y^\prime$ is a (locally) spatial diamond.
\end{lemma}

\begin{proof} We may assume that $Y$ is spatial. By Proposition~\ref{prop:quasiproetaleoverdiamond}, $Y^\prime$ is a diamond. Choose a presentation $Y=X/R$ as a quotient of an affinoid perfectoid space $X$, and let $X^\prime = X\times_Y Y^\prime$, which is finite \'etale over $X$. Let $R$ and $R^\prime$ be the respective equivalence relations. We need to find enough quasicompact open subspaces of $Y^\prime$. Pick any point $y^\prime\in |Y^\prime|$ with image $y\in |Y|$. Let $Y_y\subset Y$ be the localization of $Y$ at $y$, i.e.~the intersection of all quasicompact open subspaces containing $y$. For any qcqs perfectoid space $Z$ over $Y$, let $Z_y=Z\times_Y Y_y$. Then quasicompact open subspaces of $Z_y$ come via pullback from quasicompact open subspaces of $Z\times_Y U$ for some quasicompact open subspace $U$ of $Y$, and any two such extensions agree after shrinking $U$. By applying this to $Z=X^\prime$ and $Z=R^\prime$, we find that any quasicompact open subspace of $Y^\prime_y$ extends to a quasicompact open subspace of $Y^\prime_U$ for some $U$. This reduces us to the case $Y=Y_y$.

Thus, we may assume that $|Y|$ has a unique closed point. In this case, we can take $X=\Spa(C,C^+)$ for some algebraically closed field $C$ with an open and bounded valuation subring $C^+\subset C$. Then $X^\prime$ is a finite disjoint union of copies of $X$. We need to show that $|Y'|$ is a spectral space and $|X'|\to |Y'|$ is a spectral map. Now $|X'|$ is a disjoint union of copies of $|X|$, and $|X|$ is a totally ordered chain of specializations. As all maps are generalizing, it follows that $|Y'| = |X'|/|X'\times_{Y'} X'|$ looks like a (finite disjoint union of) tree(s): Over the generic point, one has some finite set of points, and as one goes to specializations, more branches may show up. The branching points are happening at pro-constructible subsets, i.e. the locus where two sections $X\to Y'$ induce the same map $|X|\to |Y|$ is a pro-constructible generalizing subset of $|X|$. Indeed, it is the image of $|X \times_{Y'} X|\to |X|$ (under either projection). Any such ``tree'' is spectral, with the projection from each branch being spectral, for example by writing it as an inverse limit of such trees where the branching occurs at quasicompact open subsets.
\end{proof}

\begin{lemma}\label{lem:diamondlimit} Let $Y_i$, $i\in I$, be a cofiltered inverse system of diamonds with qcqs transition maps. Then $Y=\varprojlim_i Y_i$ is a diamond with a bijective continuous map $|Y|\to \varprojlim_i |Y_i|$, and the maps $Y\to Y_i$ are qcqs. If all $Y_i$ are (locally) spatial, then $Y$ is (locally) spatial, and $|Y|\to \varprojlim_i |Y_i|$ is a homeomorphism.

If, with $\kappa$ is as in Lemma~\ref{lem:choosekappa}, the category $I$ is $\kappa$-small and all $Y_i$ are $\kappa'$-small for some $\kappa'<\kappa$, then $Y$ is $\kappa$-small.
\end{lemma}

Here being $\kappa$-small means admitting a surjection from a $\kappa$-small perfectoid space.

\begin{proof} We can assume that $I$ has a final object $0$, and that $Y_0$ is qcqs, hence all $Y_i$ are. We can choose a quasi-pro-\'etale surjection $X_0\to Y_0$ from some strictly totally disconnected space $X_0$, and quasi-pro-\'etale surjections $X_i\to Y_i\times_{Y_0} X_0$ from strictly totally disconnected $X_i$. Then $Y$ embeds into the product $Y'$ of all $Y_i$ over $Y_0$, which admits a quasi-pro-\'etale surjection from the product $X'$ of all $X_i$ over $X_0$. Then $X'\times_{Y'} Y\to Y$ is a quasi-pro-\'etale surjection, where $X'\times_{Y'} Y\subset X'$ is a quasicompact injection into some affinoid perfectoid space. By Proposition~\ref{prop:vinjdiamond}, it follows that $X'\times_{Y'} Y$, and hence $Y$, is a diamond. One also easily gets the desired bound on the size, using that if $I$ is $\lambda$-small for $\lambda<\kappa$, we can find $\kappa'\leq \kappa_\lambda<\kappa$ such that $\kappa_\lambda$ has cofinality $>\lambda$.

It follows from this construction that $Y$ is qcqs, hence in general the maps $Y\to Y_i$ are qcqs. For the claims about topological spaces, we give a refined argument. Namely, we can assume that $I$ is the category of ordinals $\mu<\lambda$ for some ordinal $\lambda$, and as before we can assume that $Y_0$ (and then all $Y_i$) are qcqs. In this situation, we will lift the diagram $Y_\mu$, $\mu<\lambda$, to a diagram of strictly totally disconnected spaces $X_\mu$, $\mu<\lambda$, with compatible quasi-pro-\'etale surjections $X_\mu\to Y_\mu$; in fact, with quasi-pro-\'etale surjections
\[
X_\mu\to Y_\mu\times_{\varprojlim_{\mu^\prime<\mu} Y_{\mu^\prime}} \varprojlim_{\mu^\prime<\mu} X_{\mu^\prime}\ .
\]
That this can be done follows easily by transfinite induction, and we note that one obtains a quasi-pro-\'etale surjection
\[
\varprojlim_{\mu<\lambda} X_\mu\to \varprojlim_{\mu<\lambda} Y_\mu
\]
in the limit. Indeed, after pullback to any strictly totally disconnected $Z$, this morphism is a transfinite composition of surjective affinoid pro-\'etale morphisms (using Lemma~\ref{lem:proetaleoverwlocal}).

In this construction, one has
\[
|Y|=|\varprojlim_{\mu<\lambda} X_\mu|/|\varprojlim_{\mu<\lambda} R_\mu| = \varprojlim_{\mu<\lambda} |X_\mu|/|R_\mu| = \varprojlim_{\mu<\lambda} |Y_\mu|\ ,
\]
as sets, where $R_\mu = X_\mu\times_{Y_\mu} X_\mu$. If all $Y_i$ are spatial, then Lemma~\ref{lem:spectralinvlimit} applied to the inverse system $|Y_\mu|$ and the map from $|\varprojlim_{\mu<\lambda} X_\mu|$ to the inverse limit shows that
\[
|\varprojlim_{\mu<\lambda} X_\mu|\to \varprojlim_{\mu<\lambda} |Y_\mu|
\]
is a quotient map, and thus $|Y|\to \varprojlim_{\mu<\lambda} |Y_\mu|$ is a homeomorphism. Moreover, any quasicompact open subspace of $|Y|$ comes via pullback from some quasicompact open subspace of $|Y_i|$, which corresponds to a quasicompact open subspace of $Y_i$, which in turn pulls back to a quasicompact open subspace of $Y$; thus, $Y$ is spatial. If the $Y_i$ are just locally spatial, one can fix a spatial open subset of some $Y_{i_0}$, and taking the preimage gives a spatial open subset of $Y$; also $|Y|=\varprojlim_i |Y_i|$ follows.
\end{proof}

Moreover, in this situation, an analogue of Proposition~\ref{prop:etmaptolim} holds true.

\begin{proposition}\label{prop:etmaptolimdiamond} Let $Y_i$, $i\in I$, be a cofiltered inverse system of qcqs diamonds, with inverse limit $Y=\varprojlim_i Y_i$.
\begin{altenumerate}
\item[{\rm (i)}] Let $Y_\fet$ denote the category of finite \'etale diamonds over $Y$, and similarly for $Y_i$. The base change functors $(Y_i)_\fet\to Y_\fet$ induce an equivalence of categories
\[
\text{2-}\varinjlim_i (Y_i)_\fet\to Y_\fet\ .
\]
\item[{\rm (ii)}] Let $Y_{\et,\qcqs}$ denote the category of (locally separated) \'etale qcqs diamonds over $Y$, and similarly for $Y_i$. The base change functors $(Y_i)_{\et,\qcqs}\to Y_{\et,\qcqs}$ induce an equivalence of categories
\[
\text{2-}\varinjlim_i (Y_i)_{\et,\qcqs}\to Y_{\et,\qcqs}\ .
\]
\item[{\rm (iii)}] Let $Y_{\et,\qc,\sep}\subset Y_{\et,\qcqs}$ be the full subcategory of quasicompact separated \'etale diamonds over $Y$ (and similarly for $Y_i$). Then the base change functors $(Y_i)_{\et,\qc,\sep}\to Y_{\et,\qc,\sep}$ induce an equivalence of categories
\[
\text{2-}\varinjlim_i (Y_i)_{\et,\qc,\sep}\to Y_{\et,\qc,\sep}\ .
\]
\end{altenumerate}
\end{proposition}

\begin{proof} As in Lemma~\ref{lem:diamondlimit}, we can assume that $I$ is the set of ordinals $\mu<\lambda$ for some fixed ordinal $\lambda$. We pick strictly totally disconnected perfectoid spaces $X_\mu$ mapping compatibly to $Y_\mu$ as in the proof of Lemma~\ref{lem:diamondlimit}. Let $R_\mu = X_\mu\times_{Y_\mu} X_\mu$ be the induced equivalence relations, and let $X=\varprojlim_\mu X_\mu$, $R=\varprojlim_\mu R_\mu$.

Let $F$ be the prestack on the category of diamonds of finite \'etale, resp. quasicompact separated \'etale morphisms. Then $F$ is a stack for the v-topology by Proposition~\ref{prop:checksheafpropvloc}. Note that $F(X/Y)$ involves the value of $F$ on $X$, $R$, and $R\times_X R$, all of which are affinoid perfectoid spaces, and similarly for $F(X_\mu/Y_\mu)$. By Proposition~\ref{prop:etmaptolim}, we see that
\[
F(Y) = F(X/Y) = \text{2-}\varinjlim_{\mu<\lambda} F(X_\mu/Y_\mu) = \text{2-}\varinjlim_{\mu<\lambda} F(Y_\mu)\ ,
\]
proving parts (i) and (iii). For part (ii), essentially the same argument applies, except that a priori the functor $F(Y)\to F(X/Y)$ is only fully faithful for the prestack $F$ of qcqs \'etale morphisms. To check essential surjectivity in (ii), pick any \'etale qcqs map $f: \tilde{Y}\to Y$, where it is understood that $f$ is locally separated. In particular, one can cover $\tilde{Y}$ by quasicompact open subsets $\tilde{Y}_j\subset \tilde{Y}$ such that $f|_{\tilde{Y}_j}$ is separated. By (iii), this comes via pullback from some finite level. Moreover, the gluing data between the different $\tilde{Y}_j$ will also be defined at some finite level, and the cocycle condition satisfied. This produces a (locally separated) \'etale qcqs map to some finite stage with pullback $\tilde{Y}$, as desired.
\end{proof}

To prove further permanence properties of spatial diamonds, we need the following result.

\begin{proposition}\label{prop:spatialunivopen} Let $Y$ be a spatial diamond. Then one can find a strictly totally disconnected perfectoid space $X$ with a surjective and universally open quasi-pro-\'etale map $f: X\to Y$ that can be written as a cofiltered inverse limit of \'etale maps which are composites of quasicompact open immersions and finite \'etale maps. If $\kappa$ is as in Lemma~\ref{lem:choosekappa} and $Y$ is $\kappa$-small, one can take $X$ to be $\kappa$-small.

Conversely, assume that $Y$ is a qcqs diamond such that $Y$ admits a surjective and universally open quasi-pro-\'etale map $X\to Y$, where $X$ is a perfectoid space. Then $Y$ is a spatial diamond.
\end{proposition}

\begin{remark}\label{rem:spatialunivopen} Of course, the proposition implies the more general converse that if $Y$ is a qcqs diamond such that $Y$ admits a surjective and universally open quasi-pro-\'etale map $Y^\prime\to Y$, where $Y^\prime$ is a (locally) spatial diamond, then $Y$ is a spatial diamond.
\end{remark}

\begin{proof} The converse is easy: If $X\to Y$ is surjective and universally open quasi-pro-\'etale, and without loss of generality $X$ is affinoid, and in fact strictly totally disconnected (cf.~Lemma~\ref{lem:strtotdisccover}), then the induced equivalence relation $R\subset X\times X$ is a qcqs perfectoid space, for which the maps $s,t: R\to X$ are open; thus, $|Y|=|X|/|R|$ is spectral and $|X|\to |Y|$ is spectral by Lemma~\ref{lem:spectralopenequivrel}, which implies that $Y$ is spatial.

Now assume that $Y$ is spatial and fix some $\kappa$ as in Lemma~\ref{lem:choosekappa} such that $Y$ is $\kappa$-small. Consider the set $I$ of isomorphism classes of surjective \'etale maps $Y^\prime\to Y$ that can be written as a composite of quasicompact open embeddings and finite \'etale maps. (To see that this is of cardinality less than $\kappa$, realize these by descent data along some fixed surjective map from a $\kappa$-small affinoid perfectoid space. Moreover, each $Y^\prime$ admits a surjection from a $\kappa^\prime$-small affinoid perfectoid space for some fixed strong limit cardinal $\kappa^\prime<\kappa$.) By Lemma~\ref{lem:finetoverspatial}, all such $Y^\prime$ are spatial diamonds. Choose a representative $Y_i\to Y$ for each $i\in I$. For any subset $J\subset I$, let $Y_J$ be the product of all $Y_i$, $i\in J$, over $Y$, which is still a composite of quasicompact open embeddings and finite \'etale maps over $Y$, and thus a spatial diamond. By Lemma~\ref{lem:diamondlimit}, the cofiltered limit $Y_\infty$ of all $Y_J$ over finite subsets $J\subset I$ is again a spatial diamond, and $Y_\infty\to Y$ is a surjective and universally open quasi-pro-\'etale map that can be written as a cofiltered inverse limit of \'etale maps which are composites of quasicompact open immersions and finite \'etale maps. Repeating the construction $Y\mapsto Y_\infty$ countably often, we can find a surjective and universally open quasi-pro-\'etale map $X\to Y$ of spatial diamonds, which can be written as a cofiltered inverse limit of \'etale maps which are composites of quasicompact open immersions and finite \'etale maps. Now any \'etale cover $\tilde{X}\to X$ that can be written as a composite of quasicompact open immersions and finite \'etale maps splits: Indeed, by Proposition~\ref{prop:etmaptolimdiamond}, any such map comes from some finite level, and becomes split at the next. Now the result follows from the next proposition.
\end{proof}

\begin{proposition}\label{prop:strtotdiscdiamond} Let $Y$ be a spatial diamond. Assume that any surjective \'etale map $\tilde{Y}\to Y$ that can be written as a composite of quasicompact open immersions and finite \'etale maps splits. Then $Y$ is a strictly totally disconnected perfectoid space.
\end{proposition}

\begin{proof} First note that any quasicompact open cover of $|Y|$ splits by assumption; thus, every connected component of $|Y|$ has a unique closed point. Thus, any connected component $Y_0\subset Y$ admits a pro-\'etale surjection $\Spa(C,C^+)\to Y_0$ for some algebraically closed field $C$ with an open and bounded valuation subring $C^+\subset C$. First, we claim this map $\Spa(C,C^+)\to Y_0$ is an isomorphism. Let $R_0 = \Spa(C,C^+)\times_{Y_0} \Spa(C,C^+)$ be the equivalence relation, which is affinoid pro-\'etale over $\Spa(C,C^+)$. To check that $\Spa(C,C^+)=Y_0$, we have to check that $R_0=\Spa(C,C^+)$. If not, then $R_0$ has another maximal point. We see that it is enough to show that $\Spa(C,\OO_C)\to Y_0^\circ$ is an isomorphism, where $Y_0^\circ\subset Y_0$ is the open subfunctor corresponding to the maximal point (which is open). But now $R_0^\circ = R_0\times_{Y_0} Y_0^\circ$ is isomorphic to $\underline{S}\times \Spa(C,\OO_C)$ for some profinite set $S$, where in fact $S=G$ is a profinite group (by the equivalence relation structure). Assume $G$ is nontrivial, and let $H\subset G$ be a proper open subgroup. Then $\underline{H}\times \Spa(C,\OO_C)\subset \underline{G}\times \Spa(C,\OO_C)=R_0^\circ$ is another equivalence relation, and the corresponding quotient of $\Spa(C,\OO_C)$ is a nontrivial finite \'etale cover of $Y_0^\circ$. Note that finite \'etale covers of $Y_0^\circ$ agree with finite \'etale covers of $Y_0$ (as this is true for the pair $\Spa(C,\OO_C)$ and $\Spa(C,C^+)$, and the pair $R_0$ and $R_0^\circ$). Now any finite \'etale cover of $Y_0$ extends to a finite \'etale cover of an open and closed subset of $Y$, which together with the complementary open and closed subset of $Y$ forms an \'etale cover as in the statement of the proposition. As this is assumed to be split, the original finite \'etale cover of $Y_0$ has to be split, which is a contradiction. Thus, $Y_0=\Spa(C,C^+)$.

In other words, we have seen that any connected component of $Y_0$ is given by $\Spa(C,C^+)$ for some algebraically closed field $C$ with an open and bounded valuation subring $C^+\subset C$. The result now follows from the next lemma.
\end{proof}

\begin{lemma}\label{lem:diamondconncomprepr} Let $Y$ be a spatial diamond. Assume that every connected component of $Y$ is representable by an affinoid perfectoid space. Then $Y$ is representable by an affinoid perfectoid space.
\end{lemma}

\begin{proof} Write $Y=X/R$ as a quotient of a strictly totally disconnected perfectoid space $X$ by an affinoid pro-\'etale equivalence relation $R\subset X\times X$. Let $X=\Spa(B,B^+)$ and $R=\Spa(C,C^+)$, and define
\[
(A,A^+) = \eq((B,B^+)\rightrightarrows (C,C^+))\ .
\]
The goal is to prove that $A$ is perfectoid, $A^+\subset A$ is open and integrally closed, and $Y=\Spa(A,A^+)$.

As a first step, we find a pseudouniformizer $\varpi_B\in B$ such that $p_1^\ast(\varpi_B)=p_2^\ast(\varpi_B)u$ for some unit $u\in 1+C^{\circ\circ}$; here, $p_1,p_2: R\to X$ denote the two projections. We can find such a pseudouniformizer over each connected component $c\in \pi_0(Y)$ (by taking one coming via pullback from the affinoid connected component of $Y$), and arbitrarily lifting it to $B$, this relation will be satisfied in the preimage of an open and closed neighborhood of $c\in \pi_0(Y)$. As the set of connected components is profinite, we can then find such a $\varpi_B\in B$ globally.

Now we check that if an element $t\in p_1^\ast(\varpi_B)^n C^{\circ\circ}=p_2^\ast(\varpi_B)^n C^{\circ\circ}$ (for some $n\geq 0$) satisfies the cocycle relation
\[
p_{23}^\ast(t) + p_{12}^\ast(t) = p_{13}^\ast(t)
\]
(corresponding to the three maps $p_{12},p_{13},p_{23}: R\times_X R\to R$), then there is an element $s\in \varpi_B^n B^{\circ\circ}$ with $p_1^\ast(s)-p_2^\ast(s)=t$. It suffices to check that we can find $s\in \varpi_B^n B^{\circ\circ}$ such that
\[
p_1^\ast(s)-p_2^\ast(s)-t\in p_1^\ast(\varpi_B)^{n+1} C^{\circ\circ}\ ,
\]
as then the result follows by induction. But in each connected component $Y_c\subset Y$, we can find such an $s$ (by almost vanishing of $H^1_v(Y_c,\OO^+)$), and lifting it arbitrarily to $B$, it will satisfy the desired congruence in a neighborhood; again, as $\pi_0(Y)$ is profinite, we can find some such $s$ globally.

Applying this to $t=p_1^\ast(\varpi_B)-p_2^\ast(\varpi_B)\in p_1^\ast(\varpi_B) C^{\circ\circ}$, we find some $s\in \varpi_B B^{\circ\circ}$ such that $p_1^\ast(s)-p_2^\ast(s)=t$. This implies that $\varpi:=\varpi_B-s\in A$, and is a topologically nilpotent unit in $B$, and thus in $A$. Now we replace $\varpi_B$ by $\varpi$, so $\varpi$ gives a compatible choice of pseudouniformizer in all rings.

From the definition of $A$, it is clear that $A$ is uniform and perfect, so $A$ is perfectoid. Also, $A^+\subset A$ is open and integrally closed. Thus, $Y^\prime:=\Spa(A,A^+)$ is an affinoid perfectoid space, and we get a natural map $Y\to Y^\prime$. To see that this is an isomorphism, it suffices by Lemma~\ref{lem:isomspatialdiamondspoints} to check on connected components. It remains to see that the construction of $(A,A^+)$ commutes with passage to connected components. But the arguments above show that
\[
(A^+/\varpi)^a = \eq((B^+/\varpi)^a\rightrightarrows (C^+/\varpi)^a)\ ,
\]
and this statement passes to open and closed subsets (i.e., direct summands on the level of rings), and filtered colimits. By the equivalence of perfectoid $A$-algebras with perfectoid $(A^+/\varpi)^a$-algebras, this gives the result on the level of the perfectoid Tate ring. For the open and integrally closed subalgebra, it suffices to check that if $Y_c=\Spa(A_c,A_c^+)$ is a connected component of $Y$, then $A^+\to A_c^+$ is surjective. But any $f\in A_c^+$ can be lifted to some function in $g\in B^+$ with $p_1^\ast(g)-p_2^\ast(g)\in \varpi C^{\circ\circ}$, and then using the claim above we can correct it by some function $s\in \varpi B^{\circ\circ}$ such that $g-s\in A^+$. Thus, $A^+\to A_c^+/\varpi$ is surjective, which implies that $A^+\to A_c^+$ is surjective, as $A^+$ and $A_c^+$ are $\varpi$-adically complete.
\end{proof}

Now we get the following two permanence properties.

\begin{corollary}\label{cor:quasiproetspatial} Let $Y$ be a locally spatial diamond, and $Y^\prime\to Y$ a quasi-pro-\'etale map of pro-\'etale sheaves. Then $Y^\prime$ is a locally spatial diamond.
\end{corollary}

Recall that by Convention~\ref{conv:etlocsep}, the map $Y^\prime\to Y$ is automatically required to be locally separated.

\begin{proof} We may assume that $Y$ is spatial, and that $Y^\prime\to Y$ is separated. By Proposition~\ref{prop:spatialunivopen}, we can find a surjective and universally open quasi-pro-\'etale map $X\to Y$, where $X$ is strictly totally disconnected. Then $X^\prime = X\times_Y Y^\prime$ is representable and pro-\'etale over $X$, and $R^\prime = X^\prime\times_{Y^\prime} X^\prime$ is representable and qcqs over $X^\prime$. Moreover, $X^\prime$ is quasiseparated, so $|Y^\prime|=|X^\prime|/|R^\prime|$ is a locally spectral space and $|X^\prime|\to |Y^\prime|$ is a spectral map, by Lemma~\ref{lem:spectralopenequivrel}. This proves that $Y^\prime$ is locally spatial.
\end{proof}

\begin{corollary}\label{cor:fibreproductspatial} A fibre product of (locally) spatial diamonds is (locally) spatial.
\end{corollary}

\begin{proof} Let $Y_1\to Y_3\leftarrow Y_2$ be a diagram of locally spatial diamonds. We may assume that $Y_3$, $Y_1$, and $Y_2$ are spatial. Assume first that $Y_1\to Y_3$ is a quasi-separated quasi-pro-\'etale map. Then the fibre product $Y_1\times_{Y_3} Y_2$ is quasi-separated quasi-pro-\'etale over $Y_2$, so the result follows from Corollary~\ref{cor:quasiproetspatial}. In general, let $X_3\to Y_3$ be a surjective and universally open quasi-pro-\'etale map from a qcqs perfectoid space $X_3$, as guaranteed by Proposition~\ref{prop:spatialunivopen}. By Remark~\ref{rem:spatialunivopen}, it is enough to prove that
\[
(Y_1\times_{Y_3} Y_2)\times_{Y_3} X_3 = (Y_1\times_{Y_3} X_3)\times_{X_3} (Y_2\times_{Y_3} X_3)
\]
is a spatial diamond. As the fibre products $Y_1\times_{Y_3} X_3$ and $Y_2\times_{Y_3} X_3$ are known to exist by the beginning of the proof (as $X_3\to Y_3$ is quasi-separated quasi-pro-\'etale), we may assume that $Y_3=X_3$ is a qcqs perfectoid space.

Now choose $X_1\to Y_1$ and $X_2\to Y_2$ surjective and universally open quasi-pro-\'etale maps. Then $X_1\times_{Y_3} X_2\to Y_1\times_{Y_3} Y_2$ is again surjective and universally open quasi-pro-\'etale, so the result follows from Proposition~\ref{prop:spatialunivopen}.
\end{proof}

Finally, there is the following 2-out-of-3 property for quasi-pro-\'etale maps.

\begin{proposition}\label{prop:qproetthirdmap} Let $f: Y_1\to Y_2$ and $g: Y_2\to Y_3$ be maps of locally spatial diamonds, with composite $h=g\circ f: Y_1\to Y_3$. Assume that $f$ is quasi-pro-\'etale and surjective, $h$ is quasi-pro-\'etale, and $g$ is separated. Then $g$ is quasi-pro-\'etale.

If $f$ and $h$ are in addition \'etale (resp.~finite \'etale), then also $g$ is \'etale (resp.~finite \'etale).
\end{proposition}

\begin{proof} By the previous corollary, we may assume that $X_3=Y_3$ is a strictly totally disconnected perfectoid space. Moreover, we can assume that $Y_2$ is spatial. Now $Y_1=X_1$ is representable and pro-\'etale over $X_3$. Replacing $X_1$ by an open cover, we can assume that $X_1$ is a disjoint union of strictly totally disconnected perfectoid spaces. A finite union already covers $Y_2$, so we can assume that $X_1$ is strictly totally disconnected, and thus affinoid pro-\'etale over $X_3$. Finally, we can replace $X_3$ by $X_3\times_{\pi_0(X_3)} \pi_0(Y_2)$, so that $\pi_0(Y_2) = \pi_0(X_3)$.

We claim that in this situation, $|f|: |Y_2|\to |X_3|$ is automatically injective. The argument is the same as in the proof of Lemma~\ref{lem:proetaleoverwlocal}, but we repeat it for convenience. We can check on connected components, so we may assume that $X_3=\Spa(C,C^+)$ is a connected strictly totally disconnected space, and then by assumption also $Y_2$ is connected. Now assume that two distinct points $y_1, y_2\in |Y_2|$ map to the same point $x\in X_3$. If $x$ corresponds to a valuation ring $(C^+)^\prime\subset C$, then $y_1, y_2$ give rise to points still denoted $y_1,y_2\in Y_2(C,(C^+)^\prime)$ (by lifting further to $X_1$). By the valuative criterion of separatedness, if $y_1\neq y_2$, then the corresponding $(C,\OO_C)$-points are still distinct, so we can assume that $(C^+)^\prime = \OO_C$. Then $x$ is the unique rank $1$ point of $X$, and the set of preimages of $x$ forms a profinite set $Y_{2,x}\subset Y_2$, as they form a spectral set without specializations. By our assumption, $Y_{2,x}$ contains at least two points, so we can find a closed and open decomposition $Y_{2,x}=U_{1,x}\sqcup U_{2,x}$, for some quasicompact open subsets $U_1,U_2\subset Y_2$. Let $V_1, V_2\subset Y_2$ be the closures of $U_{1,x}$ and $U_{2,x}$. As $U_{1,x}$ and $U_{2,x}$ are pro-constructible subsets, their closures are precisely the subsets of specializations of points in $U_{1,x}$ resp.~$U_{2,x}$. As any point of $Y_2$ generalizes to a unique point of $Y_{2,x}$, we have $Y_2=V_1\sqcup V_2$. As $V_1$ and $V_2$ are both closed, this gives a contradiction to our assumption that $Y_2$ is connected, finishing the proof that $|f|: |Y_2|\to |X_3|$ is injective.

Moreover, the image of $|f|$ is a pro-constructible and generalizing subset of $|X_3|$, and thus by Lemma~\ref{lem:subsetwlocal} there is an affinoid pro-\'etale $X_2\subset X_3$ whose image is precisely $|Y_2|$; thus $Y_2\to X_3$ factors over $Y_2\to X_2$; we can thus also replace $X_3$ by $X_2$. We claim that the map $Y_2\to X_2$ is an isomorphism. For this, we will use Lemma~\ref{lem:isomspatialdiamondspoints}. This reduces us to the case that $X_3=\Spa(C,C^+)$ where $C$ is algebraically closed, and $C^+\subset C$ is an open and bounded valuation subring. But now as $|Y_2|=|X_2|=|X_3|=|\Spa(C,C^+)|$, also $|Y_2|$ has a unique closed point, and so we can replace $X_1=\Spa(C^\prime,C^{\prime+})$ by its localization at one point. But then $(C^\prime,C^{\prime+}) = (C,C^+)$ (as $X_1\to X_2$ is affinoid pro-\'etale and surjective), and the composite $X_1\to Y_1\to X_2$ is an isomorphism, where the first map is a surjection. Thus, $X_1=Y_1=X_2$, as desired.

If $f$ and $h$ are \'etale (resp.~finite \'etale), then after the previous reductions, $Y_2=X_2$ is strictly totally disconnected. In particular, $f$ admits a section $s: Y_2\to Y_1$, which is automatically \'etale (resp.~finite \'etale) (as a section of an \'etale, resp.~finite \'etale, map), and then $g=hs$ is \'etale (resp.~finite \'etale).
\end{proof}

Another useful lemma is a structure result for the local nature of \'etale morphisms of locally spatial diamonds, generalizing Definition~\ref{def:etale}~(ii).

\begin{lemma}\label{lem:spatialetalelocallysep} Let $f: Y^\prime\to Y$ be an \'etale map of locally spatial diamonds. Then for every $y^\prime\in |Y^\prime|$ with image $y\in |Y|$, one can find open neighborhoods $V^\prime\subset Y^\prime$ of $y^\prime$ and $V\supset f(V^\prime)$ of $y$ such that $f|_{V^\prime}: V^\prime\to V$ factors as the composite of a quasicompact open immersion $V^\prime\hookrightarrow W$ and a finite \'etale map $W\to V$.
\end{lemma}

We note again that by Convention~\ref{conv:etlocsep}, the map $f$ is required to be locally separated. (This is necessary, as $f|_{V^\prime}$ is separated.)

\begin{proof} We can assume that $Y$ is spatial, and then also that $Y^\prime$ is spatial; thus $f: Y^\prime\to Y$ is a qcqs \'etale map of spatial diamonds, which we may moreover assume to be separated. We claim that it suffices to prove that for all $y\in Y$ with localization $Y_y$, the map $Y^\prime\times_Y Y_y\to Y_y$ factors as the composite of a quasicompact open immersion $Y^{\prime}\times_Y Y_y\to W_y$ and a finite \'etale map $W_y\to Y_y$. Indeed, both the finite \'etale map $W_y\to Y_y$ and the quasicompact open immersion arise via base change from $V$ for some quasicompact open subfunctor $V\subset Y$, by Proposition~\ref{prop:etmaptolimdiamond}~(i) and Lemma~\ref{lem:diamondlimit}, respectively. Up to replacing $Y$ by $V$, we get some map $V^\prime\to Y$ which is a composite of a quasicompact open immersion and a finite \'etale morphism, such that $V^\prime\times_Y Y_y\cong Y^\prime\times_Y Y_y$ over $Y_y$. Shrinking $Y$ again, we may assume by Proposition~\ref{prop:etmaptolimdiamond}~(ii) that the isomorphism $V^\prime\times_Y Y_y\cong Y^\prime\times_Y Y_y$ is defined over $Y$ already. This gives the desired result.

Thus, we can assume that $y$ is the unique closed point of $Y$. In that case, $|Y|$ is a totally ordered chain of specializations (being a quotient of $|\Spa(K,K^+)|$ for a perfectoid field $K$ with an open and bounded valuation subring $K^+\subset K$). Let $Y^\circ\subset Y$ correspond to the unique open point, and $Y^{\prime\circ} = Y^\prime\times_Y Y^\circ$. Then $Y^{\prime\circ}\to Y^\circ$ is \'etale, and thus finite \'etale: Indeed, $Y^\circ$ is covered by $\Spa(K,\OO_K)$, and $\Spa(K,\OO_K)_\et = \Spa(K,\OO_K)_\fet$. Moreover, $Y^\circ_\fet\cong Y_\fet$, as finite \'etale spaces are insensitive to $\OO^+$. We see that there is some finite \'etale space $W\to Y$ with $W\times_Y Y^\circ\cong Y^{\prime\circ}$. The composite map $Y^{\prime\circ}\cong W\times_Y Y^\circ\subset W$ extends uniquely to a map $Y^\prime\to W$ over $Y$, by checking the similar result after pullback to $\Spa(K,K^+)$ (and uniqueness). Moreover, by Lemma~\ref{lem:compactificationetalemap}, the map $Y^\prime\to W$ is an injection. It is also \'etale as a map between spaces \'etale over $Y$, and thus an open immersion; this finishes the proof.
\end{proof}

\section{Small v-stacks}

The goal of this section is two-fold. First, we generalize many basic results to the setting of general (small) v-stacks. Secondly, we give a criterion for when a v-sheaf is a diamond, Theorem~\ref{thm:vdiamondisdiamond}, without exhibiting an explicit quasi-pro-\'etale surjection. In this way, it is similar to Artin's theorem on algebraic spaces reducing smooth (or even flat) groupoids to \'etale groupoids.

As already discussed before, we make the following definition.

\begin{definition}\label{def:vdiamond} A small v-sheaf is a v-sheaf $Y$ on $\Perf$ such that there is a surjective map of v-sheaves $X\to Y$ for some perfectoid space $X$.
\end{definition}

\begin{remark}\label{rem:quasicompactsmall} Clearly, if $Y$ is a diamond, then $Y$ is a small v-sheaf. Also, if $Y$ is a v-sheaf such that there is a surjective map of v-sheaves $X\to Y$ for some diamond $X$, then $Y$ is a small v-sheaf. Finally, if $Y$ is a quasicompact v-sheaf, then it is small (as one can always cover $Y$ by the disjoint union of all maps $X\to Y$ from perfectoid spaces $X$).
\end{remark}

Perhaps surprisingly, the following proposition shows that any small v-sheaf has a reasonable geometric structure.

\begin{proposition}\label{prop:smallvsheaf} Let $Y$ be a small v-sheaf, and let $X\to Y$ be a surjective map of v-sheaves, where $X$ is a diamond. Then $R=X\times_Y X$ is a diamond, and $Y=X/R$ as v-sheaves.

If $Y$ is a quasiseparated small v-sheaf, and $X\to Y$ is a surjective map of v-sheaves, where $X$ is a locally spatial diamond, then $R=X\times_Y X$ is a locally spatial diamond. In particular, if $Y$ is a qcqs v-sheaf, and $X$ is a spatial diamond, then $R$ is a spatial diamond.
\end{proposition}

\begin{proof} As $R\subset X\times X$ is a sub-v-sheaf, and $X\times X$ is a diamond, the first part follows from Proposition~\ref{prop:vinjdiamond}. If $Y$ is quasiseparated, then $R=X\times_Y X\subset X\times X$ is a quasicompact sub-v-sheaf, so if $X$ is locally spatial, then Proposition~\ref{prop:qcvinjspatdiamond} shows that $R$ is a locally spatial diamond. For the final sentence, note that if $Y$ is quasicompact, then it is small. Now if $X$ and $Y$ are qcqs, then so is $R=X\times_Y X$, so $R$ is spatial.
\end{proof}

This allows us to go one step further and pass to v-stacks.

\begin{definition} A small v-stack is a v-stack $Y$ on $\Perf$ such that there is a surjective map of v-stacks $X\to Y$ from a perfectoid space $X$, for which $R=X\times_Y X$ is a small v-sheaf.
\end{definition}

In particular, all qcqs v-stacks are small (arguing as in Remark~\ref{rem:quasicompactsmall}). Lemma~\ref{lem:isomspatialdiamondspoints} extends to v-sheaves and more generally v-stacks.

\begin{lemma}\label{lem:isomqcqsvsheaves} Let $f: Y^\prime\to Y$ be a qcqs map of v-stacks. Then $f$ is an isomorphism if and only if for all algebraically closed perfectoid fields $K$ with an open and bounded valuation subring $K^+\subset K$, the map $f(K,K^+): Y^\prime(K,K^+)\to Y(K,K^+)$ is an equivalence of groupoids.
\end{lemma}

\begin{remark} We remind the reader of Convention~\ref{conv:stackquasisep} concerning quasiseparated maps of stacks.
\end{remark}

\begin{proof} It suffices to check the result after pullback to any affinoid perfectoid space $X$ mapping to $Y$. In that case, $Y^\prime$ is a qcqs v-stack. Write $Y^\prime=X^\prime/R^\prime$ as the quotient of an affinoid perfectoid space $X^\prime$; then $R^\prime$ is a qcqs v-sheaf. The map $X^\prime\to X$ is a v-cover, as it is a map of qcqs perfectoid spaces which by assumption is surjective on points. As both $X^\prime$ and $X$ are v-sheaves, it suffices to see that the map of v-sheaves $R^\prime\to X^\prime\times_X X^\prime$ is an isomorphism.

After this reduction, we have a new map $Y^\prime\to X$ from a qcqs v-sheaf $Y^\prime$ (the $R^\prime$ from above) to an affinoid perfectoid space $X$ (the $X^\prime\times_X X^\prime$ from above). Again, write $Y^\prime=X^\prime/R^\prime$ as the quotient of an affinoid perfectoid space $X^\prime$ by the equivalence relation $R^\prime$, which by Proposition~\ref{prop:smallvsheaf} is a spatial diamond. Repeating the arguments, it is enough to see that $R^\prime\to X^\prime\times_X X^\prime$ is an isomorphism of spatial diamonds. But the map $R(K,K^+)\to (X^\prime\times_X X^\prime)(K,K^+)$ is a bijection for all $(K,K^+)$ as in the statement of the lemma. Thus, the result follows from Lemma~\ref{lem:isomspatialdiamondspoints}.
\end{proof}

Moreover, one can define an underlying topological space for small v-stacks.

\begin{proposition}\label{prop:vtopspaceindep} Let $Y$ be a small v-stack, and let $Y=X/R$ be a presentation as a quotient of a diamond $X$ by a small v-sheaf $R\to X\times X$; let $\tilde{R}\to R$ be a surjection from a diamond. There is a canonical bijection between $|X|/|\tilde{R}|$ and the set
\[
|Y|=\{\Spa(K,K^+)\to Y\}/\sim\ ,
\]
where $(K,K^+)$ runs over all pairs of a perfectoid field $K$ with an open and bounded valuation subring $K^+\subset K$, and two such maps $\Spa(K_i,K_i^+)\to Y$, $i=1,2$, are equivalent, if there is a third pair $(K_3,K_3^+)$, and a commutative diagram
\[\xymatrix{
&\Spa(K_1,K_1^+)\ar[dr] & \\
\Spa(K_3,K_3^+)\ar[ur]\ar[rr]\ar[dr] && Y\ ,\\
& \Spa(K_2,K_2^+)\ar[ur] &
}\]
where the maps $\Spa(K_3,K_3^+)\to \Spa(K_i,K_i^+)$, $i=1,2$, are surjective.

Moreover, the quotient topology on $|Y|$ induced by the surjection $|X|\to |Y|$ is independent of the choice of presentation $Y=X/R$.
\end{proposition}

\begin{proof} The proof of Proposition~\ref{prop:topspaceindep} carries over.
\end{proof}

Thus, the following definition is independent of the choice made.

\begin{definition}\label{def:vtopspace} Let $Y$ be a small v-stack with a presentation $Y=X/R$ as a quotient of a diamond $X$ by a small v-sheaf $R\to X\times X$, and $\tilde{R}\to R$ a surjection from a diamond. The underlying topological space of $Y$ is $|Y|:=|X|/|\tilde{R}|$.
\end{definition}

Note that we are doing two steps at once here, really; one could also define this first for small v-sheaves, and then for small v-stacks as $|Y|=|X|/|R|$. But as $|\tilde{R}|\to |R|$ is surjective, this recovers the same answer.

The following version of Proposition~\ref{prop:topspaceopensubsets} holds for small v-stacks.

\begin{proposition}\label{prop:vtopspaceopensubsets} Let $Y$ be a small v-stack with underlying topological space $|Y|$. Any open sub-v-stack of $Y$ is a small v-stack. Moreover, there is bijective correspondence between open sub-v-stacks $U\subset Y$ and open subsets $V\subset |Y|$, given by sending $U$ to $V=|U|$, and $V$ to the subfunctor $U\subset Y$ of those maps $X\to Y$ for which $|X|\to |Y|$ factors over $V\subset |Y|$.

Moreover, if $f: Y^\prime\to Y$ is a surjective map of small v-stacks, then $|f|: |Y^\prime|\to |Y|$ is a quotient map.
\end{proposition}

\begin{proof} The same arguments as in the proof of Proposition~\ref{prop:topspaceopensubsets} apply.
\end{proof}

The following proposition follows from the similar result for perfectoid spaces.

\begin{proposition}\label{prop:topfiberproduct} Let $Y_1\to Y_3\leftarrow Y_2$ be a diagram of small v-stacks. Then $Y=Y_1\times_{Y_3} Y_2$ is a small v-stack, and
\[
|Y|\to |Y_1|\times_{|Y_3|} |Y_2|
\]
is surjective.$\hfill \Box$
\end{proposition}

In the following, when we say that map $f: Y^\prime\to Y$ of (small) v-stacks is surjective, we mean that it is surjective as a map of v-stacks. There is a related notion that $|f|: |Y^\prime|\to |Y|$ is surjective; if we mean this, we say it explicitly. The relation is as follows.\footnote{We thank David Hansen for related discussions.}

\begin{lemma}\label{lem:surjvstopsurj} Let $f: Y^\prime\to Y$ be a map of small v-stacks. If $f$ is a surjective map of v-stacks, then $|f|: |Y^\prime|\to |Y|$ is surjective. Conversely, if $f$ is quasicompact and $|f|: |Y^\prime|\to |Y|$ is surjective, then $f$ is a surjective map of v-stacks.
\end{lemma}

\begin{proof} The first part is clear. For the converse, assume that $f$ is quasicompact and that $|f|: |Y^\prime|\to |Y|$ is surjective. If $X$ is an affinoid perfectoid space with a map $X\to Y$, then the pullback $Y^\prime\times_Y X$ is a quasicompact v-stack, and thus admits a surjection $X^\prime\to Y^\prime\times_Y X$ from an affinoid perfectoid space. Thus, $|X^\prime|\to |Y^\prime\times_Y X|$ is surjective, and also $|Y^\prime\times_Y X|\to |Y^\prime|\times_{|Y|} |X|\to |X|$ is a composition of surjections (using Proposition~\ref{prop:topfiberproduct}). Therefore, $X^\prime\to X$ is a map of affinoid perfectoid spaces for which $|X^\prime|\to |X|$ is surjective; thus, by definition, it is a v-cover. Thus, the map $X\to Y$ lifts to $X^\prime\to Y^\prime$ after the v-cover $X^\prime\to X$, as desired.
\end{proof}

We can now introduce a notion of spatial v-sheaves.

\begin{definition} A v-sheaf $Y$ is spatial if it is qcqs (in particular, small), and $|Y|$ has a basis for the topology given by $|U|$ for quasicompact open subfunctors $U\subset Y$. More generally, $Y$ is locally spatial if it is small, and has an open cover by spatial v-sheaves.
\end{definition}

Proposition~\ref{prop:spatialfirstprop} extends to spatial v-sheaves.

\begin{proposition}\label{prop:vspatialfirstprop} Let $Y$ be a spatial v-sheaf.
\begin{altenumerate}
\item[{\rm (i)}] The underlying topological space $|Y|$ is spectral.
\item[{\rm (ii)}] Any quasicompact open subfunctor $U\subset Y$ is spatial.
\item[{\rm (iii)}] For any perfectoid space $X^\prime$ with a map $X^\prime\to Y$, the induced map of locally spectral spaces $|X^\prime|\to |Y|$ is spectral and generalizing.
\end{altenumerate}
\end{proposition}

\begin{proof} The proof is identical to the proof of Proposition~\ref{prop:spatialfirstprop}, noting that any qcqs v-sheaf $Y$ can be written in the form $Y=X/R$ for spatial diamonds $X$ and $R$.
\end{proof}

\begin{proposition}\label{prop:vlocallyspatial} Let $Y$ be a locally spatial v-sheaf.
\begin{altenumerate}
\item[{\rm (i)}] The underlying topological space $|Y|$ is locally spectral.
\item[{\rm (ii)}] Any open subfunctor $U\subset Y$ is locally spatial.
\item[{\rm (iii)}] The functor $Y$ is quasicompact (resp.~quasiseparated) if and only if $|Y|$ is quasicompact (resp.~quasiseparated).
\item[{\rm (iv)}] For any locally spatial v-sheaf $Y^\prime$ with a map $Y^\prime\to Y$, the induced map of locally spectral spaces $|Y^\prime|\to |Y|$ is spectral and generalizing.
\end{altenumerate}
In particular, if $y\in |Y|$, then the set of generalizations of $y$ is totally ordered. $\hfill \Box$
\end{proposition}

There is the following analogue of Proposition~\ref{prop:charinjection} for locally spatial v-sheaves.

\begin{proposition}\label{prop:charinjectionvsheaves} Let $f: Y^\prime\to Y$ be a map of small v-sheaves, and assume that $f$ is qcqs, or $Y^\prime$ and $Y$ are locally spatial. The following conditions are equivalent.
\begin{altenumerate}
\item The map $f$ is an injective map of v-sheaves.
\item For all perfectoid fields $K$ with an open and bounded valuation subring $K^+\subset K$, the map $f(K,K^+): Y^\prime(K,K^+)\to Y(K,K^+)$ is injective.
\item The map $|f|: |Y^\prime|\to |Y|$ is injective, and the map $f: Y^\prime\to Y$ is final in the category of maps $g: Z\to Y$ from small v-sheaves $Z$ for which $|g|: |Z|\to |Y|$ factors over a continuous map $|Z|\to |Y^\prime|$.
\end{altenumerate}
\end{proposition}

Note that (iii) says equivalently that
\[
Y^\prime = Y\times_{\underline{|Y|}} \underline{|Y^\prime|}\ .
\]

\begin{proof} Clearly, (iii) implies (i) implies (ii), so we have to see that (ii) implies (iii). Injectivity of $|f|$ follows from the description of the set $|Y|$ in Proposition~\ref{prop:vtopspaceindep}. Now assume that $g: Z\to Y$ has the property that $|Z|\to |Y|$ lifts continuously to $|Z|\to |Y^\prime|$. We claim that there is a unique map $Z\to Y^\prime$. This can be done v-locally on $Z$, so we can assume that $Z$ is a perfectoid space, or even a strictly totally disconnected perfectoid space. In case $Y$ and $Y^\prime$ are locally spatial, we can assume that $|Z|$ maps into qcqs open subsets of $|Y|$ and $|Y^\prime|$, we can work locally on $Z$, which allows us to assume that $Y$ and $Y^\prime$ are spatial, and in particular $f$ is qcqs. Now $Z\times_Y Y^\prime\to Z$ is a qcqs map of v-sheaves, and we need to see that it is an isomorphism. By Lemma~\ref{lem:isomqcqsvsheaves}, this can be checked on $(K,K^+)$-valued points, so we can further assume that $Z=\Spa(K,K^+)$. In that case, we know that $Z\times_Y Y^\prime\to Z$ is an injection. But quasicompact injections into $\Spa(K,K^+)$ are of the form $\Spa(K,(K^+)^\prime)$ by Corollary~\ref{cor:qcvinjqproet}. Looking at topological spaces shows that $(K^+)^\prime=K^+$, so that indeed $Z\times_Y Y^\prime=Z$, as desired.
\end{proof}

Moreover, we will need an analogue of Lemma~\ref{lem:finetoverspatial}.

\begin{lemma}\label{lem:finetovervspatial} Let $Y$ be a spatial v-sheaf, and let $Y^\prime\to Y$ be a finite \'etale map of v-sheaves. Then $Y^\prime$ is a spatial v-sheaf.
\end{lemma}

\begin{proof} The proof is identical to Lemma~\ref{lem:finetoverspatial}, but we repeat it for convenience. Choose a presentation $Y=X/R$ as a quotient of an affinoid perfectoid space $X$, and let $X^\prime = X\times_Y Y^\prime$, which is finite \'etale over $X$. Let $R$ and $R^\prime$ be the respective equivalence relations. We need to find enough quasicompact open subspaces of $Y^\prime$. Pick any point $y^\prime\in |Y^\prime|$ with image $y\in |Y|$. Let $Y_y\subset Y$ be the localization of $Y$ at $y$, i.e.~the intersection of all quasicompact open subspaces containing $y$. For any spatial diamond $Z$ over $Y$, let $Z_y=Z\times_Y Y_y$. Then quasicompact open subspaces of $Z_y$ come via pullback from quasicompact open subspaces of $Z\times_Y U$ for some quasicompact open subspace $U$ of $Y$, and any two such extensions agree after shrinking $U$. By applying this to $Z=X^\prime$ and $Z=R^\prime$, we find that any quasicompact open subspace of $Y^\prime_y$ extends to a quasicompact open subspace of $Y^\prime_U$ for some $U$. This reduces us to the case $Y=Y_y$.

Thus, we may assume that $|Y|$ has a unique closed point. In this case, we can take $X=\Spa(C,C^+)$ for some algebraically closed field $C$ with an open and bounded valuation subring $C^+\subset C$. Then $X^\prime$ is a finite disjoint union of copies of $X$. Now the same topological argument as in the proof of Lemma~\ref{lem:finetoverspatial} shows that $|Y'|=|X'|/|X'\times_{Y'} X'|$ is spectral, with spectral projection $|X'|\to |Y'|$.
\end{proof}

Finally, we will need an analogue of Lemma~\ref{lem:diamondlimit} and Proposition~\ref{prop:etmaptolimdiamond}.

\begin{lemma}\label{lem:vdiamondlimit} Let $Y_i$, $i\in I$, be a cofiltered inverse system of small v-sheaves with qcqs transition maps. Then $Y=\varprojlim_i Y_i$ is a small v-sheaf, the maps $Y\to Y_i$ are qcqs, and the continuous map $|Y|\to\varprojlim_i |Y_i|$ is bijective. If all $Y_i$ are locally spatial, then $Y$ is locally spatial, and $|Y|\to \varprojlim_i |Y_i|$ is a homeomorphism.

If $\kappa$ is as in Lemma~\ref{lem:choosekappa}, $I$ is $\kappa$-small and all $Y_i$ are $\kappa'$-small for some $\kappa'<\kappa$, then $Y$ is $\kappa$-small.

Moreover, if all $Y_i$ are qcqs, then $Y$ is qcqs, and the base change functors $(Y_i)_\fet\to Y_\fet$, resp.~$(Y_i)_{\et,\qcqs}\to Y_{\et,\qcqs}$, resp.~$(Y_i)_{\et,\qc,\sep}\to Y_{\et,\qc,\sep}$, induce an equivalence of categories
\[
\text{2-}\varinjlim_i (Y_i)_\fet\to Y_\fet\ ,
\]
respectively
\[
\text{2-}\varinjlim_i (Y_i)_{\et,\qcqs}\to Y_{\et,\qcqs}\ ,
\]
respectively
\[
\text{2-}\varinjlim_i (Y_i)_{\et,\qc,\sep}\to Y_{\et,\qc,\sep}\ .
\]
\end{lemma}

We remark again that by our Convention~\ref{conv:etlocsep}, all \'etale maps are locally separated.

\begin{proof} It is easy to reduce to the case that all $Y_i$ are qcqs, by fixing some index $i=i_0$ and a surjection $Y_{i_0}^\prime\to Y_{i_0}$, where $Y_{i_0}^\prime$ is a disjoint union of qcqs v-sheaves, and replacing the whole diagram by the diagram of $Y_i\times_{Y_{i_0}} Y_{i_0}^\prime$. The set-theoretic part is proved as in Lemma~\ref{lem:diamondlimit}.

We can assume that $I$ is the category of ordinals $\mu$ less than some fixed ordinal $\lambda$. By transfinite induction on $\mu$, we can find a diagram $X_\mu$ of spatial diamonds with compatible surjections $X_\mu\to Y_\mu$; in fact, such that
\[
X_\mu\to Y_\mu\times_{\varprojlim_{\mu^\prime<\mu} Y_{\mu^\prime}} \varprojlim_{\mu^\prime<\mu} X_{\mu^\prime}
\]
is surjective. To see that this can be done, assume it has been done for $\mu^\prime<\mu$. Then
\[
\varprojlim_{\mu^\prime<\mu} X_{\mu^\prime}\to \varprojlim_{\mu^\prime<\mu} Y_{\mu^\prime}
\]
is a surjection of qcqs v-sheaves: Indeed, given a map $Z\to \varprojlim_{\mu^\prime<\mu} Y_{\mu^\prime}$ from a spatial diamond, the diagram
\[
Z_{\mu^\prime} := Z\times_{\varprojlim_{\mu^\prime<\mu} Y_{\mu^\prime}} X_{\mu^\prime}\subset Z\times X_{\mu^\prime}
\]
over varying $\mu^\prime<\mu$ is a diagram of spatial diamonds (using Proposition~\ref{prop:qcvinjspatdiamond}), where all maps $Z_{\mu^\prime}\to Z$ are surjective. Then $Z_\mu:=\varprojlim_{\mu^\prime<\mu} Z_{\mu^\prime}\to Z$ is also a surjective map of spatial diamonds, using Lemma~\ref{lem:diamondlimit}. Therefore,
\[
\varprojlim_{\mu^\prime<\mu} X_{\mu^\prime}\to \varprojlim_{\mu^\prime<\mu} Y_{\mu^\prime}
\]
is a qcqs surjection of v-sheaves, and the same holds for the base change
\[
Y_\mu\times_{\varprojlim_{\mu^\prime<\mu} Y_{\mu^\prime}} \varprojlim_{\mu^\prime<\mu} X_{\mu^\prime}\to Y_\mu\ .
\]
In particular, the source is a qcqs v-sheaf, and we can find a surjection $X_\mu\to Y_\mu$ from a spatial diamond $X_\mu$.

Rereading the above arguments in the case $\mu=\lambda$, we see that $Y=\varprojlim_{\mu<\lambda} Y_\mu$ admits a qcqs surjection from the spatial diamond $X=\varprojlim_{\mu<\lambda} X_\mu$, so $Y$ is a qcqs v-sheaf.

Let $R_\mu=X_\mu\times_{Y_\mu} X_\mu$ and $R=X\times_Y X=\varprojlim_{\mu<\lambda} R_\mu$ be the equivalence relations, which are spatial diamonds. Then
\[
|Y|=|X|/|R| = \varprojlim |X_\mu|/|R_\mu| = \varprojlim |Y_\mu|
\]
as sets. If all $Y_\mu$ are spatial, then the same arguments as in the proof of Lemma~\ref{lem:diamondlimit} ensure that $|Y|=\varprojlim |Y_\mu|$ is a homeomorphism, and $Y$ is spatial.

For the final statement about finite \'etale maps, let $F$ be the prestack on the category of small v-sheaves of finite \'etale, resp.~quasicompact separated \'etale, morphisms. Then $F$ is a stack for the v-topology by Proposition~\ref{prop:checksheafpropvloc}. Note that $F(X/Y)$ involves the value of $F$ on $X$, $R$, and $R\times_X R$, all of which are spatial diamonds, and similarly for $F(X_\mu/Y_\mu)$. By Proposition~\ref{prop:etmaptolimdiamond}, we see that
\[
F(Y) = F(X/Y) = \text{2-}\varinjlim F(X_\mu/Y_\mu) = \text{2-}\varinjlim F(Y_\mu)\ ,
\]
as desired.

This leaves the case of $Y_{\et,\qcqs}$. In that case, one still gets a fully faithful functor
\[
\text{2-}\varinjlim_i (Y_i)_{\et,\qcqs}\to Y_{\et,\qcqs}\ .
\]
For essential surjectivity, note that any $\tilde{Y}\to Y$ in $Y_{\et,\qcqs}$ is per convention locally separated, and thus locally lies in $Y_{\et,\qc,\sep}$. Using that case, we can descend to some finite level, and then we can also descend the gluing data.
\end{proof}

The following theorem will be used heavily to show that certain functors are (spatial) diamonds.

\begin{theorem}\label{thm:vdiamondisdiamond} Let $Y$ be a spatial v-sheaf such that there exists a perfectoid space $X$ with a quasi-pro-\'etale map $f: X\to Y$ for which $|f|: |X|\to |Y|$ is surjective. Then $Y$ is a spatial diamond.
\end{theorem}

\begin{remark} The hypothesis here is much weaker than asking that $f$ is surjective. An equivalent condition is to ask that for any $y\in |Y|$, there exists a quasi-pro-\'etale map $\Spa(C,C^+)\to Y$ having $y$ in its image, where $C$ is some algebraically closed nonarchimedean field with an open and bounded valuation subring $C^+\subset C$. Indeed, the disjoint union of all such $\Spa(C,C^+)\to Y$ over varying $y\in |Y|$ will then give the desired map $f: X\to Y$ from a perfectoid space $X$ (which will be highly non-quasicompact, so that Lemma~\ref{lem:surjvstopsurj} does not apply).

In particular, the condition in Theorem~\ref{thm:vdiamondisdiamond} is only a condition about the points of $Y$.
\end{remark}

\begin{proof} We need to see that $Y$ is a diamond, as the notion of being spatial is compatible. To see that $Y$ is a diamond, we need to find a surjective quasi-pro-\'etale map $X\to Y$, where $X$ is a perfectoid space. For this, we follow the argument in the proof of Proposition~\ref{prop:spatialunivopen}. Consider the set $I$ of isomorphism classes of surjective \'etale maps $Y^\prime\to Y$ that can be written as a composite of quasicompact open embeddings and finite \'etale maps. (To see that this is of cardinality less than $\kappa$, realize these by descent data along some fixed surjective map from a $\kappa$-small affinoid perfectoid space. Moreover, each $Y^\prime$ is $\kappa^\prime$-small for some fixed strong limit cardinal $\kappa^\prime<\kappa$.) By Lemma~\ref{lem:finetovervspatial}, all such $Y^\prime$ are spatial v-sheaves. Choose a representative $Y_i\to Y$ for each $i\in I$. For any subset $J\subset I$, let $Y_J$ be the product of all $Y_i$, $i\in J$, over $Y$, which is still a composite of quasicompact open embeddings and finite \'etale maps over $Y$, and thus a spatial v-sheaf. Thus, by Lemma~\ref{lem:vdiamondlimit}, the cofiltered limit $Y_\infty$ of all $Y_J$ over finite subsets $J\subset I$ is again a spatial v-sheaf, and $Y_\infty\to Y$ is a surjective separated quasi-pro-\'etale map. Repeating the construction $Y\mapsto Y_\infty$ countably often, we find a surjective separated quasi-pro-\'etale map $X\to Y$ of spatial v-sheaves. Now any \'etale cover $\tilde{X}\to X$ that can be written as a composite of quasicompact open immersions and finite \'etale maps splits: Indeed, by Lemma~\ref{lem:vdiamondlimit}, any such map comes from some finite level, and becomes split at the next. Now the result follows from the following generalization of Proposition~\ref{prop:strtotdiscdiamond}, noting that the condition on points lifts to $X$ (as $X\to Y$ is quasi-pro-\'etale).
\end{proof}

\begin{proposition}\label{prop:strtotdiscvdiamond} Let $Y$ be a spatial v-sheaf such that there exists a quasi-pro-\'etale map $f: X\to Y$ from a perfectoid space $X$ for which $|f|: |X|\to |Y|$ is surjective. Assume that any surjective \'etale map $\tilde{Y}\to Y$ that can be written as a composite of quasicompact open immersions and finite \'etale maps splits. Then $Y$ is a strictly totally disconnected perfectoid space.
\end{proposition}

\begin{proof} The proof is identical to the proof of Proposition~\ref{prop:strtotdiscdiamond}, but we repeat it for convenience. First note that any quasicompact open cover of $|Y|$ splits by assumption; thus, every connected component of $|Y|$ has a unique closed point. Thus, any connected component $Y_0\subset Y$ admits a quasi-pro-\'etale surjection $\Spa(C,C^+)\to Y_0$ for some algebraically closed field $C$ with an open and bounded valuation subring $C^+\subset C$. First, we claim this map $\Spa(C,C^+)\to Y_0$ is an isomorphism. Let $R_0 = \Spa(C,C^+)\times_{Y_0} \Spa(C,C^+)$ be the equivalence relation. To check that $\Spa(C,C^+)=Y_0$, we have to check that $R_0=\Spa(C,C^+)$. If not, then $R_0$ has another maximal point. We see that it is enough to show that $\Spa(C,\OO_C)\to Y_0^\circ$ is an isomorphism, where $Y_0^\circ\subset Y_0$ is the open subfunctor corresponding to the maximal point (which is open). But now $R_0^\circ = R_0\times_{Y_0} Y_0^\circ$ is isomorphic to $\underline{S}\times \Spa(C,\OO_C)$ for some profinite set $S$ as $R_0^\circ\to \Spa(C,\OO_C)$ is a qcqs pro-\'etale map, where in fact $S=G$ is a profinite group (by the equivalence relation structure). Assume $G$ is nontrivial, and let $H\subset G$ be a proper open subgroup. Then $\underline{H}\times \Spa(C,\OO_C)\subset \underline{G}\times \Spa(C,\OO_C)=R_0^\circ$ is another equivalence relation, and the corresponding quotient of $\Spa(C,\OO_C)$ is a nontrivial finite \'etale cover of $Y_0^\circ$. Note that finite \'etale covers of $Y_0^\circ$ agree with finite \'etale covers of $Y_0$ (as this is true for the pair $\Spa(C,\OO_C)$ and $\Spa(C,C^+)$, and the pair $R_0$ and $R_0^\circ$). Now any finite \'etale cover of $Y_0$ extends to a finite \'etale cover of an open and closed subset of $Y$ by Lemma~\ref{lem:vdiamondlimit}, which together with the complementary open and closed subset of $Y$ forms an \'etale cover as in the statement of the proposition. As this is assumed to be split, the original finite \'etale cover of $Y_0$ has to be split, which is a contradiction. Thus, $Y_0=\Spa(C,C^+)$.

In other words, we have seen that any connected component of $Y_0$ is given by $\Spa(C,C^+)$ for some algebraically closed field $C$ with an open and bounded valuation subring $C^+\subset C$. The result now follows from the next lemma, which generalizes Lemma~\ref{lem:diamondconncomprepr}.
\end{proof}

\begin{lemma}\label{lem:vdiamondconncomprepr} Let $Y$ be a spatial v-sheaf. Assume that every connected component of $Y$ is representable by an affinoid perfectoid space. Then $Y$ is representable by an affinoid perfectoid space.
\end{lemma}

\begin{proof} Write $Y=X/R$ as the quotient of an affinoid perfectoid space $X$ by a spatial diamond equivalence relation $R\subset X\times X$. By Lemma~\ref{lem:diamondconncomprepr}, the spatial diamond $R$ is an affinoid perfectoid space. Now the same argument as in the proof of Lemma~\ref{lem:diamondconncomprepr} applies.
\end{proof}

\section{Spatial morphisms}

In this section, we define a notion of (locally) spatial morphisms, and prove that it behaves well. All notions we will consider will be examples of $0$-truncated maps.

\begin{definition}\label{def:reprindiamonds} A map $f: Y^\prime\to Y$ of v-stacks is \emph{representable in diamonds} if for all diamonds $X$ with a map $X\to Y$, the fibre product $Y^\prime\times_Y X$ is a diamond.
\end{definition}

The following proposition ensures that this notion is well-behaved.

\begin{proposition}\label{prop:reprindiamonds} Let $f: Y^\prime\to Y$ and $\tilde{Y}\to Y$ be maps of v-stacks, with pullback $\tilde{f}: \tilde{Y}^\prime = Y^\prime\times_Y \tilde{Y}\to \tilde{Y}$.
\begin{altenumerate}
\item If $Y$ is a diamond, then $f$ is representable in diamonds if and only if $Y^\prime$ is a diamond.
\item If $f$ is representable in diamonds, then $\tilde{f}$ is representable in diamonds.
\item If $\tilde{Y}\to Y$ is surjective \emph{as a map of pro-\'etale stacks} and $\tilde{f}$ is representable in diamonds, then $f$ is representable in diamonds.
\end{altenumerate}
\end{proposition}

\begin{proof} Part (i) follows from Proposition~\ref{prop:diamondfibreproduct}, and part (ii) is clear by definition. For part (iii), we may assume that $Y$ is a diamond. We can find a surjective quasi-pro-\'etale $X\to Y$, where $X$ is a disjoint union of strictly totally disconnected perfectoid spaces, and it suffices to see that $X\times_Y Y^\prime$ is a diamond by Proposition~\ref{prop:quotientsbydiamondequivrel}. This can be checked locally on $X$, so we can reduce to the case that $Y=X$ is a strictly totally disconnected perfectoid space. As $\tilde{Y}\to Y=X$ is surjective as a map of pro-\'etale sheaves, we can find a surjective pro-\'etale map $\tilde{X}\to X$ with a lift $\tilde{X}\to \tilde{Y}$. Then $\tilde{X}\times_X Y^\prime$ is a diamond, and thus $Y^\prime$ by Proposition~\ref{prop:quotientsbydiamondequivrel}.
\end{proof}

Now we can define locally spatial morphisms.

\begin{definition}\label{def:locallyspatialmorphism} A map $f: Y^\prime\to Y$ of v-stacks is \emph{representable in (locally) spatial diamonds} if for all (locally) spatial diamonds $X$ with a map $X\to Y$, the fibre product $Y^\prime\times_Y X$ is a (locally) spatial diamond.
\end{definition}

If $f: Y^\prime\to Y$ is a map of diamonds which is representable in (locally) spatial diamonds, we will sometimes simply say that $f: Y^\prime\to Y$ is a (locally) spatial map of diamonds.

Clearly, if a map of v-stacks $f: Y^\prime\to Y$ is representable in spatial diamonds, then it is representable in locally spatial diamonds, and the converse holds precisely when $f$ is qcqs.

\begin{proposition}\label{prop:locallyspatialmorphism} Let $f: Y^\prime\to Y$ and $\tilde{Y}\to Y$ be maps of v-stacks, with pullback $\tilde{f}: \tilde{Y}^\prime = Y^\prime\times_Y \tilde{Y}\to \tilde{Y}$.
\begin{altenumerate}
\item If $f$ is representable in locally spatial diamonds, then $f$ is representable in diamonds.
\item If $Y$ is a locally spatial diamond, then $f$ is representable in locally spatial diamonds if and only if $Y^\prime$ is a locally spatial diamond.
\item If $f$ is representable in locally spatial diamonds, then $\tilde{f}$ is representable in locally spatial diamonds.
\item If $\tilde{Y}\to Y$ is surjective as a map of pro-\'etale stacks, $f$ is quasiseparated and $\tilde{f}$ is representable in locally spatial diamonds, then $f$ is representable in locally spatial diamonds.
\item If $f$ is representable in diamonds, $\tilde{Y}\to Y$ is a surjective map of v-stacks, $f$ is quasiseparated and $\tilde{f}$ is representable in locally spatial diamonds, then $f$ is representable in locally spatial diamonds.
\end{altenumerate}
\end{proposition}

\begin{proof} Part (i) follows from Proposition~\ref{prop:reprindiamonds} (i) and (iii). Part (ii) follows from Corollary~\ref{cor:fibreproductspatial}, and part (iii) is clear by definition.

For part (iv), we may assume that $Y$ is a locally spatial diamond; in fact, we can assume that $Y$ is spatial. By Proposition~\ref{prop:reprindiamonds}~(iii), we know that $Y^\prime$ is a diamond. By Proposition~\ref{prop:spatialunivopen} and Lemma~\ref{lem:spectralopenequivrel}, we can further reduce to the case that $Y$ is strictly totally disconnected: By Proposition~\ref{prop:spatialfirstprop}, we can find a universally open map $X\to Y$ where $X$ is strictly totally disconnected, and if $X\times_Y Y^\prime$ is a locally spatial diamond, then so is $Y^\prime$ by Lemma~\ref{lem:spectralopenequivrel}.

Now assume $Y=X$ is strictly totally disconnected. We can find an affinoid pro-\'etale map $\tilde{X}\to X$ which lifts to $\tilde{Y}$, and may assume that $\tilde{Y}=\tilde{X}$. Now $\tilde{Y}^\prime := \tilde{X}\times_X Y^\prime$ is a locally spatial diamond. We claim that this implies that $Y^\prime$ is a locally spatial diamond.

Let $V\subset Y^\prime$ be an open subfunctor, and let $y\in |V|\subset |Y^\prime|$ be a point, lying in a connected component $c\in \pi_0 X$. By a subscript $_c$, we denote the fiber of all objects over $c$. Note that $\tilde{X}_c\to X_c$ splits as $X_c$ is a connected strictly totally disconnected space; this implies that $Y^\prime_c$ is locally spatial. Fix a quasicompact open subset $U_c\subset Y^\prime_c$ containing $y$. We can find a quasicompact open subset $U\subset \tilde{Y}^\prime$ contained in the preimage of $V$, whose intersection with $\tilde{Y}^\prime_c$ is given by the preimage of $U_c$. The two preimages of $U$ in $\tilde{Y}^\prime\times_{Y^\prime} \tilde{Y}^\prime$ are two quasicompact open subsets $W_1,W_2\subset \tilde{Y}^\prime\times_{Y^\prime}\tilde{Y}^\prime$ whose fibers over $c$ agree. By standard properties of spectral spaces, it follows that there is some open and closed neighborhood $U_c$ of $c$ in $\pi_0 X$ such that the intersection of $W_1$ and $W_2$ with the preimage of $U_c$ agree. Thus, the intersection of $U$ with the preimage of $U_c$ descends to $Y^\prime$, and defines a quasicompact open subfunctor of $Y^\prime$ containing $y$. This proves that $Y^\prime$ is locally spatial, as desired.

Finally, for part (v), we can use part (iv) to reduce to the case that $Y=X$ is a strictly totally disconnected perfectoid space, in which case $Y^\prime$ is a qcqs diamond. By Lemma~\ref{lem:conncomplocspat} below, we can assume that $X=\Spa(C,C^+)$ is a connected strictly totally disconnected perfectoid space. Then we can assume that similarly $\tilde{X}=\Spa(\tilde{C},\tilde{C}^+)$ is a connected strictly totally disconnected perfectoid space. Write $\tilde{X}$ as an inverse limit of rational subspaces $\tilde{X}_i$ of perfected closed balls over $X$. Let $V\subset Y^\prime$ be an open subspace, with preimage $\tilde{V}\subset \tilde{Y}^\prime$, and let $y\in |V|$. Then there is some quasicompact open subspace $\tilde{U}\subset \tilde{V}$ containing the preimage of $y$. This quasicompact open subspace spreads to a quasicompact open subspace $\tilde{U}_i\subset \tilde{Y}^\prime_i=Y^\prime\times_X \tilde{X}_i$ for $i$ sufficiently large by arguing as in Lemma~\ref{lem:conncomplocspat}. Moreover, $\tilde{U}_i\subset \tilde{V}_i$ for $i$ sufficiently large by a standard quasicompactness argument. Now the map $\tilde{X}_i\to X$ has a section by Lemma~\ref{lem:secoveralgclosed}. Pulling back $\tilde{U}_i$ under this section gives a quasicompact open subspace of $Y^\prime$ contained in $V$, and containing $y$, as desired.
\end{proof}

\begin{lemma}\label{lem:conncomplocspat} Let $Y$ be a quasiseparated small v-sheaf with a map $Y\to \underline{S}$ for some profinite set $S$. Assume that for all $s\in S$, the fiber $Y_s$ is locally spatial. Moreover, assume that there is a surjective qcqs map $X\to Y$ from a locally spatial diamond $X$. Then $Y$ is locally spatial.
\end{lemma}

We note that the existence of the surjective qcqs map $X\to Y$ from a locally spatial diamond $X$ is automatic if $Y$ is quasicompact.

\begin{proof} Choose a surjection $X\to Y$ as in the statement of the lemma, and let $R=X\times_Y X$ be the induced equivalence relation, which is a locally spatial diamond by Proposition~\ref{prop:smallvsheaf}. Let $V\subset Y$ be an open subspace with preimage $U\subset X$, and let $y\in |U|$ be a point, with image $s\in S$. Then we can find a quasicompact open subspace $V^\prime_s\subset Y_s\cap V$. Its preimage $U^\prime_s\subset X_s$ extends to a quasicompact open subspace $U^\prime_T\subset X_T\cap U$ for any sufficiently small compact open neighborhood $T$ of $s$ in $S$. Moreover, $U^\prime_T$ is invariant under the equivalence relation for $T$ sufficiently small (as the two preimages under $s,t: R_T\to X_T$ are quasicompact open subspace whose fibers over $s$ agree, so a quasicompactness argument applies), and thus descends to a quasicompact open subspace $V^\prime_T\subset Y_T\cap U$ for $T$ sufficiently small. But then in particular $V^\prime_T\subset U$ is a quasicompact open subspace of $Y$ containing $y$, which proves that $Y$ is locally spatial.
\end{proof}

The following characterization of quasi-pro-\'etale maps was suggested by L.~Fargues.

\begin{proposition}\label{prop:charqproet} Let $f: Y^\prime\to Y$ be a separated map of v-stacks. Then $f$ is quasi-pro-\'etale if and only if it is representable in locally spatial diamonds and for all complete algebraically closed fields $C$ with a map $\Spa(C,\OO_C)\to Y$, the pullback $Y^\prime\times_Y \Spa(C,\OO_C)\to \Spa(C,\OO_C)$ is pro-\'etale.
\end{proposition}

\begin{proof} If $f$ is quasi-pro-\'etale, then Proposition~\ref{prop:locallyspatialmorphism}~(iv) implies that $f$ is representable in locally spatial diamonds (noting that to check this, one can assume that $Y$ is a spatial diamond, which admits a quasi-pro-\'etale cover by a strictly totally disconnected space).

Conversely, we may assume that $Y=X$ is a strictly totally disconnected perfectoid space. Moreover, we can assume that $Y^\prime$ is spatial. In this case, we can find a surjective separated quasicompact quasi-pro-\'etale map $X^\prime\to Y^\prime$ as in Proposition~\ref{prop:spatialunivopen}. By Lemma~\ref{lem:proetaleoverwlocal}, the map $X^\prime\to X$ is affinoid pro-\'etale. By Proposition~\ref{prop:qproetthirdmap}, the map $Y^\prime\to X$ is quasi-pro-\'etale.
\end{proof}

Moreover, we will need to know that for any quasicompact separated diamond $Y$, there is a compact Hausdorff space $T$ and a map $Y\to \underline{T}$ which is representable in locally spatial diamonds. We will deduce this from a general discussion of ``Berkovich spaces''.

\begin{definition}\label{def:perfectoidberkovichspace} Let $X=\Spa(R,R^+)$ be an affinoid perfectoid space, and fix some topologically nilpotent unit $\varpi\in R$. The Berkovich space $|X|^B$ associated with $X$ is the space of all multiplicative bounded nonarchimedean seminorms $|\cdot|: R\to \mathbb R_{\geq 0}$ with $|\varpi|=\frac 12$, equipped with the weakest topology making the functions $f\mapsto |f|$ continuous for all $f\in R$.
\end{definition}

\begin{remark}\label{rem:berkovichspaceindept} The space $|X|^B$ is canonically independent of the choice of $\varpi$, for example by Proposition~\ref{prop:berkovichspacemaxhausdorffquot} below.
\end{remark}

Note that there is a natural map $|X|^B\to |X|$, as any multiplicative bounded nonarchimedean seminorm is in particular a valuation. There is also a map $|X|\to |X|^B$: Given any $x\in X$, we get a corresponding map $\Spa(K(x),K(x)^+)\to \Spa(R,R^+)$, where $K(x)$ is a complete nonarchimedean field with an open and bounded valuation subring $K(x)^+\subset K(x)$. In particular, $K(x)$ comes with a unique up to scaling nonarchimedean norm $|\cdot|: K(x)\to \mathbb R_{\geq 0}$; it can be normalized by $|\varpi|=\frac 12$. Thus, the composition $R\to K(x)\to \mathbb R_{\geq 0}$ defines a multiplicative bounded nonarchimedean seminorm, and thus a point of $|X|^B$. One checks easily that the composition $|X|^B\to |X|\to |X|^B$ is the identity.

\begin{proposition}\label{prop:berkovichspacemaxhausdorffquot} Let $X$ be an affinoid perfectoid space. The topological space $|X|^B$ is compact Hausdorff, and the map $|X|\to |X|^B$ is a continuous map identifying $|X|^B$ as the maximal Hausdorff quotient of $|X|$.
\end{proposition}

We warn the reader that the map $|X|^B\to |X|$ is not continuous.

\begin{proof} The space $|X|^B$ is Hausdorff, as if $|\cdot|$, $|\cdot|^\prime$ are two distinct points, then there is some $f\in R$ such that $|f|\neq |f|^\prime$, so there is some real number $r\in \mathbb R_{>0}$ lying strictly between $|f|$ and $|f|^\prime$; then the subsets of points giving $f$ absolute value strictly less than $r$ resp.~strictly bigger than $r$ are open subsets of $|X|^B$ which contain exactly one of $|\cdot|$ and $|\cdot|^\prime$.

Moreover, the map $|X|\to |X|^B$ is surjective (as it has a set-theoretic section) and continuous, as follows easily from the definition. Thus, $|X|^B$ is quasicompact, as $|X|$ is quasicompact. Also, the fibers of $|X|\to |X|^B$ have a unique maximal point by construction: They all share the same completed residue field $K(x)$, and thus generalize to the point given by $\Spa(K(x),\OO_{K(x)})$. Thus, any map from $|X|$ to a Hausdorff space factors set-theoretically over $|X|^B$. It remains to see that $|X|\to |X|^B$ is a quotient map, which follows from Lemma~\ref{lem:quotmaphausdorff}.
\end{proof}

We can now extend to general small v-sheaves, much as in the case of $|Y|$.

\begin{proposition}\label{prop:generalberkovichspace} There is a unique colimit-preserving functor $Y\mapsto |Y|^B$ from small v-sheaves to topological spaces extending $X\mapsto |X|^B$ on affinoid perfectoid spaces. It comes with natural transformations $|Y|^B\to |Y|\to |Y|^B$ whose composite is the identity, and the map $|Y|\to |Y|^B$ is continuous and a quotient map.
\end{proposition}

\begin{proof} For uniqueness, note that the first is first uniquely defined for disjoint unions of affinoid perfectoid spaces, and then (by taking equivalence relations) for separated perfectoid spaces, for diamonds, and for small v-sheaves. One explicit construction of $|Y|^B$ as a set is as the subset of $|Y|$ consisting of those maps $\Spa(K,K^+)\to Y$ that can be represented by a map $\Spa(K,\OO_K)\to Y$. One endows $|Y|^B$ with the quotient topology from $|X|^B$, for any surjection $X\to Y$ from a disjoint union of affinoid perfectoid spaces $X$, and checks that this is independent of the choice of $X$. From the explicit description of $|Y|^B$, we also get the natural transformations $|Y|^B\to |Y|\to |Y|^B$ (noting that any map $\Spa(K,K^+)\to Y$ in particular gives a map $\Spa(K,\OO_K)\to Y$ for the construction of $|Y|\to |Y|^B$).

Finally, to see that $|Y|\to |Y|^B$ is a quotient map, take a cover by a disjoint union of affinoid perfectoid spaces $X\to Y$. Then $|X|^B\to |Y|^B$ is a quotient map by definition, and $|X|\to |X|^B$ is a quotient map by Proposition~\ref{prop:berkovichspacemaxhausdorffquot}, which implies that $|Y|\to |Y|^B$ is a quotient map.
\end{proof}

\begin{proposition}\label{prop:qcqsvsheafberkovichcompacthausdorff} Let $Y$ be a qcqs v-sheaf. Then $|Y|^B$ is a compact Hausdorff space, and it is the maximal Hausdorff quotient of $|Y|$.
\end{proposition}

\begin{proof} First, we prove this if $Y$ is a spatial diamond. Then $Y=X/R$, where $X$ and $R$ are affinoid perfectoid spaces, and then $|Y|^B = |X|^B / |R|^B$ is a quotient of compact Hausdorff spaces, and thus compact Hausdorff itself. Now, if $Y$ is any qcqs v-sheaf, we can write $Y=X/R$ as a quotient, where $X$ and $R$ are spatial diamonds. Repeating the argument, we see that $|Y|^B$ is compact Hausdorff.

As all points in the fiber of $|Y|\to |Y|^B$ over $y\in |Y|^B$ generalize to the image of $y$ under $|Y|^B\to |Y|$, it follows that any map from $|Y|$ to a Hausdorff space factors set-theoretically over $|Y|^B$. By Lemma~\ref{lem:quotmaphausdorff}, $|Y|\to |Y|^B$ is a quotient map, so the result follows.
\end{proof}

As promised, we can now show that a general quasicompact separated diamond differs from a (locally) spatial diamond only through a map to a compact Hausdorff space.

\begin{proposition}\label{prop:qcsepdiamondmaptoberkovich} Let $Y$ be a quasicompact separated diamond. Then the map $Y\to \underline{|Y|^B}$ is representable in locally spatial diamonds.
\end{proposition}

\begin{proof} By Proposition~\ref{prop:qcqsvsheafberkovichcompacthausdorff}, the space $|Y|^B$ is compact Hausdorff. Choose a profinite set $S$ with a surjection $S\to |Y|^B$. By Proposition~\ref{prop:locallyspatialmorphism}~(iv), it is enough to show that $Y\times_{\underline{|Y|^B}} \underline{S}$ is a spatial diamond. By Lemma~\ref{lem:conncomplocspat}, it is enough to check that for any $s\in |Y|^B$, the fiber product $Y_s = Y\times_{\underline{|Y|^B}} \underline{s}$ is a spatial diamond. It then follows from the definitions that $|Y_s|^B=\{s\}$ is a point.

In other words, we have to prove that if $Y$ is a quasicompact separated diamond such that $|Y|^B$ is a point, then $Y$ is spatial. Let $y\in |Y|$ be the image of $|Y|^B\to |Y|$. Then one can find a quasi-pro-\'etale map $f: \Spa(C,\OO_C)\to Y$ with image $y$, for some algebraically closed nonarchimedean field $C$. Let $Y_y\subset Y$ be the image of $f$, so that $Y_y\to Y$ is a quasicompact injection, and $|Y_y|=\{y\}$. One finds that $\Spa(C,\OO_C)\times_{Y_y} \Spa(C,\OO_C)=\Spa(C,\OO_C)\times\underline{G}$ for some profinite group $G$ acting continuously and faithfully on $C$, as in the proof of Proposition~\ref{prop:strtotdiscdiamond}. Consider $\overline{Y}=\Spa(C,C^+)/\underline{G}$, where $C^+$ is the integral closure of $\mathbb F_p + C^{\circ\circ}$ in $C$ (cf.~Proposition~\ref{prop:compactificationprop} for a generalization of this construction). Then $\Spa(C,C^+)\to \overline{Y}$ is a $\underline{G}$-torsor, and thus universally open by Lemma~\ref{lem:gtorsoroverperfectoid}. Thus, $\overline{Y}$ is a spatial diamond by Proposition~\ref{prop:spatialunivopen}.

To finish the proof, it suffices (by Proposition~\ref{prop:qcvinjspatdiamond}) to show that there is a (necessarily quasicompact) injection $Y\to\overline{Y}$. We construct the map $Y\to \overline{Y}$ as a natural transformation on totally disconnected perfectoid spaces $X=\Spa(R,R^+)$. Given a map $X\to Y$, we note that the induced map $\Spa(R,R^\circ)\to Y$ factors over $Y_y$; indeed, $\Spa(R,R^+)\times_Y Y_y\subset \Spa(R,R^+)$ is a quasicompact injection which contains all maximal points; it then follows from Lemma~\ref{lem:subsetwlocal} that it contains $\Spa(R,R^\circ)$ (as all occuring functions will necessarily lie in $R^\circ$). The map $\Spa(R,R^\circ)\to Y_y$ gives a map $\Spa(R,R^+)\to \overline{Y}$ (cf.~Proposition~\ref{prop:compactificationprop}~(iv)). This defines the desired map $Y\to\overline{Y}$. To check that it is injective, it suffices by Proposition~\ref{prop:charinjectionvsheaves} to check injectivity on $(K,K^+)$-valued points, where $K$ is a perfectoid field with an open and bounded valuation subring $K^+\subset K$. As $Y$ and thus $Y\to \overline{Y}$ are separated, the valuative criterion of separatedness shows that it is enough to check injectivity on $(K,\OO_K)$-valued points. But on $(K,\OO_K)$-valued points, one has $Y(K,\OO_K) = Y_y(K,\OO_K) = \overline{Y}(K,\OO_K)$.
\end{proof}

For later use, we add the following result.

\begin{proposition}\label{prop:cohomberkovich} Let $Y$ be a quasicompact separated diamond, and consider the map $f: |Y|\to |Y|^B$. Then pullback $f^\ast$ induces a fully faithful functor
\[
D^+(|Y|^B,\mathbb Z)\to D^+(|Y|,\mathbb Z).
\]
In particular, for any abelian sheaf $\mathcal F$ on $|Y|^B$, one has
\[
H^i(|Y|^B,\mathcal F)\cong H^i(|Y|,f^\ast \mathcal F).
\]
\end{proposition}

\begin{proof} We have to show that the adjunction map $\mathcal F\to Rf_\ast f^\ast \mathcal F$ is an isomorphism for all abelian sheaves $\mathcal F$ on $|Y|^B$. This can be checked on stalks, so pick any $y\in |Y|^B$. As open and closed neighborhoods of $y$ are cofinal, a standard passage to the limit argument shows that
\[
(Rf_\ast f^\ast \mathcal F)_y = R\Gamma(f^{-1}(y),f^\ast \mathcal F).
\]
Note that $f^\ast \mathcal F$ restricted to $f^{-1}(y)$ is just the constant sheaf with value $\mathcal F_y$. Thus, it suffices to show that for any abelian group $M$, $R\Gamma(f^{-1}(y),M)=M$. Now $f^{-1}(y)\subset |Y|$ is exactly the closure of a rank $1$ point $\tilde{y}\in |Y|$, in particular it is a spectral space with a unique generic point, so the claim is standard, see for example~\cite[Tag 02UW]{StacksProject}.
\end{proof}

\section{Comparison of \'etale, pro-\'etale and v-cohomology}

We consider the following sites.

\begin{definition}\label{def:diffsites} Let $Y$ be a small v-stack on $\Perf$. 
\begin{altenumerate}
\item[{\rm (i)}] Assume that $Y$ is a locally spatial diamond. The \'etale site $Y_\et$ is the site whose objects are (locally separated) \'etale maps $Y^\prime\to Y$, with coverings given by families of jointly surjective maps.
\item[{\rm (ii)}] Assume that $Y$ is a diamond. The quasi-pro-\'etale site $Y_\qproet$ is the site whose objects are (locally separated) quasi-pro-\'etale maps $Y^\prime\to Y$, with coverings given by families of jointly surjective maps.
\item[{\rm (iii)}] The v-site $Y_v$ is the site whose objects are all maps $Y^\prime\to Y$ from small v-sheaves $Y^\prime$, with coverings given by families of jointly surjective maps.
\end{altenumerate}
\end{definition}

We note that here, in all cases, surjectivity refers to surjectivity as v-stacks on $\Perf$. Thus, if $X$ and $\{X_i\to X\}$ are perfectoid spaces, then surjectivity means that $\{X_i\to X\}$ is a cover in the v-topology. Again, there are variants $Y_{\qproet,\kappa}$ and $Y_{v,\kappa}$ and the same discussion as before applies; in particular, we will always restrict to small sheaves.

\begin{proposition} The categories of (small) sheaves on $Y_\et$ resp.~$Y_\qproet$ resp.~$Y_v$ for a locally spatial diamond resp.~diamond resp.~small v-stack $Y$ are algebraic. If $Y$ is $0$-truncated (i.e., if $Y$ is a small v-sheaf), then an object is quasicompact resp.~quasiseparated if and only if it is quasicompact resp.~quasiseparated as a small v-stack on $\Perf$.
\end{proposition}

\begin{proof} Left to the reader.
\end{proof}

\begin{proposition}\label{prop:etalesitepoints} Let $Y$ be a locally spatial diamond. Then the \'etale site $Y_\et$ has enough points. More precisely, for any $y\in |Y|$, choose a quasi-pro-\'etale map $\overline{y}: \Spa(C(y),C(y)^+)\to Y$ where $C(y)$ is algebraically closed and $C(y)^+\subset C(y)$ is an open and bounded valuation subring, such that $\overline{y}$ maps the closed point to $y$. Then
\[
\mathcal F\mapsto \mathcal F_{\overline{y}} = \varinjlim_{\overline{y}\to U\in Y_\et} \mathcal F(U)
\]
defines a point of the topos $Y_\et^\sim$, and a section $s\in \mathcal F(Y)$ of a sheaf $\mathcal F$ on $Y_\et$ is zero if and only if
\[
s_{\overline{y}} = 0
\]
for all $y\in |Y|$.
\end{proposition}

\begin{proof} Note that equalizers exist in $Y_\et$, so the colimit defining $\mathcal F_{\overline{y}}$ is filtered.

The functor $\mathcal F\mapsto \mathcal F_{\overline{y}}$ is the composition of pullback along $\overline{y}_\et: \Spa(C(y),C(y)^+)_\et\to Y_\et$, with global sections on $\Spa(C(y),C(y)^+)_\et$. But note that any surjective \'etale map to $\Spa(C(y),C(y)^+)$ splits and $\Spa(C(y),C(y)^+)$ is connected, so the functor of global sections is exact and commutes with all colimits, i.e.~defines a point.

Now if $s\in \mathcal F(Y)$ has the property that $s_{\overline{y}}=0$ for all $y\in |Y|$, then for all $y\in |Y|$ we can find some $\overline{y}\to U_{\overline{y}}\in Y_\et$ such that $s|_{U_{\overline{y}}}=0$. As \'etale maps are open, it follows that the disjoint union of all $U_{\overline{y}}$ is an \'etale cover of $Y$, over which $s$ becomes $0$, so that $s=0$, as desired.
\end{proof}

There are obvious functors of sites
\[
\lambda_Y: Y_v\to Y_\qproet
\]
if $Y$ is a diamond, and
\[
\nu_Y: Y_\qproet\to Y_\et
\]
if $Y$ is a locally spatial diamond. Recall that this means that the underlying functor of categories goes the other way, and is given by observing that any \'etale map is quasi-pro-\'etale, and any quasi-pro-\'etale map has source given by some small v-sheaf. These functors commute with all finite limits (which exist in all cases), so they define maps of topoi. In particular, pullback gives functors
\[
Y_\et^\sim\buildrel{\nu_Y^\ast}\over\to Y_\qproet^\sim\buildrel{\lambda_Y^\ast}\over\to Y_v^\sim
\]
on the corresponding categories of small sheaves. (So the two categories on the right are defined as the large filtered colimit over all $\kappa$ of $Y_{\qproet,\kappa}^\sim$ resp.~$Y_{v,\kappa}^\sim$.)

A surprising feature of the situation is the following observation.

\begin{lemma}\label{lem:lambdacontinuous} Let $Y$ be a diamond. Then the functor
\[
\lambda_Y^\ast: Y_\qproet^\sim\to Y_v^\sim
\]
commutes with all small limits.

Similarly, if $f: Y'\to Y$ is any map of diamonds, then $f^\ast: Y^\sim_\qproet\to Y^{\prime\sim}_\qproet$ commutes with all small limits.
\end{lemma}

\begin{proof} The assertions are quasi-pro-\'etale local on $Y$ (and $Y'$), so we can reduce to the case that $Y=X$ (and $Y'=X'$) is strictly totally disconnected.

In this situation, $X_\qproet^\sim$ is equivalent to $X_{\qproet,\qc,\sep}^\sim$, where $X_{\qproet,\qc,\sep}\subset X_\qproet$ denotes the site of quasicompact separated pro-\'etale morphisms $X^\prime\to X$. Similarly, $X_v^\sim$ is equivalent to $X_{v,\qcqs}^\sim$, where $X_{v,\qcqs}\subset X_v$ denotes the site of qcqs and representable $X^\prime\to X$. In particular, $X_{\qproet,\qc,\sep}$ and $X_{v,\qcqs}$ consist of qcqs objects (in the site-theoretic sense), stable under fibre products; it follows that $X_{\qproet,\qc,\sep}^\sim$ and $X_{v,\qcqs}^\sim$ are coherent. Now we need the following lemma.

\begin{lemma}\label{lem:univproetmap} Let $X$ be a strictly totally disconnected perfectoid space, and $X^\prime\to X$ a map from a qcqs perfectoid space $X^\prime$. Then there is a quasicompact separated pro-\'etale morphism $\lambda_{X\circ}(X^\prime)\to X$ with a factorization
\[
X^\prime\to \lambda_{X\circ}(X^\prime)\to X
\]
such that any map from $X^\prime$ to a separated pro-\'etale perfectoid space over $X$ factors uniquely over $\lambda_{X\circ}(X^\prime)$. Moreover:

\begin{altenumerate}
\item[{\rm (i)}] The map $X^\prime\to \lambda_{X\circ}(X^\prime)$ is surjective. In particular, if $X_2^\prime\to X_1^\prime$ is a surjection of quasicompact separated perfectoid spaces over $X$, then $\lambda_{X\circ}(X_2^\prime)\to \lambda_{X\circ}(X_1^\prime)$ is surjective as well.
\item[{\rm (ii)}] If $X_1^\prime\to X_3^\prime\leftarrow X_2^\prime$ is a diagram of strictly totally disconnected perfectoid spaces over $X$, then
\[
\lambda_{X\circ}(X_1^\prime\times_{X_3^\prime} X_2^\prime)\to \lambda_{X\circ}(X_1^\prime)\times_{\lambda_{X\circ}(X_3^\prime)} \lambda_{X\circ}(X_2^\prime)
\]
is an isomorphism.
\end{altenumerate}
\end{lemma}

\begin{proof} Recall that by Corollary~\ref{cor:proetaleoverwlocaltop}, the category of quasicompact separated pro-\'etale perfectoid spaces over $X$ is equivalent to the category of spectral maps of spectral spaces $T\to |X|$ for which $T\to |X|\times_{\pi_0 X} \pi_0 T$ is a pro-constructible generalizing embedding. Now for any quasicompact separated map $X^\prime\to X$, we can define $T$ as the image of $|X^\prime|\to |X|\times_{\pi_0 X} \pi_0 X^\prime$, which is of this form. The corresponding quasicompact separated pro-\'etale map $\lambda_{X\circ}(X^\prime)\to X$ has the desired universal property, and $X^\prime\to \lambda_{X\circ}(X^\prime)$ is surjective by construction.

For part (ii), we first check surjectivity, for which we need to see that
\[
X_1^\prime\times_{X_3^\prime} X_2^\prime\to \lambda_{X\circ}(X_1^\prime)\times_{\lambda_{X\circ}(X_3^\prime)} \lambda_{X\circ}(X_2^\prime)
\]
is surjective. We may assume that all spaces are connected, so $X=\Spa(C,C^+)$, and $X_i=\Spa(C_i,C_i^+)$ for $i=1,2,3$, where $C$ and $C_i$ are algebraically closed nonarchimedean fields with open and bounded valuation subrings $C^+$ resp.~$C_i^+$. The fibre product $\lambda_{X\circ}(X_1^\prime)\times_{\lambda_{X\circ}(X_3^\prime)} \lambda_{X\circ}(X_2^\prime)$ is given by $\Spa(C,(C^+)^\prime)$ for some other open and bounded valuation subring. We may pullback everything under the open immersion $\Spa(C,(C^+)^\prime)\to \Spa(C,C^+)$. After this replacement, $\lambda_{X\circ}(X_i^\prime) = X$ for all $i=1,2,3$, so $\Spa(C_i,C_i^+)\to \Spa(C,C^+)$ is surjective. As $\Spa(C_3,C_3^+)$ is a totally ordered chain of specializations and the images of $\Spa(C_i,C_i^+)\to \Spa(C_3,C_3^+)$ for $i=1,2$ are generalizing, one of them is contained in the other; we may replace $\Spa(C_3,C_3^+)$ by the image of the smaller one (and $X_1^\prime$, $X_2^\prime$ by the corresponding preimages). This preserves the condition that the maps $X_i^\prime\to X$ are surjective as now they will have the same image given by the minimal image that occured previously (which was all of $X$). After this further replacement, $X_i^\prime\to X_3^\prime$ is surjective for $i=1,2$. Thus,
\[
|X_1^\prime\times_{X_3^\prime} X_2^\prime|\to |X_1^\prime|\times_{|X_3^\prime|} |X_2^\prime|\to |X_3^\prime|\to |X|
\]
is a series of surjections, as desired.

Finally, we need to show that the map is an isomorphism. Again, we can assume that all spaces are connected, and that $X_i^\prime\to X_3^\prime$ are surjective for $i=1,2$, and $X_3^\prime\to X$ is surjective. We have to see that $X_1^\prime\times_{X_3^\prime} X_2^\prime$ is connected. This follows from the following general result.

\begin{lemma}\label{lem:geometricallyconnected} Let $X=\Spa(C,C^+)$, where $C$ is an algebraically closed nonarchimedean field with an open and bounded valuation subring $C^+\subset C$, and let $Z\to X$ be a connected affinoid perfectoid space over $X$. Let $X^\prime=\Spa(C^\prime,C^{\prime +})\to X$, where $C^\prime$ is another algebraically closed nonarchimedean field with an open and bounded valuation subring $C^{\prime +}\subset C^\prime$. Then $Z^\prime:=Z\times_X X^\prime$ is connected.
\end{lemma}

\begin{proof} Assume first that $C^\prime=C$ and $C^{\prime +}=\OO_C$. Then any disconnection of $Z^\prime$ extends to a disconnection of $Z$ by taking closures, noting that any point of $Z$ has a unique maximal generalization, which is a point of $Z^\prime$. Thus, in general, we can assume that $C^+=\OO_C$ and $C^{\prime +} = \OO_{C^\prime}$. We can write $Z$ as a cofiltered inverse limit of p-finite perfectoid spaces $Z_i$ over $(C,\OO_C)$, cf.~\cite[Lemma 6.13]{ScholzePerfectoidSpaces}. We can without loss of generality assume that all $Z_i$ are connected. By pullback, $Z^\prime=\varprojlim_i Z_i^\prime$, with $Z_i^\prime=Z_i\times_X X^\prime$. Any disconnection of $Z^\prime$ comes from a disconnection of $Z_i^\prime$ for $i$ sufficiently large. But for rigid spaces, being geometrically connected passes to extensions of algebraically closed fields, so $Z_i$ is disconnected for some $i$, which is a contradiction.
\end{proof}
\end{proof}

The first part of Lemma~\ref{lem:univproetmap} implies that for any sheaf $\mathcal F$ on $X_{\qproet,\et,\sep}$, the pullback $\lambda_X^\ast \mathcal F$ on $X_{v,\qcqs}$ is the sheafification of the presheaf sending $X^\prime\in X_{v,\qcqs}$ to $\mathcal F(\lambda_{X\circ}(X^\prime))$. Now Lemma~\ref{lem:univproetmap}~(i) implies that this defines a separated presheaf, and Lemma~\ref{lem:univproetmap}~(ii) implies that for strictly totally disconnected $X^\prime\in X_{v,\qcqs}$, the sheafification does not change anything (using that any cover can be refined by a cover by a strictly totally disconnected space); thus, for strictly totally disconnected $X^\prime\in X_{v,\qcqs}$, we have
\[
(\lambda_X^\ast \mathcal F)(X^\prime) = \mathcal F(\lambda_{X\circ}(X^\prime))\ .
\]
As evaluation commutes with limits, this equation commutes with limits of sheaves. As strictly totally disconnected spaces form a basis of $X_{v,\qcqs}$, we get the desired result for $\lambda_X^\ast$. The same argument applies to the pullback $f^\ast$.
\end{proof}

\begin{proposition}\label{prop:compproetvcohom} Let $Y$ be a diamond. The functor
\[
\lambda_Y^\ast: Y_\qproet^\sim\to Y_v^\sim
\]
is fully faithful. Moreover, for any small sheaf $\mathcal{F}$ on $Y_\qproet$, the adjunction map
\[
\mathcal{F}\to \lambda_{Y\ast} \lambda_Y^\ast \mathcal{F}
\]
is an isomorphism, and if $Y$ is a locally spatial diamond and $\mathcal{F}$ is a sheaf of abelian groups (resp.~groups) that comes via pullback from $Y_\et$, then
\[
R^i\lambda_{Y\ast} \lambda_Y^\ast \mathcal{F} = 0
\]
for all $i>0$ (resp.~for $i=1$).
\end{proposition}

We do not know whether the statement about cohomology holds true for all pro-\'etale sheaves.

\begin{proof} To prove fully faithfulness, it is enough to show that the adjunction map $\mathcal{F}\to \lambda_{Y\ast} \lambda_Y^\ast \mathcal{F}$ is an isomorphism for all sheaves $\mathcal{F}$ on $Y_\qproet$. This statement, as well as the similar statements about higher direct images, is local in $Y_\qproet$. Thus, we can assume that $Y=X$ is a strictly totally disconnected perfectoid space.

In this case, we know by the proof of Lemma~\ref{lem:lambdacontinuous} that $(\lambda_X^\ast \mathcal F)(X^\prime) = \mathcal F(\lambda_{X\circ}(X^\prime))$ for all strictly totally disconnected $X^\prime$ over $X$. In particular, if $X^\prime\in X_{\qproet,\qc,\sep}$, this implies that
\[
(\lambda_{X\ast} \lambda_X^\ast \mathcal F)(X^\prime) = (\lambda_X^\ast \mathcal F)(X^\prime) = \mathcal F(\lambda_{X\circ}(X^\prime)) = \mathcal F(X^\prime)\ .
\]
As such $X^\prime$ form a base for $X_\qproet = X_\qproet$, it follows that $\mathcal F\to \lambda_{Y\ast} \lambda_Y^\ast \mathcal F$ is an isomorphism.

For the statement about cohomology, it suffices to see that $H^i(X_v,\lambda_Y^\ast \mathcal F)=0$ for all $i>0$ when $\mathcal F$ comes via pullback from the \'etale site, and $X$ is strictly totally disconnected. Choose some minimal $i$ for which this fails (for some $X$ and $\mathcal F$) and take a nonzero class $\alpha\in H^i(X_v,\lambda_Y^\ast \mathcal F)$. There is some v-cover $f: X'\to X$ such that $f^\ast \alpha=0$. Without loss of generality, $X'$ is affinoid perfectoid. A \v{C}ech spectral sequence then shows that $\alpha$ comes from some class in $(\lambda_Y^\ast \mathcal F)(X'\times_X\ldots\times_X X')$ (with $i+1$ factors of $X'$) in the kernel of the differential. Now we can write $X'$ as a cofiltered limit of affinoid perfectoid spaces $X'_j$ that are rational subsets of finite-dimensional perfectoid balls over $X$. As $\mathcal F$ comes via pullback from the \'etale site, $\lambda_Y^\ast \mathcal F$ turns this cofiltered limit into a filtered colimit, and hence the class $\alpha$ already arises from some \v{C}ech cocycle for $X'_j\to X$, i.e.~we can assume that $X'$ is a rational subset of a finite-dimensional perfectoid ball over $X$. Using Lemma~\ref{lem:secoveralgclosed} one sees that then $X'\to X$ splits, so in fact $\alpha=0$.
\end{proof}

Next, we will need a similar result for $\nu_Y: Y_\qproet\to Y_\et$ in case $Y$ is a locally spatial diamond.

\begin{proposition}\label{prop:competproetcohom} Let $Y$ be a locally spatial diamond. The functor
\[
\nu_Y^\ast: Y_\et^\sim\to Y_\qproet^\sim
\]
is fully faithful. Moreover, for any sheaf $\mathcal{F}$ on $Y_\et$, the adjunction map
\[
\mathcal{F}\to \nu_{Y\ast} \nu_Y^\ast \mathcal{F}
\]
is an isomorphism, and if $\mathcal{F}$ is a sheaf of abelian groups (resp.~groups), then
\[
R^i\nu_{Y\ast} \nu_Y^\ast \mathcal{F} = 0
\]
for all $i>0$ (resp.~for $i=1$).
\end{proposition}

\begin{proof} We may assume that $Y$ is spatial. Let $Y_{\et,\qc,\sep}\subset Y_\et$ be the full subcategory of quasicompact separated \'etale maps. As all \'etale maps are by Convention~\ref{conv:etlocsep} locally separated, and using Corollary~\ref{cor:quasiproetspatial}, this is a basis for the topology, so $Y_{\et,\qc,\sep}^\sim\cong Y_\et^\sim$. Moreover, there is a functor
\[
\Pro(Y_{\et,\qc,\sep})\to Y_\qproet\ .
\]

As in the proof of Proposition~\ref{prop:affproetale}, Proposition~\ref{prop:etmaptolimdiamond}~(iii) implies that this functor is fully faithful. Moreover, the objects in $\Pro(Y_{\et,\qc,\sep})\subset Y_\qproet$ form a basis for the topology. Indeed, choose a map $X\to Y$ from a strictly totally disconnected space $X$ as in Proposition~\ref{prop:spatialunivopen}, so that $X\in \Pro(Y_{\et,\qc,\sep})$. If $Y^\prime\to Y$ is any quasi-pro-\'etale map, then $X^\prime:=X\times_Y Y^\prime$ is a cover of $Y^\prime$ in $Y_\qproet$, and is representable and pro-\'etale over $X$, so can be covered by perfectoid spaces which are affinoid pro-\'etale over $X$. All such affinoid pro-\'etale maps to $X$ are inverse limits of affinoid \'etale, and in particular quasicompact separated \'etale maps, to $X$. Using Proposition~\ref{prop:etmaptolimdiamond}~(iii), any such map lies in $\Pro(Y_{\et,\qc,\sep})$ again, giving the claim.

Thus, we can replace the map $\nu_Y: Y_\qproet\to Y_\et$ by $\Pro(Y_{\et,\qc,\sep})\to Y_{\et,\qc,\sep}$. Now the same arguments as in Proposition~\ref{prop:compareetproet} apply.
\end{proof}

Moreover, we have the following commutation of \'etale cohomology with inverse limits.

\begin{proposition}\label{prop:etcohomlim} Let $Y_i$, $i\in I$, be a cofiltered inverse system of spatial diamonds with inverse limit $Y$. Assume that $I$ has a final object $0\in I$, and let $\mathcal F_0$ be an \'etale sheaf on $Y_0$, with pullbacks $\mathcal F_i$ to $Y_i$, and $\mathcal F$ to $Y$. Then the natural map
\[
\varinjlim_{i\in I} H^j(Y_i,\mathcal F_i)\to H^j(Y,\mathcal F)
\]
is an isomorphism for $j=0$, resp.~$j=0,1$, resp.~all $j\geq 0$, if $\mathcal F_0$ is sheaf of sets, resp.~sheaf of groups, resp.~sheaf of abelian groups.
\end{proposition}

\begin{proof} As in the previous proof, we can replace $Y_\et$ (and $(Y_i)_\et$) by $Y_{\et,\qc,\sep}$ (and $(Y_i)_{\et,\qc,\sep}$). Then Proposition~\ref{prop:etmaptolimdiamond}~(iii) implies that $Y_{\et,\qc,\sep}^\sim=Y_\et^\sim$ is a limit of the fibred topos $(Y_i)_{\et,\qc,\sep}^\sim=(Y_i)_\et^\sim$. As all intervening topoi are coherent, the result follows from SGA 4 VI.8.7.7 (in the case of sheaves of abelian groups, the other cases being similar).
\end{proof}

Using the previous results, we can compare some derived categories. Fix a ring $\Lambda$. If $T$ is a topos, we denote by $D(T,\Lambda)$ the derived category of $\Lambda$-modules on $T$, and by $D^+(T,\Lambda),D^-(T,\Lambda),D^b(T,\Lambda)\subset D(T,\Lambda)$ the subcategories of complexes which are (cohomologically) bounded below, resp.~above, resp.~below and above. This does not literally apply to $Y_\qproet$ or $Y_v$ as these are not generated by a set of objects. In these cases, we can still write them as increasing unions of sites $Y_{\qproet,\kappa}$ (resp.~$Y_{v,\kappa}$) for varying choices of $\kappa$ as in Lemma~\ref{lem:choosekappa}. We then define $D(Y_\qproet,\Lambda)$ and $D(Y_v,\Lambda)$ by passing to a filtered colimit over all choices of $\kappa$ (the transition functors are all fully faithful).

\begin{proposition}\label{prop:pullbackfullyfaithfulderived} Let $Y$ be a locally spatial diamond. Then the pullback functors
\[
\nu_Y^\ast: D^+(Y_\et,\Lambda)\to D^+(Y_\qproet,\Lambda), \lambda_Y^\ast \nu_Y^\ast: D^+(Y_\et,\Lambda)\to D^+(Y_v,\Lambda)
\]
are fully faithful. If $Y$ is a strictly totally disconnected perfectoid space, then the pullback functors
\[
\nu_Y^\ast: D(Y_\et,\Lambda)\to D(Y_\qproet,\Lambda), \lambda_Y^\ast \nu_Y^\ast: D(Y_\et,\Lambda)\to D(Y_v,\Lambda)
\]
are fully faithful.
\end{proposition}

\begin{proof} The statements for $D^+$ follow easily from Proposition~\ref{prop:compproetvcohom} and Proposition~\ref{prop:competproetcohom}. For the unbounded statements, we need some convergence results. For this, note that $Y_\qproet$ and $Y_\et$, in case $Y$ is strictly totally disconnected, have a basis for the topology consisting of $U\in Y_\qproet$ resp.~$U\in Y_\et$ such that for all sheaves $\mathcal F$, $H^i(U,\mathcal F)=0$ for $i>0$. Indeed, in the case of $Y_\qproet$, one can take strictly w-local perfectoid spaces $U$ for which $\pi_0 U$ is extremally disconnected (as then every pro-\'etale cover splits), and in the case of $Y_\et$, one can take strictly totally disconnected $U$. This implies that for any $A\in D(Y_\qproet,\Lambda)$ resp.~$A\in D(Y_\et,\Lambda)$, we have $A=R\varprojlim_n \tau^{\geq -n} A$ by testing the values on such $U$. In fact, the similar result holds true also in $D(Y_v,\Lambda)$ as $Y_v$ is replete, using \cite[Proposition 3.3.3]{BhattScholze}.

Now the desired results follow by taking a derived limit over all $\tau^{\geq -n} A$.
\end{proof}

Note that in the first paragraph of the preceding proof, we have proved the following.

\begin{proposition}\label{prop:leftcomplete} Let $Y$ be a small v-stack.
\begin{altenumerate}
\item The derived category $D(Y_v,\Lambda)$ is left-complete.
\item If $Y$ is a diamond, then $D(Y_\qproet,\Lambda)$ is left-complete.
\item If $Y$ is a strictly totally disconnected perfectoid space, then $D(Y_\et,\Lambda)$ is left-complete.$\hfill \Box$
\end{altenumerate}
\end{proposition}

Thus, if $Y$ is a locally spatial diamond, we have full subcategories
\[
D^+(Y_\et,\Lambda)\subset D(Y_\qproet,\Lambda), D^+(Y_\et,\Lambda)\subset D(Y_v,\Lambda)\ ;
\]
if $Y$ is strictly totally disconnected, then
\[
D(Y_\et,\Lambda)\subset D(Y_\qproet,\Lambda), D(Y_\et,\Lambda)\subset D(Y_v,\Lambda).
\]
Interestingly, containment in these subcategories can be checked v-locally.

\begin{theorem}\label{thm:proetandetsheafdetectionvlocally} Let $Y$ be a locally spatial diamond and $f: Y^\prime\to Y$ a v-cover by a locally spatial diamond $Y^\prime$.
\begin{altenumerate}
\item If $A\in D(Y_\qproet,\Lambda)$ or $A\in D(Y_v,\Lambda)$ and $f^\ast A\in D^+(Y^\prime_\et,\Lambda)$, then $A\in D^+(Y_\et,\Lambda)$.
\item If $A\in D(Y_\qproet,\Lambda)$ or $A\in D(Y_v,\Lambda)$ and $Y$ and $Y^\prime$ are strictly totally disconnected and $f^\ast A\in D(Y^\prime_\et,\Lambda)$, then $A\in D(Y_\et,\Lambda)$.
\end{altenumerate}
\end{theorem}

\begin{proof} Using the convergence results from Proposition~\ref{prop:leftcomplete}, we can reduce to the case that $A=\mathcal F[0]$ is a sheaf concentrated in degree $0$.

If $\mathcal F$ is already a small sheaf on $Y_\qproet$, then we have to see that for any $\tilde{Y}=\varprojlim_j \tilde{Y}_j\in \Pro(Y_{\et,\qc,\sep})$, one has
\[
\mathcal F(\tilde{Y}) = \varinjlim_j \mathcal F(\tilde{Y}_j)\ .
\]
But if one lets $\tilde{Y}^\prime = \tilde{Y}\times_Y Y^\prime$, $\tilde{Y}^\prime_j = \tilde{Y}_j\times_Y Y^\prime$, then
\[
\mathcal F(\tilde{Y}^\prime) = \varinjlim_j \mathcal F(\tilde{Y}^\prime_j)
\]
by assumption, and a similar equation over the fibre product $Y^\prime\times_Y Y^\prime$. As equalizers commute with filtered colimits, the result follows.

Now assume that $\mathcal F$ is a small sheaf on $Y_v$. It suffices to see that the adjunction map $\lambda_Y^\ast \lambda_{Y\ast} \mathcal F\to \mathcal F$ is an isomorphism. This question is pro-\'etale local on $Y$, so we may assume that $Y=X$ is a strictly totally disconnected perfectoid space. We can also replace $Y^\prime$ by a strictly totally disconnected perfectoid space $X^\prime$. We have to see that for any strictly totally disconnected $\tilde{X}\in X_v$, the natural map
\[
\Gamma(\lambda_{X\circ}(\tilde{X})_v,\mathcal F)\to \Gamma(\tilde{X}_v,\mathcal F)
\]
is an isomorphism. It is enough to do this locally, so we may assume that $\tilde{X}\to X$ lifts to a map $\tilde{X}\to X^\prime$. For the proof of the displayed identity, we can also replace $X$ by $\lambda_{X\circ}(\tilde{X})$, and then assume that $X^\prime=\tilde{X}$. Thus, $X$ and $X^\prime$ are strictly totally disconnected spaces, $f: X^\prime\to X$ is a v-cover such that $\lambda_{X\circ}(X^\prime)=X$, and $\mathcal F$ is a small v-sheaf on $X$ such that $f^\ast \mathcal F$ is the pullback of a pro-\'etale sheaf $\mathcal G^\prime$ on $X^\prime$, and we need to show that $\mathcal F(X)=\mathcal F(X^\prime)$. Applying Theorem~\ref{thm:vproetbasechange} below to $f: Y^\prime=X^\prime\to Y=X$, the sheaf $\mathcal G^\prime$ on $X^\prime_\qproet$ and $\tilde{X}=X^\prime$, we see that the pullback map
\[
p_1^\ast: \mathcal F(X^\prime)\to \mathcal F(X^\prime\times_X X^\prime)
\]
is an isomorphism. In particular, all maps in the commutative diagram
\[\xymatrix{
\mathcal F(X^\prime)\ar[d]_{p_1^\ast}\ar@{=}[dr]\\
\mathcal F(X^\prime\times_X X^\prime)\ar[r]^-{\Delta^\ast} & \mathcal F(X^\prime)\\
\mathcal F(X^\prime)\ar[u]^{p_2^\ast}\ar@{=}[ur]
}\]
are isomorphisms, and thus
\[
\mathcal F(X) = \mathrm{eq}(\mathcal F(X^\prime)\rightrightarrows \mathcal F(X^\prime\times_X X^\prime)) = \mathrm{eq}(\mathcal F(X^\prime)\rightrightarrows \mathcal F(X^\prime)) = \mathcal F(X^\prime)\ ,
\]
as desired.
\end{proof}

Theorem~\ref{thm:proetandetsheafdetectionvlocally} shows that the following definition is reasonable.

\begin{definition}\label{def:generaletproetderivedcategory} Let $Y$ be a small v-stack. Define the full subcategory
\[
D_\et(Y,\Lambda)\subset D(Y,\Lambda)=D(Y_v,\Lambda)
\]
as consisting of those $A\in D(Y,\Lambda)$ such that for all strictly totally disconnected perfectoid spaces $f: X\to Y$, the pullback $f^\ast A\in D(X_\et,\Lambda)$.
\end{definition}

\begin{remark} By Theorem~\ref{thm:proetandetsheafdetectionvlocally}, it suffices to check the condition for one v-cover of $Y$. If $Y$ is a locally spatial diamond, $D^+_\et(Y,\Lambda)=D^+(Y_\et,\Lambda)$, and if $Y$ is a strictly totally disconnected perfectoid space, $D_\et(Y,\Lambda)=D(Y_\et,\Lambda)$. If $Y$ is a locally spatial diamond, then in general $D_\et(Y,\Lambda)\neq D(Y_\et,\Lambda)$, but one has the following result.
\end{remark}

\begin{proposition}\label{prop:locallyspatialleftcompletion} Let $Y$ be a small v-stack. Then $D_\et(Y,\Lambda)$ is left-complete. If $Y$ is a locally spatial diamond, then $D_\et(Y,\Lambda)$ is the left-completion of $D(Y_\et,\Lambda)$. In particular, if $Y$ is a locally spatial diamond, the fully faithful embedding
\[
D^+(Y_\et,\Lambda)\subset D(Y_\qproet,\Lambda)
\]
extends to a fully faithful embedding
\[
D_\et(Y,\Lambda)\subset D(Y_\qproet,\Lambda).
\]
\end{proposition}

\begin{proof} The first part is clear. For the rest, the proof is the same as in~\cite[Proposition 5.3.2]{BhattScholze}.
\end{proof}

Containment in these subcategories can be checked on cohomology sheaves.

\begin{proposition}\label{prop:etproetcheckoncohomsheaves} Let $Y$ be a small v-stack and $A\in D(Y,\Lambda)$. Then $A\in D_\et(Y,\Lambda)$ if and only if for all $i\in \mathbb Z$, the v-sheaf $\mathcal H^i(A)[0]\in D_\et(Y,\Lambda)$.
\end{proposition}

\begin{proof} As all relevant derived categories are left-complete, one can assume that $A$ is bounded below. Then $A=\varinjlim \tau^{\leq n} A$, so one can assume that $A$ is bounded. As the inclusion $D_\et(Y,\Lambda)\subset D(Y,\Lambda)$ is fully faithful, one can further reduce to the case where $A$ is concentrated in one degree, as desired.
\end{proof}

\section{Analytic adic spaces as diamonds}

There is a natural functor from rigid-analytic varieties over $\mathbb Q_p$ to locally spatial diamonds. In fact, more generally, there is a functor from analytic adic spaces over $\mathbb Z_p$ to locally spatial diamonds. Our goal in this section is to construct this functor, and show that it preserves the \'etale site.

We start with the affinoid case.

\begin{lemma}\label{lem:mapstoaffinoidvsheaf}
\begin{altenumerate}
\item[{\rm (i)}] The functor $\Spd \mathbb Z_p$ taking $X\in \Perf$ to the set of pairs $(X^\sharp,\iota)$ up to isomorphism, where $X^\sharp$ is a perfectoid space and $\iota: (X^\sharp)^\flat\cong X$ is an identification of $X^\sharp$ as an untilt of $X$, is a v-sheaf.
\item[{\rm (ii)}] Let $A$ be a Tate $\mathbb Z_p$-algebra, and let $A^+\subset A$ be an open and integrally closed subring. The functor $\Spd(A,A^+)$ taking $X\in \Perf$ to the set of pairs
\[
((X^\sharp,\iota), f: (A,A^+)\to (\OO_{X^\sharp}(X^\sharp),\OO^+_{X^\sharp}(X^\sharp))
\]
up to isomorphism, where $(X^\sharp,\iota)$ is as in (i), and $f$ is a continuous map of pairs of topological algebras, is a v-sheaf.
\end{altenumerate}
\end{lemma}

\begin{proof} Part (ii) follows from part (i) and Theorem~\ref{thm:vsub}. For part (i), we first have to see that this in fact defines a presheaf: It is not a priori clear what the maps are. But if $f: X^\prime\to X$ is a map of perfectoid spaces and $X^\sharp$ is an untilt of $X$, then by Corollary~\ref{cor:tiltingequiv}, perfectoid spaces over $X$ are equivalent to perfectoid spaces over $X^\sharp$; then we define $(X^\prime)^\sharp$ to be the perfectoid space over $X^\sharp$ corresponding to $X^\prime\to X$. Also, note that pairs $(X^\sharp,\iota)$ have no automorphisms, by the same argument.

To see that this is a v-sheaf, we have to see that if $X=\Spa(R,R^+)$ is an affinoid perfectoid space of characteristic $p$ with a v-cover $Y = \Spa(S,S^+)\to X$, and $Y^\sharp=\Spa(S^\sharp,S^{\sharp +})$ is an untilt of $Y$ such that the two corresponding untilts of $Z:=Y\times_X Y=\Spa(T,T^+)$ agree, then there is a unique untilt $X^\sharp=\Spa(R^\sharp,R^{\sharp +})$ whose pullback to $Y$ is $Y^\sharp$.

Fix a pseudouniformizer $\varpi\in R$, and consider the map $f: W(R^+)\to W(S^+)\to S^{\sharp +}$, where the second map is Fontaine's map $\theta: W(S^+)\to S^{\sharp +}$. Then $f([\varpi])\in S^{\sharp +}$ is a pseudouniformizer. In particular, since $p$ is topologically nilpotent in $S^\sharp$, there is some $N$ such that $f([\varpi])$ divides $p^N$. Replacing $\varpi$ by a $p$-power root, we may assume that $N=1$. In that case, $S^+/\varpi\cong S^{\sharp +}/f([\varpi])$, and similarly for $T^+$. It follows that
\[
(R^+/\varpi)^a=\eq((S^+/\varpi)^a\rightrightarrows (T^+/\varpi)^a) = \eq((S^{\sharp +}/f([\varpi]))^a\rightrightarrows (T^{\sharp +}/f([\varpi]))^a)\ .
\]
Moreover, extending to a long exact sequence, the higher cohomology groups vanish by Theorem~\ref{thm:vsub}. We see that
\[
(W(R^+)/[\varpi])^a\to \eq((S^{\sharp +}/f([\varpi]))^a\rightrightarrows (T^{\sharp +}/f([\varpi]))^a)
\]
is surjective, with kernel generated by $p$. Now the vanishing of higher cohomology groups and the five lemma imply inductively that one gets surjective maps
\[
(W(R^+)/[\varpi]^n)^a\to \eq((S^{\sharp +}/f([\varpi])^n)^a\rightrightarrows (T^{\sharp +}/f([\varpi])^n)^a)
\]
for all $n\geq 1$, whose kernel is generated by an element $\xi\in W(R^+)$ with $\xi\equiv p\mod [\varpi]$. In the limit, one gets an isomorphism
\[
(W(R^+)/\xi)^a = \eq(S^{\sharp + a}\rightrightarrows T^{\sharp + a})\ .
\]
In particular, $R^\sharp := (W(R^+)/\xi)[[\varpi]^{-1}]$ is a perfectoid Tate ring which satisfies
\[
R^\sharp = \eq(S^\sharp\rightrightarrows T^\sharp)\ ,
\]
which is an untilt of $R$, and then comes with a canonical $R^{\sharp +}\subset R^\sharp$ determined by $R^+\subset R$. It follows that $X^\sharp=\Spa(R^\sharp,R^{\sharp +})$ is the desired untilt.
\end{proof}

In case $A$ is a perfectoid Tate ring, it is easy to identify $\Spd(A,A^+)$.

\begin{lemma}\label{lem:spdperfectoid} Let $A$ be a perfectoid Tate ring with an open and integrally closed subring $A^+\subset A$. Then $\Spd(A,A^+)$ is representable by the affinoid perfectoid space $\Spa(A^\flat,A^{\flat +})$.
\end{lemma}

\begin{proof} This is a direct consequence of Corollary~\ref{cor:tiltingequiv}.
\end{proof}

To understand $\Spd(A,A^+)$ in general, we use the following general lemma.

\begin{lemma}\label{lem:profiniteetalecover} Let $A$ be any ring. Then one can find a cofiltered inverse system of finite groups $G_i$ with surjective transition maps, and a compatible filtered direct system of finite \'etale $G_i$-torsors $A\to A_i$ such that $A_\infty = \varinjlim_i A_i$ has no nonsplit finite \'etale covers.

If $A$ is a Tate $\mathbb Z_p$-algebra, then the uniform completion $\widehat{A_\infty}$ of $A_\infty$ is perfectoid.
\end{lemma}

Results of this type were proved for example by Colmez, \cite{ColmezEspaces}, and Faltings, \cite{FaltingsAlmostEtaleExtensions}.

\begin{proof} The first part is standard. To see that $R=\widehat{A_\infty}$ is perfectoid we find first a pseudo-uniformizer $\varpi\in R$ such that $\varpi^p\vert p$ in $R^\circ$. To do this, let $\varpi_0\in R$ be any pseudo-uniformizer.  Let $n$ be large enough so that $\varpi_0\vert p^n$.  Now look at the equation $x^{p^n}-\varpi_0x=\varpi_0$. This determines a finite \'etale $R$-algebra, and so it admits a solution $x=\varpi\in R$.  Note that $\varpi^{p^n}\vert\varpi_0$ in $R^{\circ}$, and $\varpi$ is a unit in $R$. As $p^{n-1}\geq n$, we see that $\varpi^p\vert p$, as desired.

Now we must check that $\Phi\from R^\circ/\varpi\to R^\circ/\varpi^p$ is surjective. Let $f\in R^{\circ}$, and consider the equation $x^p-\varpi^p x - f$. This determines a finite \'etale $R$-algebra, which consequently has a section, i.e.~there is some $x\in R^\circ$ with $x^p-\varpi^p x=f$. But then $x^p\equiv f\pmod{\varpi^p R^\circ}$, as desired.
\end{proof}

Now we can prove that $\Spd(A,A^+)$ is always a spatial diamond.

\begin{proposition}\label{prop:spdspatial} Let $A$ be a Tate $\mathbb Z_p$-algebra with an open and integrally closed subring $A^+\subset A$, and choose a cofiltered inverse system of finite groups $G_i$ with surjective transition maps, and a compatible filtered direct system of finite \'etale $G_i$-torsors $A\to A_i$ as in Lemma~\ref{lem:profiniteetalecover}. Let $A_i^+\subset A_i$ be the integral closure of $A^+$. Also, let $\widehat{A_\infty}^+\subset \widehat{A_\infty}$ be the closure of $\varinjlim A_i^+$ in $\widehat{A_\infty}$.

Then $\Spd(A_i,A_i^+)\to \Spd(A,A^+)$ is a $G_i$-torsor of v-sheaves, and
\[
\Spd(\widehat{A_\infty},\widehat{A_\infty}^+) = \varprojlim_i \Spd(A_i,A_i^+)\to \Spd(A,A^+)
\]
is a $\underline{G}$-torsor, where $G=\varprojlim_i G_i$. In particular,
\[
\Spa(\widehat{A_\infty}^\flat,\widehat{A_\infty}^{\flat +}) = \Spd(\widehat{A_\infty},\widehat{A_\infty}^+)\to \Spd(A,A^+)
\]
is a universally open qcqs quasi-pro-\'etale map from an affinoid perfectoid space, so that $\Spd(A,A^+)$ is a spatial diamond, with
\[
|\Spd(A,A^+)| = |\Spa(\widehat{A_\infty}^\flat,\widehat{A_\infty}^{\flat +})|/G = |\Spa(\widehat{A_\infty},\widehat{A_\infty}^+)|/G = |\Spa(A,A^+)|\ .
\]

\end{proposition}

\begin{proof} It follows from Theorem~\ref{thm:almostpurity} that $\Spd(A_i,A_i^+)\to \Spd(A,A^+)$ is a $G_i$-torsor. The rest, except for the identification of the topological space, follows formally from Lemma~\ref{lem:profiniteetalecover}, Lemma~\ref{lem:spdperfectoid}, Lemma~\ref{lem:gtorsoroverperfectoid} and Proposition~\ref{prop:spatialunivopen}.

For the identification of the topological space, only the last identification needs justification. But
\[
|\Spa(\widehat{A_\infty},\widehat{A_\infty}^+)| = \varprojlim_i |\Spa(A_i,A_i^+)|\ ,
\]
and for each $i$, $|\Spa(A_i,A_i^+)|/G_i = |\Spa(A,A^+)|$. Passing to the limit gives the result.
\end{proof}

One can now glue the functor $(A,A^+)\mapsto \Spd(A,A^+)$, and pass to adic spaces. Indeed, if $U\subset \Spa(A,A^+)$ is a rational open subset, then
\[
\Spd(\OO(U),\OO^+(U))\to \Spd(A,A^+)
\]
is the open subfunctor of $\Spd(A,A^+)$ corresponding to the open subset $U\subset |\Spa(A,A^+)|=|\Spd(A,A^+)|$.

\begin{definition}\label{def:diamondadicspace} Let $Y$ be an analytic adic space over $\mathbb Z_p$. The diamond associated with $Y$ is the v-sheaf defined by
\[
Y^\diamondsuit: X\mapsto \{((X^\sharp,\iota),f: X^\sharp\to Y)\}/\cong\ ,
\]
where $X^\sharp$ is a perfectoid space with an isomorphism $\iota: (X^\sharp)^\flat\cong X$.
\end{definition}

We note that this definition (and the next lemma) work in the generality of adic spaces as defined in \cite[Section 2.1]{ScholzeWeinstein}.

\begin{lemma}\label{lem:diamondadicspace} Let $Y$ be an analytic adic space over $\mathbb Z_p$. Then $Y^\diamondsuit$ is a locally spatial diamond, with $|Y^\diamondsuit|=|Y|$. Moreover, with $Y_\et$ and $Y_\fet$ defined as in \cite[Definition 8.2.19]{KedlayaLiu1}, one has equivalences of sites $Y^\diamondsuit_\et\cong Y_\et$ and $Y^\diamondsuit_\fet\cong Y_\fet$.
\end{lemma}

\begin{proof} All statements reduce readily to the affinoid case $Y=\Spa(A,A^+)$, where the first part follows from Proposition~\ref{prop:spdspatial}. Now, for the equivalence of finite \'etale sites, note that by \cite[Lemma 8.2.17]{KedlayaLiu1}, we need to see that the category of finite \'etale $A$-algebras is equivalent to (the opposite of) $Y^\diamondsuit_\fet$. Choosing a $\underline{G}$-torsor as in Proposition~\ref{prop:spdspatial}, this follows from the equivalence
\[
(\widehat{A}_\infty)_\fet = \text{2-}\varinjlim_i (A_i)_\fet
\]
from \cite[Lemma 7.5 (i)]{ScholzePerfectoidSpaces} (and similar results
\[
(C^0(G,\widehat{A}_\infty))_\fet = \text{2-}\varinjlim_i (C^0(G_i,A_i))_\fet\ ,
\]
\[
(C^0(G\times G,\widehat{A}_\infty))_\fet = \text{2-}\varinjlim_i (C^0(G_i\times G_i,A_i))_\fet
\]
for the algebras corresponding to the fibre products) and usual descent along the finite \'etale $G_i$-torsors $A\to A_i$.

Finally, the case of the \'etale site follows by combining the description of the topological space with the finite \'etale case, as by Lemma~\ref{lem:spatialetalelocallysep} resp.~\cite[Definition 8.2.16]{KedlayaLiu1}, in both cases an \'etale map is locally given by a composite of a quasicompact open immersion and a finite \'etale map.
\end{proof}

\section{General base change results}

In this section, we establish some general base change results. We will first prove results comparing \'etale, pro-\'etale and v-pushforwards, and then deduce more classical base change results.

\begin{theorem}\label{thm:vproetbasechange} Let $f: Y^\prime\to Y$ be a map of locally spatial diamonds, and consider the diagram of sites
\[\xymatrix{
Y^\prime_v\ar[d]_{f_v}\ar[r]^{\lambda_{Y^\prime}} & Y^\prime_\qproet\ar[d]^{f_\qproet}\\
Y_v\ar[r]^{\lambda_Y} & Y_\qproet.
}\]
Let $\mathcal F$ be a small sheaf of abelian groups on $Y^\prime_\qproet$ that comes via pullback from $Y^\prime_\et$. Then the base change morphism
\[
\lambda_Y^\ast R^i f_{\qproet\ast} \mathcal F\to R^i f_{v\ast} \lambda_{Y^\prime}^\ast \mathcal F
\]
is an isomorphism in the following cases:
\begin{altenumerate}
\item if $i=0$, or
\item for all $i\geq 0$ if $f$ is quasi-pro-\'etale, or
\item for all $i\geq 0$ if there is some integer $n$ prime to $p$ such that $n\mathcal F=0$.
\end{altenumerate}

Moreover, under the same conditions, if $Y=X$ and $\tilde{X}\in X_v$ are strictly totally disconnected and $\tilde{X}\to \lambda_{X\circ}(\tilde{X})\to X$ is the factorization from Lemma~\ref{lem:univproetmap}, then the natural map
\[
H^i((\lambda_{X\circ}(\tilde{X})\times_X Y^\prime)_\qproet,\mathcal F)\to H^i((\tilde{X}\times_X Y^\prime)_v,\lambda_{Y^\prime}^\ast \mathcal F)
\]
is an isomorphism.
\end{theorem}

\begin{remark} The final statement combines an ``invariance under change of algebraically closed base field'' statement with a ``pro-\'etale cohomology = v-cohomology''-statement. Indeed, if $X$ and $\tilde{X}$ are geometric points, then $\lambda_{X\circ}(\tilde{X})=X$, and the statement becomes
\[
H^i(Y^\prime_\qproet,\mathcal F) = H^i(Y^\prime_v,\lambda_{Y^\prime}^\ast \mathcal F) = H^i((\tilde{X}\times_X Y^\prime)_v,\lambda_{Y^\prime}^\ast \mathcal F)\ .
\]
\end{remark}

\begin{proof} The claim is pro-\'etale local on $Y$, so we can assume that $Y=X$ is strictly totally disconnected, and it suffices to prove the final statement, as this gives the statement about sheaves by sheafification.

The final statement is again pro-\'etale local on $Y^\prime$, so we may also assume that $Y^\prime=X^\prime$ is strictly totally disconnected. We have to see that if $\tilde{X}\to X\leftarrow X^\prime$ is a diagram of strictly totally disconnected perfectoid spaces, then the natural map
\[
H^i((\lambda_{X\circ}(\tilde{X})\times_X X^\prime)_\qproet,\mathcal F)\to H^i((\tilde{X}\times_X X^\prime)_v,\mathcal F)
\]
is an isomorphism for $i=0$, and for all $i\geq 0$ if $f$ is pro-\'etale or if $n\mathcal F=0$ for some $n$ prime to $p$. Here, we may replace pro-\'etale and v-cohomology by \'etale cohomology by Proposition~\ref{prop:compproetvcohom} and Proposition~\ref{prop:competproetcohom}. It suffices to prove the statement on fibres over $\pi_0 \tilde{X}\times_{\pi_0 X} \pi_0 X'$, so we can reduce to the case that all of them are connected. Moreover, we can assume that $\tilde{X}\to X$ and $X'\to X$ are surjective, replacing $X$ by the intersection of their images and $\tilde{X}, X'$ by the pullback. In this case, we have to see that for any \'etale sheaf $\mathcal F$ on $X'$, the map
\[
H^i(X'_\et,\mathcal F)\to H^i((\tilde{X}\times_X X')_\et,\mathcal F)
\]
is an isomorphism for $i=0$, and for all $i\geq 0$ if $f$ is pro-\'etale or if $n\mathcal F=0$ for some $n$ prime to $p$. As $\mathcal F$ is a sheaf on the spectral space $|X'|$ that is totally ordered under generalization, we can assume that $\mathcal F=j_! M$ for some quasicompact open subset $V\subset X'$ and some abelian group $M$, killed by $n$ prime to $p$ in case (iii).

If $V=X'$, the case $i=0$ follows from Lemma~\ref{lem:geometricallyconnected}. Otherwise, the left-hand side vanishes, and the question of vanishing of the right-hand side is independent of $M$ (it depends only on the underlying sheaf, not the abelian group structure), so we can reduce to case (iii). In case (ii), after the previous reductions we actually have $X'=X$ and $\tilde{X}\to X'$ is a surjective map of strictly local perfectoid spaces. In particular, there is no cohomology in higher degree. It remains to handle case (iii). This is the following lemma.
\end{proof}

\begin{lemma}\label{lem:etcohomfibreproductalgclosed} Let $X_i=\Spa(C_i,C_i^+)$ for $i=1,2,3$, where $C_i$ is an algebraically closed nonarchimedean field with an open and bounded valuation subring $C_i^+\subset C_i$, and let $X_1\to X_3\leftarrow X_2$ be a diagram, with $X_1\to X_3$ surjective. Let $j_2: U_2\hookrightarrow X_2$ be a quasicompact open subset and $M$ an abelian group killed by some $n$ prime to $p$. Then
\[
H^i((X_1\times_{X_3} X_2)_\et,j_{2!}M) = 0
\]
for all $i\geq 0$ unless $U_2=X_2$ and $i=0$, in which case one gets $M$.
\end{lemma}

\begin{proof} We prove more generally that if $X_2$ is any perfectoid space over $X_3$ and $j_2: U_2\hookrightarrow X_2$ any quasicompact open subset, then
\[
H^i((X_1\times_{X_3} X_2)_\et,j_{2!}M) = H^i(X_{2,\et},j_{2!}M)
\]
for all $i\geq 0$. We can write $X_2=\varprojlim_j \Spa((R_j)^\perf,(R_j^+)^\perf)$ as a cofiltered inverse limit, where each $R_j$ is topologically of finite type over $C_3$. Applying Proposition~\ref{prop:etcohomlim}, it is enough to prove the result for all $\Spa((R_j)^\perf,(R_j^+)^\perf)$, so we may assume $X_2=\Spa(R^\perf,(R^+)^\perf)$, where $R$ is topologically of finite type over $C$. Now the result follows from~\cite[Theorem 4.1.1 (c)]{Huber}, using the equivalence of sites $X_{2\et} = \Spa(R,R^+)_\et$, which follows for example from Lemma~\ref{lem:diamondadicspace} and the observation $\Spa(R,R^+)^\diamondsuit= \Spa(R^\perf,(R^+)^\perf)$.
\end{proof}

Let us note the following consequence for derived categories.

\begin{corollary}\label{cor:vproetbasechangederived} Let $f: Y^\prime\to Y$ be a map of locally spatial diamonds. Assume that $f$ is quasi-pro-\'etale, or $n\Lambda=0$ for some $n$ prime to $p$. Then for any $A\in D_\et(Y^\prime,\Lambda)\subset D(Y^\prime_\qproet,\Lambda)$, the base change morphism
\[
\lambda_Y^\ast Rf_{\qproet\ast} A\to Rf_{v\ast} \lambda_{Y^\prime}^\ast A
\]
in $D(Y_v,\Lambda)$ is an isomorphism.
\end{corollary}

\begin{proof} The statement is local in $Y_\qproet$, so we may assume $Y=X$ is a strictly totally disconnected perfectoid space. We prove the finer statement that for all $\tilde{X}\in X_v$ strictly totally disconnected,
\[
R\Gamma((\lambda_{X\circ}(\tilde{X})\times_X Y^\prime)_\qproet,A)=R\Gamma((\tilde{X}\times_X Y^\prime)_v,\lambda_{Y^\prime}^\ast A)\ .
\]
For this, note that both sides commute with the Postnikov limit $A=R\varprojlim_n \tau^{\geq -n} A$ by Proposition~\ref{prop:leftcomplete}; then it follows from Theorem~\ref{thm:vproetbasechange}~(iii).
\end{proof}

\begin{corollary}\label{cor:pushforwardpreservesproet} Let $f: Y^\prime\to Y$ be a map of locally spatial diamonds. Assume that $f$ is quasi-pro-\'etale, or $n\Lambda=0$ for some $n$ prime to $p$. If $A\in D_\et(Y^\prime,\Lambda)$ is such that $Rf_{\qproet\ast} A\in D_\et(Y,\Lambda)\subset D(Y_\qproet,\Lambda)$, then the pushforward $Rf_{v\ast} A$ lies in $D_\et(Y,\Lambda)\subset D(Y,\Lambda)$ and agrees with $Rf_{\qproet\ast} A$.
\end{corollary}

Here and in the following, we regard objects $A\in D_\et(Y,\Lambda)$ as objects of $D(Y_\qproet,\Lambda)$ or $D(Y_v,\Lambda)$ without an explicit mention of pullback functors to the quasi-pro-\'etale or v-site.

\begin{proof} This is a direct consequence of the previous corollary.
\end{proof}

We need a similar result relating the \'etale and pro-\'etale pushforward.

\begin{proposition}\label{prop:qcqsetproetbasechange} Let $f: Y^\prime\to Y$ be a qcqs morphism of locally spatial diamonds. Let $\mathcal F$ be a sheaf of abelian groups on $Y^\prime_\et$. Then the base change morphism
\[
\nu_Y^\ast R^i f_{\et\ast} \mathcal F\to R^i f_{\qproet\ast} \nu_{Y^\prime}^\ast \mathcal F
\]
of sheaves of abelian groups on $Y_\qproet$ is an isomorphism for all $i\geq 0$.
\end{proposition}

\begin{proof} The statement is \'etale local on $Y$, so we can assume that $Y$ is spatial. By the proof of Proposition~\ref{prop:competproetcohom}, it suffices to check that for all $\tilde{Y} = \varprojlim_j \tilde{Y}_j\to Y$ in $\Pro(Y_{\et,\qc,\sep})$, one has
\[
H^i((\tilde{Y}\times_Y Y^\prime)_\qproet,\mathcal F) = \varinjlim_j H^i((\tilde{Y}_j\times_Y Y^\prime)_\et,\mathcal F)\ ,
\]
where we denote by $\mathcal F$ also any of its base changes. But by Proposition~\ref{prop:competproetcohom}, the left-hand side is given by $H^i((\tilde{Y}\times_Y Y^\prime)_\et,\mathcal F)$, and then Proposition~\ref{prop:etcohomlim} gives the result.
\end{proof}

Again, we can give derived consequences.

\begin{corollary}\label{cor:etproetbasechangederived} Let $f: Y^\prime\to Y$ be a qcqs map of locally spatial diamonds. Then for any $A\in D^+(Y^\prime_\et,\Lambda)$, the base change morphism
\[
\nu_Y^\ast Rf_{\et\ast} A\to Rf_{\qproet\ast} \nu_{Y^\prime}^\ast A
\]
in $D^+(Y_\qproet,\Lambda)$ is an isomorphism. If $Y$ and $Y^\prime$ are strictly totally disconnected perfectoid spaces, then the same result holds true more generally for any $A\in D(Y^\prime_\et,\Lambda)$.
\end{corollary}

\begin{proof} The first part is a direct consequence of Proposition~\ref{prop:qcqsetproetbasechange}. If $Y$ and $Y^\prime$ are strictly totally disconnected perfectoid spaces, then as in Corollary~\ref{cor:vproetbasechangederived}, the statement follows from the convergence results in the proof of Proposition~\ref{prop:pullbackfullyfaithfulderived}.
\end{proof}

\begin{corollary}\label{cor:qcqspushforwardpreserveset} Let $f: Y^\prime\to Y$ be a qcqs map of small v-stacks.
\begin{altenumerate}
\item If $f$ is quasi-pro-\'etale, then for any $A\in D_\et(Y^\prime,\Lambda)$, the pushforward $Rf_{v\ast} A$ lies in $D_\et(Y,\Lambda)$.
\item If $n\Lambda=0$ for some $n$ prime to $p$, then for any $A\in D^+_\et(Y^\prime,\Lambda)$, the pushforward $Rf_{v\ast} A$ lies in $D^+_\et(Y,\Lambda)$. If $Y$ and $Y^\prime$ are locally spatial, then under the identifications $D^+_\et(Y^\prime,\Lambda) = D^+(Y^\prime_\et,\Lambda)$ and $D^+_\et(Y,\Lambda) = D^+(Y_\et,\Lambda)$, one has $Rf_{v\ast} = Rf_{\et\ast}$.
\end{altenumerate}
\end{corollary}

\begin{proof} We may assume that $Y$ is a strictly totally disconnected perfectoid space. In case (i), $Y^\prime$ is a qcqs perfectoid space. If $Y^\prime$ is separated, then it is strictly totally disconnected by Lemma~\ref{lem:proetaleoverwlocal}, and the result follows from Corollary~\ref{cor:etproetbasechangederived} and Corollary~\ref{cor:pushforwardpreservesproet}. In general, there is a finite quasicompact open cover of $Y^\prime$ by affinoid perfectoid spaces, where all intersections are separated. Using the corresponding Cech cover, we get the result in general.

For part (ii), choose an affinoid perfectoid space $X^\prime$ with a surjection $X^\prime\to Y^\prime$. Let $X^\prime_\bullet\to Y^\prime$ be the corresponding simplicial nerve, which is a simplicial qcqs v-sheaf. If $Y^\prime$ was already a v-sheaf, then in fact it is a simplicial spatial diamond; assume that we are in this case for the moment. Let $g_\bullet: X^\prime_\bullet\to Y$ be the resulting map. Then $Rf_{v\ast} C$ is the limit of the simplicial object $Rg_{\bullet,v\ast} C|_{X^\prime_\bullet}$. By Corollary~\ref{cor:etproetbasechangederived} and Corollary~\ref{cor:vproetbasechangederived}, all $Rg_{i,v\ast} C|_{X^\prime_i}\in D^+(Y_\et,\Lambda)$, and there is some $n$ such that all of them lie in $D^{\geq -n}(Y_\et,\Lambda)$. This implies that the derived limit $Rf_{v\ast} C$ also lies in $D^+(Y_\et,\Lambda)$.

In general (if $Y^\prime$ is any qcqs v-stack), one repeats the same argument, using that the result is now known for qcqs v-sheaves.
\end{proof}

Moreover, we get the following classical base change results.

\begin{corollary}\label{cor:proetbasechange} Let
\[\xymatrix{
Y^\prime\ar[r]^{g^\prime}\ar[d]_{f^\prime} & Y\ar[d]^f\\
X^\prime\ar[r]^g & X
}\]
be a cartesian diagram of locally spatial diamonds, and let $\mathcal F$ be a small sheaf of abelian groups on $Y_\qproet$ that comes via pullback from $Y_\et$. Then the base change morphism
\[
g^\ast_\qproet R^i f_{\qproet\ast} \mathcal F\to R^i f^\prime_{\qproet\ast} g^{\prime\ast}_\qproet \mathcal F
\]
is an isomorphism in the following cases:
\begin{altenumerate}
\item if $i=0$, or
\item for all $i\geq 0$ if $f$ or $g$ is quasi-pro-\'etale, or
\item for all $i\geq 0$ if $n\mathcal F=0$ for some $n$ prime to $p$.
\end{altenumerate}

Moreover, if $f$ or $g$ is quasi-pro-\'etale or $n\Lambda=0$ for some $n$ prime to $p$, then for any $A\in D_\et(Y,\Lambda)\subset D(Y_\qproet,\Lambda)$, the base change morphism
\[
g^\ast_\qproet Rf_{\qproet\ast} A\to Rf^\prime_{\qproet\ast} g^{\prime\ast}_\qproet A
\]
is an isomorphism.
\end{corollary}

\begin{proof} If $g$ is quasi-pro-\'etale, this is formal, by passage to slices. The other cases follow from the final statement of Theorem~\ref{thm:vproetbasechange} (and Proposition~\ref{prop:compproetvcohom}). Indeed, to show that
\[
g^\ast_\qproet R^i f_{\qproet\ast} \mathcal F\to R^i f^\prime_{\qproet\ast} g^{\prime\ast}_\qproet \mathcal F
\]
we can first of all assume that $X$ is strictly totally disconnected (by the case where $g$ is quasi-pro-\'etale). Then it suffices to identify the values at strictly totally disconnected $\tilde{X}$ quasi-pro-\'etale over $X'$. This may be factored as $\tilde{X}\to \lambda_{X\circ}(\tilde{X})\to X$ where $\lambda_{X\circ}(\tilde{X})$ is strictly totally disconnected and pro-\'etale over $X$. Then the left-hand side is
\[
H^i(\lambda_{X\circ}(\tilde{X})\times_X Y,\mathcal F)
\]
while the right-hand side is
\[
H^i(\tilde{X}\times_X Y,\mathcal F)
\]
(where we omit some pullbacks of $\mathcal F$, and whether to take pro-\'etale or v-cohomology, as those agree by Proposition~\ref{prop:compproetvcohom}). This is the final statement of Theorem~\ref{thm:vproetbasechange}.

The final statement (about complexes $A$) follows as everything commutes with Postnikov limits.
\end{proof}

\begin{corollary}\label{cor:qcqsetbasechange} Let
\[\xymatrix{
Y^\prime\ar[r]^{g^\prime}\ar[d]_{f^\prime} & Y\ar[d]^f\\
X^\prime\ar[r]^g & X
}\]
be a cartesian diagram of locally spatial diamonds where $f$ is qcqs, and let $\mathcal F$ be a sheaf of abelian groups on $Y_\et$. Then the base change morphism
\[
g^\ast_\et R^i f_{\et\ast} \mathcal F\to R^i f^\prime_{\et\ast} g^{\prime\ast}_\et \mathcal F
\]
is an isomorphism in the following cases:
\begin{altenumerate}
\item if $i=0$, or
\item for all $i\geq 0$ if $f$ or $g$ is quasi-pro-\'etale, or
\item for all $i\geq 0$ if $n\mathcal F=0$ for some $n$ prime to $p$.
\end{altenumerate}

Moreover, if $f$ or $g$ is quasi-pro-\'etale or $n\Lambda=0$ for some $n$ prime to $p$, then for any $A\in D^+(Y_\et,\Lambda)$, the base change morphism
\[
g^\ast_\et Rf_{\et\ast} A\to Rf^\prime_{\et\ast} g^{\prime\ast}_\et A
\]
is an isomorphism.
\end{corollary}

\begin{proof} This follows from Corollary~\ref{cor:proetbasechange} and Proposition~\ref{prop:qcqsetproetbasechange}. Note that here the final statement is a formal consequence of the statement about sheaves.
\end{proof}

\section{Four functors}\label{sec:fourfunctors}

We will use the formalism of the previous section to set up the first parts of our six functor formalism. This will notably introduce a pushforward functor $Rf_\ast$ that is in general different from any of the functors $Rf_{v\ast}$, $Rf_{\qproet\ast}$ and $Rf_{\et\ast}$, and we will clarify their relation.

First, recall that for any small v-stack $Y$, we have defined the full subcategory $D_\et(Y,\Lambda)\subset D(Y_v,\Lambda)$. We will at certain points invoke Lurie's $\infty$-categorical adjoint functor theorem. For this reason, we need to upgrade our constructions to functors of $\infty$-categories at certain points.

\begin{lemma}\label{lem:enrichmentexists} There is a (natural) presentable stable $\infty$-category $\mathcal{D}_\et(Y,\Lambda)$ whose homotopy category is $D_\et(Y,\Lambda)$. More precisely, for each sufficiently large $\kappa$ as in Lemma~\ref{lem:choosekappa}, the $\infty$-derived category $\mathcal{D}(Y_{v,\kappa},\Lambda)$ of $\Lambda$-modules on $Y_{v,\kappa}$ is a presentable stable $\infty$-category, and for $\kappa$ large enough, $\mathcal{D}_\et(Y,\Lambda)$ is a full presentable stable $\infty$-subcategory closed under all colimits.
\end{lemma}

\begin{proof} First, $\mathcal{D}(Y_{v,\kappa},\Lambda)$ is a presentable stable $\infty$-category, as this is true for any ringed topos. Next, we check that the full $\infty$-subcategory $\mathcal{D}_\et(Y,\Lambda)$, with objects those of $D_\et(Y,\Lambda)$, is closed under all colimits in $\mathcal{D}(Y_{v,\kappa},\Lambda)$. This is clear for cones, so we are reduced to filtered colimits. Those commute with canonical truncations, and filtered colimits of \'etale sheaves are still \'etale sheaves, as desired.

By \cite[Proposition 5.5.3.12]{LurieHTT}, it is enough to prove the claim if $Y$ is a disjoint union of strictly totally disconnected perfectoid spaces. In that case, $\mathcal{D}_\et(Y,\Lambda)=\mathcal{D}(Y_\et,\Lambda)$ (as the functor of stable $\infty$-categories $\mathcal{D}(Y_\et,\Lambda)\to \mathcal{D}(Y_v,\Lambda)$ is fully faithful (as it is on homotopy categories), and has the same objects as $\mathcal{D}_\et(Y,\Lambda)$), which is a presentable $\infty$-category.
\end{proof}

One consequence of this is the following corollary.

\begin{corollary}\label{cor:etreflection} For any small v-stack $Y$, the inclusion $D_\et(Y,\Lambda)\subset D(Y_v,\Lambda)$ has a right adjoint
\[
R_{Y\et}: D(Y_v,\Lambda)\to D_\et(Y,\Lambda)\ .
\]
\end{corollary}

In general, it is hard to compute $R_{Y\et}$. If $Y$ is a locally spatial diamond, then on $D^+(Y_v,\Lambda)$, it is given by $R(\nu\circ \lambda)_\ast$, where $\nu\circ \lambda: Y_v\to Y_\et$ is the map of sites; for strictly totally disconnected spaces, this formula even holds true on all of $D(Y_v,\Lambda)$.

\begin{proof} For each $\kappa$ the inclusion can be lifted to a colimit-preserving functor $\mathcal{D}_\et(Y,\Lambda)\to \mathcal{D}(Y_{v,\kappa},\Lambda)$ of presentable $\infty$-categories. As such, it admits a right adjoint by Lurie's $\infty$-categorical adjoint functor theorem, \cite[Corollary 5.5.2.9]{LurieHTT}. These right adjoints are necessarily compatible.
\end{proof}

We now deal first with the pullback functor. We claim that if $f: Y^\prime\to Y$ is a map of small v-stacks, there is a pullback functor
\[
f^\ast: D(Y_v,\Lambda)\to D(Y^\prime_v,\Lambda)\ ,
\]
inducing by restriction
\[
f^\ast: D_\et(Y,\Lambda)\to D_\et(Y^\prime,\Lambda)\ .
\]
This functor is easy to construct for $0$-truncated maps, as then $f$ induces a map of sites $f_v: Y^\prime_v\to Y_v$. In that case, one even gets naturally a functor of $\infty$-categories
\[
f_v^\ast: \mathcal{D}(Y_v,\Lambda)\to \mathcal{D}(Y^\prime_v,\Lambda)\ ,
\]
which induces by restriction
\[
f^\ast: \mathcal{D}_\et(Y,\Lambda)\to \mathcal{D}_\et(Y^\prime,\Lambda)\ .
\]
In general, we need simplicial techniques, so let us recall this.

\begin{proposition}\label{prop:derivedhyperdescent} Let $Y$ be a small v-stack, and let $Y_\bullet\to Y$ be a simplicial v-hypercover of $Y$ by small v-stacks $Y_\bullet$ with $0$-truncated maps $Y_i\to Y$. Consider the simplicial site $Y_{\bullet,v}$, and let $\mathcal{D}(Y_{\bullet,v},\Lambda)$ be the $\infty$-derived category of sheaves of $\Lambda$-modules on $Y_{\bullet,v}$, with homotopy category $D(Y_{\bullet,v},\Lambda)$.

Pullback along $Y_\bullet\to Y$ induces a fully faithful functor
\[
\mathcal{D}(Y_v,\Lambda)\to \mathcal{D}(Y_{\bullet,v},\Lambda)
\]
whose essential image is given by the full $\infty$-subcategory of cartesian objects. In particular, on homotopy categories,
\[
D(Y_v,\Lambda)\to D(Y_{\bullet,v},\Lambda)
\]
is fully faithful with essential image given by the full subcategory of cartesian objects.
\end{proposition}

\begin{proof} This follows from \cite[Tag 0DC7]{StacksProject} (on homotopy categories, thus on stable $\infty$-categories), cf.~also \cite[Proposition 3.3.6]{BhattScholze}.
\end{proof}

\begin{remark} One may worry that $D_\et(Y,\Lambda)\subset D(Y_v,\Lambda)$ is not contained in any one $D(Y_{v,\kappa},\Lambda)$, so would be rather large. However, by choosing a simplicial v-hypercover $Y_\bullet\to Y$ by disjoint unions of strictly totally disconnected spaces, and using that $D_\et(Y_i,\Lambda) = D(Y_{i,\et},\Lambda)$ is contained in $D(Y_{i,v,\kappa},\Lambda)$ for any $\kappa$ so that $Y_i$ is $\kappa$-small, one finds that $D_\et(Y,\Lambda) = D_{\et,\cart}(Y_\bullet,\Lambda)$ has the similar property.
\end{remark}

Now, if $f: Y^\prime\to Y$ is any map of small v-stacks, choose a perfectoid space $Y^\prime_0$ with a surjective map of v-stacks $Y^\prime_0\to Y^\prime$, and let $Y^\prime_\bullet$ be the corresponding Cech nerve, all of whose terms are small v-sheaves. Then in particular all the maps $Y^\prime_i\to Y^\prime$ and $Y^\prime_i\to Y$ are $0$-truncated, and we get a well-defined pullback functor
\[
\mathcal{D}(Y_v,\Lambda)\to \mathcal{D}(Y^\prime_{\bullet,v},\Lambda)\ ,
\]
taking values in cartesian objects (as composites of pullbacks are pullbacks). Thus, we get a functor
\[
f_v^\ast: \mathcal{D}(Y_v,\Lambda)\to \mathcal{D}_\cart(Y^\prime_{\bullet,v},\Lambda)\simeq \mathcal{D}(Y^\prime_v,\Lambda)\ .
\]
This carries $\mathcal{D}_\et(Y_v,\Lambda)$ into $\mathcal{D}_\et(Y^\prime,\Lambda)$ (as it is, at least up to homotopy, compatible with further pullbacks). Thus, we get a functor
\[
f^\ast: \mathcal{D}_\et(Y,\Lambda)\to \mathcal{D}_\et(Y^\prime,\Lambda)\ ,
\]
and in particular
\[
f^\ast: D_\et(Y,\Lambda)\to D_\et(Y^\prime,\Lambda)\ .
\]
It is easy to see that this is canonically independent of the choices made (as the possible choices of $Y^\prime_0\to Y^\prime$ form a cofiltered category), and compatible with composition, i.e.~for another map $g: Y^{\prime\prime}\to Y^\prime$ of small v-stacks with composite $f\circ g: Y^{\prime\prime}\to Y$, one has a natural equivalence $(f\circ g)^\ast\simeq f^\ast\circ g^\ast$, satisfying the usual coherences.\footnote{A different and much simpler way to construct this functoriality has been suggested by a referee: One can instead use the cocontinuous morphism of sites taking any $X'\to Y'$ in $Y'_v$ to the composite $X'\to Y'\to Y$ in $Y_v$, and the corresponding pullback functoriality.}

\begin{lemma}\label{lem:Rfastexists} For any map of small v-stacks $f: Y^\prime\to Y$, the functor
\[
f^\ast: D_\et(Y,\Lambda)\to D_\et(Y^\prime,\Lambda)
\]
has a right adjoint
\[
Rf_\ast: D_\et(Y^\prime,\Lambda)\to D_\et(Y,\Lambda)\ .
\]
\end{lemma}

\begin{proof} The functor $f_v^\ast: \mathcal{D}(Y_v,\Lambda)\to \mathcal{D}(Y^\prime_v,\Lambda)$ commutes with all colimits by construction, and thus so does $f^\ast: \mathcal{D}_\et(Y,\Lambda)\to \mathcal{D}_\et(Y^\prime,\Lambda)$. Thus, the result follows from the $\infty$-categorical adjoint functor theorem, \cite[Corollary 5.5.2.9]{LurieHTT}, and Lemma~\ref{lem:enrichmentexists}.
\end{proof}

Again, it is in general hard to compute $Rf_\ast$. If $Y^\prime$ and $Y$ are locally spatial diamonds, then it agrees with (the left-completed) $Rf_{\et\ast}$. In general, it is clear that $Rf_\ast = R_{Y\et} Rf_{v\ast}$; however, this involves the complicated functor $R_{Y\et}$. However, in the following situation, it is not necessary to apply $R_{Y\et}$; this makes use of the results of the previous section. 

\begin{proposition}\label{prop:Rfastsimple} Assume that $n\Lambda=0$ for some $n$ prime to $p$. Let $f: Y^\prime\to Y$ be a qcqs map of small v-stacks. Then for any $A\in D^+_\et(Y^\prime,\Lambda)$, one has $Rf_{v\ast} A\in D^+_\et(Y,\Lambda)$, and thus
\[
Rf_\ast A = Rf_{v\ast} A\ .
\]
Moreover, in this case, the formation of $Rf_\ast$ commutes with any base change, i.e.~for all maps $g: \tilde{Y}\to Y$ of small v-stacks with pullback $\tilde{f}: \tilde{Y}^\prime = Y^\prime\times_Y \tilde{Y}\to \tilde{Y}$, $g^\prime: \tilde{Y}^\prime\to Y^\prime$, the natural transformation
\[
g^\ast Rf_\ast A\to R\tilde{f}_\ast g^{\prime\ast} A
\]
is an equivalence for all $A\in D^+_\et(Y^\prime,\Lambda)$.

If $Rf_\ast$ has finite cohomological dimension, i.e.~there is some integer $N$ such that for all $A\in D_\et(Y^\prime,\Lambda)$ concentrated in degree $0$, one has $R^if_\ast A=0$ for $i>N$, then for all $A\in D_\et(Y^\prime,\Lambda)$, one has
\[
Rf_\ast A = Rf_{v\ast} A\in D_\et(Y,\Lambda)\ ,
\]
and moreover the base change morphism
\[
g^\ast Rf_\ast A\to R\tilde{f}_\ast g^{\prime\ast} A
\]
is an isomorphism.
\end{proposition}

We note that to get base change, we only need to assume that $Rf_\ast$ has finite cohomological dimension (and not the same about $R\tilde{f}_\ast$).

\begin{proof} Assume first that $A\in D^+_\et(Y^\prime,\Lambda)$. The statement $Rf_{v\ast} A\in D^+_\et(Y,\Lambda)$ is given by Corollary~\ref{cor:qcqspushforwardpreserveset}~(ii). The base change result follows formally as v-pushforward commutes with v-slices.

Now, if $Rf_\ast$ has finite cohomological dimension, then by left-completeness, it follows that in general
\[
Rf_{v\ast} A = R\varprojlim_n Rf_\ast \tau^{\geq -n} A\ ,
\]
where the limit becomes eventually constant in each degree. In a replete topos, the cohomological dimension of $R\varprojlim_n$ is bounded by $1$, cf.~\cite[Proposition 3.1.11]{BhattScholze}, so we see that each cohomology sheaf of $Rf_{v\ast} A$ agrees with the cohomology sheaf of $Rf_\ast \tau^{\geq -n} A$ for $n$ sufficiently large, which shows that $Rf_{v\ast} A\in D_\et(Y,\Lambda)$ by Proposition~\ref{prop:etproetcheckoncohomsheaves}, and thus $Rf_\ast A = Rf_{v\ast} A$.

Moreover, concerning base change, one has, by commuting limits,
\[\begin{aligned}
R\tilde{f}_\ast g^{\prime\ast} A &= R\varprojlim_n R\tilde{f}_\ast g^{\prime\ast} \tau^{\geq -n} A\\
&= R\varprojlim_n g^\ast Rf_\ast \tau^{\geq -n} A\ .
\end{aligned}\]
This admits a canonical map from $g^\ast Rf_\ast A$. The cone of the map
\[
g^\ast Rf_\ast A\to g^\ast Rf_\ast \tau^{\geq -n} A
\]
lies in degrees $\leq -n+N$. This implies that the $R\varprojlim_n$ of the cone is equal to zero, as we work in a replete topos, and so the derived category is left-complete.
\end{proof}

Next, we define the tensor product functor. Note that there is a natural functor
\[
-\dotimes_\Lambda - : D(Y_v,\Lambda)\times D(Y_v,\Lambda)\to D(Y_v,\Lambda)\ ,
\]
like for any ringed topos. In fact, this is already defined as a functor of $\infty$-categories
\[
-\dotimes_\Lambda - : \mathcal{D}(Y_v,\Lambda)\times \mathcal{D}(Y_v,\Lambda)\to \mathcal{D}(Y_v,\Lambda)
\]
which preserves all colimits separately in each variable.\footnote{Even better, $\mathcal{D}(Y_v,\Lambda)$ is a presentably symmetric monoidal $\infty$-category.} If $f: Y^\prime\to Y$ is any map of small v-stacks, then there is a natural equivalence
\[
f^\ast (A\dotimes_\Lambda B)\simeq f^\ast A\dotimes_\Lambda f^\ast B
\]
for all $A,B\in D(Y_v,\Lambda)$: This is clear for $0$-truncated maps by the general formalism of ringed topoi, and then follows in general by following the construction of $f^\ast$.

\begin{lemma}\label{lem:tensorprodexists} For any small v-stack $Y$, the functor
\[
-\dotimes_\Lambda - : D(Y_v,\Lambda)\times D(Y_v,\Lambda)\to D(Y_v,\Lambda)
\]
induces by restriction a functor
\[
-\dotimes_\Lambda - : D_\et(Y,\Lambda)\times D_\et(Y,\Lambda)\to D_\et(Y,\Lambda)\ .
\]
\end{lemma}

One can use this lemma (together with the previous discussion) to deduce that $\mathcal{D}_\et(Y,\Lambda)$ is a presentably symmetric monoidal $\infty$-category (in a unique way compatible with the standard symmetric monoidal structure on $\mathcal{D}(Y_v,\Lambda)$).

\begin{proof} This can be checked v-locally on $Y$, so we can assume that $Y$ is a disjoint union of strictly totally disconnected perfectoid spaces. In this case, $D_\et(Y,\Lambda) = D(Y_\et,\Lambda)$, and the result follows from the existence of the natural tensor product on $D(Y_\et,\Lambda)$ (compatible with pullback along $Y_v\to Y_\et$).
\end{proof}

\begin{lemma}\label{lem:rhomexists} For any small v-stack $Y$ and $A\in D_\et(Y,\Lambda)$, the functor
\[
D_\et(Y,\Lambda)\to D_\et(Y,\Lambda): B\mapsto B\dotimes_\Lambda A
\]
admits a right adjoint
\[
D_\et(Y,\Lambda)\to D_\et(Y,\Lambda) : C\mapsto R\sHom_\Lambda(A,C)\ ,
\]
i.e.
\[
\Hom_{D_\et(Y,\Lambda)}(B\dotimes_\Lambda A,C) = \Hom_{D_\et(Y,\Lambda)}(B,R\sHom_\Lambda(A,C))\ .
\]
For varying $A$, these assemble into a functor
\[
R\sHom_\Lambda(-,-): D_\et(Y,\Lambda)^\op\times D_\et(Y,\Lambda)\to D_\et(Y,\Lambda)\ .
\]
\end{lemma}

Again, $R\sHom_\Lambda(-,-)$ is in general hard to compute, and not compatible with the inner Hom in $D(Y_v,\Lambda)$; rather, one has to apply $R_{Y\et}$ to the inner Hom in $D(Y_v,\Lambda)$.

\begin{proof} We have a functor from $D_\et(Y,\Lambda)^\op\times D_\et(Y,\Lambda)$ to the presheaf category on $D_\et(Y,\Lambda)$, given by sending a pair $(A,C)$ to the functor from $D_\et(Y,\Lambda)^\op$ to sets,
\[
B\mapsto \Hom_{D_\et(Y,\Lambda)}(B\dotimes_\Lambda A,C)\ .
\]
The claim is that this functor factors over the full subcategory $D_\et(Y,\Lambda)$, i.e.~is representable for any pair $(A,C)$. But for any fixed $(A,C)$, the functor
\[
B\mapsto \Map_{\mathcal{D}_\et(Y,\Lambda)}(B\dotimes_\Lambda A,C)
\]
from $\mathcal{D}_\et(Y,\Lambda)^\op$ to the $\infty$-category of spaces takes all limits (in $\mathcal{D}_\et(Y,\Lambda)^\op$, which are colimits in $\mathcal{D}_\et(Y,\Lambda)$) to limits (as $-\dotimes_\Lambda A$ preserves colimits). Thus, by \cite[Proposition 5.5.2.2]{LurieHTT}, it is representable, by what is denoted $R\sHom_\Lambda(A,C)$.
\end{proof}

\begin{corollary}\label{cor:locastadjunction} Let $f: Y^\prime\to Y$ be a map of small v-stacks. There is a natural equivalence
\[
Rf_\ast R\sHom_\Lambda(f^\ast A,B)\cong R\sHom_\Lambda(A,Rf_\ast B)
\]
of functors
\[
D_\et(Y,\Lambda)^\op \times D_\et(Y^\prime,\Lambda)\to D_\et(Y,\Lambda)\ .
\]
\end{corollary}

\begin{proof} For any $C\in D_\et(Y,\Lambda)$, we have a series of natural equivalences
\[\begin{aligned}
\Hom_{D_\et(Y,\Lambda)}(C,Rf_\ast R\sHom_\Lambda(f^\ast A,B)) &\cong \Hom_{D_\et(Y^\prime,\Lambda)}(f^\ast C,R\sHom_\Lambda(f^\ast A,B))\\
&\cong \Hom_{D_\et(Y^\prime,\Lambda)}(f^\ast C\dotimes_\Lambda f^\ast A,B)\\
&\cong \Hom_{D_\et(Y^\prime,\Lambda)}(f^\ast(C\dotimes_\Lambda A),B)\\
&\cong \Hom_{D_\et(Y,\Lambda)}(C\dotimes_\Lambda A,Rf_\ast B)\\
&\cong \Hom_{D_\et(Y,\Lambda)}(C,R\sHom_\Lambda(A,Rf_\ast B))\ ,
\end{aligned}\]
giving the desired result.
\end{proof}

\section{Proper and partially proper morphisms}

As a preparation for the proper base change theorem in the next section, we define proper and partially proper morphisms of v-sheaves.

\begin{definition}\label{def:proper} Let $f: Y^\prime\to Y$ be a map of v-stacks. The map $f$ is proper if it is quasicompact, separated, and universally closed, i.e.~for all small v-sheaves $X$ with a map $X\to Y$, the map $|Y^\prime\times_Y X|\to |X|$ is closed.
\end{definition}

Note that as $f$ is quasicompact, it follows in particular that for any small v-sheaf $X$ with a map $X\to Y$, the pullback $Y^\prime\times_Y X$ is a small v-sheaf. By the definition of the topology on $|X|$ for a small v-sheaf, it suffices to check this condition for perfectoid spaces $X$, or even just for strictly totally disconnected perfectoid spaces $X$.

\begin{remark} Any closed immersion of v-sheaves is a proper map. As a somewhat amusing exercise, we leave it to the reader to check that if $T^\prime\to T$ is a map of locally compact Hausdorff spaces, then the map of v-sheaves $\underline{T^\prime}\to \underline{T}$ is proper if and only if $T^\prime\to T$ is proper in the usual sense.
\end{remark}

Again, there is a valuative criterion.

\begin{proposition}\label{prop:valcritvsheafprop} Let $f: Y^\prime\to Y$ be a map of v-stacks. Then $f$ is proper if and only if it is $0$-truncated, qcqs, and for every perfectoid field $K$ with an open and bounded valuation subring $K^+\subset K$, and any diagram
\[\xymatrix{
\Spa(K,\OO_K)\ar[d]\ar[r] & Y^\prime\ar^f[d] \\
\Spa(K,K^+)\ar[r]\ar@{-->}[ur] & Y\ ,
}\]
there exists a unique dotted arrow making the diagram commute.

Moreover, if $f$ is proper, then for every perfectoid Tate ring $R$ with an open and integrally closed subring $R^+\subset R$,  and every diagram
\[\xymatrix{
\Spa(R,R^\circ)\ar[d]\ar[r] & Y^\prime\ar^f[d] \\
\Spa(R,R^+)\ar[r]\ar@{-->}[ur] & Y\ ,
}\]
there exists a unique dotted arrow making the diagram commute.
\end{proposition}

\begin{proof} Assume first that the valuative criterion holds. By Proposition~\ref{prop:valcritsepsheaf}, we know that $f$ is separated. Replacing $f$ by a pullback, it is enough to see that $|f|: |Y^\prime|\to |Y|$ is closed if $Y$ is a small v-sheaf. Choosing a v-cover $X\to Y$, where $X$ is a perfectoid space, we can further reduce to the case that $Y=X$ is a perfectoid space. As we can moreover work locally on $X$, we can assume that $X$ is an affinoid perfectoid space. Then $Y^\prime$ is a qcqs v-sheaf; choose an affinoid perfectoid space $X^\prime$ with a v-cover $X^\prime\to Y^\prime$. Let $Z\subset |Y^\prime|$ be a closed subset. Then the image $W\subset |X|$ of $Z$ in $|Y^\prime|$ is a pro-constructible subset, as it is also the image of the preimage of $Z$ in $|X^\prime|$, under the spectral map of spectral spaces $|X^\prime|\to |X|$. To see that $W$ is closed, it suffices to see that $W$ is specializing. For this, let $w,w^\prime\in X$ be two points such that $w$ specializes to $w^\prime$, and $w\in W$. This corresponds to maps $\Spa(K,K^+)\to \Spa(K,(K^+)^\prime)\to X$ for a perfectoid field $K$ with open and bounded valuation subrings $(K^+)^\prime\subset K^+\subset K$, such that $w$ is the image of the closed point of $\Spa(K,K^+)$, and $w^\prime$ is the image of the closed point of $\Spa(K,(K^+)^\prime)$. As $w$ lies in the image of $Z$, we can assume that, after enlarging $K$, there is a commutative diagram
\[\xymatrix{
\Spa(K,K^+)\ar[d]\ar[r] & Y^\prime\ar^f[d] \\
\Spa(K,(K^+)^\prime)\ar[r]\ar@{-->}[ur] & X\ ,
}\]
such that the image of the closed point of $\Spa(K,K^+)$ in $|Y^\prime|$ lies in $Z$. Applying the valuative criterion for the two pairs $(K,K^+)$ and $(K,(K^+)^\prime)$, we see that there exists a unique dotted arrow in the diagram. As $Z$ is closed and $\Spa(K,K^+)$ is dense in $\Spa(K,(K^+)^\prime)$, the closed point of $\Spa(K,(K^+)^\prime)$ maps to $Z$. Thus, $w^\prime$ lies in the image $W$ of $Z$ in $|X|$, as desired.

Conversely, assume that $f$ is proper. We will prove that for all pairs $(R,R^+)$ of a perfectoid Tate ring with an open and integrally closed subring $R^+\subset R$, and every diagram
\[\xymatrix{
\Spa(R,R^\circ)\ar[d]\ar[r] & Y^\prime\ar^f[d] \\
\Spa(R,R^+)\ar[r]\ar@{-->}[ur] & Y\ ,
}\]
there exists a unique dotted arrow making the diagram commute. By Proposition~\ref{prop:genvalcritvsheaf}, uniqueness is clear. Thus, assume given a map $\Spa(R,R^+)\to Y$ and a lift of $\Spa(R,R^\circ)\to Y$ to a map $\Spa(R,R^\circ)\to Y^\prime$. We may replace $Y$ by $\Spa(R,R^+)$ and $Y^\prime$ by the corresponding pullback. We get a quasicompact injection
\[
\Spa(R,R^\circ)\to Y^\prime\ .
\]
By Proposition~\ref{prop:charinjectionvsheaves}, it is characterized by the subspace $|\Spa(R,R^\circ)|\subset |Y^\prime|$. Let $X^\prime\to Y^\prime$ be a surjection from a totally disconnected perfectoid space $X^\prime$. Then $\Spa(R,R^\circ)\times_{Y^\prime} X^\prime\subset X^\prime$ is a quasicompact injection, which by Corollary~\ref{cor:qcvinjqproet} corresponds to a pro-constructible and generalizing subset $W\subset |X^\prime|$, given by $W=|\Spa(R,R^\circ)|\times_{|Y^\prime|} |X^\prime|$. Consider the closure $Z$ of $|\Spa(R,R^\circ)|$ in $|Y^\prime|$. Its preimage $Z^\prime\subset |X^\prime|$ is a closed subset containing $W$. In fact, it is precisely the closure of $W$, as the closure of $W$ can be checked to be invariant under the equivalence relation $|X^\prime|\times_{|Y^\prime|} |X^\prime|$, and $|Y^\prime|$ has the quotient topology from $|X^\prime|$. Thus, $Z^\prime\subset |X^\prime|$ is a closed and generalizing subset, and therefore corresponds to a totally disconnected perfectoid space; passing back to $Z$, we see that $Z$ corresponds to a qcqs v-sheaf, still denoted by $Z$. The map $Z\to \Spa(R,R^+)$ is an injection: Indeed, to check this, it suffices by Proposition~\ref{prop:charinjectionvsheaves} to check that $Z(K,K^+)\to \Spa(R,R^+)(K,K^+)$ is injective for all perfectoid fields $K$ with an open and bounded valuation subring $K^+\subset K$. This is true if $K^+=\OO_K$ as $Z(K,\OO_K) = \Spa(R,R^\circ)(K,\OO_K)$; in general, it follows as $Z\to \Spa(R,R^+)$ is separated (as a sub-v-sheaf of the separated map $Y^\prime\to \Spa(R,R^+)$).

On the other hand, as $f$ is by assumption universally closed, the image of $|Z|$ in $|\Spa(R,R^+)|$ is closed, and thus all of $|\Spa(R,R^+)|$, as it contains the dense subspace $|\Spa(R,R^\circ)|$. Thus, $|Z|\to |\Spa(R,R^+)|$ is surjective, thus a quotient map by Lemma~\ref{lem:quotientmap} (applied to a cover of $Z$ by an affinoid perfectoid space), and therefore (using Proposition~\ref{prop:charinjectionvsheaves} again) $Z$ maps isomorphically to $\Spa(R,R^+)$. Composing the inverse with the natural map $Z\to Y^\prime$ gives the desired map $\Spa(R,R^+)\to Y^\prime$.
\end{proof}

The following generalization of proper maps is useful.

\begin{definition}\label{def:partiallyproper} Let $f: Y^\prime\to Y$ be a map of v-stacks. Then $f$ is partially proper if $f$ is separated, and for every perfectoid Tate ring $R$ with an open and integrally closed subring $R^+\subset R$, and every diagram
\[\xymatrix{
\Spa(R,R^\circ)\ar[d]\ar[r] & Y^\prime\ar^f[d] \\
\Spa(R,R^+)\ar[r]\ar@{-->}[ur] & Y\ ,
}\]
there exists a (necessarily unique) dotted arrow making the diagram commute.

Moreover, the v-stack $Y$ is partially proper if $Y\to \ast$ is partially proper.
\end{definition}

By the valuative criterion of separatedness, Proposition~\ref{prop:valcritsepsheaf}, we could in Definition~\ref{def:partiallyproper} replace the condition ``$f$ is separated'' by ``$f$ is $0$-truncated and quasiseparated''; we will use this remark in the proof of Proposition~\ref{prop:compactificationprop}~(i).

\begin{remark} We refrain from making the definition that $Y$ is proper if $Y\to \ast$ is proper. The problem here is that proper should be equivalent to being partially proper and quasicompact; however, the notions of $Y$ being quasicompact and of $Y\to \ast$ being quasicompact are different. In the partially proper case, this issue does not come up.
\end{remark}

A useful feature of the world of analytic adic spaces is the presence of canonical compactifications, cf.~\cite[Theorem 5.1.5]{Huber}. In our context, this is the following construction.

\begin{proposition}\label{prop:cancomp} Let $f: Y^\prime\to Y$ be a separated map of v-stacks. The functor sending any totally disconnected perfectoid space $X=\Spa(R,R^+)$ to
\[
Y^\prime(R,R^\circ)\times_{Y(R,R^\circ)} Y(R,R^+)
\]
extends to a v-stack $\overline{Y^\prime}^{/Y}$ with a map $\overline{f}^{/Y}: \overline{Y^\prime}^{/Y}\to Y$, and a natural map $Y^\prime\to \overline{Y^\prime}^{/Y}$ over $Y$. The map $\overline{f}^{/Y}$ is partially proper, and for every partially proper map $g: Z\to Y$ of v-stacks, composition with $Y^\prime\to \overline{Y^\prime}^{/Y}$ induces a bijection
\[
\Hom_Y(\overline{Y^\prime}^{/Y},Z)\to \Hom_Y(Y^\prime,Z)\ .
\]
\end{proposition}

We will refer to $\overline{f}^{/Y}: \overline{Y^\prime}^{/Y}\to Y$ as the canonical compactification of $f: Y^\prime\to Y$. The formation of $\overline{f}^{/Y}$ is clearly functorial in $f$.

\begin{proof} To see that $\overline{Y^\prime}^{/Y}$ is a v-stack, observe first that if $\Spa(S,S^+)\to \Spa(R,R^+)$ is a v-cover of totally disconnected perfectoid spaces, then also $\Spa(S,S^\circ)\to \Spa(R,R^\circ)$ is a v-cover. This is true as $\Spa(R,R^\circ)\subset \Spa(R,R^+)$ is the minimal pro-constructible generalizing subset containing all rank-$1$-points by Lemma~\ref{lem:subsetwlocal}, and the image of $\Spa(S,S^\circ)\to \Spa(R,R^\circ)$ is a pro-constructible and generalizing subset containing all rank-$1$-points. Also, to check that $\overline{Y^\prime}^{/Y}$ is a v-stack, we may work locally on $Y$, and in particular assume that $Y$ is representable, in which case $Y^\prime$ is a v-sheaf, and $\overline{Y^\prime}^{/Y}$ is a presheaf. Choose a surjection $\Spa(T,T^+)\to \Spa(S,S^+)\times_{\Spa(R,R^+)} \Spa(S,S^+)$, where $\Spa(T,T^+)$ is totally disconnected. We need to see that
\[
\overline{Y^\prime}^{/Y}(R,R^+) = \eq(\overline{Y^\prime}^{/Y}(S,S^+)\rightrightarrows \overline{Y^\prime}^{/Y}(T,T^+))\ .
\]
As $\Spa(S,S^\circ)\to \Spa(R,R^\circ)$ is a v-cover, we see that the map
\[
\overline{Y^\prime}^{/Y}(R,R^+)=Y^\prime(R,R^\circ)\times_{Y(R,R^\circ)} Y(R,R^+)\to \overline{Y^\prime}^{/Y}(S,S^+)=Y^\prime(S,S^\circ)\times_{Y(S,S^\circ)} Y(S,S^+)
\]
is injective. Now let $a\in \overline{Y^\prime}^{/Y}(S,S^+)=Y^\prime(S,S^\circ)\times_{Y(S,S^\circ)} Y(S,S^+)$ be a section whose two pullbacks to $\overline{Y^\prime}^{/Y}(T,T^+)$ agree. In particular, we get a section $b\in Y(R,R^+)$ as $Y$ is a v-sheaf. On $Y^\prime$, we want to see that the section in $Y^\prime(S,S^\circ)$ descends to $Y^\prime(R,R^\circ)$. For this, we need to see that the two induced sections in $Y^\prime(T,(T^+)^\prime)$ agree, where $\Spa(T,(T^+)^\prime) = \Spa(S,S^\circ)\times_{\Spa(R,R^\circ)} \Spa(S,S^\circ)$. But we know that they agree in $Y^\prime(T,T^\circ)$ and $Y(T,(T^+)^\prime)$, and $Y^\prime\to Y$ is separated, so the result follows from Proposition~\ref{prop:genvalcritvsheaf}.

It remains to see that $\overline{f}^{/Y}$ is partially proper, for then the definition of partially proper maps ensures that any map $Y^\prime\to Z$ to a partially proper $Z\to Y$ factors uniquely over $\overline{Y^\prime}^{/Y}$. This question is v-local on $Y$, so we may assume that $Y$ is an affinoid perfectoid space (and in particular separated). Now the result follows from Proposition~\ref{prop:compactificationprop}~(i) below.
\end{proof}

In fact, for any separated v-sheaf $Y$, one can define a new v-sheaf $\overline{Y}$ by setting
\[
\overline{Y}(R,R^+)=Y(R,R^\circ)
\]
for any totally disconnected space $\Spa(R,R^+)$, which comes with an injective map $Y\to \overline{Y}$. The verification that $\overline{Y}$ is a v-sheaf is identical to the verification for $\overline{Y^\prime}^{/Y}$ above (and uses that $Y$ is separated). With this notation, we have
\[
\overline{Y^\prime}^{/Y} = \overline{Y^\prime}\times_{\overline{Y}} Y\ .
\]
We will need the following properties.

\begin{proposition}\label{prop:compactificationprop} Let $Y$ be a separated v-sheaf.
\begin{altenumerate}
\item The v-sheaf $\overline{Y}$ is partially proper.
\item If $Y$ is a small v-sheaf, then $\overline{Y}$ is a small v-sheaf.
\item If $Y$ is a diamond, then $\overline{Y}$ is a diamond.
\item If $Y=\Spa(R,R^+)$ is an affinoid perfectoid space, then $\overline{Y}=\Spa(R,(R^+)^\prime)$, where $(R^+)^\prime\subset R^+$ is the smallest open and integrally closed subring (which is the integral closure of $\mathbb F_p+R^{\circ\circ}$).
\item The functor $Y\mapsto \overline{Y}$ commutes with all limits.
\item If $f: Y^\prime\to Y$ is a surjective map of separated v-sheaves, then $\overline{f}: \overline{Y^\prime}\to \overline{Y}$ is a surjective map of v-sheaves.
\item If $f: Y^\prime\to Y$ is a quasicompact map of separated v-sheaves, then $\overline{f}: \overline{Y^\prime}\to \overline{Y}$ is a proper map of v-sheaves.
\item If $f: Y^\prime\to Y$ is a quasi-pro-\'etale map of separated v-sheaves, then $\overline{f}: \overline{Y^\prime}\to \overline{Y}$ is quasi-pro-\'etale.
\end{altenumerate}
\end{proposition}

\begin{proof} Part (iv) follows from the equations
\[
\Hom((R,(R^+)^\prime),(S,S^+)) = \Hom(R,S) = \Hom((R,R^+),(S,S^\circ))\ ,
\]
which hold for any pair $(S,S^+)$. Part (v) is clear from the definition. For part (vi), consider a totally disconnected space $X=\Spa(R,R^+)$ with a map $X\to \overline{Y}$, which corresponds to a map $\Spa(R,R^\circ)\to Y$. Then, after replacing $X$ by a v-cover, we can lift this to a map $\Spa(R,R^\circ)\to Y^\prime$, which corresponds to a map $X\to \overline{Y^\prime}$ lifting $X\to \overline{Y}$.

For part (i), we check first that $\overline{Y}\to \ast$ is quasiseparated. If $Z=\mathrm{Spa}(\mathbb F_p((t^{1/p^\infty})))$ and $X_1,X_2\to \overline{Y}\times Z$ are two affinoid perfectoid spaces mapping to $\overline{Y}\times Z$, we have to show that $X_1\times_{\overline{Y}\times Z} X_2$ is quasicompact. But using (v) (and $Z=\overline{Z}$), we have
\[
X_1\times_{\overline{Y}\times Z} X_2 = X_1\times_{\overline{X_1}} \overline{X_1\times_{Y\times Z} X_2}\times_{\overline{X_2}} X_2\ ,
\]
where the maps $X_i\to \overline{X_i}$ are quasicompact by (iv), and $\overline{X_1\times_{Y\times Z} X_2}$ is quasicompact, using (vi) and the assumption that $X_1\times_{Y\times Z} X_2$ is quasicompact (as $Y\to \ast$ is quasiseparated, i.e.~$Y\times Z$ is quasiseparated). Now, since $\overline{Y}\to \ast$ is quasiseparated, it follows that $\overline{Y}$ is partially proper, as it satisfies the desired valuative criterion by construction (for totally disconnected $\Spa(R,R^+)$, which gives the general case by v-descent); this finishes the proof of (i).

Part (ii) follows directly from (iv) and (vi). For part (vii), note that $\overline{f}$ is always partially proper by (i); to show that $\overline{f}$ is proper, it remains to see that $\overline{f}$ is quasicompact. But this follows from (iv), (v) and (vi) (similarly to the argument that $\overline{Y}$ is quasicompact).

For part (viii), we need to see that for any strictly totally disconnected space $X\to \overline{Y}$, $\overline{Y^\prime}\times_{\overline{Y}} X\to X$ is quasi-pro-\'etale. Let $X=\Spa(R,R^+)$ and $X^\circ=\Spa(R,R^\circ)$. Then the map $X\to \overline{Y}$ corresponds to a map $X^\circ\to Y$. We get separated quasi-pro-\'etale maps $Y^\prime\times_Y X^\circ\to X^\circ\to X$; let $g: Z:=Y^\prime\times_Y X^\circ\to X$ denote their composite. Then $\overline{Y^\prime}\times_{\overline{Y}} X\to X$ can be identified with $\overline{Z}^{/X}\to X$. But recall that by Corollary~\ref{cor:proetaleoverwlocaltop}, $Z\to X$ has a canonical factorization $Z\to X\times_{\pi_0 X} \pi_0 Z\to X$. It is then easy to see that $\overline{Z}^{/X} = X\times_{\pi_0 X}\pi_0 Z$, which is quasi-pro-\'etale over $X$.

Finally, part (iii) follows from (iv), (vi) and (viii).
\end{proof}

Coming back to the setting of Proposition~\ref{prop:cancomp}, there is the following relative version.

\begin{corollary}\label{cor:cancomprelative} Let $Y^\prime\to Y$ be a separated map of v-stacks.
\begin{altenumerate}
\item The map of v-stacks $\overline{Y^\prime}^{/Y}\to Y$ is partially proper.
\item If $Y$ and $Y^\prime$ are small v-stacks, then $\overline{Y^\prime}^{/Y}$ is a small v-stack.
\item If $Y$ and $Y^\prime$ are diamonds, then $\overline{Y^\prime}^{/Y}$ is a diamond.
\item The functor $Y^\prime\mapsto \overline{Y^\prime}^{/Y}$ from v-sheaves over $Y$ to v-sheaves over $Y$ commutes with all limits.
\item If $f: Y^\prime_2\to Y_1^\prime$ is a surjective map of separated v-stacks over $Y$, then $\overline{f}: \overline{Y_2^\prime}^{/Y}\to \overline{Y_1^\prime}^{/Y}$ is a surjective map of v-stacks.
\item If $f: Y^\prime_2\to Y_1^\prime$ is a quasicompact map of separated v-stacks over $Y$, then $\overline{f}: \overline{Y_2^\prime}^{/Y}\to \overline{Y_1^\prime}^{/Y}$ is a proper map of v-stacks.
\item If $f: Y^\prime_2\to Y_1^\prime$ is a quasi-pro-\'etale map of separated v-stacks over $Y$, then $\overline{f}: \overline{Y_2^\prime}^{/Y}\to \overline{Y_1^\prime}^{/Y}$ is quasi-pro-\'etale.
\end{altenumerate}
\end{corollary}

\begin{proof} If $Y$ is separated (and thus a sheaf), then also $Y^\prime$ is separated (and a sheaf), and the results follow from Proposition~\ref{prop:compactificationprop} and the formula $\overline{Y^\prime}^{/Y} = \overline{Y^\prime}\times_{\overline{Y}} Y$. In general, the formation of $\overline{Y^\prime}^{/Y}$ commutes with base change in $Y$, so all properties reduce to the case where $Y$ is affinoid perfectoid (and in particular separated).
\end{proof}

In the case of sheaves, another characterization of partially proper maps is given by the following proposition.

\begin{proposition}\label{prop:partpropfilcolimprop} Let $f: Y^\prime\to Y$ be a map of v-sheaves. Then $f$ is partially proper if and only if $f$ can be written as a (possibly large) filtered colimit of proper maps $f_i: Y^\prime_i\to Y$ along closed immersions $Y_i^\prime\to Y_j^\prime$. If $Y^\prime$ is small, then the filtered colimit is also small.
\end{proposition}

\begin{proof} It follows from Proposition~\ref{prop:valcritvsheafprop} that if $f$ can be written as such a filtered colimit of proper maps, then $f$ is partially proper. Conversely, assume that $f$ is partially proper, and consider the category of all closed immersions $Y^\prime_i\subset Y^\prime$ for which the composite $Y^\prime_i\to Y^\prime\to Y$ is proper. If $Y^\prime$ is small, this category is small. Any map in this category is a closed immersion. To finish the proof, we need to see that it is filtered, and that $Y^\prime$ is the colimit of all those $Y^\prime_i$.

First, we claim that if $Z\to Y$ is any proper map of v-sheaves and $Z\to Y^\prime$ is a map over $Y$, then the (sheaf-theoretic) image $Z^\prime$ of $Z$ in $Y^\prime$ is proper over $Y$, and closed in $Y^\prime$. (In this step, we only use that $f$ is separated.) Indeed, $Z^\prime$ is separated as it is a subsheaf of $Y^\prime$, and quasicompact as a quotient of $Z$, and the valuative criterion for $Z^\prime$ follows from that for $Z$. That the image is closed follows from the definition of properness, as it can be written as the image of the graph of $Z\to Y^\prime$, which is a closed subset of $Z\times_Y Y^\prime$, under the projection $Z\times_Y Y^\prime\to Y^\prime$.

In particular, if $Y^\prime_i\subset Y^\prime$, $i=1,2$, are two closed subsets which are proper over $Y$, then the image of $Y^\prime_1\sqcup Y^\prime_2$ is another such subset, showing that the category is filtered. On the other hand, if $Z=\Spa(R,R^+)\to Y^\prime$ is any map from an affinoid perfectoid space, then the map is separated (as $Z$ is separated), and so we have the canonical compactification $\overline{Z}^{/Y}$, which is proper over $Y$ by Corollary~\ref{cor:cancomprelative}~(vi). Now the image of $\overline{Z}^{/Y}\to Y^\prime$ is closed in $Y^\prime$ and proper over $Y$, and the map $Z\to Y^\prime$ factors over it (cf.~Proposition~\ref{prop:cancomp}), finishing the proof.
\end{proof}

In the case of locally spatial v-sheaves, there is another description of partially proper maps.

\begin{proposition}\label{prop:partiallyproperlocallyspatial} Let $f: Y^\prime\to Y$ be a map from a locally spatial v-sheaf $Y^\prime$ to a spatial v-sheaf $Y$. Then $f$ is partially proper if and only if $|Y^\prime|$ is taut, and for every perfectoid field $K$ with an open and bounded valuation subring $K^+\subset K$, and any diagram
\[\xymatrix{
\Spa(K,\OO_K)\ar[d]\ar[r] & Y^\prime\ar^f[d] \\
\Spa(K,K^+)\ar[r]\ar@{-->}[ur] & Y\ ,
}\]
there exists a unique dotted arrow making the diagram commute.
\end{proposition}

Recall that a locally spectral topological space $T$ is {\it taut} if it is quasiseparated, and for every quasicompact open subset $U\subset T$, the closure $\bar{U}\subset T$ of $U$ is quasicompact, cf.~\cite[Definition 5.1.2]{Huber}.

\begin{proof} In one direction, we have to see that if $f$ is partially proper, then $|Y^\prime|$ is taut. Fix a surjection of v-sheaves $\Spa(S,S^+)\to Y$ from an affinoid perfectoid space. Let $|U|\subset |Y^\prime|$ be a quasicompact open subspace, which necessarily comes from some quasicompact open sub-v-sheaf $U\subset Y^\prime$. Choose an affinoid perfectoid space $\Spa(R,R^+)$ with a surjective map of v-sheaves $\Spa(R,R^+)\to U\times_Y \Spa(S,S^+)$. Let $(R^+)^\prime\subset R$ be the smallest open and integrally closed subring containing the image of $S^+\to R^+$. As $f$ is partially proper, we get a unique map $\Spa(R,(R^+)^\prime)\to Y^\prime\times_Y \Spa(S,S^+)$ extending $\Spa(R,R^+)\to U\times_Y \Spa(S,S^+)$. We claim that the image of $|\Spa(R,(R^+)^\prime)|\to |Y^\prime|$ is given by $\bar{U}$, which is thus quasicompact. Note that as $\Spa(R,R^+)$ is dense in $\Spa(R,(R^+)^\prime)$, the image is contained in $\bar{U}$. On the other hand, if $x\in \bar{U}$ is any point, then it corresponds to a map $\Spa(K,K^+)\to Y^\prime$ such that $\Spa(K,\OO_K)\subset \Spa(K,K^+)\to Y^\prime$ factors over $U$ (as $\bar{U}$ is the set of specializations of points of $U$). One can lift the image of $x$ in $Y$ to $\Spa(S,S^+)$. After enlarging $K$, we can find a map $(R,R^+)\to (K,\OO_K)$ inducing the given $(K,\OO_K)$-point of $U\times_Y \Spa(S,S^+)$. Then we get a map $(R,(R^+)^\prime)\to (K,K^+)$, so that $x$ lies in the image of $|\Spa(R^+,(R^+)^\prime)|\to |Y^\prime|$, as desired.

For the converse, note that if $|Y^\prime|$ is taut, we can write it as an increasing union of quasicompact closed generalizing subsets (along closed immersions): Indeed, write $|Y^\prime|=\bigcup_{i\in I} U_i$ as an increasing union of quasicompact open subsets $U_i$; then their closures $\bar{U_i}\subset |Y^\prime|$ form an increasing union of quasicompact closed generalizing subsets. Now each $\bar{U_i}\subset |Y^\prime|$ corresponds to a spatial sub-v-sheaf $Y^\prime_i\subset Y^\prime$, which still satisfies the valuative criterion for pairs $(K,K^+)$. It follows from Proposition~\ref{prop:valcritvsheafprop} that $Y^\prime_i\to Y^\prime$ is proper, so that $Y^\prime\to Y$ satisfies the criterion of Proposition~\ref{prop:partpropfilcolimprop}.
\end{proof}

\section{Proper base change}

In this section, we prove the proper base change theorem for diamonds. It states a certain commutation between pushforward under proper maps, and the functor $j_!$ for open morphisms $j$. Let us first introduce the latter functor, in the generality of \'etale maps.

\begin{defprop}\label{def:Rfshrieketale} Let $f: Y^\prime\to Y$ be an \'etale morphism of small v-stacks. Then $f^\ast: D_\et(Y,\Lambda)\to D_\et(Y^\prime,\Lambda)$ has a left adjoint $Rf_!: D_\et(Y^\prime,\Lambda)\to D_\et(Y,\Lambda)$. Moreover, for any map $g: \tilde{Y}\to Y$ with pullback $\tilde{f}: \tilde{Y}^\prime = Y^\prime\times_Y \tilde{Y}\to \tilde{Y}$, $g^\prime: \tilde{Y}^\prime\to Y^\prime$, the natural transformation
\[
R\tilde{f}_! g^{\prime\ast}\to g^\ast Rf_!
\]
of functors $D_\et(Y^\prime,\Lambda)\to D_\et(\tilde{Y},\Lambda)$, adjoint to
\[
g^{\prime\ast}\to g^{\prime\ast} f^\ast Rf_! = \tilde{f}^\ast g^\ast Rf_!\ ,
\]
is an equivalence.
\end{defprop}

As $Rf_!$ commutes with canonical truncations, we will often denote it simply by $f_!: D_\et(Y^\prime,\Lambda)\to D_\et(Y,\Lambda)$.

\begin{proof} We start with the case that $Y$ is a perfectoid space. Then $Y^\prime$ is also a perfectoid space. In that case, $Y^\prime_\et$ is a slice of $Y_\et$, and thus the functor $f^\ast$ on \'etale sheaves has an exact left adjoint $f_!$. This induces a corresponding functor on (non-left-completed) derived categories, which is left adjoint to $f^\ast$ as a functor on non-left-completed derived categories; moreover, both $f^\ast$ and $Rf_!$ commute with canonical truncations. This implies that one gets a similar adjunction on left-completions. To check that $Rf_!$ commutes with base change along maps of perfectoid spaces, use that all functors commute with canonical truncations, so it is enough to prove the assertion if $\mathcal F$ is an \'etale sheaf on $Y^\prime$. Moreover, all functors are defined on all sheaves (not just sheaves of abelian groups), and the adjunction exists in that setting; as all functors commute with all colimits, one reduces to the case of the sheaf represented by some \'etale map $Z\to Y^\prime$. In that case, $f_!$ is given by the sheaf represented by the composite \'etale map $Z\to Y^\prime\to Y$. It is clear from this description that $f_!$ commutes with base change.

In general, as $Y$ is a small v-stack, we can find a simplicial v-hypercover of $Y$ by a simplicial perfectoid space $Y_\bullet$; for simplicity, we assume that all $Y_i$ are strictly totally disconnected. Let $f_\bullet: Y^\prime_\bullet\to Y_\bullet$ be the pullback of $Y^\prime\to Y$, which is again a simplicial perfectoid space. We have the simplicial topos of \'etale sheaves on $Y_\bullet$ (and $Y^\prime_\bullet$), roughly given as systems of sheaves $\mathcal F_i$ on $Y_i$ together with maps $g^\ast \mathcal F_i\to \mathcal F_j$ for all maps $g: \Delta^j\to \Delta^i$. Correspondingly, we have the derived category $D(Y_{\bullet,\et},\Lambda)$ of sheaves of $\Lambda$-modules on the simplicial site $Y_{\bullet,\et}$, and its full subcategory $D_\cart(Y_{\bullet,\et},\Lambda)$ consisting of those objects for which the maps $g^\ast A_i\to A_j$ are all equivalences. Then pullback along $Y_\bullet\to Y$ induces an equivalence
\[
D_\et(Y,\Lambda)\cong D_\cart(Y_{\bullet,\et},\Lambda)\ .
\]
Similarly,
\[
D_\et(Y^\prime,\Lambda)\cong D_\cart(Y^\prime_{\bullet,\et},\Lambda)\ .
\]

The functor $f_\bullet^\ast: D(Y_{\bullet,\et},\Lambda)\to D(Y^\prime_{\bullet,\et},\Lambda)$ has a left adjoint given by $Rf_{\bullet !}$ (which in every degree is given by $Rf_{i!}$). By the base change result, this carries $D_\cart(Y_{\bullet,\et},\Lambda)$ into $D_\cart(Y^\prime_{\bullet,\et},\Lambda)$, and thus gives the desired left adjoint $Rf_!$. Compatibility with base change follows for pullback along $Y_0\to Y$ by construction, and this implies the general case.
\end{proof}

Now we can state the proper base change theorem.

\begin{theorem}\label{thm:properbasechangepartpropvsopen} Let $f: Y^\prime\to Y$ be a proper morphism of small v-stacks, and let $j: U\subset Y$ be an open immersion, with pullback $g: U^\prime=V\times_Y Y^\prime\to U$, and $j^\prime: U^\prime\subset Y^\prime$. There is a natural transformation of functors
\[
j_! Rg_\ast\to Rf_\ast j^\prime_!: D_\et(U^\prime,\Lambda)\to D_\et(Y,\Lambda)\ .
\]
If $f$ is quasi-pro-\'etale or $n\Lambda=0$ for some $n$ prime to $p$, then for any $A\in D^+_\et(U^\prime,\Lambda)$, the map
\[
j_! Rg_\ast A\to Rf_\ast j^\prime_! A
\]
is an isomorphism.

If $Rf_\ast$ has finite cohomological dimension, i.e.~there is some integer $N$ such that for all $A\in D_\et(Y^\prime,\Lambda)$ concentrated in degree $0$, one has $R^if_\ast A=0$ for $i>N$, then for all $A\in D_\et(U^\prime,\Lambda)$, the map
\[
j_! Rg_\ast A\to Rf_\ast j^\prime_! A
\]
is an isomorphism.
\end{theorem}

\begin{remark} This does not quite look like a base change theorem. However, morally it is equivalent to the statement $Rf_!$ commutes with pullback to the complementary closed subset $Y\setminus U$. However, in general $Y\setminus U\subset Y$ is not generalizing, and so does not correspond to a closed sub-v-sheaf. In \cite{Huber}, Huber defines a notion of pseudo-adic spaces which allows one to treat such subsets as spaces in their own right. In this language, the theorem becomes a base change theorem. We have decided not to introduce an analogue of pseudo-adic spaces in our setup, and so leave the statement of Theorem~\ref{thm:properbasechangepartpropvsopen} as it is.
\end{remark}

\begin{proof} The natural transformation $j_! Rg_\ast\to Rf_\ast j^\prime_!$ is adjoint to
\[
f^\ast j_! Rg_\ast = j^\prime_! g^\ast Rg_\ast\to j^\prime_!\ ,
\]
using base change for $j_!$ in Proposition~\ref{def:Rfshrieketale}.

First, we note that if $Rf_\ast$ has finite cohomological dimension, then the general case reduces to the case $A\in D^+_\et(U^\prime,\Lambda)$. Indeed, in that case also $Rg_\ast = j^\ast Rf_\ast j^\prime_!$ has finite cohomological dimension. Thus, any cohomology sheaf of both sides of
\[
j_! Rg_\ast A\to Rf_\ast j^\prime_! A
\]
agrees with those of
\[
j_! Rg_\ast \tau^{\geq -n} A\to Rf_\ast j^\prime_! \tau^{\geq -n} A
\]
for $n$ sufficiently large.

Now let $A\in D^+_\et(U^\prime,\Lambda)$; we want to show that the map
\[
j_! Rg_\ast A\to Rf_\ast j^\prime_! A
\]
is an isomorphism. By Proposition~\ref{prop:Rfastsimple} and Proposition~\ref{def:Rfshrieketale}, we can assume that $Y=X$ is a strictly totally disconnected perfectoid space. The map
\[
j_! Rg_\ast A\to Rf_\ast j^\prime_! A
\]
becomes a map in $D^+_\et(X,\Lambda) = D^+(X_\et,\Lambda)$. To check that it is an isomorphism, we can check on stalks, i.e.~after pullback to maps $\Spa(C,C^+)\to X$, where $C$ is algebraically closed and $C^+\subset C$ is an open and bounded valuation subring. Thus, we can assume that $X=\Spa(C,C^+)$, and we need only check the statement on global sections. If $U=X$, the result is clear. Otherwise, $j_! Rg_\ast A$ has trivial global sections, so we need to see that
\[
R\Gamma(X,Rf_\ast j^\prime_! A)=0\ .
\]
But this is given by
\[
R\Gamma(X,Rf_\ast j^\prime_! A) = R\Gamma(Y^\prime,j^\prime_! A)\ .
\]
If $f$ is quasi-pro-\'etale, then $Y^\prime$ is a qcqs perfectoid space proper and pro-\'etale over $X=\Spa(C,C^+)$, i.e.~$Y^\prime = \Spa(C,C^+)\times \underline{S}$ for some profinite set $S$. In that case, the result is clear. Thus, from now on assume that $n\Lambda=0$ for some $n$ prime to $p$.

Let $X^\prime$ be a strictly totally disconnected space with a surjection $X^\prime\to Y^\prime$. By Corollary~\ref{cor:cancomprelative}~(vi), the canonical compactification $\overline{X^\prime}^{/X}$ is proper over $X$, and $\overline{X^\prime}^{/X}=\overline{X^\prime}\times_{\overline{X}} X$ is an affinoid perfectoid space by Proposition~\ref{prop:compactificationprop}~(iv). There is a unique extension of $X^\prime\to Y^\prime$ to $\overline{X^\prime}^{/X}\to Y^\prime$ by Proposition~\ref{prop:cancomp}. The fibre product $\overline{X^\prime}^{/X}\times_{Y^\prime} \overline{X^\prime}^{/X}$ is again proper over $X$; continuing, we can produce a v-hypercover $Y^\prime_\bullet\to Y^\prime$ where each $Y^\prime_i=\overline{X^\prime_i}^{/X}$ is the canonical compactification of some strictly totally disconnected $X^\prime_i\to X$. Then, by unbounded cohomological descent in a replete topos, cf.~\cite[Proposition 3.3.6]{BhattScholze},
\[
R\Gamma(Y^\prime,j^\prime_! A)
\]
is the (derived) limit of the simplicial object
\[
R\Gamma(\overline{X^\prime_i}^{/X},j^\prime_! A)\ .
\]
Thus, it is enough to handle the case that $Y^\prime=\overline{X^\prime}^{/X}$ is the canonical compactification of some strictly totally disconnected $X^\prime\to X$. In particular, in this case $Y^\prime$ is an affinoid perfectoid space, and $A\in D^+_\et(Y^\prime,\Lambda) = D^+(Y^\prime_\et,\Lambda)$. Let $X^\prime\to \pi_0 X^\prime$ be the projection; this extends to a projection $Y^\prime\to \pi_0 X^\prime=\pi_0 Y^\prime$. Computing $R\Gamma(Y^\prime,-)$ via a Leray spectral sequence along $Y^\prime\to \pi_0 Y^\prime$, it is enough to check that the fibres of the pushforward vanish; but these are given by the cohomology of the fibers. Thus, we can assume that $Y^\prime$ is connected. In other words, $X^\prime=\Spa(C^\prime,C^{\prime +})$ for some complete algebraically closed field $C^\prime$ (over $C$), and some open and bounded valuation subring $C^{\prime +}\subset C^\prime$. Then $Y^\prime=\overline{X^\prime}^{/X}$ is given by
\[
Y^\prime = \Spa(C^\prime,C^{\prime\circ\circ} + C^+)\ .
\]
Now $Y^\prime$ has only one rank-$1$-point $\Spa(C^\prime,\OO_{C^\prime})$, where $C^\prime$ is algebraically closed; it follows that any map in $Y^\prime_\et$ is a local isomorphism, and thus the topos $(Y^\prime_\et)^\sim$ is equivalent to the topos $|Y^\prime|^\sim$. Let $K^\prime$ resp.~$K$ be the (algebraically closed) residue field of $C^\prime$ resp.~$C$, and let $V\subset K$ correspond to $C^+\subset \OO_C$, i.e.~$V=C^+/C^{\circ\circ}\subset K=\OO_C/C^{\circ\circ}$. Then $|Y^\prime|=|\Spa(K^\prime,V)|$. Thus, we are reduced to the following lemma about the Zariski--Riemann spaces of algebraically closed fields.
\end{proof}

The following lemma is related to results of Huber, \cite{HuberZariskiRiemann}.

\begin{lemma}\label{lem:properbasechangezariskiriemann} Let $K\subset K^\prime$ be algebraically closed fields, and let $V\subset K$ be a valuation ring of $K$. Let $T^\prime=\Spa(K^\prime,V)$ be the spectral space of all valuation rings $V^\prime\subset K^\prime$ that contain $V$, and let $T=\Spa(K,V)$, with the natural projection map $f: T^\prime\to T$ sending $V^\prime\subset K^\prime$ to $V^\prime\cap K\subset K$. Let $s\in T$ denote the unique closed point, corresponding to the valuation ring $V\subset K$. Let $\mathcal F$ be a sheaf of torsion abelian groups on $T^\prime$ such that for all $x^\prime\in f^{-1}(s)$, the stalk $\mathcal F_{x^\prime} = 0$. Then
\[
R\Gamma(T^\prime,\mathcal F)=0\ .
\]
\end{lemma}

\begin{proof} Consider the category $C$ of all proper $V$-schemes $X$ with a point $x\in X(K^\prime)$. This category is cofiltered: Given such $(X,x)$ and $(X^\prime,x^\prime)$, the product $X\times_V X^\prime$ with product point $(x,x^\prime)$ dominates both, and if $f,g: X\to X^\prime$ are two morphisms mapping $x$ to $x^\prime$, then the equalizer of $f$ and $g$ is also proper over $V$ and contains $x$. Then $T^\prime=\varprojlim_C |X|$; in fact, one could replace $C$ by the subcategory of integral and projective $X$ for which the point $x$ is dominant, cf.~\cite[Lemma 2.1]{HuberZariskiRiemann}. Moreover, in the limit, the topoi $\varprojlim_C |X|^\sim$ and $\varprojlim_C X_\et^\sim$ are equivalent (as any \'etale map becomes a local isomorphism after some pullback), cf.~\cite[Lemma 2.4]{HuberZariskiRiemann}.

We may assume that $\mathcal F$ is constructible. Then it comes via pullback from some $\mathcal F_X$ on $X_\et^\sim$ (with trivial restriction to the special fiber) for some $X\in C$, and the result follows from the proper base change theorem in \'etale cohomology applied to all $X^\prime\in C_{/X}$, noting that
\[
R\Gamma(T^\prime,\mathcal F) = \varinjlim_{X^\prime\in C_{/X}} R\Gamma(X^\prime_\et,\mathcal F_{X^\prime})\ ,
\]
where $\mathcal F_{X^\prime}$ denotes the pullback of $\mathcal F_X$ to $X^\prime$.
\end{proof}

As a first application of the proper base change theorem, we can complete the proof of invariance under change of algebraically closed base field. Let us recall the statement from the introduction.

\begin{theorem}\label{thm:changebasefield} Let $Y$ be a small v-stack, and assume that $n\Lambda=0$ for some $n$ prime to $p$.
\begin{altenumerate}
\item Assume that $Y$ lives over $k$, where $k$ is a discrete algebraically closed field of characteristic $p$, and $k^\prime/k$ is an extension of discrete algebraically closed base fields, $Y^\prime = Y\times_k k^\prime$. Then the pullback functor
\[
D_\et(Y,\Lambda)\to D_\et(Y^\prime,\Lambda)
\]
is fully faithful.
\item Assume that $Y$ lives over $k$, where $k$ is an algebraically closed discrete field of characteristic $p$. Let $C/k$ be an algebraically closed complete nonarchimedean field, and $Y^\prime = Y\times_k \Spa(C,C^+)$ for some open and bounded valuation subring $C^+\subset C$ containing $k$. Then the pullback functor
\[
D_\et(Y,\Lambda)\to D_\et(Y^\prime,\Lambda)
\]
is fully faithful.
\item Assume that $Y$ lives over $\Spa(C,C^+)$, where $C$ is an algebraically closed complete nonarchimedean field with an open and bounded valuation subring $C^+\subset C$, $C^\prime/C$ is an extension of algebraically closed complete nonarchimedean fields, and $C^{\prime +}\subset C^\prime$ an open and bounded valuation subring containing $C^+$, such that $\Spa(C^\prime,C^{\prime +})\to \Spa(C,C^+)$ is surjective. Then for $Y^\prime = Y\times_{\Spa(C,C^+)} \Spa(C^\prime,C^{\prime +})$, the pullback functor
\[
D_\et(Y,\Lambda)\to D_\et(Y^\prime,\Lambda)
\]
is fully faithful.
\end{altenumerate}
\end{theorem}

\begin{proof} We note that (i) follows from (ii) and (iii). Moreover, (ii) follows from (iii) and the restricted version of (ii) where we demand that $C$ is the completed algebraic closure of $k((t))$.

Let us start by proving (iii). Let $f: Y^\prime\to Y$ be the map, which is qcqs. We need to see that the adjunction map
\[
A\to Rf_\ast f^\ast A
\]
is an isomorphism for all $A\in D_\et(X,\Lambda)$. First, using Postnikov-completeness of all intervening categories (and commutation of $f^\ast$ with Postnikov truncations), we note that by writing $A$ as the homotopy limit of $\tau^{\geq -n} A$, it is enough to prove this for $A\in D^+_\et(X,\Lambda)$. In that case, $Rf_\ast = Rf_{v\ast}$ commutes with any base change by Proposition~\ref{prop:Rfastsimple}. Thus, we may assume that $Y=X$ is a strictly totally disconnected perfectoid space. Now the statement follows from (the final paragraph of) Theorem~\ref{thm:vproetbasechange}.

It remains to prove (ii) in the case that $C$ is the completed algebraic closure of $k((t))$. Let $Y_\bullet\to Y$ be a simplicial v-hypercover of $Y$ by disjoint unions of strictly totally disconnected spaces $Y_i$; then $Y^\prime_i = Y^\prime\times_Y Y_i$ is a perfectoid space. Under the identifications
\[
D_\et(Y,\Lambda)\simeq D_{\et,\cart}(Y_\bullet,\Lambda)\ ,\ D_\et(Y^\prime,\Lambda)\simeq D_{\et,\cart}(Y^\prime_\bullet,\Lambda)\ ,
\]
it is enough to prove that
\[
D_\et(Y_i,\Lambda)\to D_\et(Y^\prime_i,\Lambda)
\]
is fully faithful for all $i$. In other words, we can assume that $Y=\Spa(A,A^+)$ is a strictly totally disconnected space.

Now fix a pseudouniformizer $\varpi\in A$. In that case, $Y^\prime$ is the increasing union of the affinoid perfectoid spaces
\[
Y^\prime_n = \{|t|^n\leq |\varpi|\leq |t|^{1/n}\}\subset Y^\prime\ .
\]
As $R\Hom$'s in $D_\et(Y^\prime,\Lambda)$ are the derived limit of $R\Hom$'s in $D_\et(Y^\prime_n,\Lambda)$, it is enough to prove that for all $n\geq 1$, the functor
\[
D_\et(Y,\Lambda)\to D_\et(Y^\prime_n,\Lambda)
\]
is fully faithful. Let $f_n: Y^\prime_n\to Y$ be the map of affinoid perfectoid spaces. We need to see that the adjunction map
\[
A\to Rf_{n\ast} f_n^\ast A
\]
is an equivalence for all $A\in D_\et(Y,\Lambda)$. Again, this reduces to the case $A\in D^+_\et(Y,\Lambda)$. The desired statement can be checked on stalks, so we can assume that $Y=\Spa(C^\prime,C^{\prime +})$, where $C^\prime$ is an algebraically closed nonarchimedean field with an open and bounded valuation subring $C^{\prime +}\subset C^\prime$, and we only need to verify the statement on global sections. Thus, for any $A\in D_\et(Y,\Lambda)$, we need to see that
\[
R\Gamma(Y,A) = R\Gamma(Y^\prime_n,f_n^\ast A)\ .
\]
By Theorem~\ref{thm:properbasechangepartpropvsopen}, both sides vanish if $A=j_! A_0$ for some $A_0\in D^+_\et(U,\Lambda)$, where $U=Y\setminus \{s\}\subset Y$ is the complement of the closed point $s\in Y$. Thus, replacing $A$ by the cone of $j_! j^\ast A\to A$, we can assume that $A$ is concentrated at the closed point $s$. Repeating the argument in the other direction, we can assume that $A$ is constant. We can also assume that $\Lambda=\mathbb Z/\ell^n\mathbb Z$, and then by triangles that $\Lambda=\Fl$. In that case, $A$ is a direct sum of shifted copies of $\Fl$ (but bounded below). This finally reduces us to the case $A=\Fl$.

Thus, it remains to see that
\[
R\Gamma(Y^\prime_n,\Fl)=\Fl\ .
\]
But note that $Y^\prime_n$ is the inverse limit over all finite extensions $L\subset C$ of $k((t^{1/p^\infty}))$ of the system of affinoid perfectoid spaces given by
\[
Y^\prime_{n,L} = \{|t|^n\leq |\varpi|\leq |t|^{1/n}\}\subset \Spa(L,\OO_L)\times_k \Spa(C^\prime,C^{\prime +})\ ,
\]
which implies
\[
H^i(Y^\prime_n,\Fl) = \varinjlim H^i(Y^\prime_{n,L},\Fl)
\]
for all $i\geq 0$ by Proposition~\ref{prop:etcohomlim}. On the other hand, all $L$ are isomorphic to $k((t_L^{1/p^\infty}))$ for some pseudo-uniformizer $t_L$, as one can take for $t_L$ a uniformizer of the finite separable extension of $k((t))$ corresponding to $L$ (noting that finite separable extensions of $k((t))$ are equivalent to finite extensions of $k((t^{1/p^\infty}))$). Thus, $Y^\prime_{n,L}$ is an annulus over $\Spa(C,C^+)$ (the defining inequalities can also be written as $\{|\varpi|^n\leq |t|\leq |\varpi|^{1/n}\}$), which implies that
\[
H^i(Y^\prime_{n,L},\Fl) = \left\{\begin{array}{cc} \Fl & i=0,1\\ 0 &\mathrm{else.} \end{array}\right .
\]
Taking the direct limit over all $L$ kills the class in degree $1$ by extracting $\ell$-power roots of the uniformizer; this finishes the proof.
\end{proof}

Before going on and using the proper base change theorem to define the functor $Rf_!$, we pause to obtain certain criteria guaranteeing finite cohomological dimension for $Rf_\ast$.

\section{Constructible sheaves}

It will be convenient to have a general notion of constructible sheaves. This notion is slightly subtle as locally closed subsets of adic spaces are not adic spaces themselves. This is one reason that Huber considers pseudo-adic spaces in \cite{Huber}. We will work around this issue.

\begin{definition}\label{def:constructible} Let $\Lambda$ be a noetherian ring.
\begin{altenumerate}
\item Let $X$ be a strictly totally disconnected perfectoid space, and identify the topos of \'etale sheaves on $X$ with the topos of sheaves on $|X|$. A sheaf $\mathcal F$ of $\Lambda$-modules on $X_\et$ (equivalently, on $|X|$) is constructible if there is stratification of $X$ into constructible locally closed subsets $S_i\subset |X|$ such that $\mathcal F|_{S_i}$ is the constant sheaf on $S_i$ associated with some finitely generated $\Lambda$-module.
\item Let $Y$ be a small v-stack, and $\mathcal F$ a small sheaf of $\Lambda$-modules on $Y_v$. Then $\mathcal F$ is constructible if $\mathcal F\in D_\et(Y,\Lambda)\subset D(Y,\Lambda)$, and for every strictly totally disconnected space $f: X\to Y$, the pullback $f^\ast \mathcal F$ is constructible.
\end{altenumerate}
\end{definition}

\begin{remark} Note that in part (i), we made a switch from $X$ as a perfectoid space to $|X|$ as a mere topological space, which allowed us to restrict to the topological space $S_i$.
\end{remark}

\begin{remark}\label{rem:constrabelian} It is clear from the definition that the class of constructible sheaves is stable under kernels, cokernels, images, and extensions; in particular, constructible sheaves form an abelian category. Indeed, this reduces immediately to the case where $X$ is strictly totally disconnected, in which case it reduces further (by passing to suitably refined stratifications) to the case of a map of constant sheaves (associated with finitely generated $\Lambda$-modules) on a spectral space, where finally it reduces (on open and closed subsets) to the category of finitely generated $\Lambda$-modules.
\end{remark}

We will show that this property has good descent properties, and will give a better definition in the case of spatial diamonds. First, we analyze constructible sheaves on strictly totally disconnected perfectoid spaces $X$. In fact, this reduces to sheaves on the spectral space $|X|$, where we have the following general lemma.

\begin{lemma}\label{lem:constructiblecompact} Let $\Lambda$ be a noetherian ring, let $X$ be a spectral space, and let $\mathcal F$ be a sheaf of $\Lambda$-modules on $X$. Then $\mathcal F$ is constructible if and only if $\mathcal F$ is compact in the category of sheaves of $\Lambda$-modules on $X$, i.e.~$\Hom(\mathcal F,-)$ commutes with filtered colimits. Moreover, any sheaf of $\Lambda$-modules on $X$ can be written as a filtered colimit of constructible sheaves.
\end{lemma}

\begin{proof} First, we check that constructible sheaves are compact. For this, note that both $j_!$ for constructible open immersions $j$ and $i_\ast$ for constructible closed immersions $i$ preserve compact objects (as their right adjoints $j^\ast$ resp.~$i^!$ preserve filtered colimits). Passing to a filtration, this reduces the problem to the case that $\mathcal F$ is the constant sheaf associated with some finitely generated $\Lambda$-module, where the result follows by taking a finite free $2$-term resolution.

Next, we show that every sheaf of $\Lambda$-modules on $X$ can be written as a filtered colimit of constructible sheaves. Note that any sheaf $\mathcal F$ admits a surjection from a direct sum $\mathcal G$ of sheaves of the form $j_! \Lambda$, where $j: U\hookrightarrow X$ ranges over quasicompact open immersions into $X$. Applying the same to the kernel of $\mathcal G\to \mathcal F$, we get a $2$-term resolution
\[
\mathcal H\to \mathcal G\to \mathcal F\to 0\ ,
\]
where $\mathcal G$ and $\mathcal H$ are direct sums of sheaves of the form $j_! \Lambda$. The map $\mathcal H\to \mathcal G$ is then a filtered colimit of similar maps $\mathcal H_i\to \mathcal G_i$, where $\mathcal H_i$ and $\mathcal G_i$ are finite sums of sheaves of the form $j_! \Lambda$ (as such sheaves are compact). The cokernels $\mathcal F_i$ of the maps $\mathcal H_i\to \mathcal G_i$ are constructible, and $\mathcal F$ is their filtered colimit, as desired.

Now, if $\mathcal F$ is compact, then write $\mathcal F=\varinjlim_i \mathcal F_i$ as a filtered colimit of constructible sheaves. By compactness, one can factor the map $\mathcal F\to \varinjlim \mathcal F_i=\mathcal F$ over $\mathcal F_i$ for $i$ sufficiently large. This shows that $\mathcal F$ is a direct summand of $\mathcal F_i$, and in particular constructible itself.
\end{proof}

\begin{proposition}\label{prop:constrdescends} Let $\Lambda$ be a noetherian ring, let $f: \tilde{Y}\to Y$ be a surjective map of small v-stacks, and let $\mathcal F$ be a small sheaf of $\Lambda$-modules on $Y_v$. If $f^\ast \mathcal F$ is constructible, then $\mathcal F$ is constructible.
\end{proposition}

\begin{proof} We may assume that both $Y=X$ and $\tilde{Y}=\tilde{X}$ are strictly totally disconnected perfectoid spaces. In that case, it follows from Lemma~\ref{lem:constructiblecompact} as $\Hom_X(\mathcal F,-)$ can be written as the equalizer of
\[
\Hom_{\tilde{X}}(f^\ast \mathcal F,f^\ast-)\rightrightarrows \Hom_{\tilde{\tilde{X}}} (g^\ast \mathcal F,g^\ast-)\ ,
\]
where $\tilde{\tilde{X}}\to \tilde{X}\times_X \tilde{X}$ (with composite $g: \tilde{\tilde{X}}\to X$) is some strictly totally disconnected cover, and these functors commute with filtered colimits by assumption.
\end{proof}

\begin{proposition}\label{prop:charconstructible} Let $\Lambda$ be a noetherian ring, and let $Y$ be a spatial diamond. Let $\mathcal F$ be an \'etale sheaf of $\Lambda$-modules on $Y$. The following conditions are equivalent.
\begin{altenumerate}
\item The sheaf $\mathcal F$ is constructible.
\item The sheaf $\mathcal F$ is compact in the category of \'etale sheaves of $\Lambda$-modules on $Y$.
\item There is a stratification of $|Y|$ into constructible locally closed subsets $S_i\subset |Y|$ such that the restriction of $\mathcal F$ to $S_i$ satisfies the following condition: For any strictly totally disconnected perfectoid space $f: X\to Y$, the pullback $f^\ast \mathcal F|_{f^{-1}(S_i)}$ is the constant sheaf associated with some finitely generated $\Lambda$-module.
\end{altenumerate}

Moreover, any \'etale sheaf of $\Lambda$-modules on $Y$ can be written as a filtered colimit of constructible sheaves.
\end{proposition}

\begin{proof} First, (i) implies (ii), as compactness descends over a v-cover $X\to Y$ by a strictly totally disconnected perfectoid space, as in the proof of Proposition~\ref{prop:constrdescends}. Moreover, by definition (iii) implies (i).

It remains to see that (ii) implies (iii). Let us for the moment call a sheaf $\mathcal F$ satisfying the hypothesis of (iii) strongly constructible. Then strongly constructible sheaves are constructible. It now suffices to show that any \'etale sheaf of $\Lambda$-modules on $Y$ can be written as a filtered colimit of strongly constructible sheaves. Indeed, it will then follow that any compact $\mathcal F$ is a direct summand of a strongly constructible sheaf, as in the proof of Lemma~\ref{lem:constructiblecompact}, and therefore strongly constructible itself.

Thus, it remains to see that any \'etale sheaf of $\Lambda$-modules on $Y$ can be written as a filtered colimit of strongly constructible sheaves. As in Remark~\ref{rem:constrabelian}, the category of strongly constructible sheaves is closed under kernels, cokernels, images, and extensions. By Lemma~\ref{lem:spatialetalelocallysep}, any sheaf of $\Lambda$-modules on $Y_\et$ can be written as a quotient of a direct sum of sheaves of the form $j_! \Lambda$ where $j: U\to Y$ ranges over \'etale maps which can be written as a composite $U\hookrightarrow V\to W\hookrightarrow Y$, where $U\hookrightarrow V$ and $W\hookrightarrow Y$ are quasicompact open immersions, and $V\to W$ is finite \'etale. Arguing as in Lemma~\ref{lem:constructiblecompact}, it is enough to prove that $j_! \Lambda$ is strongly constructible.

For this, let us consider first the case that $Y$ is strictly totally disconnected perfectoid space $X$. Then $U$ can be embedded as a quasicompact open subset of $X\times \{1,\ldots,m\}$ for some $m$. Consider the function $a: |X|\to \mathbb Z_{\geq 0}$ which assigns to any point $x\in X$ the number of geometric points of $U$ above a fixed geometric point $\bar{x}$ over $x$. We claim that the level sets $a^{-1}(n)$ are constructible locally closed subsets of $|X|$ for all $n\geq 0$. As $a$ is bounded, it suffices to show that $a^{-1}(\mathbb Z_{\geq n})\subset |X|$ is a quasicompact open subset for all $n\geq 0$. But if $x\in X$ has $n$ preimages in $U$, then there are $n$ elements $i\in \{1,\ldots,m\}$ such that $(x,i)\in U\subset X\times\{1,\ldots,n\}$, and then $\bigcap_{i: (x,i)\in U} U\cap (X\times \{i\})\subset X$ is a quasicompact open subset of $X$ containing $x$ lying in $a^{-1}(\mathbb Z_{\geq n})$.

Let $S_i=a^{-1}(i)$ for $i\geq 0$; this is a constructible stratification of $X$. Moreover, the above arguments show that $(j_! \Lambda)|_{S_i}\cong \Lambda^i$ is constant. Indeed, for this we can replace $X$ by $a^{-1}(\mathbb N_{\geq i})$, and then locally $U$ is the disjoint union of $i$ copies of $X$, and some component that maps to $X\setminus S_i$. Thus, $j_! \Lambda$ is indeed strongly constructible.

In general, we can still consider the function $a: |Y|\to \mathbb Z_{\geq 0}$ which assigns to any point $y\in Y$ the number of geometric points of $U$ above a fixed geometric point $\bar{y}$ over $y$. The level sets $a^{-1}(n)$ are constructible locally closed subsets of $|Y|$ for all $n\geq 0$, by checking over a strictly totally disconnected cover. Let $S_i = a^{-1}(i)$ for $i\geq 0$. It remains to see that for all strictly totally disconnected perfectoid spaces $f: X\to Y$, the pullback $f^\ast(j_! \Lambda)|_{f^{-1}(S_i)}$ is the constant sheaf on $f^{-1}(S_i)$ associated with $\Lambda^i$; but this was verified above.
\end{proof}

\begin{proposition}\label{prop:constructiblelimit} Let $Y_i$, $i\in I$, be a cofiltered inverse system of spatial diamonds with inverse limit $Y=\varprojlim_i Y_i$, which is again a spatial diamond. Let $\Lambda$ be a noetherian ring, and denote by $\Cons(Y,\Lambda)$ (resp.~$\Cons(Y_i,\Lambda)$) the category of constructible \'etale sheaves of $\Lambda$-modules on $Y$ (resp.~$Y_i$). Then the natural functor
\[
\text{2-}\varinjlim_i \Cons(Y_i,\Lambda)\to \Cons(Y,\Lambda)
\]
is an equivalence of categories.
\end{proposition}

\begin{proof} First, we check fully faithfulness. For this, let $\mathcal F_{i_0}, \mathcal G_{i_0}\in \Cons(Y_{i_0},\Lambda)$ for some $i_0$ with pullbacks $\mathcal F_i,\mathcal G_i\in \Cons(Y_i,\Lambda)$ and $\mathcal F,\mathcal G\in \Cons(Y,\Lambda)$. Let $f_j: Y\to Y_j$ and $f_{ij}: Y_i\to Y_j$ be the natural maps. Then
\[\begin{aligned}
\Hom_Y(\mathcal F,\mathcal G)&=\Hom_{Y_{i_0}}(\mathcal F_{i_0},f_{i_0\ast} \mathcal G)\\
&=\Hom_{Y_{i_0}}(\mathcal F_{i_0},\varinjlim_i f_{i,i_0\ast} \mathcal G_i)\\
&=\varinjlim_i \Hom_{Y_{i_0}}(\mathcal F_{i_0}, f_{i,i_0\ast} \mathcal G_i)\\
&=\varinjlim_i \Hom_{Y_i}(\mathcal F_i,\mathcal G_i)\ ,
\end{aligned}\]
using obvious adjunctions, (the relative version of) Proposition~\ref{prop:etcohomlim} to write $f_{i_0\ast} \mathcal G = \varinjlim_i f_{i,i_0\ast} \mathcal G_i$, and compactness of $\mathcal F_{i_0}$.

Now, for essential surjectivity, note that any constructible sheaf $\mathcal F$ on $Y$ is a quotient of a map of finite direct sums of $j_! \Lambda$, where $j: U\to Y$ runs through quasicompact separated \'etale maps, as the proof of Proposition~\ref{prop:charconstructible} shows. Using Proposition~\ref{prop:etmaptolimdiamond} and fully faithfulness, we see that this data is defined over some $Y_i$, as desired.
\end{proof}

Now we can prove a slightly more explicit characterization of constructibility, which will be necessary for the passage to derived categories.

\begin{proposition}\label{prop:charconstructible2} Let $\Lambda$ be a noetherian ring, and let $Y$ be a spatial diamond. Then an \'etale sheaf $\mathcal F$ of $\Lambda$-modules on $Y$ is constructible if and only if $\mathcal F$ has a filtration whose graded pieces $\mathcal F_i$ are of the form $j_! (\mathcal{L}|_Z)$, where $j: U\to Y$ is a quasicompact separated \'etale map, $Z\subset U$ is a constructible closed subset, and $\mathcal L$ is a sheaf of $\Lambda$-modules on $U$ that is locally on $U_\et$ isomorphic to the constant sheaf associated with some finitely generated $\Lambda$-module.
\end{proposition}

Here, $\mathcal L|_Z$ is the cokernel of the injective map $j^\prime_! \mathcal L|_{U\setminus Z}\to \mathcal L$, where $j^\prime: U\setminus Z\to U$ is the open embedding.

\begin{proof} It is clear that any sheaf admitting such a filtration is constructible. Conversely, we can find a stratification $S_1,\ldots,S_m$ of $Y$ as in Proposition~\ref{prop:charconstructible}~(iii), with the extra property that for all $i=1,\ldots,m$, the union $U_i=S_1\sqcup\ldots\sqcup S_i$ is (quasicompact and) open. One can filter $\mathcal F$ by the extensions by $0$ of $\mathcal F|_{U_i}$; passing to associated gradeds, we can assume that there is some constructible locally closed subset $S\subset |Y|$ such that the stalks of $\mathcal F$ at all points outside of $S$ vanish, and $\mathcal F|_S$ satisfies the final condition of Proposition~\ref{prop:charconstructible}~(iii): For all strictly totally disconnected perfectoid spaces $f: X\to Y$, the restriction $f^\ast\mathcal F|_{f^{-1}(S)}$ is constant.

Moreover, the claim can be checked locally: Indeed, assume it is true for the restrictions $\mathcal F|_{V_i}$ for an open covering of $\{V_i\}$ of $Y$. By induction, we may assume that there are only two subsets $V_1,V_2\subset Y$; let $j_i: V_i\to Y$ be the open subsets for $i=1,2$, as well as $j_{12}: V_{12} = V_1\cap V_2\subset Y$. Then $\mathcal F$ has a subsheaf $j_{1!} \mathcal F|_{V_1}$ of the desired form, and the quotient $ \mathcal F/j_{1!} \mathcal F|_{V_1}$ can be identified with $j_{2!}(\mathcal F|_{V_2\setminus V_{12}})$, which is again of the desired form.

Now we claim that locally, $\mathcal F$ is indeed of the form $j_!(\mathcal L|_Z)$ for a quasicompact separated \'etale map $j: U\to Y$, some $\mathcal L$ on $U$ as in the statement, and a constructible closed subset $Z\subset U$. This can be checked after pullback to the localization $Y_y$ at varying points $y\in |Y|$, by Proposition~\ref{prop:constructiblelimit}.

Thus, we can assume that $Y$ is local, so let $\Spa(C,C^+)\to Y$ be a quasi-pro-\'etale surjection, where $C$ is algebraically closed and $C^+\subset C$ an open and bounded valuation subring. Then $|Y|=|\Spa(C,C^+)|$ is a totally ordered chain of points with a unique closed point $s\in |Y|$. Moreover, by the first paragraph, there is some constructible locally closed subset $S\subset |Y|$ such that $\mathcal F|_{\Spa(C,C^+)\setminus S}=0$ and $\mathcal F|_{S\subset \Spa(C,C^+)}$ is constant, with value some finitely generated $\Lambda$-module $M$. We can assume that $s\in S$, so that $S$ is actually closed. Let $\eta_S\in S$ be the generic point of $S$, and let $G_S=G_{\eta_S}$ be the profinite group given as the fibre of $|R|\to |Y|$, $R=\Spa(C,C^+)\times_Y \Spa(C,C^+)$, over $\eta_S$ (where the group structure comes from the equivalence relation structure). Then $\mathcal F$ is given by a continuous action of $G_S$ on $M$. Similarly, let $\eta\in |Y|$ be the generic point, and $G_\eta$ the profinite group that is the fibre of $|\Spa(C,C^+)\times_Y \Spa(C,C^+)|\to |Y|$ over $\eta$ (so that the open subspace $Y^\circ$ of $Y$ with underlying space $\{\eta\}$ is given by $\Spa(C,\OO_C)/\underline{G_\eta}$). There is natural closed immersion $G_S\hookrightarrow G_\eta$ of groups, given by generalization. One can find an open subgroup $H\subset G_\eta$ containing $G_S$ such that the action of $G_S$ on $M$ extends to $H$. Let $R_H\subset R$ be the open and closed subspace given as the closure of $H\subset G_\eta = R\times_{Y} \{\eta\}$. Then $U = \Spa(C,C^+)/\underline{H}$ is separated and \'etale over $Y$, and the continuous $H$-module $M$ defines a local system $\mathcal{L}$ on $U$, which is easily checked to have the right property.
\end{proof}

We will also need a compactness result in the derived category. We start with an easy result.

\begin{proposition}\label{prop:constructiblecompactderivedplus} Let $Y$ be a spatial diamond, and assume that $\mathcal F$ is a constructible \'etale sheaf of $\Fl$-vector spaces on $Y$. Then for all filtered direct systems $C_j\in \mathcal D^{\geq -n}_\et(Y,\Fl)$, $j\in J$, of complexes uniformly bounded to the left, the natural map
\[
\varinjlim_j \Hom_{D_\et(Y,\Fl)}(\mathcal F[0],C_j)\to \Hom_{D_\et(Y,\Fl)}(\mathcal F[0],\varinjlim_j C_j)
\]
is an isomorphism.
\end{proposition}

\begin{proof} By descent (and using boundedness), this can be reduced to the case that $Y=X$ is strictly totally disconnected. Decomposing $\mathcal F$ into finitely many triangles, we can assume that $\mathcal F = j_! \Fl$ for some quasicompact open immersion $j: U\hookrightarrow X$. Then the statement becomes
\[
\varinjlim_{j\in J} R\Gamma(U_\et,C_j)\buildrel\simeq\over\to R\Gamma(U_\et,\varinjlim_{j\in J} C_j)\ ,
\]
which holds true as $U_\et$ is coherent and the $C_j$ are uniformly bounded to the left.
\end{proof}

One gets a stronger result if one assumes that $Y$ is locally of finite $\ell$-cohomological dimension.

\begin{proposition}\label{prop:constructiblecompactderivedfull} Let $Y$ be a spatial diamond, and assume that there is an integer $N$ such that for all $\ell$-torsion sheaves $\mathcal F$ on $Y_\et$, one has $H^i(Y_\et,\mathcal F)=0$ for $i>N$.

Then $D(Y_\et,\Fl)$ is left-complete (thus $D(Y_\et,\Fl)=D_\et(Y,\Fl)$), compactly generated, and a complex $C\in D_\et(Y,\Fl)$ is compact if and only if it is bounded and all cohomology sheaves are constructible.
\end{proposition}

\begin{proof} We observe first that the hypothesis on cohomological dimension passes (with the same $N$) to any quasicompact separated \'etale $U\to Y$, by Remark~\ref{rem:qcsepqproetpushforward}. Now left-completeness of $D(Y_\et,\Fl)$ follows from \cite[Tag 0719]{StacksProject}. This implies $D(Y_\et,\Fl)=D_\et(Y,\Fl)$ by Proposition~\ref{prop:locallyspatialleftcompletion}.

First, we check that a complex is compact if it is bounded with all cohomology sheaves constructible. This reduces immediately to the case that $C=\mathcal F[0]$ for some constructible sheaf $\mathcal F$. Using Proposition~\ref{prop:charconstructible2}, this reduces further to the case $\mathcal F=j_!\mathcal L$ for some quasicompact separated \'etale map $j: U\to Y$ and $\Fl$-local system $\mathcal L$ on $U$. But then
\[
\Hom_{D(Y_\et,\Fl)}(j_!\mathcal L,-) = \Hom_{D(U_\et,\Fl)}(\mathcal L,-)=R\Gamma(U_\et,\mathcal L^\vee\otimes_\Fl -)\ ,
\]
so we have to prove that $R\Gamma(U_\et,-)$ commutes with direct sums in $D(U_\et,\Fl)$. This follows easily from left-completeness of $D_\et(U,\Fl)$, the assumption that $U_\et$ has finite $\ell$-cohomological dimension, and Proposition~\ref{prop:constructiblecompactderivedplus}.

As the derived category is generated by $j_! \Fl$ for varying quasicompact separated \'etale maps $j: U\to Y$, this shows that $D(Y_\et,\Fl)$ is compactly generated. By abstract nonsense, any compact object is a direct summand of a finite complex whose terms are finite direct sums of sheaves of the form $j_! \Fl$ for varying quasicompact separated \'etale maps $j: U\to Y$. All of those are bounded with constructible cohomology sheaves, as desired.
\end{proof}

It will be convenient to have generalizations of some of the previous results on derived categories to the case that $\Lambda\neq \Fl$. In that case, the good notion of constructibility in the derived category is less directly related to the notion of constructibility for abelian sheaves. To highlight the difference, we call these objects perfect-constructible.

\begin{definition}\label{def:perfectconstructible} Let $\Lambda$ be any ring.
\begin{altenumerate}
\item Let $X$ be a strictly totally disconnected perfectoid space. A complex $A\in D_\et(X,\Lambda)\cong D(|X|,\Lambda)$ of $\Lambda$-modules on $X_\et$ (equivalently, on $|X|$) is perfect-constructible if there is stratification of $X$ into constructible locally closed subsets $S_i\subset |X|$ such that $A|_{S_i}$ is the constant sheaf on $S_i$ associated with some perfect complex of $\Lambda$-modules.
\item Let $Y$ be a small v-stack, and $A\in D_\et(Y,\Lambda)$. Then $A$ is perfect-constructible if for every strictly totally disconnected space $f: X\to Y$, the pullback $f^\ast A\in D_\et(X,\Lambda)$ is perfect-constructible.
\end{altenumerate}
\end{definition}

It is clear that the category of perfect-constructible complexes forms a thick triangulated subcategory $D_{\et,pc}(Y,\Lambda)\subset D_\et(Y,\Lambda)$. The following proposition clarifies the relation to the notion of constructible sheaves.

\begin{proposition}\label{prop:perfectconstrvsconstr} Assume that $\Lambda$ is noetherian and let $Y$ be a small v-stack and $A\in D_\et(Y,\Lambda)$. The complex $A$ is perfect-constructible if and only if it is locally bounded, each cohomology sheaf $\mathcal H^i(A)$ is a constructible sheaf of $\Lambda$-modules, and all geometric stalks of $A$ are perfect complexes of $\Lambda$-modules.
\end{proposition}

\begin{proof} We can assume that $Y=X$ is a strictly totally disconnected perfectoid space. It is clear that if $A$ is perfect-constructible, then it satisfies the stated properties, so we have to prove the converse. Passing to a stratification of $|X|$, it is enough to prove that if $S$ is a totally disconnected spectral space and $B\in D(S,\Lambda)$ is a bounded complex such that each $\mathcal H^i(B)$ is the constant sheaf associated with some finitely generated $\Lambda$-module, then $B$ is locally the constant sheaf associated with some complex of $\Lambda$-modules. Indeed, this implies that if the stalks of $B$ are perfect, then these constant sheaves are associated with perfect complexes of $\Lambda$-modules.

We argue by induction on the length of $B$, so assume $B\in D^{[a,b]}(S,\Lambda)$. If $a=b$, there is nothing to prove. In general, consider the triangle $\tau^{<b} B\to B\to H^b(B)[-b]$. After localization, $\tau^{<b} B$ is the constant sheaf associated with some complex of $\Lambda$-modules $C^{<b}$, and the extension $B$ is given by a class in
\[
\Hom_{D(S,\Lambda)}(H^b(B)[-b],C^{<b}[1])\ .
\]
As $S$ is a totally disconnected spectral space, one shows that
\[
R\Hom_{D(S,\Lambda)}(M,C) = \mathrm{Cont}(S,R\Hom_\Lambda(M,C))
\]
for any bounded complex of $\Lambda$-modules $C$ and any finitely generated $\Lambda$-module $M$, by choosing a finite free resolution of $M$ to reduce to the assertion
\[
H^i(S,C) = \mathrm{Cont}(S,H^i(C))\ .
\]
Thus, after passing to an open and closed cover of $S$, the extension $B$ is constant, as desired.
\end{proof}

Again, being perfect-constructible can be checked v-locally.

\begin{proposition}\label{prop:perfectconstrdescends} Let $\Lambda$ be a ring, let $f: \tilde{Y}\to Y$ be a surjective map of small v-stacks, and let $A\in D_\et(Y,\Lambda)$. If $f^\ast A\in D_\et(\tilde{Y},\Lambda)$ is perfect-constructible, then $A$ is perfect-constructible.
\end{proposition}

\begin{proof} We may assume that both $Y=X$ and $\tilde{Y}=\tilde{X}$ are strictly totally disconnected perfectoid spaces. In that case $D_\et(Y,\Lambda)=D(|X|,\Lambda)$ and $D_\et(\tilde{Y},\Lambda)=D(|\tilde{Y}|,\Lambda)$. Replacing $\tilde{Y}$ by $Y\times_{|Y|} |\tilde{Y}|$, we may then assume that $\tilde{Y}\to Y$ is affinoid pro-\'etale. Write $\tilde{Y}$ as a cofiltered limit of affinoid \'etale maps $\tilde{Y}_j\to Y$. Then the stratification witnessing the perfect-constructibility of $f^\ast A$ is defined on some $\tilde{Y}_j$. As $\tilde{Y}_j\to Y$ admits a splitting, it is enough to show that $A|_{\tilde{Y}_j}$ is perfect-constructible, so we can assume that $Y=\tilde{Y}_j$; in other words, we can assume that there is a stratification of $|Y|$ into constructible locally closed subsets $S_i$ such that $f^\ast A|_{f^{-1}(S_i)}$ is constant with perfect value. Now the map from the perfect complex to $f^\ast A|_{f^{-1}(S_i)}$ is also defined over some $\tilde{Y}_j$, and is already an isomorphism there. Thus, after replacing $Y$ by the split cover $\tilde{Y}_j$, we see that indeed $A$ is perfect-constructible.
\end{proof}

Perfect-constructible complexes on spatial diamonds satisfy a restricted version of compactness.

\begin{proposition}\label{prop:perfectconstructiblecompactderivedplus} Let $Y$ be a spatial diamond, and assume that $A\in D_\et(Y,\Lambda)$ is perfect-constructible. Then for all filtered direct systems $C_j\in \mathcal D^{\geq -n}_\et(Y,\Fl)$, $j\in J$, of complexes uniformly bounded to the left, the natural map
\[
\varinjlim_j \Hom_{D_\et(Y,\Fl)}(A,C_j)\to \Hom_{D_\et(Y,\Fl)}(A,\varinjlim_j C_j)
\]
is an isomorphism.
\end{proposition}

\begin{proof} By descent (and using boundedness), this can be reduced to the case that $Y=X$ is strictly totally disconnected. Decomposing $A$ into finitely many triangles, we can assume that $A = j_! \Lambda$ for some quasicompact open immersion $j: U\hookrightarrow X$. Then the statement becomes
\[
\varinjlim_j R\Gamma(U_\et,C_j)\buildrel\simeq\over\to R\Gamma(U_\et,\varinjlim_j C_j)\ ,
\]
which holds true as $U_\et$ is coherent and the $C_j$ are uniformly bounded to the left.
\end{proof}

The notion behaves well with respect to passage to limits.

\begin{proposition}\label{prop:constructiblelimitderived} Let $Y_i$, $i\in I$, be a cofiltered inverse system of spatial diamonds with inverse limit $Y=\varprojlim_i Y_i$, which is again a spatial diamond. Let $\Lambda$ be a ring. The natural functor
\[
\text{2-}\varinjlim_i D_{\et,pc}(Y_i,\Lambda)\to D_{\et,pc}(Y,\Lambda)
\]
is an equivalence of categories.

Similarly, if $Y$ is any spatial diamond and $\Lambda_i$, $i\in I$, is a filtered direct system of rings with colimit $\Lambda=\varinjlim_i \Lambda_i$, then the natural functor
\[
\text{2-}\varinjlim_i D_{\et,pc}(Y,\Lambda_i)\to D_{\et,pc}(Y,\Lambda)
\]
is an equivalence of categories.
\end{proposition}

\begin{proof} We handle the case of spaces; the case of rings is proved similarly. First, we check fully faithfulness. For this, let $A_{i_0}, B_{i_0}\in D_{\et,pc}(Y_{i_0},\Lambda)$ for some $i_0$ with pullbacks $A_i,B_i\in D_{\et,pc}(Y_i,\Lambda)$ and $A,B\in D_{\et,pc}(Y,\Lambda)$. Let $f_j: Y\to Y_j$ and $f_{ij}: Y_i\to Y_j$ be the natural maps. Then
\[\begin{aligned}
\Hom_Y(A,B)&=\Hom_{Y_{i_0}}(A_{i_0},f_{i_0\ast} B)\\
&=\Hom_{Y_{i_0}}(A_{i_0},\varinjlim_i f_{i,i_0\ast} B_i)\\
&=\varinjlim_i \Hom_{Y_{i_0}}(A_{i_0}, f_{i,i_0\ast} B_i)\\
&=\varinjlim_i \Hom_{Y_i}(A_i,B_i)\ ,
\end{aligned}\]
using obvious adjunctions, (the relative version of) Proposition~\ref{prop:etcohomlim} to write $f_{i_0\ast} B = \varinjlim_i f_{i,i_0\ast} B_i$, and Proposition~\ref{prop:perfectconstructiblecompactderivedplus}.

For essential surjectivity, we use Proposition~\ref{prop:charperfectconstructible} below, whose proof only requires the fully faithfulness part of the current proposition. In the notation of that proposition, this reduces us to the case $A=j_!(\mathcal L|_Z)$, and then in fact further to the case $A=\mathcal L$, so we can assume that $A$ is locally constant with perfect values. Choose a quasicompact separated \'etale map $Y^\prime\to Y$ such that $A|_{Y^\prime}$ is constant, of necessarily finite perfect amplitude. We can assume that $Y^\prime\to Y$ is the pullback of $Y^\prime_i\to Y_i$ for $i$ large enough. As $D_\et(Y,\Lambda)\cong D_{\et,\cart}(Y^\prime_\bullet,\Lambda)$ as in Proposition~\ref{prop:derivedhyperdescent} where $Y^\prime_\bullet$ is the \v{C}ech nerve of $Y^\prime\to Y$, and similarly $D_\et(Y_i,\Lambda)\cong D_{\et,\cart}(Y^\prime_{i,\bullet},\Lambda)$, and moreover the descent for perfect complexes of given finite perfect amplitude only needs a truncation of the \v{C}ech nerve, one gets the result by noting that $A|_{Y^\prime}$ spreads to $Y^\prime_i$ as $A|_{Y^\prime}$ is constant, and the descent datum spreads by the fully faithfulness already proved.
\end{proof}

Moreover, for spatial diamonds, a stratification witnessing constructibility is defined on the spatial diamond itself, and one can obtain an analogue of Proposition~\ref{prop:charconstructible2}.

\begin{proposition}\label{prop:charperfectconstructible} Let $\Lambda$ be a ring, let $Y$ be a spatial diamond and let $A\in D_\et(Y,\Lambda)$. The following conditions are equivalent.
\begin{altenumerate}
\item The complex $A$ is perfect-constructible.
\item There is a stratification of $|Y|$ into constructible locally closed subsets $S_i\subset |Y|$ such that the restriction of $A$ to $S_i$ satisfies the following condition: For any strictly totally disconnected perfectoid space $f: X\to Y$, the pullback $f^\ast A|_{f^{-1}(S_i)}$ is the constant sheaf associated with some perfect complex of $\Lambda$-modules.
\item The complex $A$ has a finite filtration whose graded pieces $A_i$ are of the form $j_! (\mathcal{L}|_Z)$, where $j: U\to Y$ is a quasicompact separated \'etale map, $Z\subset U$ is a constructible closed subset, and $\mathcal L\in D_\et(U,\Lambda)$ is locally constant with perfect values.
\end{altenumerate}
\end{proposition}

\begin{proof} By definition (ii) and (iii) imply (i). Next, we check that (i) implies (ii). For this, take a surjective quasi-pro-\'etale $f: X\to Y$ from a strictly totally disconnected perfectoid space $X$ as in Proposition~\ref{prop:spatialunivopen}, so $X=\varprojlim_i Y_i$ where $Y_i\to Y$ is a surjective quasicompact \'etale map that can be written as a composite of quasicompact open immersions and finite \'etale maps. If $A$ is perfect-constructible, then $f^\ast A$ becomes constant with perfect values over a constructible stratification of $X$. This stratification is pulled back from $|Y_i|$ for some $i$, so assume $U_{0,i}\subset U_{1,i}\subset \ldots\subset U_{n,i}=|Y_i|$ is a filtration by quasicompact open subsets such that $f^\ast A$ becomes constant with perfect values over the preimage of $U_{j,i}\setminus U_{j-1,i}$ for $j=0,\ldots,n$. Let $U_j\subset |Y|$ be the quasicompact open image of $U_{j,i}\subset |Y_i|$. Then $f^\ast A$ becomes constant with perfect values over the preimage of $S_j=U_j\setminus U_{j-1}$ for $j=0,\ldots,n$, as the image of $U_{j,i}\setminus U_{j-1,i}$ contains $U_j\setminus U_{j-1}$. This shows that (i) implies (ii).

Finally, we have to see that (ii) implies (iii). We can assume that there is some constructible locally closed subset $S\subset |Y|$ such that the stalks of $A$ at all points outside of $S$ vanish, and $A|_S$ satisfies that for all strictly totally disconnected perfectoid spaces $f: X\to Y$, the restriction $f^\ast A|_{f^{-1}(S)}$ is constant.

The desired result about $A$ can be checked locally. Indeed, assume it is true for the restrictions $A|_{V_i}$ for a quasicompact open covering of $\{V_i\}$ of $Y$. By induction, we may assume that there are only two subsets $V_1,V_2\subset Y$; let $j_i: V_i\to Y$ be the open subsets for $i=1,2$, as well as $j_{12}: V_{12} = V_1\cap V_2\subset Y$. Then one has a triangle
\[
j_{1!} A|_{V_1}\to A\to j_{2!}(A|_{V_2\setminus V_{12}})\ ,
\]
where both pieces are of the desired form.

Now we claim that locally, $A$ is of the form $j_!(\mathcal L|_Z)$ for a quasicompact separated \'etale map $j: U\to Y$, some $\mathcal L\in D_\et(U,\Lambda)$ that is locally constant with perfect values, and a constructible closed subset $Z\subset U$. This can be checked after pullback to the localization $Y_y$ at varying points $y\in |Y|$, by (the fully faithfulness part of) Proposition~\ref{prop:constructiblelimitderived}.

Thus, we can assume that $Y$ is local, so let $\Spa(C,C^+)\to Y$ be a quasi-pro-\'etale surjection, where $C$ is algebraically closed and $C^+\subset C$ an open and bounded valuation subring. Then $|Y|=|\Spa(C,C^+)|$ is a totally ordered chain of points with a unique closed point $s\in |Y|$. Moreover, by the first paragraph, there is some constructible locally closed subset $S\subset |Y|$ such that $A|_{\Spa(C,C^+)\setminus S}=0$ and $A|_{S\subset \Spa(C,C^+)}$ is constant, with value some perfect complex of $\Lambda$-modules $B$. We can assume that $s\in S$, so that $S$ is actually closed. Let $\eta_S\in S$ be the generic point of $S$, and let $G_S=G_{\eta_S}$ be the profinite group given as the fibre of $|R|\to |Y|$, $R=\Spa(C,C^+)\times_Y \Spa(C,C^+)$, over $\eta_S$ (where the group structure comes from the equivalence relation structure). Similarly, let $\eta\in |Y|$ be the generic point, and $G_\eta$ the profinite group that is the fibre of $|\Spa(C,C^+)\times_Y \Spa(C,C^+)|\to |Y|$ over $\eta$ (so that the open subspace $Y^\circ$ of $Y$ with underlying space $\{\eta\}$ is given by $\Spa(C,\OO_C)/\underline{G_\eta}$). There is natural closed immersion $G_S\hookrightarrow G_\eta$ of groups, given by generalization. For any open subgroup $H\subset G_\eta$ containing $G_S$, let $R_H\subset R$ be the open and closed subspace given as the closure of $H\subset G_\eta = R\times_{Y} \{\eta\}$. Then $U = \Spa(C,C^+)/\underline{H}$ is separated and \'etale over $Y$, and an isomorphism over $S$. Passing to the inverse limit over all such $H$, we may by Proposition~\ref{prop:constructiblelimitderived} assume that $G_S=G_\eta$.

For a profinite group $G$, let $BG$ denote the site of finite $G$-sets. There is a pullback functor from the topos of sheaves on $Y$ that are concentrated on $S$ to the topos of sheaves on $BG_S$ given by the fibre over $\eta_S$, and a pushforward functor from the topos of sheaves on $BG_S=BG_\eta$ to the topos of sheaves on $Y$ induced by the open embedding $\{\eta\}\subset Y$. Applying the composite functor to $A$ produces some $\tilde{A}\in D_\et(Y,\Lambda)$ whose pullback to $\Spa(C,C^+)$ is constant and has restriction $A$ to $S$. Thus, it suffices to see that $\tilde{A}$ is locally constant with perfect values. Over $\Spa(C,C^+)$, it is isomorphic to the constant sheaf associated with some perfect complex $B$; the map from $B$ to this pullback is already defined over some \'etale map to $Y$ by Proposition~\ref{prop:perfectconstructiblecompactderivedplus}, and is an isomorphism there (as can be checked on the v-cover by $\Spa(C,C^+)$).
\end{proof}

Assuming finite cohomological dimension, we get the expected relation to compact objects.

\begin{proposition}\label{prop:perfectconstructiblecompact} Let $\Lambda$ be a ring, and let $Y$ be a spatial diamond of bounded $\Lambda$-cohomological dimension. Then $D(Y_\et,\Lambda)$ is left-complete, so that $D_\et(Y,\Lambda)=D(Y_\et,\Lambda)$. 

A complex $A\in D_\et(Y,\Lambda)$ is compact if and only if it is perfect-constructible, and $D_\et(Y,\Lambda)$ is compactly generated. A set of compact generators is given by $j_!\Lambda$ for $j: U\to X$ ranging over quasicompact separated \'etale maps.
\end{proposition}

\begin{proof} Again, Remark~\ref{rem:qcsepqproetpushforward} shows that the hypothesis on cohomological dimension of $Y$ passes to all quasicompact separated \'etale $U\to Y$; then left-completeness follows from \cite[Tag 0719]{StacksProject}. The objects $j_!\Lambda$ for $j: U\to X$ a quasicompact separated \'etale map are compact generators as $R\Hom(j_!\Lambda,-)=R\Gamma(U,-)$ commutes with arbitrary direct sums (by finite cohomological dimension) and a complex $A\in D_\et(Y,\Lambda)=D(Y_\et,\Lambda)$ vanishes as soon as $R\Gamma(U,A)=0$ for all such $U$. This implies in particular that $D_\et(Y,\Lambda)$ is compactly generated. The given compact generators are perfect-constructible; it follows that any $A\in D_\et(Y,\Lambda)$ can be written as a filtered homotopy colimit of perfect-constructible complexes. If $A$ is compact, it is then a retract of a perfect-constructible complex, and thus perfect-constructible itself.

Conversely, if $A$ is perfect-constructible, then by Proposition~\ref{prop:charperfectconstructible}, it admits a finite filtration with graded pieces given by complexes of the form $j_!(\mathcal L|_Z)$ where $j: U\to X$ is a quasicompact separated \'etale map, $Z\subset U$ is a constructible closed subset, and $\mathcal L\in D_\et(U,\Lambda)$ is locally constant with perfect values. To see that $A$ is compact, we may thus assume that $A=j_!(\mathcal L|_Z)$, and then by passage to a $2$-term resolution that $A=j_!\mathcal L$. But then $R\Hom(A,-)=R\Gamma(U,\mathcal L^\vee\dotimes_\Lambda-)$ and tensoring with the dual $\mathcal L^\vee=R\sHom_\Lambda(\mathcal L,\Lambda)$ of $\mathcal L$ preserves direct sums.
\end{proof}

\section{Dimensions}

To discuss duality, it will be important to impose a ``finite-dimensionality'' assumption on morphisms $f$. For this, we restrict to locally spatial morphisms, and we will put a condition on transcendence degrees of residue fields.

For a morphism $f: X^\prime\to X$ of perfectoid spaces, one has the definition of $\dim f$ as given in \cite[Definition 1.8.4]{Huber}, and of $\dimtr f$ as given in \cite[Definition 1.8.4]{Huber}. Let us briefly recall those.

\begin{definition}\label{def:dim}$ $\begin{altenumerate}
\item Let $X$ be a locally spectral space. The dimension $\dim X\in \mathbb Z_{\geq 0}\cup \{-\infty,\infty\}$ of $X$ is the supremum of all integers $n$ for which there exists a chain $x_0,\ldots,x_n\in X$ of distinct points in $X$ such that $x_i$ is a specialization of $x_{i+1}$ for $i=0,\ldots,n-1$.
\item Let $f: X^\prime\to X$ be a spectral map of locally spectral spaces. Then
\[
\dim f = \sup_{x\in X} \dim f^{-1}(x)\in \mathbb Z_{\geq 0}\cup \{-\infty,\infty\}\ .
\]
\end{altenumerate}
\end{definition}

\begin{definition}\label{def:dimtr} Let $K\subset K^\prime$ be an extension of complete algebraically closed nonarchimedean fields. Then the topological transcendence degree $\trc(K^\prime/K)\in \mathbb Z_{\geq 0}\cup \{\infty\}$ of $K^\prime$ over $K$ is the minimum over all integers $n$ such that there exists a dense subfield $L\subset K^\prime$ containing $K$ with transcendence degree $n$ over $K$.
\end{definition}

The definition in \cite[Definition 1.8.2]{Huber} is slightly different, but in fact they agree in the case of algebraically closed fields, as one easily checks using Krasner's lemma.

\begin{lemma}\label{lem:trcsubadditive}$ $
\begin{altenumerate}
\item Let $K\subset K^\prime\subset K^{\prime\prime}$ be extensions of algebraically closed complete nonarchimedean fields. Then $\trc(K^{\prime\prime}/K)\leq \trc(K^{\prime\prime}/K^\prime) + \trc(K^\prime/K)$.
\item Let $K\subset K^\prime$ and $L\subset L^\prime$ be extensions of algebraically closed complete nonarchimedean fields such that there is an embedding $K^\prime\hookrightarrow L^\prime$ sending $K$ into $L$, and such that the algebraic closure of $L\cdot K^\prime$ in $L^\prime$ is dense. Then $\trc(L^\prime/L)\leq \trc(K^\prime/K)$.
\end{altenumerate}
\end{lemma}

\begin{proof} The first part is \cite[Remark 1.8.3 (i)]{Huber}. More precisely, assume that $\trc(K^\prime/K)=n$ and $\trc(K^{\prime\prime}/K^\prime)=m$. Then there exist $n$ elements $x_1,\ldots,x_n\in K^\prime$ such that $K^\prime$ is the minimal complete and algebraically closed subfield of $K^\prime$ containing $K$ and $x_1,\ldots,x_n$, and there exist $m$ elements $x_{n+1},\ldots,x_{n+m}\in K^{\prime\prime}$ such that $K^{\prime\prime}$ is the minimal complete and algebraically closed subfield of $K^{\prime\prime}$ containing $K^\prime$ and $x_{n+1},\ldots,x_{n+m}$. Then the minimal complete and algebraically closed subfield of $K^{\prime\prime}$ containing $K$ and $x_1,\ldots,x_{n+m}$ is $K^{\prime\prime}$: Indeed, it contains $K^\prime$, and then is all of $K^{\prime\prime}$. This shows that $\trc(K^{\prime\prime}/K)\leq n+m$, as desired.

Similarly, one proves the second part, which is also \cite[Remark 1.8.3 (ii)]{Huber}.
\end{proof}

Unfortunately, we were not able to resolve the following question.

\begin{question}\label{q:transcendence} Let $K\subset K^\prime\subset K^{\prime\prime}$ be extensions of algebraically closed complete nonarchimedean fields. Is is true that $\trc(K^\prime/K)\leq \trc(K^{\prime\prime}/K)$?
\end{question}

There is the following theorem of Temkin: In the situation of Question~\ref{q:transcendence}, if $\trc(K^\prime/K)<\infty$, then $\trc(K^\prime/K)\leq \trc(K^{\prime\prime}/K)$. This is vacuous if $\trc(K^{\prime\prime}/K)=\infty$. Otherwise, \cite[Theorem 3.2.3]{TemkinTranscendence} identifies $\trc(K^{\prime\prime}/K)$ with the maximal number of topologically algebraically independent elements, and similarly for $\trc(K^\prime/K)$ by assumption $\trc(K^\prime/K)<\infty$. This invariant is monotonic by \cite[Lemma 2.2.2]{TemkinTranscendence}.

As we cannot answer Question~\ref{q:transcendence} in general, we consider instead the variant $\widetilde{\trc}(K^\prime/K)$, which is defined as the minimum over $\trc(K^{\prime\prime}/K)$ over all complete algebraically closed extensions $K^{\prime\prime}$ of $K^\prime$. Clearly, $\widetilde{\trc}(K^\prime/K)\leq \trc(K^\prime/K)$. Moreover, Lemma~\ref{lem:trcsubadditive} holds true with $\widetilde{\trc}$ in place of $\trc$.

\begin{definition}\label{def:dimtrg} Let $f: X^\prime\to X$ be a map of analytic adic spaces. Then
\[
\dimtrg f\in \mathbb Z_{\geq 0}\cup \{-\infty,\infty\}
\]
is the supremum over all $x^\prime\in X^\prime$ of $\widetilde{\trc}(C(x^\prime)/C(x))$, where $x=f(x^\prime)\in X$, and $C(x^\prime)$ resp.~$C(x)$ denote completed algebraic closures of the completed residue fields.
\end{definition}

\begin{lemma}\label{lem:compdim} Let $f: X^\prime\to X$ be a map of analytic adic spaces. Then
\[
\dim f\leq \dimtrg f\ .
\]
\end{lemma}

\begin{proof} As the rationalized value group does not change under passage to a completed algebraic closure (and only gets larger under passage to extensions), the proof of~\cite[Lemma 1.8.5 (i)]{Huber} applies to show that $\dim f\leq \dimtrg f$.
\end{proof}

The definition of $\dimtrg$ extends to maps of diamonds.

\begin{definition}\label{def:dimtrgdiamond} Let $f: Y^\prime\to Y$ be a map of diamonds. For each $y^\prime\in |Y^\prime|$ with image $y\in |Y|$, choose quasi-pro-\'etale maps $\Spa(C(y),C(y)^+)\to Y$ with $y$ in the image, and $\Spa(C(y^\prime),C(y^\prime)^+)\to \Spa(C(y),C(y)^+)\times_Y Y^\prime$ with $y^\prime$ in the image. Then
\[
\dimtrg f\in \mathbb Z_{\geq 0}\cup \{-\infty,\infty\}
\]
is defined as the supremum of $\widetilde{\trc}(C(y^\prime)/C(y))$ over all $y^\prime\in |Y^\prime|$.

More generally, if $f: Y^\prime\to Y$ is a map of v-stacks that is representable in diamonds, then $\dimtrg f$ is the supremum of $\dimtrg (f\times_Y X)$ over all maps $X\to Y$ from diamonds $X$.
\end{definition}

\begin{remark} Note that a pullback will not increase $\dimtrg$ by Lemma~\ref{lem:trcsubadditive}~(ii), so the definition of $\dimtrg f$ for morphisms of v-sheaves agrees with the previous definition if $Y$ and $Y^\prime$ are diamonds.
\end{remark}

Note that if $Y$ is a locally spatial diamond or if $f: Y^\prime\to Y$ is map of v-stacks that is representable in locally spatial diamonds, we also have a definition of $\dim Y$ and of $\dim f$, by Definition~\ref{def:dim}. In the second case, we take the supremum over all locally spatial diamonds $X$ with a map $X\to Y$ of $\dim(f\times_Y X)$. To evaluate this, it is enough to range over $X$ of the form $\Spa(C,C^+)$.

Next, we want to prove a bound on the cohomological dimension of spatial diamonds. For this, we need to briefly discuss points of diamonds.

\begin{proposition}\label{prop:diamondpoint} Let $Y$ be a quasiseparated diamond such that $|Y|$ consists of only one point. Then $Y=\Spa(C,\OO_C)/\underline{G}$ for some algebraically closed nonarchimedean field $C$ and some profinite group $G$ acting continuously and faithfully on $C$. Moreover, the pair $(C,G)$ is unique up to (non-unique) isomorphism.
\end{proposition}

\begin{proof} As $|Y|$ consists of only one point, we can find a quasi-pro-\'etale surjection $\Spa(C,\OO_C)\to Y$, which is necessarily separated (as $\Spa(C,\OO_C)$ is separated). Then the fibre product $R=\Spa(C,\OO_C)\times_Y \Spa(C,\OO_C)$ is quasicompact separated pro-\'etale over $\Spa(C,\OO_C)$, and in particular affinoid pro-\'etale by Lemma~\ref{lem:proetaleoverwlocal}. Thus, $R=\Spa(C,\OO_C)\times \underline{S}$ for some profinite set $S$. As $R$ is an equivalence relation, $S$ becomes a profinite group, which we denote by $G$. The equivalence relation structure $R=\Spa(C,\OO_C)\times\underline{G}\subset \Spa(C,\OO_C)\times \Spa(C,\OO_C)$ is then equivalent to a continuous and faithful action of $G$ on $C$, and $Y=\Spa(C,\OO_C)/\underline{G}$.

For uniqueness, note that if $(C^\prime,G^\prime)$ is another such pair, then $\Spa(C^\prime,\OO_{C^\prime})\times_Y \Spa(C,\OO_C)$ is affinoid pro-\'etale over both $\Spa(C,\OO_C)$ and $\Spa(C^\prime,\OO_{C^\prime})$; the choice of a point will then give an isomorphism $C\cong C^\prime$ commuting with the maps to $Y$. In this case, the group $G$ is also the same, as it is the group of automorphisms of the map $\Spa(C,\OO_C)\to Y$.
\end{proof}

In particular, it follows that in the situation of Proposition~\ref{prop:diamondpoint}, \'etale sheaves on $Y$ are equivalent to continuous discrete $G$-modules, and \'etale cohomology is equivalent to continuous $G$-cohomology.

\begin{definition}\label{def:cohomdimpoints} Let $Y$ be a quasiseparated diamond, and let $\ell$ be a prime. For each maximal point $y\in |Y|$, let $Y_y\subset Y$ be the corresponding subdiamond with $|Y_y| = \{y\}\subset |Y|$, and write $Y_y=\Spa(C_y,\OO_{C_y})/\underline{G}_y$. Then the $\ell$-cohomological dimension of $Y$ at $y$ is
\[
\cdim_\ell y = \cdim_\ell G_y\ .
\]
\end{definition}

We get the following bound on the cohomological dimension of spatial diamonds, which is an analogue of \cite[Corollary 2.8.3]{Huber}.

\begin{proposition}\label{prop:cohombounded} Let $Y$ be a spatial diamond, let $\ell$ be a prime, and let $\mathcal F$ be an $\ell$-power-torsion \'etale sheaf on $Y$ such that $\mathcal F|_U = 0$ for some open subset $U\subset Y$. Then
\[
H^i(Y,\mathcal F)=0
\]
for $i>\dim (Y\setminus U) + \sup_y \cdim_\ell y$, where $y\in Y$ runs through maximal points of $Y$.
\end{proposition}

\begin{remark} Even if we are only interested in the statement for $U=\emptyset$, the proof runs by an induction that involves the statement for general $U$. However, we will actually need the statement for general $U$ below.
\end{remark}

\begin{proof} We argue by induction on $\dim(Y\setminus U)$ (noting that if it is infinite, there is nothing to prove). Also, we may assume that $\mathcal F$ is $\ell$-torsion.

We may assume that there is a constructible sheaf $\mathcal F_0$ such that
\[
\mathcal F = \mathcal F_0 / j_{U!} \mathcal F_0|_U\ ,
\]
where $j_U: U\to Y$ denotes the open inclusion. Filtering $\mathcal F_0$, we may assume that there is a constructible locally closed subset $S\subset |Y|$ such that the restriction of $\mathcal F_0$ to $|Y|\setminus S$ is trivial (i.e., the stalk at all geometric points above $|Y|\setminus S$ vanishes), and for any strictly totally disconnected $f: X\to Y$, the pullback $f^\ast \mathcal F_0|_{f^{-1}(S)}$ is the constant sheaf associated with some finite abelian group $M$ with $\ell M=0$. In particular, $S=V\cap Z_0$ for a quasicompact open subset $V\subset |Y|$ and a constructible closed subset $Z_0\subset |Y|$. In this case,
\[
\mathcal F = \mathcal F_0 / j_{U!} \mathcal F_0|_U
\]
is concentrated (and constant after pullback to a strictly totally disconnected space) on the locally closed subset $S\setminus U = V\cap Z$, where $Z=Z_0\cap (Y\setminus U)$.

We also denote by $V\subset Y$ the corresponding open subdiamond, and let $j: V\hookrightarrow Y$ be the quasicompact open immersion. Then $\mathcal F=j_! \mathcal F_V$ for some sheaf $\mathcal F_V$ on $V$. Let $\tilde{\mathcal{F}}=j_\ast \mathcal F_V$. Note that $R^ij_\ast \mathcal F_V=0$ for $i>0$ by Lemma~\ref{lem:qcopenpushforward} below. Thus, $H^i(Y,\tilde{\mathcal{F}}) = H^i(V,\mathcal{F}_V)$. Moreover, we have an injection $\mathcal F\hookrightarrow \tilde{\mathcal F}$; let $\mathcal G$ be its cokernel. Then $\mathcal G$ is concentrated on $\overline{V\setminus U}\setminus V$. Note that $\dim(\overline{V\setminus U}\setminus V)<\dim(Y\setminus U)$: Indeed, $\overline{V\setminus U}\setminus V\subset Y\setminus U$, and any chain of specializations inside $\overline{V\setminus U}\setminus V$ can be prolonged by including a point of $V\setminus U$ as a proper generalization.

By induction, it follows that $H^i(Y,\mathcal G)=0$ for
\[
i>\dim(\overline{V\setminus U}\setminus V)+ \sup_y \cdim_\ell y\ ,
\]
and in particular for $i>\dim(Y\setminus U)+ \sup_y \cdim_\ell y - 1$. Thus, the long exact sequence
\[
\ldots \to H^{i-1}(Y,\mathcal G)\to H^i(Y,\mathcal F)\to H^i(Y,\tilde{\mathcal F})\to \ldots
\]
reduces us to proving the desired vanishing for $\tilde{\mathcal F}$. As $H^i(Y,\tilde{\mathcal F}) = H^i(V,\mathcal F_V)$, we can replace $Y$ by $V$ (and $\mathcal F$ by $\mathcal F_V$) and assume that there is a closed subset $Z\subset |Y|$ such that $\mathcal F$ is concentrated on $Z$, and for any strictly totally disconnected perfectoid space $f: X\to Y$, the pullback $f^\ast \mathcal F|_{f^{-1}(Z)}$ is the constant sheaf associated with some finite abelian group $M$ with $\ell M=0$.

Now we use the Leray spectral sequence for $g: Y_\et\to |Y|$. Note that $Rg_\ast \mathcal F|_{|Y|\setminus Z} = 0$, so the derived pushforward is concentrated on $Z$, which is a spectral space contained in $Y\setminus U$, and thus is of dimension $\leq \dim(Y\setminus U)$. By \cite[Corollary 4.6]{Scheiderer}, the cohomological dimension of $Z$ is bounded by $\dim (Y\setminus U)$.

Thus, it suffices to prove that $R^ig_\ast \mathcal F=0$ for $i>\sup_y \cdim_\ell y$. This can be checked on stalks, so we can assume that $|Y|$ is local. In that case, $|Y|$ is a totally ordered chain of specializations, and $Z\subset |Y|$ is a closed subset of finite dimension. It follows that $Z$ is a finite set of points totally ordered under specialization; let $\eta\in Z$ be the generic point. Then the generalizations of $\eta$ in $|Y|$ form a quasicompact open subspace corresponding to a quasicompact open subdiamond $j:V\hookrightarrow Y$. Now the adjunction map $\mathcal F\to j_\ast j^\ast \mathcal F$ is an isomorphism (as can be checked after pullback to a strictly totally disconnected cover) by our assumption on $\mathcal F|_Z$. Using Lemma~\ref{lem:qcopenpushforward} again, we can assume replace $Y$ by $V$, and so assume that $Z$ is the closed point of $|Y|$. Finally, we are reduced to Proposition~\ref{prop:cohomdimdiamondpoint} below.
\end{proof}

\begin{lemma}\label{lem:qcopenpushforward} Let $j: U\to Y$ be a quasicompact injection of locally spatial diamonds. Then for any sheaf of abelian groups $\mathcal F$ on $U_\et$, one has $R^i j_{\et\ast} \mathcal F = 0$ for $i>0$.
\end{lemma}

\begin{remark}\label{rem:qcsepqproetpushforward} In fact, this holds true more generally for quasicompact separated quasi-pro-\'etale morphisms $j: U\to X$. The proof reduces to the case where $X$ is strictly totally disconnected, in which case $U$ is also strictly totally disconnected by Lemma~\ref{lem:proetaleoverwlocal}. Then it follows from vanishing of \'etale cohomology on strictly totally disconnected spaces.
\end{remark}

\begin{proof} By Corollary~\ref{cor:qcqsetbasechange}, and as the result can be checked on stalks, we can assume that $Y=\Spa(C,C^+)$ for an algebraically closed nonarchimedean field $C$ with an open and bounded valuation subring $C^+\subset C$, and we need to check that $H^i(U_\et,\mathcal F)=0$ for $i>0$. But $U_\et$ and $|U|$ define equivalent topoi, and $|U|$ has a unique closed point, so $H^i(U_\et,\mathcal F)=H^i(|U|,\mathcal F)=0$ for $i>0$, as desired.
\end{proof}

\begin{proposition}\label{prop:cohomdimdiamondpoint} Let $Y$ be a spatial diamond such that $|Y|$ is local with closed point $s\in |Y|$ and complement $U=Y\setminus \{s\}\subset Y$. Then for any $\ell$-torsion \'etale sheaf $\mathcal F$ on $Y_\et$ such that $\mathcal F|_U=0$,
\[
H^i(Y_\et,\mathcal F)=0
\]
for any $i>\cdim_\ell \eta$, where $\eta\in |Y|$ is the unique generic point.
\end{proposition}

\begin{proof} By assumption, there is a surjective quasi-pro-\'etale map $f: \Spa(C,C^+)\to Y$, where $C$ is an algebraically closed nonarchimedean field, and $C^+\subset C$ is an open and bounded valuation subring.

As $\Spa(C,C^+)$ is separated, the map $f$ is separated, and thus $R=\Spa(C,C^+)\times_Y \Spa(C,C^+)$ is quasicompact separated and pro-\'etale over $\Spa(C,C^+)$, i.e.~$R$ is affinoid pro-\'etale over $\Spa(C,C^+)$ by Lemma~\ref{lem:proetaleoverwlocal}.

Note that the map $|\Spa(C,C^+)|\to |Y|$ is bijective. For any point $y\in |Y|$, let $\tilde{y}\in |\Spa(C,C^+)|$ be the unique lift, which corresponds to a valuation subring $C^+_y\subset \OO_C$ containing $C^+$. Then the fiber of $|R|_{\tilde{y}}$ of $|R|$ over $\tilde{y}$ identifies with the sections of $f$ over $\Spa(C,C^+_y)$. It is a profinite set, which in fact is a profinite group $G_y$ by the equivalence relation structure. If $y^\prime$ is a generalization of $y$, there is a natural generalization map $|R|_{\tilde{y}}\to |R|_{\tilde{y}^\prime}$ as all maps are uniquely generalizing, so one gets an inclusion of profinite groups $G_y\subset G_{y^\prime}$. In particular, $G_s$ is a closed subgroup of $G_\eta$, and so $\cdim_\ell G_s\leq \cdim_\ell G_\eta=\cdim_\ell \eta$.

It remains to see that the sheaf $\mathcal F$ can be identified with a discrete continuous $G_s$-module, and its cohomology with continuous $G_s$-cohomology. But giving $\mathcal F$ is equivalent to giving a sheaf on $|\Spa(C,C^+)|$ concentrated at the closed point, i.e.~an abelian group, and the descent data over $R$ amount to a continuous $G_s$-action. Computing cohomology via the Cartan--Leray spectral sequence for the covering $g: \Spa(C,C^+)\to Y$ gives the complex of continuous cochains, as desired.
\end{proof}

Moreover, we need the following result bounding the cohomological dimension of points.

\begin{proposition}\label{prop:cohomdimpointbound} Let $f: Y\to \Spa(C,C^+)$ be a map of locally spatial diamonds, and let $y\in |Y|$ be a maximal point. Then, for all $\ell\neq p$, one has
\[
\cdim_\ell y\leq \dimtrg f\ .
\]
\end{proposition}

\begin{proof} We may replace $Y$ by $Y_y$ and $\Spa(C,C^+)$ by $\Spa(C,\OO_C)$. Then $Y=\Spa(C^\prime,\OO_{C^\prime})/\underline{G}$ for a profinite group $G$ acting continuously and faithfully on $C^\prime$, fixing $C$ pointwise. In particular, $G$ acts continuously on the quotient
\[
C^{\prime\times} / (1+C^{\prime\circ\circ})
\]
of $C^{\prime\times}$. Let $P\subset G$ be the normal closed subgroup which acts trivially; this is an analogue of the wild inertia subgroup. First, we note that $P$ is a pro-$p$-group. This follows from Lemma~\ref{lem:wildinertiaprop} below. Thus, $\cdim_\ell G=\cdim_\ell G/P$ for $\ell\neq p$.

Let $k^\prime/k$ be the extension of residue fields of $C^\prime/C$, and let $\Gamma\subset \Gamma^\prime\subset \mathbb R_{>0}$ be the value groups. Then the group $C^{\prime\times} / (1+C^{\prime\prime\times})$ is an extension
\[
1\to k^{\prime\times}\to C^{\prime\times} / (1+C^{\prime\prime\times})\to \Gamma^\prime\to 1\ .
\]
As $\Gamma^\prime\subset \mathbb R_{>0}$, the action of $G/P$ on $\Gamma^\prime$ is trivial. Considering the action of $G/P$ on $k^\prime$, let $I$ be the pointwise stabilizer of $k^\prime$; this is an analogue of the inertia group. Then $\overline{G}=G/I$ acts faithfully and continuously on the discrete field $k^\prime$, fixing $k$ pointwise. It follows that $k^\prime_0 = k^{\prime\overline{G}}$ is a perfect field with algebraic closure $k^\prime$, and $\Gal(k^\prime/k^\prime_0) = \overline{G}$. Thus,
\[
\cdim_\ell \overline{G}\leq \trdeg(k^\prime_0/k) = \trdeg(k^\prime/k)\ .
\]

Finally, we have to understand the ``tame inertia'' group $I/P$. This acts through maps $\Gamma^\prime\to k^\times$ which are trivial on $\Gamma$, i.e.~through maps $\Gamma^\prime/\Gamma\to k^\times$. Moreover, as the elements are topologically nilpotent, the image of $\Gamma^\prime/\Gamma\to k^\times$ lands in the roots of unity; assuming that $\Gamma^\prime/\Gamma$ is finitely-dimensional over $\mathbb Q$, this gives an embedding
\[
I/P\hookrightarrow \Hom(\Gamma^\prime/\Gamma,\mu_\infty(k))\cong (\mathbb A_f^p)^{\dim_{\mathbb Q} \Gamma^\prime/\Gamma}\ .
\]
This implies that
\[
\cdim_\ell I/P\leq \dim_{\mathbb Q} \Gamma^\prime/\Gamma\ .
\]
Therefore we get the desired inequality
\[
\cdim_\ell G = \cdim_\ell G/P\leq \cdim_\ell \overline{G} + \cdim_\ell I/P\leq \trdeg(k^\prime/k) + \dim_{\mathbb Q} \Gamma^\prime/\Gamma\leq \widetilde{\trc}(C^\prime/C)\ ,
\]
using \cite[VI.10.3 Corollary 1]{BourbakiCommutativeAlgebra} in the final inequality (which also shows that indeed $\Gamma^\prime/\Gamma$ is finite-dimensional in case $\trc(C^\prime/C)<\infty$); here, we use that rationalized value groups and transcendence degrees of residue fields do not change under completion, and only increase after passage to further extensions.
\end{proof}

\begin{lemma}\label{lem:wildinertiaprop} Let $C$ be an algebraically closed complete nonarchimedean field of characteristic $p$, and let $\gamma$ be a continuous automorphism of $C$ such that $\gamma^{n!}$ converges (pointwise) to the identity for $n\to \infty$. Assume that $\gamma$ acts trivially on $C^\times / (1+C^{\circ\circ})$. Then $\gamma^{p^n}$ converges (pointwise) to the identity for $n\to \infty$.
\end{lemma}

\begin{proof} Let $G$ be the subgroup of the continuous automorphisms of $C$ generated by $\gamma$; then $G$ is a cyclic profinite group, and we have to see that $G$ is pro-$p$. If $G$ is not pro-$p$, then take some prime $\ell\neq p$ dividing the pro-order of $G$; replacing $G$ by a pro-$\ell$-Sylow subgroup (and $\gamma$ by a generator of it), we can assume that $G$ is pro-$\ell$. Then $\gamma^{\ell^n}$ converges pointwise to the identity for $n\to \infty$, and we have to see that $\gamma$ is the identity. For any $x\in C^\times$, the quotient $\frac{\gamma(x)}{x}\in C^\times$ lies in $1+C^{\circ\circ}$ by assumption. If it is not equal to $1$, pick some nonzero $y\in C^{\circ\circ}$ such that
\[
\frac{\frac{\gamma(x)}{x}-1}y\in \OO_C^\times\ .
\]
Then the association
\[
g\in G\mapsto \frac{\frac{g(x)}{x}-1} y\in \OO_C\to k=\OO_C/C^{\circ\circ}
\]
defines a continuous group homomorphism $G\to k$. As $k$ is $p$-torsion and $G$ is pro-$\ell$, this implies that the map is $0$, which contradicts the assumption on $y$. Thus, for all $x\in C^\times$, one has $\gamma(x)=x$, as desired.
\end{proof}

\section{Proper pushforward}\label{sec:properpushforward}

In this section, we finally define the functor $Rf_!$.

\begin{convention} In this section, we will always work with unbounded derived categories and therefore assume that $f$ is representable in locally spatial diamonds with locally $\dimtrg f<\infty$, which will locally guarantee the assumptions on bounded cohomological dimension in Theorem~\ref{thm:properbasechangepartpropvsopen} and Proposition~\ref{prop:Rfastsimple}. Everything below works as well for $D^+$ without the assumption $\dimtrg f<\infty$. Moreover, we assume throughout that $n\Lambda=0$ for some $n$ prime to $p$.
\end{convention}

As usual, one needs to restrict attention to morphisms admitting a compactification.

\begin{definition}\label{def:compactifiable} A morphism $f: Y^\prime\to Y$ of v-stacks is compactifiable if it can be written as a composite of an open immersion and a partially proper morphism.
\end{definition}

A key difference to the world of schemes is the presence of a canonical compactification.

\begin{proposition}\label{prop:compactifiable} Let $f: Y^\prime\to Y$ and $\tilde{Y}\to Y$ be morphisms of v-stacks, with pullback $\tilde{f}: \tilde{Y}^\prime = Y^\prime\times_Y \tilde{Y}\to \tilde{Y}$.
\begin{altenumerate}
\item The morphism $f$ is compactifiable if and only if it is separated and the natural map $Y^\prime\to \overline{Y^\prime}^{/Y}$ is an open immersion.
\item If $f$ is compactifiable, then $\tilde{f}$ is compactifiable.
\item If $\tilde{f}$ is compactifiable and $\tilde{Y}\to Y$ is a surjective map of v-stacks, then $f$ is compactifiable.
\item If $Y_1\to Y_2$ and $Y_2\to Y_3$ are compactifiable morphisms of v-stacks, then the composite $Y_1\to Y_3$ is compactifiable.
\item If $f$ is separated and representable in locally spatial diamonds and $Y^\prime$ admits a cover by open subfunctors $V\subset Y^\prime$ such that $f|_V: V\to Y$ is compactifiable, then $f$ is compactifiable. Moreover, for any such open subfunctor $V\subset Y^\prime$, the composite $V\to Y^\prime\to \overline{Y^\prime}^{/Y}$ is an open immersion.
\item If $f$ is separated and \'etale, then $f$ is compactifiable.
\item If $f$ is separated and representable in locally spatial diamonds and there is a separated surjective map $g: Z\to Y^\prime$ of v-stacks that after pullback to strictly totally disconnected spaces has local sections and such that $f\circ g$ is compactifiable, then $f$ is compactifiable.
\item If $g: Y\to Z$ is a separated map of v-stacks such that $g\circ f$ is compactifiable, then $f$ is compactifiable.
\end{altenumerate}
\end{proposition}

In particular, the property of being compactifiable is v-local on the target by (iii), and for separated maps which are representable in locally spatial diamonds, it is \'etale local on the source by (vii).\footnote{In a previous version, it was claimed that in (vii), the local splitting condition on $g$ is not required, but there are counterexamples to that version. We do not know the optimal conditions on $g$; in particular, is it true when $g$ is universally open? This question is relevant to the question whether in Proposition~\ref{prop:smooth2outof3}, the condition ``$f$ is compactifiable'' is automatic.}

\begin{proof} For part (i), one direction is clear as $\overline{Y^\prime}^{/Y}$ is partially proper. For the converse, assume that $f$ is compactifiable, and let $Y^\prime\hookrightarrow Z$ be an open immersion into a partially proper $Z\to Y$. As open immersions and partially proper morphisms are separated, we see that $f$ is separated. Moreover, we get a map $\overline{Y^\prime}^{/Y}\to Z$, which is still an injection. Thus $Y^\prime\to \overline{Y^\prime}^{/Y}$ is a pullback of the open immersion $Y^\prime\to Z$, and so an open immersion itself.

Now parts (ii) and (iii) follow easily from (i). In Part (iv), we may now assume that $Y_3$, and thus all $Y_i$, are separated. Then $Y_2\to \overline{Y_2}^{/Y_3}$ and $Y_1\to \overline{Y_1}^{/Y_2}$ are open immersions, thus so is
\[
Y_1\to \overline{Y_1}^{/Y_2}=\overline{Y_1}\times_{\overline{Y_2}} Y_2\to \overline{Y_1}\times_{\overline{Y_2}} \overline{Y_2}^{/Y_3}=\overline{Y_1}\times_{\overline{Y_3}} Y_3 = \overline{Y_1}^{/Y_3}\ ,
\]
as desired.

In part (v), it is enough to prove that for any open subspace $V\subset Y^\prime$, the composite $V\to \overline{Y^\prime}^{/Y}$ is an open immersion. For this, we can assume $Y=\overline{Y^\prime}^{/Y}$, and we can work v-locally on $Y$, so we can assume $Y=\Spa(A,A^+)$ is strictly totally disconnected. In that case, by Proposition~\ref{prop:vinjindrepr}, $Y^\prime$ is a filtered union of open subspaces (of $Y^\prime$; not a priori of $Y$!) of the form $\Spa(A,(A^+)^\prime)$ for certain varying rings of integral elements $(A^+)^\prime\subset A$ containing $A^+$. One can assume that the subspaces $V$ are quasicompact, in which case $V\subset \Spa(A,(A^+)^\prime)\subset Y^\prime$ for some such $(A^+)^\prime$, and for any $x\in V$ there are $f_1,\ldots,f_n,g\in A$ generating the unit ideal such that the rational subset $U=\{|f_i|\leq |g|\}\subset Y$ satisfies $x\in \Spa(A,(A^+)^\prime)\cap U\subset V$. But then $U\subset \overline{V} = \overline{V}^{/Y}$. As $V\subset \overline{V}^{/Y}$ is an open immersion by assumption, this implies that $U\cap V\subset U\cap \overline{V}=U$ is an open immersion. In other words, $x\in U\cap V\subset U\subset Y$ is an open neighborhood of $x$ in $Y$ which is contained in $Y^\prime$. These cover $Y^\prime$, so that $Y^\prime$ is open in $Y$, as desired.

For part (vi), we may assume (by (iii)) that $Y=X$ is a strictly totally disconnected perfectoid space. Then $Y^\prime=X^\prime$ is a perfectoid space which is separated and \'etale over $X$. By (v), we may assume that $Y^\prime$ is quasicompact. In that case, $Y^\prime$ decomposes as a disjoint union of quasicompact open subspaces of $Y$, so the result is clear.

For part (vii), we start as in part (v): We can assume that $Y=\overline{Y^\prime}^{/Y}$, and we need to see that $f: Y^\prime\hookrightarrow Y$ is an open immersion. This can be checked v-locally on $Y$, so we can assume that $Y=\Spa(A,A^+)$ is a strictly totally disconnected space, so that $Y^\prime$ is a locally spatial diamond. By part (v), we can assume that $Y^\prime$ is quasicompact, and then it is itself strictly totally disconnected. But then $g$ is locally split by assumption, so the result follows from (v).

In part (viii), we have injections
\[
Y^\prime\hookrightarrow \overline{Y^\prime}^{/Y}\hookrightarrow \overline{Y^\prime}^{/Z}\ .
\]
If the composite is an open immersion, then so is the first map (as it is a pullback of the composite). 
\end{proof}

Now if $f: Y^\prime\to Y$ is a quasicompact compactifiable map of small v-stacks, we can write $f$ as a composite of an open immersion $j: Y^\prime\hookrightarrow \overline{Y^\prime}^{/Y}$ and the proper map $\overline{f}^{/Y}: \overline{Y^\prime}^{/Y}\to Y$. In that case, we will momentarily define $Rf_!$ as the composite $R\overline{f}^{/Y}_\ast \circ j_!$.

However, we will be interested in the case where $f$ is not necessarily quasicompact. In that case $\overline{f}^{/Y}: \overline{Y^\prime}^{/Y}\to Y$ is only partially proper, and one should restrict to sections with proper support. In a fully derived setting, this is somewhat nontrivial to do. But we are already restricting to maps which are representable in locally spatial diamonds, so we can get around this difficulty. Namely, if $Y$ and $Y^\prime$ are locally spatial diamonds, and say that $Y$ is even spatial, then one can write $Y^\prime$ as an increasing union of quasicompact open subspaces $V\subset Y^\prime$, and $Rf_!$ should be defined as the filtered colimit of $R(f|_V)_!$, which are defined as before. One can then descend this construction to the case where $f$ is merely representable in locally spatial diamonds, as $Rf_!$ satisfies base change.

We will now execute this line of thoughts, one step at a time.

\begin{definition}\label{def:Rfshriekquasicompact} Let $f: Y^\prime\to Y$ be a compactifiable map of small v-stacks which is representable in spatial diamonds with $\dimtrg f<\infty$, and assume $n\Lambda=0$ for some $n$ prime to $p$. Write $f$ as the composite of the open immersion $j: Y^\prime\hookrightarrow \overline{Y^\prime}^{/Y}$ and the proper map $\overline{f}^{/Y}: \overline{Y^\prime}^{/Y}\to Y$. Then we define
\[
Rf_! = R\overline{f}^{/Y}_\ast\circ j_!: D_\et(Y^\prime,\Lambda)\to D_\et(Y,\Lambda)\ .
\]
\end{definition}

Our first aim is to prove that this definition is well-behaved on unbounded derived categories. For this reason, we need to know that $R\overline{f}^{/Y}_\ast$ has bounded cohomological dimension.

\begin{theorem}\label{thm:dimtrgfinitecohomdim} Let $f: Y^\prime\to Y$ be a compactifiable map of small v-stacks which is representable in spatial diamonds with $\dimtrg f<\infty$, and assume $n\Lambda=0$ for some $n$ prime to $p$. Write $f$ as the composite of the open immersion $j: Y^\prime\hookrightarrow \overline{Y^\prime}^{/Y}$ and the proper map $\overline{f}^{/Y}: \overline{Y^\prime}^{/Y}\to Y$. Then $R\overline{f}^{/Y}_\ast$ has bounded cohomological dimension; more precisely, for all $A\in D_\et(\overline{Y^\prime}^{/Y},\Lambda)$ concentrated in degree $0$, one has
\[
R^i \overline{f}^{/Y}_\ast A=0
\]
for $i>3\dimtrg f$.

In particular, as $j_\ast$ is exact (by Remark~\ref{rem:qcsepqproetpushforward}), also
\[
R^i f_\ast A=0
\]
for all $A\in D_\et(Y^\prime,\Lambda)$ concentrated in degree $0$ and $i>3\dimtrg f$, so that $Rf_\ast$ has bounded cohomological dimension.
\end{theorem}

The constant $3$ may well be an artifact of the proof; if $\overline{f}^{/Y}$ is still representable in spatial diamonds, it can be replaced by $2$, as expected.

\begin{proof} By Proposition~\ref{prop:Rfastsimple}, we can assume that $Y$ is a strictly totally disconnected space. In fact, it is enough to check a statement about stalks, so we can reduce to the case $Y=\Spa(C,C^+)$, where $C$ is an algebraically closed nonarchimedean field and $C^+\subset C$ an open and bounded valuation subring, and we only need to check the statement on global sections. In that case, we need to prove that
\[
H^i(\overline{Y^\prime}^{/Y},A)=0
\]
for $i> 3\dimtrg f$. Let $s\in Y$ be the closed point and $U=Y\setminus \{s\}$, $V=f^{-1}(U)\subset \overline{Y^\prime}^{/Y}$, with inclusion $j: V\hookrightarrow \overline{Y^\prime}^{/Y}$. By Theorem~\ref{thm:properbasechangepartpropvsopen}, we know that
\[
R\Gamma(\overline{Y^\prime}^{/Y},j_!j^\ast A)=0\ .
\]
Thus, replacing $A$ by the cone of $j_!j^\ast A\to A$, we can assume that $j^\ast A=0$.

Now, if $\overline{Y^\prime}^{/Y}$ is a spatial diamond, note that $\dim(\overline{Y^\prime}^{/Y}\setminus V) = \dim (\overline{f}^{/Y})^{-1}(s)\leq \dim \overline{f}^{/Y}\leq \dimtrg \overline{f}^{/Y} = \dimtrg f$ and $\cdim_\ell y^\prime\leq \dimtrg f$ for all maximal points $y^\prime\in \overline{Y^\prime}^{/Y}$ (which all lie in $Y^\prime$) and $\ell\neq p$ by Proposition~\ref{prop:cohomdimpointbound}, so the result follows from Proposition~\ref{prop:cohombounded}.

Unfortunately, for a compactifiable map $f: Y^\prime\to Y$ of spatial diamonds, we do not know whether the canonical compactification $\overline{Y^\prime}^{/Y}$ is spatial. However, by Proposition~\ref{prop:qcsepdiamondmaptoberkovich} and Proposition~\ref{prop:berkovichspacemaxhausdorffquot}, there is a compact Hausdorff space $T$ (given as the maximal Hausdorff quotient of $|\overline{Y^\prime}^{/Y}|$, which is also the maximal Hausdorff quotient of the spectral space $|Y^\prime\times_{\Spa(C,C^+)} \Spa(C,\OO_C)|$, as they share the same maximal points) and a map $\overline{Y^\prime}^{/Y}\to \underline{T}$ which is representable in locally spatial diamonds. In particular, one gets a map
\[
g: \overline{Y^\prime}^{/Y}\to \underline{T}\times Y
\]
of proper diamonds over $Y$, which is thus necessarily proper; moreover, it is representable in spatial diamonds. The previous arguments imply that
\[
R^ig_\ast A=0
\]
for $i>2\dimtrg g=2\dimtrg f$. Let $h: \underline{T}\times Y\to Y$ be the projection. It remains to see that for any $B\in D_\et(\underline{T}\times Y,\Lambda)$ concentrated in degree $0$ and with trivial restriction to $\underline{T}\times U$, we have $H^i(\underline{T}\times Y,B)=0$ for $i>\dimtrg f$.

In fact, we claim that $H^i(\underline{T}\times Y,B)=0$ for $i>\dim f$. This follows from Proposition~\ref{prop:compacthausdorffDet} below, noting that if $B$ is trivial on $\underline{T}\times U$, one gets
\[
H^i(\underline{T}\times Y,B)=H^i(T,B|_{T\times \{s\}})\ ,
\]
and by Proposition~\ref{prop:cohomberkovich}, we have $H^i(T,-) = H^i(|Y^\prime\times_{\Spa(C,C^+)} \Spa(C,\OO_C)|,-)$, the latter of which is a spectral space of dimension $\leq \dim f$, so we conclude using \cite[Corollary 4.6]{Scheiderer}.
\end{proof}

In the final step, we used a result on proper diamonds over strictly totally disconnected spaces of $\dimtrg =0$.

\begin{proposition}\label{prop:0dimpartpropermap} Let $X$ be a strictly totally disconnected perfectoid space. Then the category of proper diamonds $f: Y\to X$ with $\dimtrg f=0$ is equivalent to the category of compact Hausdorff spaces over $\pi_0 X$, via sending a compact Hausdorff space $T\to \pi_0 X$ to $X\times_{\underline{\pi_0 X}} \underline{T}$.
\end{proposition}

As a special case, if $X=\Spa(C,C^+)$, then the category of proper diamonds $f: Y\to X$ of $\dimtrg f=0$ is equivalent to the category of compact Hausdorff spaces. This result was suggested to the author by M.~Rapoport.

\begin{proof} Let $f: Y\to X$ be any proper map of diamonds with $\dimtrg f=0$. Let $\tilde{X}\to Y$ be a quasi-pro-\'etale surjection from a strictly totally disconnected space. Then $\overline{\tilde{X}}^{/X}\to X$ is a proper map of strictly totally disconnected perfectoid spaces which induces an isomorphism on completed residue fields (by the assumption $\dimtrg f=0$); this implies that $\overline{\tilde{X}}^{/X} = X\times_{\underline{\pi_0 X}} \underline{S}$ for some profinite set $S\to \pi_0 X$. Write $X_S = X\times_{\underline{\pi_0 X}} \underline{S}$. We get a surjective map $X_S\to Y$ over $X$. The equivalence relation $R=X_S\times_Y X_S$ is again proper and pro-\'etale over $X$, so $R=X\times_{\underline{\pi_0 X}} \underline{S^\prime}$ for some profinite $S^\prime\subset S\times S$. Moreover, the maps $s,t: R\to X_S$ over $X$ are given by maps $S^\prime\to S$. Now $T=S/S^\prime$ is a compact Hausdorff space, and $Y=X\times_{\underline{\pi_0 X}} \underline{T}$. Using Example~\ref{ex:compacthausdorff}, we see that this gives the desired equivalence of categories.
\end{proof}

In this situation, we also need a characterization of the category $D^+_\et(Y,\Lambda)$.

\begin{proposition}\label{prop:compacthausdorffDet} Let $X$ be a strictly totally disconnected perfectoid space, and let $f: Y\to X$ be a proper map of diamonds of $\dimtrg f=0$, so that by Proposition~\ref{prop:0dimpartpropermap}, one has $Y=X\times_{\underline{\pi_0 X}} \underline{T}$ for some compact Hausdorff space $T\to \pi_0 X$. Then pullback under $Y\to |Y|$ defines an equivalence
\[
D^+(|Y|,\Lambda)\buildrel\simeq\over\to D^+_\et(Y,\Lambda)\ ,
\]
where $D^+(|Y|,\Lambda)$ is the derived category of sheaves of $\Lambda$-modules on the topological space $|Y|$.
\end{proposition}

\begin{proof} There is a natural map of topoi $t: Y_v\to |Y|$. It follows from the definitions that if $\mathcal F$ is a sheaf of $\Lambda$-modules on $|Y|$, then $t^\ast \mathcal F$ lies in $D^+_\et(Y_v,\Lambda)$. Thus, it is enough to prove the following assertions.
\begin{altenumerate}
\item For any sheaf of abelian groups $\mathcal F$ on $|Y|$, the adjunction map $\mathcal F\to Rt_\ast t^\ast \mathcal F$ is an isomorphism.
\item If $\mathcal G$ is a small sheaf of abelian groups on $Y_v$ such that for some cover $g: \tilde{Y}\to Y$ by some strictly totally disconnected perfectoid space $\tilde{Y}$, the restriction $g^\ast \mathcal G$ lies in $\tilde{Y}_\et^\sim\subset \tilde{Y}_v^\sim$, then the adjunction map $t^\ast Rt_\ast \mathcal G\to \mathcal G$ is an isomorphism.
\end{altenumerate}

First, note that for any small v-sheaf $Y$, if $C\in D^+_\et(Y,\Lambda)$, then for any cofiltered inverse system of qcqs diamonds $Z_i\to Y$ with inverse limit $Z$, one has
\[
R\Gamma(Z,C)=\varinjlim_i R\Gamma(Z_i,C)\ .
\]
This follows from Proposition~\ref{prop:etcohomlim} if all $Z_i$ are spatial; in general, writing all $Z_i$ compatibly as quotients of affinoid perfectoid spaces $X_i$ as in the proof of Lemma~\ref{lem:diamondlimit}, the equivalence relations $R_i=X_i\times_{Z_i} X_i$ are spatial diamonds by Proposition~\ref{prop:smallvsheaf}. The result then follows by descent from the case of the limits of the $X_i$, the $R_i$, the $R_i\times_{X_i} R_i$, etc.~.

To check claim (i) above, we have to check an equality of stalks. But note that if $y\in |Y|$, then the open neighborhoods of $y$ are cofinal with the neighborhoods of the form $U\times_{\underline{\pi_0 X}} \underline{S}$, where $U\subset X$ is a quasicompact open neighborhood of the image of $y$ in $X$, and $S\subset T$ is the closure of an open neighborhood of the image of $y$ in $T$; indeed, this follows from the similar observation for compact Hausdorff spaces. Note that such $U\times_{\underline{\pi_0 X}} \underline{S}$ are qcqs diamonds which form a cofiltered inverse system with inverse limit the set of generalizations of $Y_y$ in $|Y|$, which is of the form $\Spa(C(y),C(y)^+)$, where $C(y)$ is algebraically closed, and $C(y)^+\subset C(y)$ an open and bounded valuation subring. Thus, the stalk of $Rt_\ast t^\ast \mathcal F$ can also be computed as a direct limit of the cohomologies of $t^\ast \mathcal F$ over these neighborhoods $U\times_{\underline{\pi_0 X}} \underline{S}$, which by the remark above reduces to
\[
R\Gamma(Y_y,t^\ast \mathcal F)\ .
\]
But this is given by the stalk of $\mathcal F$ at $y$, as desired.

Now for claim (ii), we have to similarly check a claim on stalks, which again follows from the cofinality of open and qcqs neighborhoods.
\end{proof}

Now we come back to the analysis of the functor $Rf_!$.

\begin{proposition}\label{prop:Rfshriekquasicompactbasechange} Let $f: Y^\prime\to Y$ be a compactifiable map of small v-stacks which is representable in spatial diamonds with $\dimtrg f<\infty$, and assume $n\Lambda=0$ for some $n$ prime to $p$. Let $g: \tilde{Y}\to Y$ be any map of small v-stacks, with base change $\tilde{f}: \tilde{Y}^\prime = Y^\prime\times_Y \tilde{Y}\to \tilde{Y}$ and $g^\prime: \tilde{Y}^\prime\to Y^\prime$.

There is a natural base change equivalence
\[
g^\ast Rf_!\simeq R\tilde{f}_! g^{\prime\ast}
\]
of functors $D_\et(Y^\prime,\Lambda)\to D_\et(\tilde{Y},\Lambda)$.
\end{proposition}

\begin{proof} This follows by combining Theorem~\ref{thm:dimtrgfinitecohomdim}, Proposition~\ref{prop:Rfastsimple} and Proposition~\ref{def:Rfshrieketale}.
\end{proof}

\begin{proposition}\label{prop:Rfshriekquasicompactcomp} Let $g: Y^{\prime\prime}\to Y^\prime$ and $f: Y^\prime\to Y$ be compactifiable maps of small v-stacks which are representable in spatial diamonds with $\dimtrg f,\dimtrg g <\infty$, and assume $n\Lambda=0$ for some $n$ prime to $p$. Then there is a natural equivalence
\[
Rf_!\circ Rg_!\simeq R(f\circ g)_!
\]
of functors $D_\et(Y^{\prime\prime},\Lambda)\to D_\et(Y,\Lambda)$.
\end{proposition}

\begin{proof} Consider the following diagram
\[\xymatrix{
Y^{\prime\prime}\ar[r]^{j_1}\ar[dr]^g & \overline{Y^{\prime\prime}}^{/Y^\prime}\ar[r]^{j_2}\ar[d]^{g_1} & \overline{Y^{\prime\prime}}^{/Y}\ar[d]^{g_2}\\
& Y^\prime\ar[r]^{j_3}\ar[dr]^f & \overline{Y^\prime}^{/Y}\ar[d]^{f_1}\\
&& Y\ .
}\]
Then
\[
Rf_! Rg_! = Rf_{1\ast} j_{3!} Rg_{1\ast} j_{1!}
\]
and
\[
R(f\circ g)_! = Rf_{1\ast} Rg_{2\ast} j_{2!} j_{1!}\ .
\]
However, by Theorem~\ref{thm:properbasechangepartpropvsopen} and Theorem~\ref{thm:dimtrgfinitecohomdim} (applied to $Y^{\prime\prime}\to \overline{Y^\prime}^{/Y}$), we have $j_{3!} Rg_{1\ast} = Rg_{2\ast} j_{2!}$, as desired.
\end{proof}

In case of overlap, the current definition of $Rf_!$ agrees with Definition~\ref{def:Rfshrieketale}.

\begin{proposition}\label{prop:compRfshrieketaleqc} Let $f: Y^\prime\to Y$ be a quasicompact separated \'etale map of small v-stacks, and assume $n\Lambda=0$ for some $n$ prime to $p$. Then $Rf_!$ as defined in Definition~\ref{def:Rfshrieketale} agrees with $Rf_!$ as defined in Definition~\ref{def:Rfshriekquasicompact}.
\end{proposition}

\begin{proof} Let $Rf_!^\et$ denote the one from Definition~\ref{def:Rfshrieketale} temporarily, which is the left adjoint of $f^\ast$. There is a natural transformation
\[
Rf_!^\et\to Rf_!
\]
adjoint to the map $\id\to f^\ast Rf_!$ coming from the base change identity $f^\ast Rf_! = R\pi_{2!} \pi_1^\ast$, where $\pi_1,\pi_2: Y^\prime\times_Y Y^\prime\to Y^\prime$ are the two projections, and the map
\[
\id = R\pi_{2!} R\Delta_! \Delta^\ast \pi_1^\ast \to R\pi_{2!} \pi_1^\ast\ ,
\]
using that $\Delta: Y^\prime\hookrightarrow Y^\prime\times_Y Y^\prime$ is an open and closed immersion.

To check that this is an equivalence, we use that both functors commute with any base change to reduce to the case that $Y$ is strictly totally disconnected. In that case, $Y^\prime$ decomposes into a disjoint union of quasicompact open subspaces of $Y$, so we can reduce to that case, in which the claim is clear.
\end{proof}

We also get the projection formula.

\begin{proposition}\label{prop:projectionformulaqc} Let $f: Y^\prime\to Y$ be a compactifiable map of small v-sheaves which is representable in spatial diamonds with $\dimtrg f<\infty$, and let $\Lambda$ be a ring with $n\Lambda=0$ for some $n$ prime to $p$. Then there is a functorial isomorphism
\[
Rf_! B\dotimes_\Lambda A\simeq Rf_!(B\dotimes_\Lambda f^\ast A)
\]
in $B\in D_\et(Y^\prime,\Lambda)$ and $A\in D_\et(Y,\Lambda)$.
\end{proposition}

\begin{proof} If $f=j$ is an open immersion, there is a natural map
\[
j_!(B\dotimes_\Lambda j^\ast A)\to j_! B\dotimes_\Lambda A
\]
adjoint to
\[
B\dotimes_\Lambda j^\ast A= j^\ast j_! B\dotimes_\Lambda A = j^\ast(j_! B\dotimes_\Lambda A)\ .
\]
To check whether this is an equivalence, we can assume that $Y$ is a strictly totally disconnected space (as all operations commute with pullback). In that case $D_\et(Y,\Lambda) = D_\et(|Y|,\Lambda)$; let $i: Z\hookrightarrow |Y|$ be the closed complement. As the map is clearly an isomorphism away from $Z$ and the left-hand side vanishes on $Z$, the statement follows from
\[
i^\ast(j_! B\dotimes_\Lambda A) = i^\ast j_! B\dotimes_\Lambda i^\ast A = 0
\]
as $i^\ast j_! B=0$.

In general, we now get a map
\[
Rf_! B\dotimes_\Lambda A = R\overline{f}^{/Y}_\ast j_! B\dotimes_\Lambda A\to R\overline{f}^{/Y}_\ast(j_! B\dotimes_\Lambda \overline{f}^{/Y\ast} A)\simeq R\overline{f}^{/Y}_\ast j_!(B\dotimes_\Lambda j^\ast \overline{f}^{/Y\ast} A) = Rf_!(B\dotimes_\Lambda f^\ast A)\ ,
\]
where the middle map comes from a standard adjunction. To check whether this is an isomorphism, we may by Proposition~\ref{prop:Rfastsimple} assume that $Y$ is a strictly totally disconnected space. In that case, we have to prove a statement on stalks, so we can assume that $Y=\Spa(C,C^+)$ is connected, and it is enough to check the statement on global sections. Let $s\in Y$ be the closed point, $j_U: U=Y\setminus \{s\}\hookrightarrow Y$ the open immersion. If $A=j_{U!} A_U$ for some $A_U\in D_\et(U,\Lambda)$, then
\[
R\Gamma(Y,Rf_! B\dotimes_\Lambda j_{U!} A_U) = 0 = R\Gamma(\overline{Y^\prime}^{/Y},j_!(B\dotimes_\Lambda f^\ast j_{U!} A_U)) = R\Gamma(Y,Rf_!(B\dotimes_\Lambda f^\ast j_{U!} A_U)\ ,
\]
using Theorem~\ref{thm:properbasechangepartpropvsopen} in the middle equality. Thus, we may replace $A$ by the cone of $j_{U!} j_U^\ast A\to A$, and assume that $A$ is a complex of $\Lambda$-modules concentrated on $\{s\}\subset Y$. On the other hand, we can also replace $A$ by the constant complex of $\Lambda$-modules on $Y$. Repeating this argument in the other direction, we may assume that $A$ is a constant complex of $\Lambda$-modules. As such, it is a (derived) filtered colimit of perfect complexes of $\Lambda$-modules. The case of a perfect complex is clear, and so the general case follows from the next proposition.
\end{proof}

\begin{proposition}\label{prop:Rfshriekquasicompactunboundeddirectsums} Let $f: Y^\prime\to Y$ be a compactifiable map of small v-sheaves which is representable in spatial diamonds with $\dimtrg f<\infty$, and let $\Lambda$ be a ring with $n\Lambda=0$ for some $n$ prime to $p$. Then for any collection of objects $A_i\in D_\et(Y^\prime,\Lambda)$, $i\in I$ for some set $I$, the natural map
\[
\bigoplus_i Rf_! A_i\to Rf_!(\bigoplus_i A_i)
\]
is an isomorphism in $D_\et(Y,\Lambda)$.
\end{proposition}

\begin{proof} By Proposition~\ref{prop:Rfshriekquasicompactbasechange}, we can assume that $Y=\Spa(C,C^+)$, and it suffices to check on global sections. In any given degree $d$, the groups $H^d(Y,\bigoplus_i Rf_! A_i)$ and $H^d(Y,Rf_!(\bigoplus_i A_i))$ depend only on the truncations $\tau^{\geq d-3\dimtrg f} A_i$; thus, we may assume that all $A_i$ are uniformly bounded below. We need to see that
\[
R\Gamma(Y,\bigoplus_i R\overline{f}^{/Y}_\ast j_! A_i) = R\Gamma(Y,R\overline{f}^{/Y}_\ast j_!(\bigoplus_i A_i))\ .
\]
Here, the left-hand side is given by $\bigoplus_i R\Gamma(Y,R\overline{f}^{/Y}_\ast j_! A_i) = \bigoplus_i R\Gamma(\overline{Y^\prime}^{/Y},j_! A_i)$ as $Y$ is strictly local, and the right-hand side is given by $R\Gamma(\overline{Y^\prime}^{/Y},\bigoplus_i j_! A_i)$. In other words, we have to check that $R\Gamma(\overline{Y^\prime}^{/Y},-)$ commutes with arbitrary direct sums in $D^{\geq -n}_\et(\overline{Y^\prime}^{/Y},\Lambda)$. We recall that these cohomology groups here are always taken to be v-cohomology. Now the sheaves are small, and the ($\kappa$-small) v-topos of a qcqs v-sheaf is coherent, so we get the result by SGA 4 VI Corollaire 5.2.
\end{proof}

This finishes the case that $f$ is quasicompact. It remains to consider the non-quasicompact case, and obtain all previous results in this generality. The idea here is simple: If $f: Y^\prime\to Y$ is a compactifiable map of small v-stacks that is representable in locally spatial diamonds with locally $\dimtrg f<\infty$ (i.e., after pullback to any spatial diamond $Z\to Y$, any quasicompact open subspace of $Y'\times_Y Z$ has finite $\dimtrg$ over $Z$), then, at least v-locally on $Y$, we can write $Y^\prime$ as an increasing union of quasicompact open subspaces, to which we can apply the preceding discussion; now we take the filtered direct limit. Unfortunately, it is somewhat tricky to resolve all homotopy coherence issues in this approach. For this reason, we make rather heavy use of Lurie's book, \cite{LurieHTT}, in the construction of $Rf_!$. Specifically, we will use the theory of left Kan extensions (along full embeddings), cf.~\cite[Section 4.3.2]{LurieHTT}.

First, we consider the case where $Y$ is itself a (quasiseparated) locally spatial diamond.

\begin{definition}\label{def:Rfshrieklocspat} Let $f: Y^\prime\to Y$ be a compactifiable map of quasiseparated locally spatial diamonds with locally $\dimtrg f<\infty$, and assume $n\Lambda=0$ for some $n$ prime to $p$. Let $j: Y^\prime\hookrightarrow \overline{Y^\prime}^{/Y}$ and $\overline{f}^{/Y}: \overline{Y^\prime}^{/Y}\to Y$ be the usual maps, where $\overline{f}^{/Y}$ is partially proper. Let $\mathcal{D}_{\et,\prop/Y}(Y^\prime,\Lambda)\subset \mathcal{D}_\et(Y^\prime,\Lambda)$ be the full $\infty$-subcategory of $A\in \mathcal{D}_\et(Y^\prime,\Lambda)$ such that $A\simeq j_{V!} j_V^\ast A$ for some open subspace $j_V: V\hookrightarrow Y^\prime$ that is quasicompact over $Y$.

The functor
\[
Rf_!: \mathcal{D}_\et(Y^\prime,\Lambda)\to \mathcal{D}_\et(Y,\Lambda)
\]
is defined as the left Kan extension of
\[
R\overline{f}^{/Y}_\ast j_!: \mathcal{D}_{\et,\prop/Y}(Y^\prime,\Lambda)\to \mathcal{D}_\et(Y,\Lambda)
\]
along the full inclusion $\mathcal{D}_{\et,\prop/Y}(Y^\prime,\Lambda)\subset \mathcal{D}_\et(Y^\prime,\Lambda)$.
\end{definition}

In other words, on the full $\infty$-subcategory $\mathcal{D}_{\et,\prop/Y}(Y^\prime,\Lambda)\subset \mathcal{D}_\et(Y^\prime,\Lambda)$ of sheaves ``with proper support over $Y$'', the functor $Rf_!$ is given by $R\overline{f}^{/Y}_\ast j_!$. For a general object $A\in \mathcal{D}_\et(Y^\prime,\Lambda)$, one can write $A$ as a filtered colimit of objects $A_V = j_{V!} j_V^\ast A\in \mathcal{D}_{\et,\prop/Y}(Y^\prime,\Lambda)$, and then by definition (cf.~\cite[Definition 4.3.2.2]{LurieHTT})
\[
Rf_! A = \varinjlim_V Rf_! A_V\ .
\]
Note that here, $Rf_! A_V = R(f|_V)_! (j_V^\ast A)$, as
\[
R\overline{f}^{/Y}_\ast j_! A_V = R\overline{f}^{/Y}_\ast (j\circ j_V)_! (j_V^\ast A) = R\overline{f}^{/Y}_\ast g_\ast j^\prime_{V!} (j_V^\ast A) = R(f|_V)_! (j_V^\ast A)\ ,
\]
where $j^\prime_V: V\hookrightarrow \overline{V}^{/Y}$ is the open immersion, and $g: \overline{V}^{/Y}\hookrightarrow \overline{Y^\prime}^{/Y}$ the closed immersion.

Before going on with the definition of $Rf_!$ in the general case, we need to establish base change in this situation. For this, it is helpful to first establish that $Rf_!$ commutes with all colimits.

\begin{proposition}\label{prop:Rfshriekunboundeddirectsumslocspat} Let $f: Y^\prime\to Y$ be a compactifiable map of quasiseparated locally spatial diamonds with locally $\dimtrg f<\infty$, and assume $n\Lambda=0$ for some $n$ prime to $p$. The functor
\[
Rf_!: \mathcal{D}_\et(Y^\prime,\Lambda)\to \mathcal{D}_\et(Y,\Lambda)
\]
commutes with all direct sums; equivalently (cf.~\cite[Proposition 1.4.4.1 (2)]{LurieHA}), with all colimits.
\end{proposition}

\begin{proof} Let $A_i$, $i\in I$, be any set of objects of $\mathcal{D}_\et(Y^\prime,\Lambda)$. Write each $A_i$ as the colimit of $A_{i,V} = j_{V!} j_V^\ast A_i$ over all open immersions $j_V: V\hookrightarrow Y^\prime$ which are quasicompact over $Y$. Then $Rf_! A_{i,V} = R(f|_V)_! (j_V^\ast A_i)$ by construction, which commutes with arbitrary direct sums by Proposition~\ref{prop:Rfshriekquasicompactunboundeddirectsums}. In general, we see that $Rf_! \bigoplus_i A_i$ is the colimit of
\[
R(f|_V)_! j_V^\ast  \bigoplus_i A_i = \bigoplus_i R(f|_V)_! j_V^\ast A_i\ ,
\]
which gives the desired result.
\end{proof}

\begin{proposition}\label{prop:Rfshriekbasechangelocspat} Let $f: Y^\prime\to Y$ be a compactifiable map of quasiseparated locally spatial diamonds with locally $\dimtrg f<\infty$, and assume $n\Lambda=0$ for some $n$ prime to $p$. Let $g: \tilde{Y}\to Y$ be any map of locally spatial diamonds, with base change $\tilde{f}: \tilde{Y}^\prime = Y^\prime\times_Y \tilde{Y}\to \tilde{Y}$ and $g^\prime: \tilde{Y}^\prime\to Y^\prime$.

There is a natural base change equivalence
\[
g^\ast Rf_!\simeq R\tilde{f}_! g^{\prime\ast}
\]
of functors $D_\et(Y^\prime,\Lambda)\to D_\et(\tilde{Y},\Lambda)$.
\end{proposition}

\begin{proof} To get a natural transformation from $g^\ast Rf_!$ to $R\tilde{f}_! g^{\prime\ast}$, note that in the definition of $Rf_!$ as the left Kan extension of $R\overline{f}^{/Y}_\ast j_!$, the functor $j_!$ commutes with any base change, and there is a general base change adjunction for $R\overline{f}^{/Y}_\ast$. On $\mathcal{D}_{\et,\prop/Y}(Y^\prime,\Lambda)$, the base change map is an equivalence by Proposition~\ref{prop:Rfshriekquasicompactbasechange}. In general, write any $A\in \mathcal{D}_\et(Y^\prime,\Lambda)$ as the filtered colimit of $A_V = j_{V!} j_V^\ast A$. The functor $g^\ast Rf_!$ commutes with this by definition; but so does $R\tilde{f}_! g^{\prime\ast}$, by Proposition~\ref{prop:Rfshriekunboundeddirectsumslocspat}.
\end{proof}

Let us now define $Rf_!$ in general. Thus, let $f: Y^\prime\to Y$ be a compactifiable map of small v-stacks that is representable in locally spatial diamonds with locally $\dimtrg f<\infty$. In that case, we pick a simplicial v-hypercover $Y_\bullet\to Y$ such that all $Y_i$ are quasiseparated locally spatial diamonds, and let $Y^\prime_\bullet = Y^\prime\times_Y Y_\bullet$, which is a simplicial v-hypercover of $Y^\prime$. The association $i\mapsto \mathcal{D}_\et(Y^\prime_i,\Lambda)$ defines a functor $\Delta\to \sCat_\infty$ (with functors given by pullback); cf.~\cite[Section 2]{LiuZheng} for a construction even as presentable symmetric monoidal $\infty$-categories for any diagram of ringed topoi (the extra condition $_\et$ only amounts to passage to full $\infty$-subcategories, and poses no homotopy coherence issues). This is encoded in a coCartesian fibration $\mathcal{D}_\et(Y^\prime_\bullet,\Lambda)^0\to \Delta$, whose $\infty$-category of sections is the $\infty$-derived category $\mathcal{D}_\et(Y^\prime_\bullet,\Lambda)$ of the simplicial space $Y^\prime_\bullet$.

Let
\[
\mathcal{D}_{\et,\prop/Y_\bullet}(Y^\prime_\bullet,\Lambda)^0\subset \mathcal{D}_\et(Y^\prime_\bullet,\Lambda)^0
\]
be the full $\infty$-subcategory whose fibre over any $i\in \Delta$ is given by
\[
\mathcal{D}_{\et,\prop/Y_i}(Y^\prime_i,\Lambda)\subset \mathcal{D}_\et(Y^\prime_i,\Lambda)\ .
\]
As pullbacks preserve this condition, this is still a coCartesian fibration over $\Delta$. Similarly, we have coCartesian fibrations $\mathcal{D}_\et(\overline{Y^\prime_\bullet}^{/Y_\bullet},\Lambda)^0$ and $\mathcal{D}_\et(Y_\bullet,\Lambda)^0$ over $\Delta$. The restriction functor
\[
j^\ast_\bullet: \mathcal{D}_\et(\overline{Y^\prime_\bullet}^{/Y_\bullet},\Lambda)^0\to \mathcal{D}_\et(Y^\prime_\bullet,\Lambda)^0
\]
has a fully faithful left adjoint $j_{\bullet !}$, which is in every fibre over $i\in \Delta$ given by $j_{i!}$, where $j_i: Y^\prime_i\hookrightarrow \overline{Y^\prime_i}^{/Y_i}$ is the open immersion. Indeed, this follows from the fact that $j_{i!}$ commutes with base change. Moreover, one has the pushforward functor
\[
R\overline{f_\bullet}^{/Y_\bullet}_\ast: \mathcal{D}_\et(\overline{Y^\prime_\bullet}^{/Y_\bullet},\Lambda)^0\to \mathcal{D}_\et(Y_\bullet,\Lambda)^0\ ,
\]
right adjoint to $\overline{f_\bullet}^{/Y_\bullet,\ast}$, which in every fibre over $i\in \Delta$ is given by $R\overline{f_i}^{/Y_i}_\ast$. In particular, we get the functor
\[
R\overline{f_\bullet}^{/Y_\bullet}_\ast j_{\bullet !}: \mathcal{D}_{\et,\prop/Y_\bullet}(Y^\prime_\bullet,\Lambda)^0\to \mathcal{D}_\et(Y_\bullet,\Lambda)^0\ .
\]
We let
\[
Rf_{\bullet !}^0: \mathcal{D}_\et(Y^\prime_\bullet,\Lambda)^0\to \mathcal{D}_\et(Y_\bullet,\Lambda)^0
\]
be its left Kan extension, which exists by \cite[Corollary 4.3.2.14]{LurieHTT}.

\begin{lemma}\label{lem:leftkanext} The functor $Rf_{\bullet !}^0$ is given by
\[
Rf_{i!}: \mathcal{D}_\et(Y^\prime_i,\Lambda)\to \mathcal{D}_\et(Y_i,\Lambda)
\]
in the fibre over $i\in \Delta$.
\end{lemma}

\begin{proof} As $\mathcal{D}_{\et,\prop/Y_\bullet}(Y^\prime_\bullet,\Lambda)^0$ is itself a coCartesian fibration over $\Delta$, one sees that for all $A\in \mathcal{D}_\et(Y^\prime_i,\Lambda)$, the functor
\[
\mathcal{D}_{\et,\prop/Y_i}(Y^\prime_i,\Lambda)_{/A}\to \mathcal{D}_{\et,\prop/Y_\bullet}(Y^\prime_\bullet,\Lambda)^0_{/A}
\]
is cofinal. Moreover, the index category is filtered. Taken together, these imply that the relevant colimits agree (using \cite[Proposition 4.3.1.7]{LurieHTT} to see that one may pass to a cofinal subcategory).
\end{proof}

\begin{lemma}\label{lem:cocartesian} The functor
\[
Rf_{\bullet !}^0: \mathcal{D}_\et(Y^\prime_\bullet,\Lambda)^0\to \mathcal{D}_\et(Y_\bullet,\Lambda)^0
\]
sends coCartesian edges to coCartesian edges.
\end{lemma}

\begin{proof} This follows from the previous lemma and Proposition~\ref{prop:Rfshriekbasechangelocspat}.
\end{proof}

Passing to coCartesian sections over $\Delta$, we get a functor
\[
Rf_!: \mathcal{D}_\et(Y^\prime,\Lambda)\simeq \mathcal{D}_{\et,\cart}(Y^\prime_\bullet,\Lambda)\to \mathcal{D}_{\et,\cart}(Y_\bullet,\Lambda)\simeq \mathcal{D}_\et(Y,\Lambda)\ ,
\]
as desired.

\begin{definition}\label{def:Rfshriek} Let $f: Y^\prime\to Y$ be a compactifiable map of small v-stacks that is representable in locally spatial diamonds with locally $\dimtrg f<\infty$, and assume $n\Lambda=0$ for some $n$ prime to $p$. The functor
\[
Rf_!: D_\et(Y^\prime,\Lambda)\to D_\et(Y,\Lambda)
\]
is the functor obtained from the previous discussion by passage to homotopy categories.
\end{definition}

It is easy to see that this is independent of the choice of the simplicial v-hypercover, by passing to common refinements. 
Again, we need to check that this still satisfies all desired properties, starting with base change.

\begin{proposition}\label{prop:Rfshriekbasechange} Let $f: Y^\prime\to Y$ be a compactifiable map of small v-stacks which is representable in locally spatial diamonds with locally $\dimtrg f<\infty$, and assume $n\Lambda=0$ for some $n$ prime to $p$. Let $g: \tilde{Y}\to Y$ be any map of small v-stacks, with base change $\tilde{f}: \tilde{Y}^\prime = Y^\prime\times_Y \tilde{Y}\to \tilde{Y}$ and $g^\prime: \tilde{Y}^\prime\to Y^\prime$.

There is a natural base change equivalence
\[
g^\ast Rf_!\simeq R\tilde{f}_! g^{\prime\ast}
\]
of functors $D_\et(Y^\prime,\Lambda)\to D_\et(\tilde{Y},\Lambda)$.
\end{proposition}

\begin{proof} Let $Y_\bullet\to Y$ be a simplicial v-hypercover by quasiseparated locally spatial diamonds $Y_i$, and similarly let $\tilde{Y}_\bullet\to \tilde{Y}$ be a simplicial v-hypercover by quasiseparated locally spatial diamonds $\tilde{Y}_i$, such that $g: \tilde{Y}\to Y$ extends to a map of simplicial spaces $g_\bullet: \tilde{Y}_\bullet\to Y_\bullet$. Repeating the previous discussion over $\Delta$ again over $\Delta\times \Delta^1$ gives the result.
\end{proof}

\begin{proposition}\label{prop:Rfshriekunboundeddirectsums} Let $f: Y^\prime\to Y$ be a compactifiable map of small v-stacks that is representable in locally spatial diamonds with locally $\dimtrg f<\infty$, and assume $n\Lambda=0$ for some $n$ prime to $p$. The functor
\[
Rf_!: \mathcal{D}_\et(Y^\prime,\Lambda)\to \mathcal{D}_\et(Y,\Lambda)
\]
commutes with all direct sums; equivalently (cf.~\cite[Proposition 1.4.4.1 (2)]{LurieHA}), with all colimits.
\end{proposition}

\begin{proof} This follows from Proposition~\ref{prop:Rfshriekunboundeddirectsumslocspat} and Proposition~\ref{prop:Rfshriekbasechange}.
\end{proof}

\begin{proposition}\label{prop:Rfshriekcomp} Let $g: Y^{\prime\prime}\to Y^\prime$ and $f: Y^\prime\to Y$ be compactifiable maps of small v-stacks which are representable in locally spatial diamonds with locally $\dimtrg f,\dimtrg g <\infty$, and assume $n\Lambda=0$ for some $n$ prime to $p$. Then there is a natural equivalence
\[
Rf_!\circ Rg_!\simeq R(f\circ g)_!
\]
of functors $D_\et(Y^{\prime\prime},\Lambda)\to D_\et(Y,\Lambda)$.
\end{proposition}

\begin{proof} Let $Y_\bullet\to Y$ be a simplicial v-hypercover by quasiseparated locally spatial diamonds, as usual, and let $Y^\prime_\bullet\to Y^\prime$, $Y^{\prime\prime}_\bullet\to Y^{\prime\prime}$ be the pullbacks, which are again simplicial v-hypercovers by quasiseparated locally spatial diamonds. We seek to find an equivalence of the composite of the functors
\[
Rg_{\bullet!}^0: \mathcal{D}_\et(Y^{\prime\prime}_\bullet,\Lambda)^0\to \mathcal{D}_\et(Y^\prime_\bullet,\Lambda)^0
\]
and
\[
Rf_{\bullet!}^0: \mathcal{D}_\et(Y^\prime_\bullet,\Lambda)^0\to \mathcal{D}_\et(Y_\bullet,\Lambda)^0
\]
with
\[
R(f\circ g)_{\bullet!}^0: \mathcal{D}_\et(Y^{\prime\prime}_\bullet,\Lambda)^0\to \mathcal{D}_\et(Y_\bullet,\Lambda)^0\ .
\]
This gives the result by passage to coCartesian sections. But one can directly identify the functors on the full $\infty$-subcategory $\mathcal{D}_{\et,\prop/Y_\bullet}(Y^{\prime\prime}_\bullet,\Lambda)^0$ (as there is a natural transformation, which is an equivalence, by Proposition~\ref{prop:Rfshriekquasicompactcomp} and its proof). On the other hand, all functors commute with all colimits by Proposition~\ref{prop:Rfshriekunboundeddirectsums}, so the full functors can be recovered by left Kan extension (as above).
\end{proof}

In case of overlap, the current definition of $Rf_!$ agrees with Definition~\ref{def:Rfshrieketale}.

\begin{proposition}\label{prop:compRfshrieketale} Let $f: Y^\prime\to Y$ be a separated \'etale map of small v-stacks, and assume $n\Lambda=0$ for some $n$ prime to $p$. Then $Rf_!$ as defined in Definition~\ref{def:Rfshrieketale} agrees with $Rf_!$ as defined in Definition~\ref{def:Rfshriek}.
\end{proposition}

\begin{proof} As in the proof of Proposition~\ref{prop:compRfshrieketaleqc}, one has a natural transformation. To check whether it is an equivalence, one can reduce to the case where $Y$ is strictly totally disconnected. As moreover both functors commute with all direct sums, one can reduce to the case where $Y^\prime$ is quasicompact, where it follows from Proposition~\ref{prop:compRfshrieketaleqc}.
\end{proof}

Finally, we also get the projection formula in general.

\begin{proposition}\label{prop:projectionformula} Let $f: Y^\prime\to Y$ be a compactifiable map of small v-sheaves which is representable in locally spatial diamonds with locally $\dimtrg f<\infty$, and let $\Lambda$ be a ring with $n\Lambda=0$ for some $n$ prime to $p$. Then there is a functorial isomorphism
\[
Rf_! B\dotimes_\Lambda A\simeq Rf_!(B\dotimes_\Lambda f^\ast A)
\]
in $B\in D_\et(Y^\prime,\Lambda)$ and $A\in D_\et(Y,\Lambda)$.
\end{proposition}

\begin{proof} We fix $A\in D_\et(Y,\Lambda)$, and define an isomorphism
\[
Rf_! B\dotimes_\Lambda A\simeq Rf_!(B\dotimes_\Lambda f^\ast A)
\]
functorial in $B\in D_\et(Y^\prime,\Lambda)$. One then checks that varying $A$, the relevant diagrams commute in the derived category.

Choose a simplicial v-hypercover $Y_\bullet\to Y$ by locally spatial diamonds as in the definition of $Rf_!$, with pullback $Y^\prime_\bullet\to Y^\prime$. Consider the functor
\[
Rf_{\bullet !}^0: \mathcal{D}_\et(Y^\prime_\bullet,\Lambda)^0\to \mathcal{D}_\et(Y_\bullet,\Lambda)^0\ .
\]
Both $\infty$-categories have an endofunctor given by tensoring with the pullback of $A$; we denote this operation by $-\dotimes_\Lambda A|_{Y^\prime_\bullet}$ respectively $-\dotimes_\Lambda A|_{Y_\bullet}$. We claim that there is a natural equivalence of functors
\[
Rf_{\bullet !}^0(B\dotimes_\Lambda A|_{Y^\prime_\bullet})\simeq Rf_{\bullet !}^0 B\dotimes_\Lambda A|_{Y_\bullet}
\]
from $B\in \mathcal{D}_\et(Y^\prime_\bullet,\Lambda)^0$ to $\mathcal{D}_\et(Y_\bullet,\Lambda)^0$. This implies the desired result by passage to coCartesian sections (noting that both functors map coCartesian edges to coCartesian edges, as the second functor does).

To construct the natural equivalence, note that both functors are left Kan extensions from $\mathcal{D}_{\et,\prop/Y_\bullet}(Y^\prime_\bullet,\Lambda)^0$. On this full $\infty$-subcategory, the first functor is the composite
\[
\mathcal{D}_{\et,\prop/Y_\bullet}(Y^\prime_\bullet,\Lambda)^0\buildrel{\dotimes_\Lambda A|_{Y^\prime_\bullet}}\over\longrightarrow \mathcal{D}_{\et,\prop/Y_\bullet}(Y^\prime_\bullet,\Lambda)^0\buildrel{j_{\bullet !}}\over\longrightarrow \mathcal{D}_\et(\overline{Y^\prime_\bullet}^{/Y_\bullet},\Lambda)^0\buildrel{R\overline{f_\bullet}^{/Y_\bullet}_\ast}\over\longrightarrow \mathcal{D}_\et(Y_\bullet,\Lambda)^0\ .
\]
This admits a natural map to the composite
\[
\mathcal{D}_{\et,\prop/Y_\bullet}(Y^\prime_\bullet,\Lambda)^0\buildrel{j_{\bullet !}}\over\longrightarrow \mathcal{D}_\et(\overline{Y^\prime_\bullet}^{/Y_\bullet},\Lambda)^0\buildrel{\dotimes_\Lambda A|_{\overline{Y^\prime_\bullet}^{/Y_\bullet}}}\over\longrightarrow
\mathcal{D}_\et(\overline{Y^\prime_\bullet}^{/Y_\bullet},\Lambda)^0\buildrel{R\overline{f_\bullet}^{/Y_\bullet}_\ast}\over\longrightarrow \mathcal{D}_\et(Y_\bullet,\Lambda)^0\ ,
\]
by using that $j_{\bullet !}$ is left adjoint to $j_\bullet^\ast$, and that pullback commutes with $-\dotimes_\Lambda -$. This natural transformation is an equivalence (even without postcomposition with $R\overline{f_\bullet}^{/Y_\bullet}_\ast$), as can be checked in each fibre, where it reduces to the projection formula for an open embedding.

On the other hand, the second functor is the composite
\[
\mathcal{D}_{\et,\prop/Y_\bullet}(Y^\prime_\bullet,\Lambda)^0\buildrel{j_{\bullet !}}\over\longrightarrow \mathcal{D}_\et(\overline{Y^\prime_\bullet}^{/Y_\bullet},\Lambda)^0
\buildrel{R\overline{f_\bullet}^{/Y_\bullet}_\ast}\over\longrightarrow \mathcal{D}_\et(Y_\bullet,\Lambda)^0\buildrel{\dotimes_\Lambda A|_{Y_\bullet}}\over\longrightarrow \mathcal{D}_\et(Y_\bullet,\Lambda)^0\ .
\]
This also admits a natural map to the composition
\[
\mathcal{D}_{\et,\prop/Y_\bullet}(Y^\prime_\bullet,\Lambda)^0\buildrel{j_{\bullet !}}\over\longrightarrow \mathcal{D}_\et(\overline{Y^\prime_\bullet}^{/Y_\bullet},\Lambda)^0\buildrel{\dotimes_\Lambda A|_{\overline{Y^\prime_\bullet}^{/Y_\bullet}}}\over\longrightarrow
\mathcal{D}_\et(\overline{Y^\prime_\bullet}^{/Y_\bullet},\Lambda)^0\buildrel{R\overline{f_\bullet}^{/Y_\bullet}_\ast}\over\longrightarrow \mathcal{D}_\et(Y_\bullet,\Lambda)^0
\]
considered above, by using that $R\overline{f_\bullet}^{/Y_\bullet}_\ast$ is right adjoint to $\overline{f_\bullet}^{/Y_\bullet,\ast}$, and that pullback commutes with $-\dotimes_\Lambda -$. Again, one can check that it is a natural equivalence by checking it on fibres, where it reduces to the projection formula in the quasicompact case, Proposition~\ref{prop:projectionformulaqc}. Combining these observations finishes the proof.
\end{proof}

\section{Cohomologically smooth morphisms}\label{sec:cohomsmooth}

We start with the following theorem, which is essentially a corollary of Proposition~\ref{prop:Rfshriekunboundeddirectsums}.

\begin{theorem}\label{thm:Rfuppershriekexists} Let $f: Y^\prime\to Y$ be a compactifiable map of small v-stacks which is representable in locally spatial diamonds with locally $\dimtrg f<\infty$, and let $\Lambda$ be a ring with $n\Lambda=0$ for some $n$ prime to $p$. Then the functor
\[
Rf_!: D_\et(Y^\prime,\Lambda)\to D_\et(Y,\Lambda)
\]
admits a right adjoint
\[
Rf^!: D_\et(Y,\Lambda)\to D_\et(Y^\prime,\Lambda)\ .
\]
\end{theorem}

\begin{proof} This follows from Lurie's $\infty$-categorical adjoint functor theorem, \cite[Corollary 5.5.2.9]{LurieHTT}, and Proposition~\ref{prop:Rfshriekunboundeddirectsums}.
\end{proof}

\begin{remark}\label{rem:uppershriekchangeofcoeff} The functor $Rf^!$ is compatible with change of rings $g: \Lambda^\prime\to \Lambda$ in the following sense. The map $g$ induces a map $D_\et(Y^\prime,\Lambda)\to D_\et(Y^\prime,\Lambda^\prime)$ by restricting the action along $g$. Then the diagram
\[\xymatrix{
D_\et(Y,\Lambda)\ar[r]^{Rf^!}\ar[d] & D_\et(Y^\prime,\Lambda)\ar[d]\\
D_\et(Y,\Lambda^\prime)\ar[r]^{Rf^!} & D_\et(Y^\prime,\Lambda^\prime)
}\]
commutes. Indeed, this identity of functors is the right adjoint of the identity of functors
\[
Rf_!(-\otimes_{\Lambda^\prime} \Lambda) = Rf_!\otimes_{\Lambda^\prime} \Lambda\ ,
\]
which follows from the projection formula, Proposition~\ref{prop:projectionformula}.
\end{remark}

Before going on, we note that the following results follow from the definitions and Proposition~\ref{prop:projectionformula}.

\begin{proposition}\label{prop:localverdier} Let $f: Y^\prime\to Y$ be a compactifiable map of small v-stacks that is representable in locally spatial diamonds with locally $\dimtrg f<\infty$, and let $\Lambda$ be a ring with $n\Lambda=0$ for some $n$ prime to $p$.
\begin{altenumerate}
\item For all $A\in D_\et(Y^\prime,\Lambda)$, $B\in D_\et(Y,\Lambda)$, one has
\[
R\sHom_\Lambda(Rf_!A,B)\cong Rf_\ast R\sHom_\Lambda(A,Rf^! B)\ .
\]
\item For all $A, B\in D_\et(Y,\Lambda)$, one has
\[
Rf^! R\sHom_\Lambda(A,B)\cong R\sHom_\Lambda(f^\ast A,Rf^! B)\ .
\]
\end{altenumerate}
\end{proposition}

\begin{proof} For part (i), note that for all $C\in D_\et(Y,\Lambda)$, one has
\[\begin{aligned}
\Hom_{D_\et(Y,\Lambda)}(C,R\sHom_\Lambda(Rf_!A,B))&=\Hom_{D_\et(Y,\Lambda)}(C\dotimes_\Lambda Rf_! A,B)\\
&=\Hom_{D_\et(Y,\Lambda)}(Rf_!(f^\ast C\dotimes_\Lambda A),B)\\
&=\Hom_{D_\et(Y^\prime,\Lambda)}(f^\ast C\dotimes_\Lambda A,Rf^!B)\\
&=\Hom_{D_\et(Y^\prime,\Lambda)}(f^\ast C,R\sHom_\Lambda(A,Rf^!B))\\
&=\Hom_{D_\et(Y,\Lambda)}(C,Rf_\ast R\sHom_\Lambda(A,Rf^!B))\ ,
\end{aligned}\]
so the result follows from the Yoneda lemma. Similarly, for part (ii), note that for all $C\in D_\et(Y^\prime,\Lambda)$, one has
\[\begin{aligned}
\Hom_{D_\et(Y^\prime,\Lambda)}(C,Rf^! R\sHom_\Lambda(A,B))&=\Hom_{D_\et(Y,\Lambda)}(Rf_!C,R\sHom_\Lambda(A,B))\\
&=\Hom_{D_\et(Y,\Lambda)}(Rf_!C\dotimes_\Lambda A,B)\\
&=\Hom_{D_\et(Y,\Lambda)}(Rf_!(C\dotimes_\Lambda f^\ast A),B)\\
&=\Hom_{D_\et(Y^\prime,\Lambda)}(C\dotimes_\Lambda f^\ast A, Rf^! B)\\
&=\Hom_{D_\et(Y^\prime,\Lambda)}(C,R\sHom_\Lambda(f^\ast A,Rf^!B))\ .
\end{aligned}\]
\end{proof}

As a preparation for the definition of smooth morphisms, we prove the following proposition.

\begin{proposition}\label{prop:shriekvsusualpullback} Let $X$ be a strictly totally disconnected perfectoid space, let $f: Y\to X$ be a compactifiable map from a locally spatial diamond $Y$ of locally $\dimtrg f<\infty$, and fix a prime $\ell\neq p$. The following conditions are equivalent.
\begin{altenumerate}
\item The natural transformation
\[
Rf^!\Fl\otimes_\Fl f^\ast\to Rf^!: D_\et(X,\Fl)\to D_\et(Y,\Fl)
\]
adjoint to
\[
Rf_!(Rf^!\Fl\otimes_\Fl f^\ast-) = Rf_! Rf^!\Fl\otimes_\Fl \mathrm{id}\to \mathrm{id}
\]
is an equivalence (using the projection formula, Proposition~\ref{prop:projectionformula}, in the equality).
\item The functor $Rf^!: D_\et(X,\Fl)\to D_\et(Y,\Fl)$ is equivalent to a functor of the form $A\otimes_\Fl f^\ast$ for some $A\in D_\et(Y,\Fl)$.
\item The functor $Rf^!: D_\et(X,\Fl)\to D_\et(Y,\Fl)$ commutes with arbitrary direct sums, and for all connected components $X_0=\Spa(C,C^+)\subset X$ with an open subset $j: U\subset X_0$ and pullbacks
\[\xymatrix{
V\ar@{^(->}[r]^{j^\prime}\ar[d]^{f_U}& Y_0\ar[r]\ar[d]^{f_0}& Y\ar[d]^f\\
U\ar@{^(->}[r]^j& X_0\ar[r] &X\ ,
}\]
the map
\[
j^\prime_! Rf_U^! \Fl \to Rf_0^! j_! \Fl
\]
adjoint to
\[
Rf_{0!} j^\prime_! Rf_U^!\Fl = j_! Rf_{U!} Rf_U^!\Fl\to j_!\Fl
\]
is an equivalence.
\item For all affinoid pro-\'etale maps $g: X^\prime\to X$ with pullback
\[\xymatrix{
Y^\prime\ar[r]^h\ar[d]^{f^\prime}& Y\ar[d]^f\\
X^\prime\ar[r]^g & X\ ,
}\]
the natural transformation of functors
\[
h^\ast Rf^!\to Rf^{\prime !} g^\ast: D_\et(X,\Fl)\to D_\et(Y^\prime,\Fl)
\]
adjoint to
\[
Rf^\prime_! h^\ast Rf^! = g^\ast Rf_! Rf^!\to g^\ast
\]
is an equivalence, and for all connected components $X_0=\Spa(C,C^+)\subset X$ with an open subset $j: U\subset X_0$ and pullbacks
\[\xymatrix{
V\ar@{^(->}[r]^{j^\prime}\ar[d]^{f_U}& Y_0\ar[r]\ar[d]^{f_0}& Y\ar[d]^f\\
U\ar@{^(->}[r]^j& X_0\ar[r] &X\ ,
}\]
the natural transformation of functors
\[
j^\prime_! Rf_U^!\to Rf_0^! j_!: D_\et(U,\Fl)\to D_\et(Y_0,\Fl)
\]
adjoint to
\[
Rf_{0!} j^\prime_! Rf_U^! = j_! Rf_{U!} Rf_U^!\to j_!
\]
is an equivalence.
\end{altenumerate}

Moreover, under these conditions, for any $\ell$-power-torsion ring $\Lambda$, the natural transformation
\[
Rf^! \Lambda\otimes_\Lambda f^\ast\to Rf^!: D_\et(X,\Lambda)\to D_\et(Y,\Lambda)
\]
is an equivalence.
\end{proposition}

\begin{remark} As in the discussion around Theorem~\ref{thm:properbasechangepartpropvsopen}, the second part of condition (iv) can be regarded as a version of the first part for pullback to closed but non-generalizing subsets. Thus, condition (iv) is expressing a version of the idea that ``$Rf^!$ commutes with arbitrary pro-\'etale base change''. Of course, if $Rf^!$ is essentially given by $f^\ast$, this should be true. What is maybe surprising is that the converse also holds.
\end{remark}

\begin{proof} It is clear that (i) implies (ii), and (ii) implies (iii). To see that (iii) implies (iv), we first check that if $Rf^!$ commutes with arbitrary direct sums, then the first part of condition (iv) is satisfied. For this, let $g: X^\prime=\varprojlim X^\prime_i\to X$ be an inverse limit of affinoid \'etale maps $g_i: X^\prime_i\to X$, and consider the cartesian diagrams
\[\xymatrix{
Y^\prime\ar[r]^h\ar[d]^{f^\prime}& Y\ar[d]^f& &Y^\prime_i \ar[r]^{h_i}\ar[d]^{f^\prime_i} & Y\ar[d]^f\\
X^\prime\ar[r]^g & X& & X^\prime_i\ar[r]^{g_i} & X\ .
}\]
We need to check that the natural transformation
\[
h^\ast Rf^!\to Rf^{\prime !} g^\ast
\]
is an equivalence. Note that by Remark~\ref{rem:qcsepqproetpushforward}, the functor $h_\ast=Rh_\ast$ is exact, and one checks easily that it is also conservative (like pushforward along any quasicompact separated quasi-pro-\'etale map). Thus, it suffices to show that the natural transformation
\[
h_\ast h^\ast Rf^!\to h_\ast Rf^{\prime !} g^\ast = Rf^! g_\ast g^\ast
\]
is an equivalence. But note that $g_\ast g^\ast$ is the filtered colimit of $g_{i\ast} g_i^\ast$, and correspondingly $h_\ast h^\ast$ is the filtered colimit of $h_{i\ast} h_i^\ast$. If $Rf^!$ commutes with arbitrary direct sums, its $\infty$-categorical version $\mathcal Rf^!$ commutes with arbitrary colimits, and thus the preceding argument reduces us to the case of $g_i: X^\prime_i\to X$. But if $g=g_i$ is \'etale, then $g^\ast=Rg^!$ and $h^\ast = Rh^!$, so $h^\ast Rf^! = Rh^! Rf^! = Rf^{\prime !} Rg^! = Rf^{\prime !} g^\ast$, as desired.

Now, for the second part of condition (iv), one checks that the full subcategory of all $K\in D_\et(U,\Fl)$ which satisfy the conclusion is triangulated and stable under arbitrary direct sums, and contains $j^\prime_! \Fl$ for all open immersions $j^\prime: U^\prime\hookrightarrow U$. This implies that it is all of $D_\et(U,\Fl)$.

For (iv) implies (i), we have to check that
\[
Rf^!\Fl\otimes f^\ast\to Rf^!: D_\et(X,\Fl)\to D_\et(Y,\Fl)
\]
is an equivalence. This can be checked in fibers over points of $X$. As both sides commute with base change by the first part of condition (iv), we can thus assume that $X=\Spa(C,C^+)$ is strictly local. Fix a geometric point $y$ of $Y$; we want to prove that the map $(Rf^!\Fl\otimes f^\ast)_y\to (Rf^!)_y$ of stalks at $y$ is an isomorphism. Let $x\in X$ be the image of $y$, let $X_1=\Spa(C,(C^+)^\prime)\subset X$ be the set of generalizations of $x$, and let $U=X_1\setminus \{x\}$. Then we can pullback further under the open immersion $X_1\hookrightarrow X$, and assume that $X=X_1$. For any $K\in D_\et(X,\Fl)$, we have a triangle $j_! j^\ast K\to K\to K^\prime$, where $j: U\hookrightarrow X_1=X$ denotes the open immersion. By the second condition of (iv), we see that the stalk at $y$ of the target of the map
\[
Rf^!\Fl\otimes f^\ast j_! K|_U\to Rf^! j_! K|_U
\]
vanishes; it also vanishes on the source, as the second factor $f^\ast j_! K|_U$ does. Thus, we can replace $K$ by $K^\prime$, and assume that $K|_U=0$. In that case, $K=i_\ast K_0$ for some complex $K_0\in D(\Fl)$, where $i: \{x\}\to |X|$ denotes the closed inclusion. Repeating the argument with the displayed triangle, we can also replace $K$ by the constant sheaf $K_0$.

As a further reduction step, we reduce to the case that $K_0$ is concentrated in degree $0$. Assume for the moment that the result holds true in this case. By triangles, it holds if $K_0\in D^b(\Fl)$ is bounded. We claim that it also holds if $K_0\in D^+(\Fl)$. For this, it is enough to show that if $K_0\in D^{\geq n}(\Fl)$, then for some constant $c$, both sides lie in $D^{\geq n-c}(Y,\Lambda)$. This in turn reduces to proving a similar result for $Rf^!$, which follows formally from the fact that $Rf_!$ has finite cohomological dimension. It remains to handle the case of $K_0\in D^-(\Fl)$. Writing this as a limit of its Postnikov truncations and using that $Rf^!$ commutes with derived limits, it is enough to prove that $Rf^!\Fl\in D^{\leq 0}_\et(Y,\Lambda)$. This can be checked on stalks, which are given by colimits of $R\Hom_X(Rf^\prime_! \Fl,\Fl)$ for $f^\prime: Y^\prime\to Y\to X$ the composition of an \'etale map $Y^\prime\to Y$ with $f: Y\to X$. But if $j_\eta: \{\eta\}\to X$ is the inclusion of the maximal point, then using Lemma~\ref{lem:qcopenpushforward}
\[
R\Hom_X(-,\Fl) = R\Hom_X(-,Rj_{\eta\ast} \Fl) = R\Hom_\eta(j^\ast-,\Fl) = \Hom_\eta(j^\ast-,\Fl)\ ,
\]
as on a geometric point, $\Hom(-,\Fl)$ is exact. As $Rf^\prime_! \Fl\in D^{\geq 0}(X,\Fl)$, this implies that
\[
R\Hom_X(Rf^\prime_! \Fl,\Fl)\in D^{\leq 0}(\Fl)\ ,
\]
as desired.

Thus, it remains to handle the case that $K_0$ is concentrated in degree $0$. In this case, $K_0=V[0]$ for some $\Fl$-vector space $V$. We can assume that $V=C^0(S,\Fl)$ is the space of continuous functions on some profinite set $S$. Let $X^\prime=X\times \underline{S}$, so that $h: X^\prime\to X$ is quasicompact separated pro-\'etale. Let $f^\prime: Y^\prime=X^\prime\times_X Y\to X^\prime$ be the pullback, which is given by $Y^\prime=Y\times\underline{S}$. Using the first part of condition (iv), we get
\[
Rf^{\prime!} h^\ast \Fl = h^{\prime\ast} Rf^! \Fl\ .
\]
Applying $Rh^\prime_\ast$ and using that $Rh^\prime_\ast Rf^{\prime!} = Rf^! Rh_\ast$ by Proposition~\ref{prop:smoothbasechange}~(i), we get
\[
Rf^! Rh_\ast h^\ast \Fl = Rh^\prime_\ast h^{\prime\ast} Rf^! \Fl\ .
\]
By Lemma~\ref{lem:stonecechinfinitesum} below, this translates into
\[
Rf^! V= V\otimes_\Fl Rf^! \Fl\ ,
\]
as desired.

This shows the equivalence of conditions (i) through (iv). To handle the case of a general $\ell$-power-torsion ring $\Lambda$, choose $m$ such that $\ell^m=0$ in $\Lambda$. Then $\Lambda$ is a $\mathbb Z/\ell^m\mathbb Z$-algebra, and using Remark~\ref{rem:uppershriekchangeofcoeff}, one reduces to the case $\Lambda=\mathbb Z/\ell^m\mathbb Z$. Any $K\in D_\et(X,\Lambda)$ is filtered by $m$ copies of $K\otimes_{\mathbb Z/\ell^m\mathbb Z} \Fl$, so one can assume that $K$ comes from $D_\et(X,\Fl)$ via restriction of coefficients. Using the result for $\Fl$, this reduces to showing that
\[
Rf^! \mathbb Z/\ell^m\mathbb Z\otimes_{\mathbb Z/\ell^m\mathbb Z} \Fl\to Rf^! \Fl
\]
is an isomorphism. For this, note that condition (iv) also holds for $\mathbb Z/\ell^m\mathbb Z$ in place of $\Fl$, by filtering by $m$ copies of $-\otimes_{\mathbb Z/\ell^m\mathbb Z}\Fl$. Thus, as above, we can reduce to the case that $X=\Spa(C,C^+)$ is connected. Now one uses the standard infinite resolution of $\Fl$ as a $\mathbb Z/\ell^m\mathbb Z$-module, and the observation $Rf^! \mathbb Z/\ell^m\mathbb Z\in D^{\leq 0}_\et(Y,\mathbb Z/\ell^m\mathbb Z)$. The latter is proved in the same way as $Rf^! \Fl\in D^{\leq 0}_\et(Y,\Fl)$ above, using that $\mathbb Z/\ell^m\mathbb Z$ is self-injective.
\end{proof}

\begin{lemma}\label{lem:stonecechinfinitesum} Let $S$ be a profinite set, let $Y$ be a small v-sheaf, and consider $h: Y\times\underline{S}\to Y$. Then for any ring $\Lambda$ and any $C\in D_\et(Y,\Lambda)$, one has a natural isomorphism
\[
Rh_\ast h^\ast C\simeq C^0(S,\Lambda)\otimes_\Lambda C\ ,
\]
where $C^0(S,\Lambda)$ is the $\Lambda$-module of continuous maps $S\to \Lambda$.
\end{lemma}

\begin{proof} Note that $h$ is proper, so by the projection formula (Proposition~\ref{prop:projectionformulaqc}), we get
\[
Rh_\ast h^\ast C\cong Rh_! h^\ast C\cong Rh_! h^\ast \Lambda\dotimes_\Lambda C\cong Rh_\ast h^\ast \Lambda\dotimes_\Lambda C.
\]
Writing $S=\varprojlim_i S_i$ as an inverse limit of finite sets, one computes $Rh_\ast h^\ast \Lambda=C^0(S,\Lambda)$, using Proposition~\ref{prop:etcohomlim}.
\end{proof}

\begin{proposition}\label{prop:Rfshriekconstructible} Let $X$ be a strictly totally disconnected perfectoid space, $f: Y\to X$ a compactifiable map from a spatial diamond $Y$ with $\dimtrg f<\infty$, and fix a prime $\ell\neq p$.

The functor
\[
Rf^! : D_\et(X,\Fl)\to D_\et(Y,\Fl)
\]
commutes with arbitrary direct sums if and only if for all constructible sheaves $\mathcal F$ of $\Fl$-vector spaces on $Y_\et$ and all $i\geq 0$, the $!$-pushforward $R^if_! \mathcal F$ is constructible on $X_\et$.
\end{proposition}

\begin{proof} First, we claim that $Y$ satisfies the hypothesis of Proposition~\ref{prop:constructiblecompactderivedfull}. For this, we have to see that for any quasicompact separated \'etale map $j: U\to Y$, the $\ell$-cohomological dimension of $U_\et$ is bounded by $N$, for some fixed integer $N$. Here, we can take $N=3\dimtrg f$ by Theorem~\ref{thm:dimtrgfinitecohomdim} (and as $X$ has cohomological dimension $0$).

Now it follows from Proposition~\ref{prop:constructiblecompactderivedfull} that both $D_\et(Y,\Fl)$ and $D_\et(X,\Fl)$ are compactly generated, with compact objects given by bounded complexes whose cohomology sheaves are constructible. It is easy to see that the second condition is equivalent to the condition that $Rf_!: D_\et(Y,\Fl)\to D_\et(X,\Fl)$ preserves compact objects. It follows from adjunction that $Rf^!$ commutes with arbitrary direct sums if and only if $Rf_!$ preserves compact objects (this simple but powerful observation goes back to Neeman, \cite{NeemanJAMS}), so we get the result.
\end{proof}

Finally, we can give the definition of cohomological smoothness.

\begin{definition}\label{def:cohsmooth} Let $f: Y^\prime\to Y$ be a separated map of small v-stacks that is representable in locally spatial diamonds, and let $\ell\neq p$ be a prime. Then $f$ is $\ell$-cohomologically smooth if $f$ is compactifiable, locally $\dimtrg f<\infty$, and for any strictly totally disconnected perfectoid space $X$ with a map $X\to Y$ with pullback $f_X: Y^\prime\times_Y X\to X$, the functor
\[
Rf_X^!: D_\et(X,\Fl)\to D_\et(Y^\prime\times_Y X,\Fl)
\]
is equivalent to a functor of the form $D_{f_X}\otimes_\Fl f^\ast$ for some invertible object $D_{f_X}\in D_\et(Y^\prime\times_Y X,\Fl)$.
\end{definition}

Here, for a locally spatial diamond $Y$, an object $D\in D_\et(Y,\Fl)$ is \emph{invertible} if it is locally (equivalently, in the v-, quasi-pro-\'etale, or \'etale topology of $Y$) isomorphic to $\Fl[n]$ for some integer $n\in \mathbb Z$. We note that we are not a priori asking that the equivalence between $Rf_X^!$ and $D_{f_X}\otimes_\Fl f^\ast$ is natural, that $D_{f_X}$ commutes with base change, or similar results. These are however automatic consequences, cf.~Proposition~\ref{prop:smoothshriekbasechange}.

\begin{remark} The definition is phrased in a way that leaves room for a more general definition of $\ell$-cohomological smoothness for morphisms which are not separated, or not representable in locally spatial diamonds. For example, the map $\underline{M}\to \ast$ for a topological manifold $M$ should be considered smooth, but it is not representable in locally spatial diamonds. Moreover, smoothness should be a local condition, while being separated is not.
\end{remark}

Before investigating cohomologically smooth morphisms in detail, we give a different characterization that is easier to check in practice.

\begin{proposition}\label{prop:propsmooth} Let $f: Y^\prime\to Y$ be a compactifiable map of small v-stacks which is representable in spatial diamonds, and let $\ell\neq p$ be a prime. Then $f$ is $\ell$-cohomologically smooth if and only if the following conditions are satisfied.
\begin{altenumerate}
\item The dimension $\dimtrg f<\infty$ is finite.
\item For any strictly totally disconnected perfectoid space $X$ with a map $g: X\to Y$ and pullback $f_X: Y^\prime\times_Y X\to X$, and any constructible \'etale sheaf $\mathcal F$ of $\Fl$-vector spaces on the spatial diamond $Y^\prime\times_Y X$, the $!$-pushforward $R^i f_{X!} \mathcal F$ is constructible on $X$, for all $i\geq 0$.
\item For any $X=\Spa(C,C^+)$ with $C$ an algebraically closed field and $C^+\subset C$ an open and bounded valuation subring, with a quasicompact open subset $j: U\to X$, and pullbacks
\[\xymatrix{
Y^\prime\times_Y U\ar@{^(->}[r]^{j^\prime}\ar[d]^{f_U}& Y^\prime\times_Y X\ar[r]\ar[d]^{f_X}& Y^\prime\ar[d]^f\\
U\ar@{^(->}[r]^j& X\ar[r] &Y\ ,
}\]
the natural map
\[
j^\prime_! Rf_U^!\Fl\to Rf_X^! j_!\Fl
\]
adjoint to
\[
Rf_{X!} j^\prime_! Rf_U^!\Fl = j_! Rf_{U!} Rf_U^!\Fl\to j_!\Fl
\]
is an equivalence.
\item For any strictly totally disconnected perfectoid space $X$ with a map $g: X\to Y$ and pullback $f_X: Y^\prime\times_Y X \to X$, the $!$-pullback $Rf_X^! \Fl\in D_\et(Y^\prime\times_Y X,\Fl)$ is invertible.
\end{altenumerate}
\end{proposition}

\begin{proof} Note that given (ii), the condition of (iii) for quasicompact $U$ implies the same for all $U$ by Proposition~\ref{prop:Rfshriekconstructible} and passage to a filtered colimit over all quasicompact open subspaces. Thus, the result follows from the equivalence of (ii) and (iii) in Proposition~\ref{prop:shriekvsusualpullback} and Proposition~\ref{prop:Rfshriekconstructible}.
\end{proof}

One can deduce that cohomologically smooth maps are universally open.

\begin{proposition}\label{prop:cohomsmoothopen} Let $f: Y^\prime\to Y$ be separated map of small v-stacks that is representable in locally spatial diamonds and $\ell$-cohomologically smooth for some $\ell\neq p$. Then $f$ is universally open, i.e.~for all $X\to Y$, the map $|Y^\prime\times_Y X|\to |X|$ is open.
\end{proposition}

\begin{proof} We can assume that $Y=X$ is strictly totally disconnected, and $Y^\prime$ is spatial; it is enough to see that the image of $|Y^\prime|\to |X|$ is open. As the image is generalizing, it is enough to see that the image is constructible. Consider $Rf_! Rf^! \Fl\in D_\et(X,\Fl)$. As $f$ is $\ell$-cohomologically smooth, $Rf^! \Fl$ is invertible, and in particular constructible, and thus by Proposition~\ref{prop:propsmooth}~(ii), one sees that $Rf_! Rf^! \Fl$ is constructible. In particular, its support is a constructible subset of $X$. We claim that the support of $Rf_! Rf^! \Fl$ agrees with the image of $|Y^\prime|\to |X|$. For this, we can assume $X=\Spa(C,C^+)$ is strictly local (using that by Proposition~\ref{prop:shriekvsusualpullback} and Proposition~\ref{prop:propsmooth}, $Rf^!\Fl$ commutes with quasi-pro-\'etale base change). If the closed point $s\in |X|$ is not in the image of $f$, then $Rf^!\Fl$ is concentrated on the preimage on $X\setminus \{s\}$, which implies the same for $Rf_! Rf^! \Fl$ (by commuting $Rf_!$ with $j_!$). Now assume $s$ lies in the image of $f$, so $f$ is surjective. Let $j: U=X\setminus \{s\}\hookrightarrow X$ and $\tilde{j}: V=f^{-1}(U)\hookrightarrow Y^\prime$ be the open immersions, and $f_U: V\to U$ the pullback of $f$.

We have
\[
Rf_! Rf^! j_! \Fl = Rf_! \tilde{j}_! Rf_U^! \Fl = j_! Rf_{U!} Rf_U^! \Fl = j_! j^\ast(Rf_! Rf^! \Fl)\ ,
\]
using Proposition~\ref{prop:propsmooth}~(iii) in the first equality. In particular, if we define $\Fl|_s$ as the cone of $j_!\Fl\to \Fl$, one finds that $Rf_! Rf^! \Fl|_s$ is concentrated at $s$, with stalk equal to the stalk of $Rf_! Rf^! \Fl$. Now note that
\[
\Hom_{D_\et(X,\Fl)}(Rf_! Rf^! \Fl|_s,\Fl) = \Hom_{D_\et(Y^\prime,\Fl)}(Rf^! \Fl|_s,Rf^!\Fl)\ ,
\]
and $Rf^! \Fl|_s$ is the restriction $Rf^! \Fl|_{f^{-1}(s)}$ of the invertible sheaf $Rf^!\Fl$ to $f^{-1}(s)$. In particular, there is a natural map $Rf^! \Fl|_s\to Rf^!\Fl$, which is nonzero as soon as the fibre $f^{-1}(s)$ is nonempty. Thus, one has $Rf_! Rf^! \Fl|_s\neq 0$, and therefore $s$ lies in the support of $Rf_! Rf^! \Fl$.
\end{proof}

For cohomologically smooth morphisms, the functor $Rf^!$ commutes with any base change, and has a very simple form.

\begin{proposition}\label{prop:smoothshriekbasechange} Let $f: Y^\prime\to Y$ be a separated map of small v-stacks that is representable in locally spatial diamonds such that $f$ is $\ell$-cohomologically smooth, and let $\Lambda$ be an $\ell$-power-torsion ring.
\begin{altenumerate}
\item[{\rm (i)}] The natural transformation
\[
Rf^!\Lambda\otimes_\Lambda f^\ast\to Rf^!: D_\et(Y,\Lambda)\to D_\et(Y^\prime,\Lambda)
\]
is an equivalence, and the dualizing complex $D_f := Rf^!\Lambda\in D_\et(Y^\prime,\Lambda)$ is invertible, in fact \'etale locally isomorphic to $\Lambda[n]$ for some integer $n\in \mathbb Z$.
\item[{\rm (ii)}] If $f$ is quasicompact, then for any perfect-constructible $A\in D_\et(Y^\prime,\Lambda)$, the proper pushforward $Rf_! A\in D_\et(Y,\Lambda)$ is perfect-constructible.
\item[{\rm (iii)}] If $g: \tilde{Y}\to Y$ is a map of small v-sheaves with base change
\[\xymatrix{
\tilde{Y}^\prime\ar[r]^{g^\prime}\ar[d]^{\tilde{f}} & Y^\prime\ar[d]^f\\
\tilde{Y}\ar[r]^g & Y\ ,
}\]
then the natural transformation of functors
\[
g^{\prime\ast} Rf^!\to R\tilde{f}^! g^\ast : D_\et(Y,\Lambda)\to D_\et(\tilde{Y}^\prime,\Lambda)
\]
adjoint to
\[
R\tilde{f}_! g^{\prime\ast} Rf^! = g^\ast Rf_!Rf^!\to g^\ast
\]
is an equivalence. In particular, the dualizing complex $D_f$ commutes with base change, i.e.
\[
g^{\prime\ast} D_f = D_{\tilde{f}}\ .
\]
\end{altenumerate}
\end{proposition}

\begin{proof} By Proposition~\ref{prop:shriekvsusualpullback}, we know that if $Y=X$ is strictly totally disconnected, then
\[
Rf^! \Lambda\otimes_\Lambda f^\ast\to Rf^!
\]
is an equivalence. Moreover, we see that $Rf^!\Lambda\otimes_\Lambda \Fl=Rf^! \Fl$ is invertible, which implies formally that $Rf^!\Lambda$ is invertible (as it is an extension of $m$ copies of $\Fl[n]$). Thus, part (i) holds true in case $Y$ is a strictly totally disconnected perfectoid space.

Next, we check that part (iii) holds true in case $Y=X$ and $\tilde{Y}=\tilde{X}$ are both strictly totally disconnected perfectoid spaces. By part (i), it suffices to check that $g^{\prime\ast} D_f=D_{\tilde{f}}$. Using part (iv) of Proposition~\ref{prop:shriekvsusualpullback}, we can assume that $X$ and $\tilde{X}$ are connected, and in fact we can assume that $X=\Spa(C,\OO_C)$ and $\tilde{X} = \Spa(\tilde{C},\OO_{\tilde{C}})$, using that $Rf^!$ and $g^{\prime\ast}$ commute with $j_\ast$ for quasicompact open immersions. Also, the equality $g^{\prime\ast} D_f = D_{\tilde{f}}$ can be checked locally on $Y^\prime$, so we can assume that $Y^\prime$ is spatial.

Thus, assume that $X=\Spa(C,\OO_C)$ and $\tilde{X}=\Spa(\tilde{C},\OO_{\tilde{C}})$, and that $Y^\prime$ is spatial. First, note that by assumption $Rf^! \Fl$ is invertible and in particular constructible, so by Proposition~\ref{prop:Rfshriekconstructible}, $Rf_! Rf^! \Fl\in D_\et(X,\Fl)=D(\Fl)$ is constructible, i.e.~a bounded complex of finite-dimensional vector spaces. Thus, $\Hom_X(Rf_! Rf^! \Fl,\Fl)=\Hom_{Y^\prime}(Rf^! \Fl,Rf^! \Fl) = \Hom_{Y^\prime}(\Fl,\Fl)$ is finite-dimensional, which implies that $\pi_0 Y^\prime$ is finite. Passing to a connected component, we can thus assume that $Y^\prime$ is connected. In that case, $\tilde{Y}^\prime$ is still connected (for example by Theorem~\ref{thm:changebasefield}~(iii)), and so both $D_f$ and $D_{\tilde{f}}$ are isomorphic to $\mathbb L[n]$ resp.~$\widetilde{\mathbb L}[\tilde{n}]$ for some $\Fl$-local systems $\mathbb L$ on $Y^\prime$ resp.~$\widetilde{\mathbb L}$ on $\tilde{Y}^\prime$, and integers $n$ resp.~$\tilde{n}$. The map $g^{\prime\ast} D_f\to D_{\tilde{f}}$ is given by a map
\[
g^{\prime\ast}\mathbb L[n]\to \widetilde{\mathbb L}[\tilde{n}]\ ,
\]
such that for any $K\in D_\et(Y^\prime,\Fl)$, one has a commutative diagram
\[\xymatrix{
\Hom_{D_\et(Y^\prime,\Fl)}(K,\mathbb L[n])\ar[r]\ar[d]^\cong & \Hom_{D_\et(\tilde{Y}^\prime,\Fl)}(g^{\prime\ast} K,g^{\prime\ast} \mathbb L[n])\ar[r] & \Hom_{D_\et(\tilde{Y}^\prime,\Fl)}(g^{\prime\ast} K,\widetilde{\mathbb L}[\tilde{n}])\ar[d]^\cong\\
\Hom_{D_\et(X,\Fl)}(Rf_! K,\Fl)\ar[r]^\cong & \Hom_{D_\et(\tilde{X},\Fl)}(g^\ast Rf_! K,\Fl)\ar[r]^\cong & \Hom_{D_\et(\tilde{X},\Fl)}(R\tilde{f}_! g^{\prime\ast} K,\Fl)\ .
}\]
If we apply this to $K=\mathbb L[n]$, then by invariance of \'etale cohomology under change of algebraically closed base field, the upper left arrow is the isomorphism $R\Gamma(Y^\prime_\et,\Fl)\cong R\Gamma(\tilde{Y}^\prime_\et,\Fl)$, which is nonzero. Then the diagram implies that the map $g^{\prime\ast} \mathbb L[n]\to \widetilde{\mathbb L}[\tilde{n}]$ in $D_\et(\tilde{Y}^\prime,\Fl)$ cannot be the zero map. This implies that $\tilde{n}\geq n$, and that it is an isomorphism if $\tilde{n}=n$.

Thus, if for any $\tilde{C}/C$ as above, one has $\tilde{n}=n$, then it follows that base change holds. In general, $\tilde{n}\leq 3\dimtrg f$ is bounded in terms of $f$, so at most finitely many values appear, and we can choose some $\tilde{C}/C$ achieving the maximal value. It follows that $\tilde{f}: \tilde{Y}^\prime\to \tilde{X}=\Spa(\tilde{C},\OO_{\tilde{C}})$ has the property that $D_{\tilde{f}}$ commutes with any base change to a strictly totally disconnected perfectoid space over $\tilde{X}$.

Now choose a v-hypercover of $X=\Spa(C,\OO_C)$ by strictly totally disconnected perfectoid spaces $\tilde{X}_\bullet\to X$, with $\tilde{X}_0 = \tilde{X}$, and let $\tilde{f}_\bullet: \tilde{Y}^\prime_\bullet\to \tilde{X}_\bullet$ be the pullback. Then we have the derived category $D(\tilde{X}_\bullet,\Lambda)$ of sheaves of $\Lambda$-modules on the simplicial space $\tilde{X}_\bullet$, and the full subcategory $D_{\cart,\et}(\tilde{X}_\bullet,\Lambda)\subset D(\tilde{X}_\bullet,\Lambda)$ where all pullback maps are isomorphisms, and all terms lie in $D_\et(\tilde{X}_i,\Lambda)$. By Proposition~\ref{prop:derivedhyperdescent}, one has an equivalence of categories
\[
D_\et(X,\Lambda)\cong D_{\cart,\et}(\tilde{X}_\bullet,\Lambda)\ .
\]
Similarly,
\[
D_\et(Y^\prime,\Lambda)\cong D_{\cart,\et}(\tilde{Y}^\prime_\bullet,\Lambda)\ .
\]
Moreover, the functor $Rf_!: D_\et(Y^\prime,\Lambda)\to D_\et(X,\Lambda)$ gets identified with the functor
\[
R\tilde{f}_{\bullet !}: D_{\cart,\et}(\tilde{Y}^\prime_\bullet,\Lambda)\to D_{\cart,\et}(\tilde{X}_\bullet,\Lambda)
\]
which is termwise given by $R\tilde{f}_{i !}$. It preserves the ``cartesian'' condition by proper base change.

Now note that all terms $R\tilde{f}_{i !}$ admit right adjoints $R\tilde{f}_i^!$, and these right adjoints preserve the ``cartesian'' condition by the base change already established. This implies that the right adjoint to
\[
R\tilde{f}_{\bullet !}: D_{\cart,\et}(\tilde{Y}^\prime_\bullet,\Lambda)\to D_{\cart,\et}(\tilde{X}_\bullet,\Lambda)
\]
is given by the functor
\[
R\tilde{f}_\bullet^!: D_{\cart,\et}(\tilde{X}_\bullet,\Lambda)\to D_{\cart,\et}(\tilde{Y}^\prime_\bullet,\Lambda)
\]
that is termwise $R\tilde{f}_i^!$. This gives a commutative diagram
\[\xymatrix{
D_\et(X,\Lambda)\ar[r]^{Rf^!}\ar[d]^\cong & D_\et(Y^\prime,\Lambda)\ar[d]^\cong\\
D_{\cart,\et}(\tilde{X}_\bullet,\Lambda)\ar[r]^{R\tilde{f}_\bullet^!}\ar[d] & D_{\cart,\et}(\tilde{Y}^\prime_\bullet,\Lambda)\ar[d] \\
D_\et(\tilde{X},\Lambda)\ar[r]^{R\tilde{f}^!} & D_\et(\tilde{Y}^\prime,\Lambda)\ ,
}\]
where the vertical maps are given by pullback. This gives the desired base change result along $\Spa(\tilde{C},\OO_{\tilde{C}})\to \Spa(C,\OO_C)$ in general, which finishes the proof of the base change result for $Rf^!$ along maps of strictly totally disconnected perfectoid spaces.

In the general case, we can now find a v-hypercover $\tilde{X}_\bullet\to Y$ of $Y$ such that all $\tilde{X}_i$ are disjoint unions of strictly totally disconnected perfectoid spaces. Applying the previous discussion again, we get the desired base change for $Rf^!$ along the map $\tilde{X}_0\to Y$; this implies also that $Rf^! \Lambda\otimes_\Lambda f^\ast\to Rf^!$ is an equivalence, as this can now be checked after base change to $\tilde{X}_0$, where we know it. Moreover, $Rf^!\Lambda$ is v-locally isomorphic to $\Lambda[n]$ for some integer $n\in \mathbb Z$, which implies the same \'etale locally, as the integer $n$ is constant on an open and closed stratification of $Y^\prime$, and then the space of isomorphisms between $Rf^! \Lambda$ and $\Lambda[n]$ is separated, \'etale and surjective over $Y^\prime$ as follows by v-descent.

Now, for the general base change identity, we can cover any map $\tilde{Y}\to Y$ by a map between disjoint unions of strictly totally disconnected perfectoid spaces $\tilde{X}\to X$; we know base change for $\tilde{X}\to \tilde{Y}$, $X\to Y$ and $\tilde{X}\to X$, which implies it for $\tilde{Y}\to Y$.

Finally, for part (ii), we can by proper base change assume that $Y$ is a strictly totally disconnected perfectoid space, in which case $Y^\prime$ is a spatial diamond of bounded cohomological dimension (as $Rf_\ast$ has bounded cohomological dimension since $\dimtrg f<\infty$, and $Y$ is strictly totally disconnected). Thus, by Proposition~\ref{prop:perfectconstructiblecompact}, $D_\et(Y^\prime,\Lambda)$ and $D_\et(Y,\Lambda)$ are compactly generated with compact objects given by the perfect-constructible complexes. We need to see that $Rf_!$ preserves compact objects, but this is equivalent to the condition that $Rf^!$ preserves arbitrary direct sums, which follows from part (i).
\end{proof}

Using Proposition~\ref{prop:smoothshriekbasechange}, one can show that composites of cohomologically smooth morphisms are cohomologically smooth.

\begin{proposition}\label{prop:smooth2outof3} Let $g: Y^{\prime\prime}\to Y^\prime$ and $f: Y^\prime\to Y$ be separated morphisms of small v-stacks and let $\ell\neq p$ be a prime. If $f$ and $g$ are representable in locally spatial diamonds and $\ell$-cohomologically smooth, then $f\circ g$ is representable in locally spatial diamonds and $\ell$-cohomologically smooth. Conversely, if $g$ and $f\circ g$ are representable in locally spatial diamonds and $\ell$-cohomologically smooth, $g$ is surjective, and $f$ is representable in diamonds and compactifiable, then $f$ is representable in locally spatial diamonds and $\ell$-cohomologically smooth.
\end{proposition}

\begin{remark} The first part of the proof shows that if $f: Y^\prime\to Y$ is a separated $0$-truncated map of small v-stacks that is representable in diamonds, and $g: Y^{\prime\prime}\to Y^\prime$ is a universally open (e.g., $\ell$-cohomologically smooth), separated and surjective map of small v-stacks such that $f\circ g$ is representable in locally spatial diamonds, then $f$ is representable in locally spatial diamonds.
\end{remark}

We do not know if in the second part, one can omit the hypothesis ``$f$ is compactifiable''; by Proposition~\ref{prop:compactifiable}~(vii), it can be omitted if $g$ has local sections after pullback to strictly totally disconnected spaces. This is satisfied for all $\ell$-cohomologically smooth maps that occur in practice. In fact, they are usually also formally smooth in the sense of \cite[Definition IV.3.1]{FarguesScholze}, and this implies this splitting condition by \cite[Proposition IV.3.5]{FarguesScholze}.

\begin{proof} The first part follows from Proposition~\ref{prop:smoothshriekbasechange}. For the second part, we check first that $f$ is representable in locally spatial diamonds. By Proposition~\ref{prop:locallyspatialmorphism}, we can check this after pullback to a strictly totally disconnected $Y$. In that case, $Y^{\prime\prime}$ is separated and locally spatial over $Y$. For any quasicompact open subspace $V\subset Y^{\prime\prime}$, the image $U\subset Y^\prime$ is open by Proposition~\ref{prop:cohomsmoothopen}. Moreover, it is quasicompact, as $V$ surjects onto $U$. Also, $|U|$ is the quotient of the spectral space $|V|$ by an open qcqs equivalence relation (given by the image of $|V\times_U V|$), so $|U|$ is spectral and $|V|\to |U|$ is a spectral map by Lemma~\ref{lem:spectralopenequivrel}. Thus, $U$ is spatial, and $Y^\prime$ is locally spatial.

One easily checks that locally $\dimtrg f\leq \dimtrg (f\circ g)<\infty$ (where one uses the modified $\widetilde{\trc}$ in the definition of $\dimtrg$), and $f$ is compactifiable by assumption. We may assume $Y=X$ is strictly totally disconnected, and we have to check that
\[
Rf^!\Fl\otimes_\Fl f^\ast \to Rf^!
\]
is an equivalence, and that $Rf^!\Fl$ is invertible. But as $f\circ g$ is $\ell$-cohomologically smooth,
\[
R(f\circ g)^! \Fl = Rg^!(Rf^!\Fl) = Rg^! \Fl \otimes_\Fl g^\ast Rf^! \Fl
\]
is invertible, and $Rg^! \Fl$ is invertible as $g$ is $\ell$-cohomologically smooth; this implies that $g^\ast Rf^! \Fl$ is invertible, which shows that $Rf^! \Fl$ is invertible. Similarly, to check whether $Rf^! \Fl\otimes_\Fl f^\ast\to Rf^!$ is an equivalence, we can check after applying $Rg^! = Rg^! \Fl\otimes_\Fl g^\ast$. But
\[
Rg^! Rf^! = R(f\circ g)^! = R(f\circ g)^! \Fl\otimes_\Fl (f\circ g)^\ast = Rg^!\Fl\otimes_\Fl g^\ast Rf^!\Fl\otimes_\Fl g^\ast f^\ast = Rg^!(Rf^!\Fl\otimes_\Fl f^\ast)\ ,
\]
as desired.
\end{proof}

Also, we can check cohomological smoothness v-locally on the target.

\begin{proposition}\label{prop:cohomsmoothvlocal} Let $f: Y^\prime\to Y$ be a separated map of small v-stacks that is representable in locally spatial diamonds, and let $\ell\neq p$ be a prime. Let $g: \tilde{Y}\to Y$ be map of small v-stacks with pullback $\tilde{f}: \tilde{Y}^\prime = Y^\prime\times_Y \tilde{Y}\to \tilde{Y}$.

If $f$ is $\ell$-cohomologically smooth, then $\tilde{f}$ is $\ell$-cohomologically smooth. Conversely, if $\tilde{f}$ is $\ell$-cohomologically smooth, $g$ is a surjective map of v-stacks, and locally $\dimtrg f<\infty$, then $f$ is $\ell$-cohomologically smooth.
\end{proposition}

It is not clear to us whether the condition that locally $\dimtrg f<\infty$ can be checked v-locally on the target.

\begin{proof} The first statement is clear from the definition. For the converse, we can assume that $Y=X$ and $\tilde{Y}=\tilde{X}$ are strictly totally disconnected. Then it follows from the argument involving a simplicial v-hypercover $\tilde{X}_\bullet\to X$ with $\tilde{X}_0=\tilde{X}$ in the proof of Proposition~\ref{prop:smoothshriekbasechange}.
\end{proof}

Finally, let us note that we get the expected smooth base change results.

\begin{proposition}\label{prop:smoothbasechange} Let
\[\xymatrix{
Y^\prime\ar[r]^{\tilde{g}}\ar[d]^{f^\prime} &Y\ar[d]^f\\
X^\prime\ar[r]^g & X
}\]
be a cartesian diagram of small v-stacks, and assume that $n\Lambda=0$ for some $n$ prime to $p$.
\begin{altenumerate}
\item Assume that $g$ is compactifiable and representable in locally spatial diamonds with locally $\dimtrg g<\infty$, and $A\in D_\et(Y,\Lambda)$. Then
\[
Rg^! Rf_\ast A\cong Rf^\prime_\ast R\tilde{g}^! A\ .
\]
\item Assume that $g$ is separated, representable in locally spatial diamonds and $\ell$-cohomologically smooth, and $\Lambda$ is $\ell$-power torsion. Then for all $A\in D_\et(Y,\Lambda)$, the base change morphism
\[
g^\ast Rf_\ast A\to Rf^\prime_\ast \tilde{g}^\ast A
\]
is an isomorphism.
\item Assume that $g$ is separated, representable in locally spatial diamonds and $\ell$-cohomologically smooth, and $\Lambda$ is $\ell$-power torsion. Then for all $A\in D_\et(X,\Lambda)$, the map
\[
\tilde{g}^\ast Rf^! A\to Rf^{\prime !} g^\ast A
\]
adjoint to $Rf^! A\to Rf^! Rg_\ast g^\ast A = R\tilde{g}_\ast Rf^{\prime !} g^\ast A$ is an equivalence.
\end{altenumerate}
\end{proposition}

\begin{proof} Part (i) follows from Proposition~\ref{prop:Rfshriekbasechange} by passing to right adjoints. Now part (ii) follows from part (i) and Proposition~\ref{prop:smoothshriekbasechange}. Finally, for part (iii), we can tensor by $R\tilde{g}^! \Lambda$ which is invertible. The left-hand side becomes
\[
\tilde{g}^\ast Rf^! A\dotimes_\Lambda R\tilde{g}^! \Lambda = R\tilde{g}^! Rf^! A = Rf^{\prime !} Rg^! A = Rf^{\prime !}(g^\ast A\dotimes_\Fl Rg^! \Fl)
\]
using $\ell$-cohomological smoothness of $g$ and $\tilde{g}$. Finally, as $Rg^!\Fl$ is invertible,
\[
Rf^{\prime !}(g^\ast A\dotimes_\Fl Rg^!\Fl) = Rf^{\prime !} g^\ast A\dotimes_\Fl f^{\prime\ast} Rg^!\Fl\ ,
\]
and $f^{\prime\ast} Rg^! \Fl = R\tilde{g}^!\Fl$ by Proposition~\ref{prop:smoothshriekbasechange}.
\end{proof}

\begin{proposition}\label{prop:smoothhompullback} Let $f: Y\to X$ be a separated map of small v-stacks that is representable in locally spatial diamonds and $\ell$-cohomologically smooth. There is a functorial isomorphism
\[
f^\ast R\sHom(A,B)\cong R\sHom(f^\ast A,f^\ast B)
\]
in $A,B\in D_\et(X,\Lambda)$.
\end{proposition}

\begin{proof} This follows from Proposition~\ref{prop:localverdier}~(ii) together with the identification $Rf^! = f^\ast\otimes_\Lambda Rf^! \Lambda$, where $Rf^! \Lambda$ is invertible.
\end{proof}

\section{Examples of smooth morphisms}\label{sec:geomsmooth}

In the previous section, we have defined cohomologically smooth morphisms, and checked that in this case $!$-pullback takes a very simple form. Moreover, the class of cohomologically smooth morphisms has good stability properties. However, we have not seen any example of cohomologically smooth morphisms. In this section, we will give criteria guaranteeing cohomological smoothness.

We start with the following important theorem, which uses the main results of Huber's book.

\begin{theorem}\label{thm:ballcohomsmooth} Let $f: \mathbb B\to \ast$ be the projection from the ball to the point, where $\mathbb B(R,R^+)=R^+$. Then $f$ is $\ell$-cohomologically smooth for any prime $\ell\neq p$, and for any ring $\Lambda$ with $n\Lambda=0$ for some $n$ prime to $p$, one has a canonical isomorphism
\[
D_f = Rf^!\Lambda\cong \Lambda(1)[2]\ .
\]
\end{theorem}

\begin{proof} Fix a prime $\ell\neq p$. We check the conditions of Proposition~\ref{prop:propsmooth}. Condition (i) is easy (in fact, $\dimtrg f=1$). For condition (ii), let $X$ be a strictly totally disconnected perfectoid space, and let $\mathcal F$ be a constructible $\ell$-torsion sheaf on $\mathbb B\times X$. We can find a map $X\to \Spa(K,\OO_K)$, where $K$ is a perfectoid field (given by $\mathbb F_p((\varpi^{1/p^\infty}))$ for a pseudo-uniformizer $\varpi$). Then we can write $X$ as a cofiltered inverse limit of spaces $X_i\to \Spa(K,\OO_K)$ which are topologically of finite type, cf.~\cite[Lemma 6.13]{ScholzePerfectoidSpaces}, and $\mathcal F$ is the preimage of some constructible sheaf $\mathcal F_i$ on $\mathbb B\times X_i$ for $i$ large enough by Proposition~\ref{prop:constructiblelimit}. Consider the cartesian diagram
\[\xymatrix{
\mathbb B\times X\ar[d]^{f_X}\ar[r]^h & \mathbb B\times X_i\ar[d]^{f_{X_i}}\\
X\ar[r]^g & X_i\ .
}\]
Then for all $j\geq 0$, one has
\[
R^j f_{X!} \mathcal F = R^j f_{X!} h^\ast \mathcal F_i = g^\ast R^j f_{X_i !} \mathcal F_i\ .
\]
By~\cite[Theorem 6.2.2]{Huber}, the sheaf $R^j f_{X_i !} \mathcal F_i$ is constructible on $X_i$, and thus its pullback along $g^\ast$ is a constructible sheaf on $X$, as desired.

Condition (iii) is a condition about $f_X: \mathbb B\times X\to X$ in case $X=\Spa(C,C^+)$ is strongly Noetherian. Thus, Huber's results and in particular \cite[Theorem 7.5.3]{Huber} applies.

For condition (iv), take again a strictly totally disconnected perfectoid space $X$. We can in fact find a map $X\to \Spa(C,\OO_C)$, where $C$ is algebraically closed: As above, we have a map $X\to \Spa(K,\OO_K)$ where $K=\mathbb F_p((\varpi^{1/p^\infty}))$, but in fact this factors over $\Spa(C,\OO_C)$, where $C$ is the completed algebraic closure of $K$, as $X$ has no nonsplit finite \'etale covers. Consider the cartesian diagram
\[\xymatrix{
\mathbb B\times X\ar[d]^{f_X}\ar[r]^{\!\!\!\!\! h} & \mathbb B\times \Spa(C,\OO_C)\ar[d]^{f_C}\\
X\ar[r]^{\!\!\!\!\! g} & \Spa(C,\OO_C)\ .
}\]
There is a trace map
\[
Rf_{C!} \Fl(1)\to \Fl[-2]
\]
by \cite[Theorem 7.2.2]{Huber} (noting that $R^i f_{C!} \Fl(1)=0$ for $i>2$ by \cite[Proposition 5.5.8]{Huber}). By base change, we get a trace map
\[
Rf_{X!} \Fl(1) = g^\ast Rf_{C!} \Fl(1)\to \Fl[-2]\ ,
\]
which by adjunction gives a map
\[
\alpha: \Fl(1)[2]\to Rf_X^! \Fl\ .
\]
We claim that $\alpha$ is an isomorphism. Note that we have already proved that condition (iii) of Proposition~\ref{prop:shriekvsusualpullback} is satisfied, so also condition (iv) is satisfied. In particular, $Rf_X^! \Fl$ commutes with quasi-pro-\'etale base change, and so we can check whether $\alpha$ is an equivalence after pullback to connected components of $X$. Thus, we can assume $X=\Spa(C,C^+)$. In this case, the result follows from \cite[Theorem 7.5.3]{Huber}.
\end{proof}

We also need the following result.

\begin{proposition}\label{prop:propquotientsmooth} Let $f: Y^\prime\to Y$ be a separated map of small v-stacks that is representable in locally spatial diamonds and $\ell$-cohomologically smooth for some prime $\ell\neq p$. Assume that there is a profinite group $K$ of pro-order prime to $\ell$ with a free action $\underline{K}\times Y^\prime\to Y^\prime$ over $Y$, i.e.~such that $\underline{K}\times Y^\prime\to Y^\prime\times_Y Y^\prime$ is an injection.

Then the quotient map $f/\underline{K}: Y^\prime/\underline{K}\to Y$ is separated, representable in locally spatial diamonds, and $\ell$-cohomologically smooth. Moreover, if we fix the $\Lambda$-valued Haar measure on $K$ with total volume $1$, there is a natural equivalence of functors
\[
Rf^! = q^\ast R(f/\underline{K})^!: D_\et(Y,\Lambda)\to D_\et(Y^\prime,\Lambda)\ ,
\]
where $q: Y^\prime\to Y^\prime/\underline{K}$ denotes the quotient map.
\end{proposition}

One also has the identity
\[
Rf^! = Rq^! R(f/\underline{K})^!\ .
\]
However, it is not true that $Rq^! = q^\ast$. For example, after a base change, $q$ becomes $\underline{S}\times \Spa(C,\OO_C)\to \Spa(C,\OO_C)$ for some profinite set $S$ and algebraically closed perfectoid field $C$. In this case, $Rq^! \Lambda$ can be identified with the sheaf sending any open and closed subset $T\subset S$ to the distributions $\Hom(C^0(T,\Lambda),\Lambda)$, as follows easily from the adjunction defining $Rq^!$. Under the choice of a Haar measure, we get however a natural map $q^\ast\to Rq^!$, as in the proof below.

\begin{proof} We can assume that $Y=X$ is a strictly totally disconnected perfectoid space. Then $Y^\prime/\underline{K}$ is locally spatial (and quasiseparated) by Lemma~\ref{lem:spectralopenequivrel} and Lemma~\ref{lem:gtorsoroverperfectoid}. It is also separated by the valuative criterion. Its canonical compactification $\overline{Y^\prime/\underline{K}}^{/Y}$ is given by $\overline{Y^\prime}^{/Y}/\underline{K}$. To check whether the inclusion is an open immersion, one uses that $\overline{Y^\prime}^{/Y}\to \overline{Y^\prime}^{/Y}/\underline{K}$ is a quotient map.

It remains to see that $f/\underline{K}$ is $\ell$-cohomologically smooth. It is clear that locally $\dimtrg f/\underline{K} = \dimtrg f<\infty$. Now, fixing the $\Lambda$-valued Haar measure on $K$ with total volume $1$, we first construct a natural transformation
\[
q^\ast\to Rq^!: D_\et(Y^\prime/\underline{K},\Lambda)\to D_\et(Y^\prime,\Lambda)\ .
\]
This is adjoint to a map $Rq_! q^\ast = Rq_\ast q^\ast = q_\ast q^\ast\to \id$. For any open subgroup $H\subset K$, consider the projection $q_{H,K}: Y^\prime/\underline{H}\to Y^\prime/\underline{K}$. Then $q_\ast q^\ast$ is the filtered colimit of $q_{H,K\ast} q_{H,K}^\ast$, and there are natural trace maps $q_{H,K\ast} q_{H,K}^\ast\to \id$, which one can divide by $[K:H]$ to make them compatible. (Note that here, we are implicitly using the choice of the Haar measure.) This gives the desired natural transformation $q_\ast q^\ast\to \id$, and thus $q^\ast\to Rq^!$.

In particular, we get a natural transformation
\[
q^\ast R(f/\underline{K})^!\to Rq^! R(f/\underline{K})^! = Rf^!\ .
\]
We claim that this is an equivalence. It suffices to check this for $\Lambda=\Fl$. Note that once this is known, it follows that $f/\underline{K}$ satisfies condition (iii) of Proposition~\ref{prop:shriekvsusualpullback}, and $R(f/\underline{K})^! \Fl$ is invertible, which shows that $f/\underline{K}$ is $\ell$-cohomologically smooth.

Thus, it remains to check that
\[
q^\ast R(f/\underline{K})^!\to Rq^! R(f/\underline{K})^! = Rf^!: D_\et(X,\Fl)\to D_\et(Y^\prime,\Fl)
\]
is an equivalence. This can be checked locally on $Y^\prime$, so we can assume that $Y^\prime$ is spatial. By Proposition~\ref{prop:constructiblecompactderivedfull}, it suffices to check that for any constructible sheaf $\mathcal F$ on $Y^\prime$ and any $A\in D_\et(X,\Fl)$, one has
\[
\Hom_{D_\et(Y^\prime,\Fl)}(\mathcal F,q^\ast R(f/\underline{K})^! A) = \Hom_{D_\et(Y^\prime,\Fl)}(\mathcal F, Rf^! A)\ .
\]
Let $q_H: Y^\prime\to Y^\prime/\underline{H}$ be the projection. By Proposition~\ref{prop:constructiblelimit}, we can assume that $\mathcal F = q_H^\ast \mathcal F_H$ for some constructible sheaf $\mathcal F_H$ on $Y^\prime/\underline{H}$. Then
\[\begin{aligned}
\Hom_{D_\et(Y^\prime,\Fl)}(\mathcal F,q^\ast R(f/\underline{K})^! A) &= \Hom_{D_\et(Y^\prime,\Fl)}(q_H^\ast \mathcal F,q^\ast R(f/\underline{K})^! A)\\
&= \Hom_{D_\et(Y^\prime/\underline{H},\Fl)}(\mathcal F_H,q_{H\ast} q^\ast R(f/\underline{K})^! A)\\
&= \Hom_{D_\et(Y^\prime/\underline{H},\Fl)}(\mathcal F_H,\varinjlim_{H^\prime\subset H} q_{H^\prime,H\ast} q_{H^\prime,K}^\ast R(f/\underline{K})^! A)\\
&= \varinjlim_{H^\prime\subset H} \Hom_{D_\et(Y^\prime/\underline{H},\Fl)}(\mathcal F_H, q_{H^\prime,H\ast} q_{H^\prime,K}^\ast R(f/\underline{K})^! A)\\
&= \varinjlim_{H^\prime\subset H} \Hom_{D_\et(Y^\prime/\underline{H^\prime},\Fl)}(q_{H^\prime,H}^\ast \mathcal F_H, q_{H^\prime,K}^\ast R(f/\underline{K})^! A)\\
&= \varinjlim_{H^\prime\subset H} \Hom_{D_\et(Y^\prime/\underline{H^\prime},\Fl)}(q_{H^\prime,H}^\ast \mathcal F_H, Rq_{H^\prime,K}^! R(f/\underline{K})^! A)\\
&= \varinjlim_{H^\prime\subset H} \Hom_{D_\et(Y^\prime/\underline{H^\prime},\Fl)}(q_{H^\prime,H}^\ast \mathcal F_H, R(f/\underline{H^\prime})^! A)\\
&= \varinjlim_{H^\prime\subset H} \Hom_{D_\et(X,\Fl)}(R(f/\underline{H^\prime})_! q_{H^\prime,H}^\ast \mathcal F_H, A)\ .
\end{aligned}\]
Let $\mathcal F_{H^\prime} = q_{H^\prime,H}^\ast \mathcal F_H$. We claim that the filtered direct system
\[
R(f/\underline{H^\prime})_! \mathcal F_{H^\prime}\in D_\et(X,\Fl)
\]
is eventually constant, i.e.~for all sufficiently small $H^\prime$, all transition maps are isomorphisms. For this, note first that all transition maps are split injective (using the trace maps as above). On the other hand, $Rf_! \mathcal F\in D_\et(X,\Fl)$ is bounded with constructible cohomology sheaves (i.e., compact), and
\[\begin{aligned}
Rf_! \mathcal F &= Rf_! q_H^\ast \mathcal F = R(f/\underline{H})_! Rq_{H!} q_H^\ast \mathcal F\\
&= R(f/\underline{H})_! q_{H\ast} q_H^\ast \mathcal F\\
&= R(f/\underline{H})_! \varinjlim_{H^\prime\subset H} q_{H^\prime,H\ast} q_{H^\prime,H}^\ast \mathcal F\\
&= \varinjlim_{H^\prime\subset H} R(f/\underline{H})_! q_{H^\prime,H\ast} \mathcal F_{H^\prime}\\
&= \varinjlim_{H^\prime\subset H} R(f/\underline{H^\prime})_! \mathcal F_{H^\prime}\ .
\end{aligned}\]
The map to the direct limit factors over some term in the direct limit by compactness of $Rf_! \mathcal F$. As all transition maps are split injective, this shows that indeed the system $R(f/\underline{H^\prime})_! \mathcal F_{H^\prime}$ is eventually constant, and eventually equal to $Rf_! \mathcal F$. Continuing the equations above, we see that
\[\begin{aligned}
\Hom_{D_\et(Y^\prime,\Fl)}(\mathcal F,q^\ast R(f/\underline{K})^! A) &=\varinjlim_{H^\prime\subset H} \Hom_{D_\et(X,\Fl)}(R(f/\underline{H^\prime})_! q_{H^\prime,H}^\ast \mathcal F_H, A)\\
&= \Hom_{D_\et(X,\Fl)}(Rf_! \mathcal F,A)\\
&= \Hom_{D_\et(Y^\prime,\Fl)}(\mathcal F, Rf^! A)\ ,
\end{aligned}\]
as desired.
\end{proof}

There is a variant of the proposition when the action is not free, but we know already that all fibres are smooth.

\begin{proposition}\label{prop:propnonfreequotientsmooth} Let $f: Y^\prime\to Y$ be a separated map of small v-stacks that is representable in locally spatial diamonds and $\ell$-cohomologically smooth for some prime $\ell\neq p$. Assume that there is a profinite group $K$ of pro-order prime to $\ell$ with an action $\underline{K}\times Y^\prime\to Y^\prime$ over $Y$ for which the map $\underline{K}\times Y^\prime\to Y^\prime\times_Y Y^\prime$ is $0$-truncated and qcqs.

Let $Y^\prime/\underline{K}$ be the quotient of $Y^\prime$ by the equivalence relation which is given as the image of $\underline{K}\times Y^\prime\to Y^\prime\times_Y Y^\prime$. Then the quotient map $f/\underline{K}: Y^\prime/\underline{K}\to Y$ is separated and representable in locally spatial diamonds. Moreover, if for all complete algebraically closed fields $C$ with an open and bounded valuation subring $C^+\subset C$, and all maps $\Spa(C,C^+)\to Y$, the pullback $Y^\prime/\underline{K}\times_Y \Spa(C,C^+)\to \Spa(C,C^+)$ is $\ell$-cohomologically smooth, then $f/\underline{K}: Y^\prime/\underline{K}\to Y$ is $\ell$-cohomologically smooth, and if we fix the $\Lambda$-valued Haar measure on $K$ with total volume $1$, there is a natural equivalence of functors
\[
Rf^! = q^\ast R(f/\underline{K})^!: D_\et(Y,\Lambda)\to D_\et(Y^\prime,\Lambda)\ ,
\]
where $q: Y^\prime\to Y^\prime/\underline{K}$ denotes the quotient map.
\end{proposition}

\begin{proof} First, we check that $Y^\prime/\underline{K}\to Y$ is compactifiable, representable in locally spatial diamonds, with locally $\dimtrg <\infty$. For this, we can assume that $Y=X$ is a strictly totally disconnected perfectoid space. Then $Y^\prime/\underline{K}$ is locally spatial (and quasiseparated) by Lemma~\ref{lem:spectralopenequivrel} and Lemma~\ref{lem:gtorsoroverperfectoid}. It is also separated by the valuative criterion. Its canonical compactification $\overline{Y^\prime/\underline{K}}^{/Y}$ is given by $\overline{Y^\prime}^{/Y}/\underline{K}$. To check whether the inclusion is an open immersion, one uses that $\overline{Y^\prime}^{/Y}\to \overline{Y^\prime}^{/Y}/\underline{K}$ is a quotient map. It is clear that locally $\dimtrg f/\underline{K} = \dimtrg f<\infty$.

It remains to see that $f/\underline{K}$ is $\ell$-cohomologically smooth, and that $Rf^! = q^\ast R(f/\underline{K})^!$. We start by constructing a natural transformation
\[
q^\ast\to Rq^!: D_\et(Y^\prime/\underline{K},\Lambda)\to D_\et(Y^\prime,\Lambda)\ .
\]
Recall that there is always a natural transformation
\[
Rq^! \Lambda\dotimes_\Lambda q^\ast\to Rq^!\ ;
\]
thus, it is enough to construct a natural map $\Lambda\to Rq^!\Lambda$. This is adjoint to a trace map $Rq_! \Lambda\to \Lambda$. But note that $q$ is proper and quasi-pro-\'etale, so $Rq_! = Rq_\ast$, and $Rq_\ast=q_\ast$ is of cohomological dimension $0$. Thus, it remains to construct a map of sheaves $q_\ast \Lambda\to \Lambda$. This map can be constructed v-locally, and is given by integrating over the $K$-action, fixing the $\Lambda$-valued Haar measure on $K$ with total volume $1$.

Now we prove that $f/\underline{K}$ is $\ell$-cohomologically smooth. For this, we can assume that $X$ is strictly totally disconnected, and $Y^\prime$ is quasicompact. We check the criteria of Proposition~\ref{prop:propsmooth}. Part (i) is clear. For part (ii), note that for any constructible sheaf $\mathcal F$ on $Y^\prime/\underline{K}$, one has natural maps
\[
\mathcal F\to q_\ast q^\ast \mathcal F = Rq_! q^\ast \mathcal F\to Rq_! Rq^! \mathcal F\to \mathcal F
\]
whose composite is an isomorphism. This implies that $R(f/\underline{K})_! \mathcal F$ is a direct summand of $Rf_! q^\ast \mathcal F = R(f/\underline{K})_! Rq_! q^\ast \mathcal F$, which is thus constructible. Now part (iii) is a question about geometric fibres that holds by hypothesis. In particular, condition (iii) of Proposition~\ref{prop:shriekvsusualpullback} holds true, so $R(f/\underline{K})^!$ commutes with quasi-pro-\'etale base change. Thus, checking whether
\[
q^\ast R(f/\underline{K})^!\to Rf^!
\]
is an equivalence can be done on geometric fibres; but here it holds by assumption. In particular, this implies that condition (iv) of Proposition~\ref{prop:propsmooth} is satisfied. Thus, $f/\underline{K}$ is $\ell$-cohomologically smooth; moreover, we have already checked that $q^\ast R(f/\underline{K})^!\to Rf^!$ is an equivalence, as desired.
\end{proof}

In particular, we note that smooth morphisms of analytic adic spaces give examples of cohomologically smooth morphisms.

\begin{proposition}\label{prop:classicallysmoothisgeomsmooth} Let $f: Y^\prime\to Y$ be a separated smooth morphism of analytic adic spaces over $\Spa \mathbb Z_p$, i.e.~$f$ is locally on $Y^\prime$ the composition of an \'etale map $Y^\prime\to \mathbb B_Y^n$ and the projection $\mathbb B_Y^n\to Y$, cf.~\cite[Corollary 1.6.10]{Huber}. Then $f^\diamondsuit: (Y^\prime)^\diamondsuit\to Y^\diamondsuit$ is $\ell$-cohomologically smooth.
\end{proposition}

Here, for $Y=\Spa(A,A^+)$ affinoid, the ball $\mathbb B_Y^n$ is defined by $\Spa(A\langle T_1,\ldots,T_n\rangle,A^+\langle T_1,\ldots,T_n\rangle)$.

\begin{proof} By Lemma~\ref{lem:diamondadicspace}, $f$ is representable in locally spatial diamonds. Using Proposition~\ref{prop:compactifiable}, one sees that $f$ is compactifiable.

By Proposition~\ref{prop:smooth2outof3}, it suffices to check that $\mathbb B_Y^n\to Y$ is $\ell$-cohomologically smooth. By induction, we can assume that $n=1$. As we can cover $\mathbb B_Y$ by two copies of $\mathbb T_Y\subset \mathbb B_Y$ of the annulus (with ``center'' $0$ and $1$ respectively), it suffices to show that $\mathbb T_Y\to Y$ is $\ell$-cohomologically smooth. If we work over $\Spa \mathbb Q_p$ instead, we can even base change to $Y=\Spa \mathbb C_p$.

In that case, there is a $\mathbb Z_p$-torsor
\[
\widetilde{\mathbb T}_{\mathbb C_p}=\Spa \mathbb C_p\langle T^{\pm 1/p^\infty}\rangle \to \mathbb T_{\mathbb C_p} = \Spa \mathbb C_p\langle T^{\pm 1}\rangle\ ,
\]
so that
\[
\mathbb T_{\mathbb C_p}^\diamondsuit = \widetilde{\mathbb T}_{\mathbb C_p}^\diamondsuit/\underline{\mathbb Z_p}\to Y^\diamondsuit = (\Spa \mathbb C_p)^\diamondsuit
\]
is $\ell$-cohomologically smooth by Proposition~\ref{prop:propquotientsmooth} and Theorem~\ref{thm:ballcohomsmooth}.

On the other hand, if $Y$ is of characteristic $p$, the result follows directly from Theorem~\ref{thm:ballcohomsmooth}. In mixed characteristic, we argue as follows. We need to see that $\mathbb T_Y\to Y$ is $\ell$-cohomologically smooth. We can assume, by a v-cover, that $Y$ lives over $\Spa \mathbb Z_p^{\mathrm{cycl}}$. In that case, $\mathbb T_Y$ is a quotient of $\widetilde{\mathbb T}_Y$ by a nonfree $\underline{\mathbb Z_p}$-action. The result follows from Proposition~\ref{prop:propnonfreequotientsmooth}, as we have already verified that all geometric fibres are $\ell$-cohomologically smooth.
\end{proof}

Another example is the following.

\begin{proposition}\label{prop:Qpsmooth} The map $(\Spa \mathbb Q_p)^\diamondsuit\to \ast$ is $\ell$-cohomologically smooth.
\end{proposition}

\begin{proof} Let $K_\infty$ be the cyclotomic $\mathbb Z_p$-extension of $\mathbb Q_p$. Then $K_\infty^\flat\cong \mathbb F_p((t^{1/p^\infty}))$, and
\[
(\Spa \mathbb Q_p)^\diamondsuit = \Spa \mathbb F_p((t^{1/p^\infty})) / \underline{\mathbb Z_p}\ .
\]
In particular, after pullback along the v-cover $\Spa(C,\OO_C)\to\ast$, where $C$ is some algebraically closed nonarchimedean field of characteristic $p$, one has
\[
(\Spa \mathbb Q_p)^\diamondsuit \times \Spa(C,\OO_C) = \mathbb D^\times_C / \underline{\mathbb Z_p}\ ,
\]
where $\mathbb D^\times_C = \Spa \mathbb F_p((t^{1/p^\infty}))\times \Spa(C,\OO_C)$ denotes the punctured open unit disc over $\Spa(C,\OO_C)$, which is an open subset of $\mathbb B\times \Spa(C,\OO_C)$. Using Proposition~\ref{prop:propquotientsmooth}, this shows that $(\Spa \mathbb Q_p)^\diamondsuit\times \Spa(C,\OO_C)\to \Spa(C,\OO_C)$ is $\ell$-cohomologically smooth, and thus also $(\Spa \mathbb Q_p)^\diamondsuit\to \ast$.
\end{proof}

Occasionally, it is useful to check cohomological smoothness in a different way, not by reduction to the case that the base is strictly totally disconnected, but by keeping a large geometric base. In that case, the following criterion is useful.

\begin{proposition}\label{prop:cohomsmoothlargebase} Let $X$ be a perfectoid space, let $Y$ be a locally spatial diamond, and let $f: Y\to X$ be compactifiable with locally $\dimtrg f<\infty$. Then $f$ is $\ell$-cohomologically smooth if and only if the following conditions are satisfied.
\begin{altenumerate}
\item[{\rm (i)}] The sheaf $Rf^!\Fl\in D_\et(Y,\Fl)$ is invertible, i.e.~\'etale locally isomorphic to $\Fl[d]$ for some $d\in \mathbb Z$.
\item[{\rm (ii)}] For any $\tilde{X}\to X$ that is an open subset of a finite-dimensional ball $\mathbb B^n_X$ over $X$ with pullback $\tilde{f}: \tilde{Y}=Y\times_X \tilde{X}\to \tilde{X}$, the natural transformation
\[
R\tilde{f}^! \Fl\dotimes_\Fl \tilde{f}^\ast\to R\tilde{f}^!
\]
of functors $D_\et(\tilde{X},\Fl)\to D_\et(\tilde{Y},\Fl)$ is an equivalence.
\end{altenumerate}
\end{proposition}

\begin{proof} The conditions are necessary by Proposition~\ref{prop:smoothshriekbasechange}. Conversely, we may assume that $X$ is affinoid perfectoid and $Y$ is spatial.

As preparation, note that in the situation of condition (ii), denoting the map $g: \tilde{Y}\to Y$, the natural map
\[
g^\ast Rf^!\Fl\to R\tilde{f}^!\Fl
\]
is an equivalence by Proposition~\ref{prop:smoothbasechange}~(iii), as $g$ is $\ell$-cohomologically smooth by Theorem~\ref{thm:ballcohomsmooth}. In particular, $R\tilde{f}^!\Fl$ is invertible.

Now let $X^\prime$ be any strictly totally disconnected space with a map $g: X^\prime\to X$ to $X$, and let $f^\prime: Y^\prime=X^\prime\times_X Y\to X^\prime$ be the pullback, $h: Y^\prime\to Y$. We will prove that the composite transformation
\[
h^\ast Rf^! \Fl\dotimes_\Fl f^{\prime\ast}\to Rf^{\prime!} \Fl\dotimes_\Fl f^{\prime\ast}\to Rf^{\prime !}
\]
of functors $D_\et(X^\prime,\Fl)\to D_\et(Y^\prime,\Fl)$ is an equivalence, which gives the desired result. We can write $X^\prime$ as an inverse limit of affinoid perfectoid spaces $\tilde{X}_i\to X$ that are open subsets of finite-dimensional balls over $X$; let $\tilde{f}_i: \tilde{Y}_i = \tilde{X}_i\times_X Y\to \tilde{X}_i$ be the corresponding pullbacks.

It is enough to check that the transformation
\[
h^\ast Rf^! \Fl\dotimes_\Fl f^{\prime\ast}\to Rf^{\prime!} \Fl\dotimes_\Fl f^{\prime\ast}\to Rf^{\prime !}
\]
becomes an equivalence after evaluation on any quasicompact separated \'etale map $V^\prime\to Y^\prime$. By \ref{prop:etmaptolimdiamond}~(iii), this comes via pullback from some quasicompact separated \'etale map $\tilde{V}_i\to \tilde{Y}_i$ for $i$ large enough. Replacing $X$ by $\tilde{X}_i$ and $Y$ by $\tilde{V}_i$, which still satisfy conditions (i) and (ii), we can reduce to the case $V^\prime=Y^\prime$. In other words, it is enough to check that the transformation
\[
h^\ast Rf^! \Fl\dotimes_\Fl f^{\prime\ast}\to Rf^{\prime!} \Fl\dotimes_\Fl f^{\prime\ast}\to Rf^{\prime !}
\]
becomes an equivalence on global sections; for this, it is enough to check that it becomes an equivalence after applying $Rh_\ast$. By computation, we obtain the desired equality
\[
Rh_\ast Rf^{\prime !} = Rf^! Rg_\ast = Rf^! \Fl\dotimes_\Fl f^\ast Rg_\ast = Rf^!\Fl\dotimes_\Fl Rh_\ast f^{\prime\ast} = Rh_\ast(h^\ast Rf^! \Fl\dotimes_\Fl f^{\prime\ast})\ ,
\]
using Proposition~\ref{prop:smoothbasechange}~(i) in the first equation, condition (ii) in the second equation, Proposition~\ref{prop:Rfastsimple} in the third equation, and a simple projection formula for $Rh_\ast$ (noting that $Rf^!\Fl$ is invertible).
\end{proof}

Let us end this section with a question, which gives one way of making precise the intuition that ``regular rings are smooth over an absolute base''.

\begin{question} Let $R$ be a regular $\mathbb Z_p$-algebra that is $I$-adically complete for some ideal $I$ containing $p$, and topologically of finite type over $\mathbb Z_p$. Is the map $(\Spa R)^\diamondsuit\to \ast$ cohomologically smooth?
\end{question}

Here, $(\Spa R)^\diamondsuit$ is the small v-sheaf parametrizing untilts over $R$. Even the case $R=\mathbb Z_p$ is not known to us. We note that in general, this question is related to Grothendieck's purity conjecture (the theorem of Thomason--Gabber).

\section{Biduality}

In this section, we prove a biduality result for cohomologically smooth diamonds over a geometric point $\Spa(C,\OO_C)$. More general results in the case of rigid spaces can be found in work of Gaisin--Welliaveetil, \cite{GaisinWelliaveetil}.

\begin{theorem}\label{thm:constructiblereflexive} Let $C$ be a complete algebraically closed nonarchimedean field of characteristic $p$ with ring of integers $\OO_C$, and let $X$ be a locally spatial diamond that is separated and $\ell$-cohomologically smooth over $\Spa(C,\OO_C)$ for some $\ell\neq p$. Let $A\in D_\et(X,\Fl)$ be a bounded complex with constructible cohomology. Then the double (naive) duality map
\[
A\to R\sHom_\Fl(R\sHom_\Fl(A,\Fl),\Fl)
\]
is an equivalence; equivalently, as $Rf^! \Fl$ is invertible, the double Verdier duality map
\[
A\to R\sHom_\Fl(R\sHom_\Fl(A,Rf^! \Fl),Rf^! \Fl)
\]
is an equivalence. Moreover, if $X$ is quasicompact, then $H^i(X,A)$ is finite for all $i\in \mathbb Z$.
\end{theorem}

\begin{remark} This theorem does not hold if one replaces $\Spa(C,\OO_C)$ by $\Spa(C,C^+)$ for $C^+\neq \OO_C$, even in the simplest case $X=\Spa(C,C^+)$. Indeed, if $j: U=\Spa(C,\OO_C)\hookrightarrow X=\Spa(C,C^+)$ denotes the open immersion, then $A=j_! \Fl$ has dual $R\sHom_\Fl(j_! \Fl,\Fl) = Rj_\ast \Fl = j_\ast \Fl=\Fl$, and so the double dual is also equal to $\Fl$, which is different from $A=j_!\Fl$.
\end{remark}

\begin{remark} The biduality theorem also holds true for a general coefficient ring $\Lambda$ of $\ell$-power-torsion if $A\in D_\et(X,\Lambda)$ is perfect-constructible, and in that case $R\Gamma(X,A)$ is a perfect complex of $A$-modules if $X$ is quasicompact. Indeed, by Proposition~\ref{prop:perfectconstructiblecompact} this reduces to the case $A=j_! \Lambda$ for some quasicompact separated \'etale map $j: U\to X$, in which case it follows from the theorem applied to $j_! \Fl$ (which implies the same for $j_! \mathbb Z/\ell^m\mathbb Z$ by the $5$-lemma, and then for $j_! \Lambda$ by extension of scalars).\footnote{The last step requires explanation for the biduality statement. Note that $Rj_\ast \Lambda = Rj_\ast \mathbb Z/\ell^m\mathbb Z\dotimes_{\mathbb Z/\ell^m \mathbb Z} \Lambda$ as $j$ is qcqs and of finite cohomological dimension so that $Rj_\ast$ commutes with all colimits, and then
\[
R\sHom_\Lambda(R\sHom_\Lambda(j_! \Lambda,\Lambda),\Lambda) = R\sHom_\Lambda(Rj_\ast \Lambda,\Lambda) = R\sHom_{\mathbb Z/\ell^m\mathbb Z}(Rj_\ast \mathbb Z/\ell^m\mathbb Z,\Lambda)\ .
\]
It remains to see that $R\sHom_{\mathbb Z/\ell^m\mathbb Z}(Rj_\ast \mathbb Z/\ell^m\mathbb Z,\Lambda)=j_! \Lambda$, which follows from the proof of the theorem.}
\end{remark}

\begin{proof} We can assume that $X$ is spatial and that $A$ is concentrated in degree $0$. By Proposition~\ref{prop:charconstructible2}, we can assume that $A=j_!\mathcal L$ for some quasicompact separated \'etale map $j: U\to X$ and some $\Fl$-local system $\mathcal L$ on $U$. As the statement is \'etale local, we can replace $X$ by an \'etale cover; doing so, we can assume that $j$ decomposes into a disjoint union of open embeddings. Thus, we can assume that $j$ is an open immersion. For any geometric point $\overline{x}\to X$, there is an \'etale neighborhood of $V$ of $\overline{x}$ such that $\mathcal L$ is constant on $V\times_X U$. Indeed, by Proposition~\ref{prop:constructiblelimit}, it is enough to check this on the strict localization $\Spa(C(x),C(x)^+)$ of $X$ at $\overline{x}$, and then it follows from the observation that any quasicompact open subspace of $\Spa(C(x),C(x)^+)$ is still strictly local. Using this observation, we can also assume that $\mathcal L$ is constant.

We are reduced to the case $A=j_! \Fl$ for a quasicompact open immersion $j: U\to X$. In this case $R\sHom_\Fl(j_!\Fl,\Fl) = Rj_\ast\Fl$, and the biduality statement says that the natural map
\[
j_!\Fl\to R\sHom_\Fl(Rj_\ast \Fl,\Fl)
\]
is an equivalence. We claim more generally that for any $M\in D_\et(\Fl)$, the map
\[
j_!M\to R\sHom_\Fl(Rj_\ast \Fl,M)
\]
is an equivalence. This does in fact imply that $H^i(X,j_! \Fl)$ is finite. Indeed, by further \'etale localization, we can assume that $Rf^!\Fl=\Fl[d]$ is trivial. One has
\[
R\Gamma(X,j_!M)=R\Hom_{D_\et(X,\Fl)}(Rj_\ast \Fl,M) = R\Hom_{D(\Fl)}(Rf_! Rj_\ast \Fl,M)[-d]\ .
\]
As the left-hand side commutes with all colimits by Proposition~\ref{prop:constructiblecompactderivedfull}, it follows that $Rf_! Rj_\ast \Fl\in D(\Fl)$ is compact, i.e.~bounded with finite cohomology groups. Applying the displayed equality for $M=\Fl$ then gives the finiteness of $H^i(X,j_!\Fl)$.

Thus, we have to prove that for all $M\in D_\et(\Fl)$, the map
\[
j_! M\to R\sHom_\Fl(Rj_\ast \Fl,M)
\]
is an equivalence. Choose a quasi-pro-\'etale surjection $q: \tilde{X}\to X$ from a strictly totally disconnected perfectoid space that can be written as an inverse limit of quasicompact separated \'etale maps $q_i: \tilde{X}_i\to X$ as in Proposition~\ref{prop:spatialunivopen}. Let $\tilde{j}_i: \tilde{U}_i\to \tilde{X}_i$ and $\tilde{j}: \tilde{U}\to \tilde{X}$ be the pullbacks of $j: U\to X$. Then $\tilde{j}: \tilde{U}\to \tilde{X}$ is a quasicompact open subset of the strictly totally disconnected space $\tilde{X}$; by Lemma~\ref{lem:subsetwlocal}, there is some function $f\in H^0(\tilde{X},\OO_{\tilde{X}})$ such that $\tilde{U}$ is the locus $\{|f|\leq 1\}$. Multiplying $f$ by a suitable pseudouniformizer $\varpi\in C$, we can instead arrange that $f\in H^0(\tilde{X},\OO_{\tilde{X}}^+)$ so that $\tilde{U}$ is the locus $\{|f|\leq |\varpi|\}$. In particular, $f$ defines a section of $\OO_{\tilde{X}}^+/\varpi$. But
\[
H^0(\tilde{X},\OO_{\tilde{X}}^+/\varpi) = \varinjlim_i H^0(\tilde{X}_i,\OO_{\tilde{X}_i}^+/\varpi)
\]
by Proposition~\ref{prop:etcohomlim} (noting that $\OO_X^+/\varpi$ is an \'etale sheaf with pullback $\OO_{\tilde{X}}^+/\varpi$ to $\tilde{X}$; indeed this holds true by definition as pro-\'etale sheaves, but the pro-\'etale sheaf $\OO_{\tilde{X}}^+/\varpi$ is actually an \'etale sheaf, which implies that the pro-\'etale sheaf $\OO_X^+/\varpi$ is an \'etale sheaf by Theorem~\ref{thm:proetandetsheafdetectionvlocally}). Thus, replacing $\tilde{X}$ by $\tilde{X}_i$ for $i$ large enough if necessary, we can assume that there is a section $\overline{f}\in H^0(X,\OO_X^+/\varpi)$ such that $U$ is the locus where $\overline{f}=0$. Let $\mathbb B_C = \Spa(C\langle T^{1/p^\infty}\rangle,\OO_C\langle T^{1/p^\infty}\rangle)$ be the ball over $\Spa(C,\OO_C)$. We can consider the open subset
\[
Y=\{\overline{f}-\overline{T}=0\}\subset X\times_{\Spa(C,\OO_C)} \mathbb B_C\ ,
\]
where $\overline{T}$ denotes the image of $T$ in $\OO^+/\varpi$. The projection map $Y\to X$ is cohomologically smooth (as the composite of an open immersion and a pullback of $\mathbb B_C\to \Spa(C,\OO_C)$, cf.~Theorem~\ref{thm:ballcohomsmooth}) and surjective as a map of v-sheaves; thus, it suffices to check the statement after pullback to $Y$ (using Proposition~\ref{prop:smoothhompullback} and base change results). The pullback of $U$ to $Y$ is given by the locus where $\overline{T}=0$ (as $\overline{f}-\overline{T}=0$ on $Y$). This also comes as the pullback of the locus $\{\overline{T}=0\}$ inside $\mathbb B_C$, and the projection $Y\to \mathbb B_C$ is cohomologically smooth (again as a composite of an open immersion and the pullback of the cohomologically smooth map $X\to \Spa(C,\OO_C)$). Thus, we can replace $j: U\hookrightarrow X$ by $\{\overline{T}=0\}\subset \mathbb B_C$.

Thus, we can assume $X$ is (the diamond associated) a quasicompact smooth rigid-analytic curve over $C$, and $U$ is a quasicompact open subspace,\footnote{A case that is handled by a different method in~\cite{GaisinWelliaveetil}.} and allowing this generality, it is enough to prove the result on global sections, i.e.
\[
R\Gamma(X,j_!M)\to R\Hom_{D_\et(X,\Fl)}(Rj_\ast \Fl,M)
\]
is an equivalence for all $M\in D(\Fl)$. The cone of the map
\[
j_! M\to R\sHom_\Fl(Rj_\ast \Fl,M)
\]
is concentrated on $\overline{U}\setminus U$, which is a discrete finite set of points. In particular, the cone of $R\Gamma(X,j_!M)\to R\Hom_{D_\et(X,\Fl)}(Rj_\ast \Fl,M)$ decomposes into a direct sum of contributions over $\overline{U}\setminus U$. We may embed $X$ into a proper smooth rigid-analytic curve $X^\prime$ by \cite[Theorem 5.3]{Luetkebohmert}, and we continute to denote by $j: U\to X^\prime$ the open immersion. The previous discussion implies that the cone of
\[
R\Gamma(X,j_!M)\to R\Hom_{D_\et(X,\Fl)}(Rj_\ast\Fl,M)
\]
is a direct summand of the cone of
\[
R\Gamma(X^\prime,j_!M)\to R\Hom_{D_\et(X^\prime,\Fl)}(Rj_\ast\Fl,M)\ .
\]
Thus, it is enough to prove that the latter is trivial. In other words, we may assume that $X$ is proper. But then, if $M=M^\prime\dotimes_\Fl Rf^! \Fl = Rf^! M^\prime$, one has
\[\begin{aligned}
R\Hom_{D_\et(X,\Fl)}(Rj_\ast \Fl,M) &= R\Hom_{D(\Fl)}(Rf_! Rj_\ast \Fl,M^\prime) = R\Hom_{D(\Fl)}(Rf_\ast Rj_\ast \Fl,M^\prime)\\
& = R\Hom_{D(\Fl)}(R\Gamma(U,\Fl),M^\prime)
\end{aligned}\]
and
\[
R\Gamma(X,j_! M) = R\Gamma_c(X,j_! M) = R\Gamma_c(U,M)=R\Gamma_c(U,Rf^! M^\prime)\ ,
\]
so the result follows from Poincar\'e duality on $U$; indeed, $R\Gamma_c(U,Rf^!\Fl)$ and its dual $R\Gamma(U,\Fl)$ are finite by Proposition~\ref{prop:propsmooth}~(ii), so both sides commute with all colimits in $M$, and we may reduce to $M=\Fl$. Poincar\'e duality says that $R\Gamma(U,\Fl)$ is the dual of $R\Gamma_c(U,Rf^!\Fl)$, but in vector spaces, this implies that $R\Gamma_c(U,Rf^!\Fl)$ is the dual of $R\Gamma(U,\Fl)$, as desired; we leave it to the reader to check that the maps are correct.
\end{proof}

It is sometimes useful to combine this biduality result with a conservativity result for the Verdier dual.

\begin{proposition}\label{prop:dualconservative} Let $C$ be an algebraically closed nonarchimedean field of characteristic $p$ with ring of integers $\OO_C$, and let $X$ be a locally spatial diamond with $f: X\to \Spa(C,\OO_C)$ compactifiable with locally $\dimtrg f<\infty$. Assume that $A\in D_\et(X,\Fl)$ satisfies $R\sHom_\Fl(A,Rf^! \Fl)=0$. Then $A=0$.
\end{proposition}

\begin{remark} Again, a similar result fails over $\Spa(C,C^+)$ if $C^+\neq \OO_C$. Indeed, if $X=\Spa(C,C^+)$ and $i: \{s\}\to X$ is the inclusion of the closed point, then $R\sHom_\Lambda(i_\ast \Lambda,\Lambda)=0$.
\end{remark}

\begin{remark} A similar result holds with coefficients in a ring $\Lambda$ (killed by $n$ prime to $p$) as soon as for any $M\in D(\Lambda)$ with $R\Hom_\Lambda(M,\Lambda)=0$, one has $M=0$. This fails in some cases, for example if $\Lambda=\OO_K$ with $K$ spherically complete and $M=k$ the residue field. We do not know in which generality it holds (for example if $\Lambda$ is noetherian).
\end{remark}

\begin{proof} We may assume that $X$ is quasicompact. Let $g: \tilde{X}\to X$ be a quasi-pro-\'etale surjective map from a strictly totally disconnected space as in Proposition~\ref{prop:spatialunivopen}, and let $\overline{g}^{/X}: \overline{\tilde{X}}^{/X}\to X$ be its compactification. Then $\overline{g}^{/X}$ is proper with $\dimtrg \overline{g}^{/X}=0$, so $R\overline{g}^{/X!}$ is well-defined and
\[
0=R\overline{g}^{/X!}R\sHom_\Fl(A,Rf^!\Fl) = R\sHom_\Fl(\overline{g}^{/X\ast} A,R(f\circ \overline{g}^{/X})^! \Fl)
\]
by Proposition~\ref{prop:localverdier}~(ii). It is enough to prove that $\overline{g}^{/X\ast} A=0$, so we may replace $X$ by $\overline{\tilde{X}}^{/X}$ and $A$ by $\overline{g}^{/X\ast} A$.

In other words, we can assume that $X$ is an affinoid perfectoid space, compactifiable over $\Spa(C,\OO_C)$, whose connected components are of the form $\Spa(C^\prime,C^{\prime +})$ for varying complete algebraically closed nonarchimedean fields $C^\prime$ and certain open and integrally closed subrings $C^{\prime +}\subset \OO_{C^\prime}$ (not necessarily valuation rings). In that case, $D_\et(X,\Fl) = D(|X|,\Fl)$. In fact, replacing $X$ by its compactification over $\Spa(C,\OO_C)$ and $A$ by its corresponding extension by zero, we can assume that $X$ is proper over $\Spa(C,\OO_C)$. For any open subset $U\subset X$, we have
\[
R\Hom_{D(\Fl)}(R\Gamma_c(U,A),\Fl) = R\Hom_{D_\et(U,\Fl)}(A|_U,Rf^! \Fl|_U) = R\Gamma(U,R\sHom_\Fl(A,Rf^!\Fl))=0\ ,
\]
which implies that $R\Gamma_c(U,A)=0$. Using this for $U=X$ and $U=X\setminus \{x\}$ for some closed point $x\in X$, one sees that also the cone of $R\Gamma_c(U,A)\to R\Gamma_c(X,A)=R\Gamma(X,A)$ is zero, which is the stalk $A_x$ of $A$ at $x$. In other words, the stalks of $A$ at all closed points vanish, which implies that $A=0$, as the closed points are very dense in $|X|$. To see the latter, we can check on connected components, and (using that $X$ is compactifiable) these are open inside the Zariski--Riemann space of $\OO_{C^\prime}/\mathfrak m_{C^\prime}$ over $\OO_C/\mathfrak m_C$. But Zariski--Riemann spaces of fields have a very dense set of closed points (by writing them as a cofiltered limit of projective varieties).
\end{proof}

\section{Adic sheaves}

The formalism developed in this paper extends to $\ell$-adic sheaves. As coefficients, we take a ring $\Lambda$ that is complete for and equipped with the $I$-adic topology for some finitely generated ideal $I\subset \Lambda$. For simplicity, we assume that $I$ is generated by a regular sequence and contains an integer prime to $p$.

\begin{definition} For any small v-stack $Y$, define
\[
D_\et(Y,\Lambda)\subset D(Y_v,\Lambda)
\]
as the full subcategory of the derived category of $\Lambda$-modules on $Y_v$ of all $A\in D(Y_v,\Lambda)$ such that $A$ is derived $I$-complete, cf.~\cite[Definition 3.4.1, Lemma 3.4.12]{BhattScholze}, and $A\dotimes_{\Lambda} \Lambda/I$ lies in $D_\et(Y,\Lambda/I)$.\footnote{There is some obvious conflict of notation with $D_\et(Y,\Lambda)$ for $\Lambda$ considered as a discrete ring. We believe that in practice such as for $D_\et(Y,\mathbb Z_\ell)$ this will not cause any problems.}
\end{definition}

There is a natural $\infty$-categorical enrichment.

\begin{proposition}\label{prop:enrichmentexistsadic} There is a presentable stable $\infty$-category $\mathcal{D}_\et(Y,\Lambda)$ whose homotopy category is $D_\et(Y,\Lambda)$, given as the full $\infty$-subcategory of $\mathcal{D}(Y_v,\Lambda)$ spanned by the objects of $D_\et(Y,\Lambda)$.

The natural functor $\mathcal{D}_\et(Y,\Lambda)\to \varprojlim_n\mathcal{D}_\et(Y,\Lambda/I^n)$ is an equivalence.
\end{proposition}

However, $\mathcal{D}_\et(Y,\Lambda)\to \mathcal{D}(Y_v,\Lambda)$ does not preserve all colimits (and so is not a map in the category $\mathcal{P}r^L$ of \cite{LurieHTT}): Rather, colimits have to be completed in $\mathcal{D}_\et(Y,\Lambda)$.

\begin{proof} By \cite[Proposition 5.5.3.13]{LurieHTT}, it is enough to prove that the natural functor
\[
\mathcal{D}_\et(Y,\Lambda)\to \varprojlim_n \mathcal{D}_\et(Y,\Lambda/I^n)
\]
is an equivalence of $\infty$-categories. For this, it is enough to prove that if $\mathcal{D}_{\mathrm{comp}}(Y_v,\Lambda)\subset \mathcal{D}(Y_v,\Lambda)$ denotes the full $\infty$-subcategory spanned by the complete objects, then the functor
\[
\mathcal{D}_{\mathrm{comp}}(Y_v,\Lambda)\to \varprojlim_n \mathcal{D}(Y_v,\Lambda/I^n)
\]
is an equivalence. This follows from \cite[Lemma 3.5.7]{BhattScholze} (noting that if $I$ is generated by a regular sequence, the hypothesis that $\Lambda$ is noetherian is not necessary).
\end{proof}

Again, one has six operations. First, for any morphism of small v-stacks $f: Y^\prime\to Y$, one has a pullback functor $f^\ast: D_\et(Y,\Lambda)\to D_\et(Y^\prime,\Lambda)$ by restriction from the functor $f_v^\ast: D(Y_v,\Lambda)\to D(Y^\prime_v,\Lambda)$ constructed after Proposition~\ref{prop:derivedhyperdescent}. Indeed, this functor preserves completeness as $f_v^\ast$ commutes with all limits (as it is essentially a restriction), as well as the reduction modulo $I$.

As $f^\ast$ preserves arbitrary direct sums (as can be checked modulo $I$), it admits a right adjoint $Rf_\ast: D_\et(Y^\prime,\Lambda)\to D_\et(Y,\Lambda)$. Also, the completion of the usual tensor product on $D(Y_v,\Lambda)$ defines a symmetric monoidal tensor product $-\widehat{\dotimes}_\Lambda-$ on $D_\et(Y,\Lambda)$. By the usual adjunction, this also defines $R\sHom_\Lambda(-,-)$ on $D_\et(Y,\Lambda)$.

For the functor $Rf_!$, we rerun the construction for the whole pro-system $(\Lambda/I^n)_n$ of coefficient rings, and pass to the limit in the end. If $f: Y^\prime\to Y$ is compactifiable, representable in locally spatial diamonds, and with locally $\dimtrg f<\infty$, this defines the functor $Rf_!: D_\et(Y^\prime,\Lambda)\to D_\et(Y,\Lambda)$. It preserves arbitrary direct sums (and admits an $\infty$-categorical enhancement), and thus it admits a right adjoint $Rf^!: D_\et(Y,\Lambda)\to D_\et(Y^\prime,\Lambda)$.

\begin{remark} The functors $f^\ast$, $Rf_!$ and $-\widehat{\dotimes}_\Lambda-$ are compatible with the equivalence $\mathcal{D}_\et(Y,\Lambda)\cong \varprojlim_n \mathcal{D}_\et(Y,\Lambda/I^n)$. The same is true for the functors $Rf_\ast$, $Rf^!$ and $R\sHom_\Lambda(-,-)$: they commute with $-\dotimes_\Lambda \Lambda/I$ and the similar operations on $D_\et(Y,\Lambda/I)$ as $I$ is generated by a regular sequence and they are exact, and then by induction on $n$ one sees that they commute with $-\dotimes_\Lambda \Lambda/I^n$ and the similar operations on $D_\et(Y,\Lambda/I^n)$, thus are given by the limit over $n$ of these operations for all $\Lambda/I^n$.

In particular, given any natural transformation such as a base change map, to check whether it is an equivalence we may replace $\Lambda$ by $\Lambda/I$, and so reduce to the known results.
\end{remark}

Using these remarks, we see in particular that all the results stated in the introduction hold true for $D_\et(Y,\Lambda)$ in the adic case as well.

\section{Comparison to schemes}

Finally, we state some comparison results between the theory developed here and the classical theory for schemes (of characteristic $p$). For any scheme $X$ of characteristic $p$, one can define a small v-sheaf $X^\diamond$ sending any $S\in \Perf$ to the set of maps $S\to X$ in the category of adic spaces, where we embed schemes into adic spaces via $\Spec R\mapsto \Spa(R,R^+)$ where $R^+\subset R$ is the integral closure of $\mathbb F_p$. This functor factors over the category of perfect schemes. This induces a functor of sites $c_X: X^\diamond_v\to X_\et$, pullback along which defines a functor
\[
c_X^\ast: D(X_\et,\Lambda)\to D_\et(X^\diamond,\Lambda)\ ,
\]
where $\Lambda$ is any (discrete) ring. In fact, more generally, we can work with adic rings $\Lambda$ complete for the adic topology generated by an ideal $I$ which is generated by a regular sequence and contains some integer prime to $p$, if we work on the pro-\'etale site. Indeed, there is a natural map of sites still denoted $c_X: X^\diamond_v\to X_\proet$, completed pullback along which defines a functor
\[
c_X^\ast: D_\et(X,\Lambda)\to D_\et(X^\diamond,\Lambda)\ ,
\]
where $D_\et(X,\Lambda)\subset D(X_\proet,\Lambda)$ denotes the full subcategory of all derived $I$-complete $A\in D(X_\proet,\Lambda)$ such that $A\dotimes_\Lambda \Lambda/I$ lies in the left-completion $D_\et(X,\Lambda):=\widehat{D}(X_\et,\Lambda)\subset D(X_\proet,\Lambda)$; equivalently, all cohomology sheaves are in the essential image of $X_\et^\sim\subset X_\proet^\sim$. If $D(X_\et,\Lambda)$ is left-complete, for example under hypotheses of finite cohomological dimension, this is equivalent to $A\dotimes_\Lambda \Lambda/I\in D(X_\et,\Lambda)$.

\begin{proposition}\label{prop:schemestensorpullback} The comparison functor $c_X^\ast$ commutes with the following operations.
\begin{altenumerate}
\item[{\rm (i)}] The derived tensor product $-\widehat{\dotimes}_\Lambda-$.
\item[{\rm (ii)}] For any map $f: Y\to X$ of schemes of characteristic $p$, one has $(f^\diamond)^\ast c_X^\ast = c_Y^\ast f^\ast$.
\end{altenumerate}
\end{proposition}

\begin{proof} This is clear from the definition.
\end{proof}

\begin{proposition}\label{prop:schemesfullyfaithful} If $X$ is any scheme of characteristic $p$, the functor
\[
c_X^\ast: D_\et(X,\Lambda)\to D_\et(X^\diamond,\Lambda)
\]
is fully faithful and admits a right adjoint $Rc_{X\ast}$.
\end{proposition}

\begin{proof} As usual, the existence of the right adjoint $Rc_{X\ast}$ follows from the adjoint functor theorem (and the existence of natural $\infty$-categorical enhancements). We note that this right adjoint commutes with derived reduction modulo $I$ (as it is exact and $I$ is generated by a regular sequence) and when restricted to $D_\et(X^\diamond,\Lambda/I)$, it agrees with the right adjoint for $\Lambda/I$ in place of $\Lambda$ (as the diagram of the left adjoints, which are given by $c_X^\ast$ and derived reduction modulo $I$, commutes).

Now we have to prove that the adjunction map $A\to Rc_{X\ast} c_X^\ast A$ is an isomorphism for all $A\in D_\et(X,\Lambda)$. By the above remarks, this can be checked modulo $I$, so we can assume that $\Lambda$ is discrete and killed by some $n$ prime to $p$. Moreover, by passing to a Postnikov limit, we can assume that $A\in D^+_\et(X,\Lambda)=D^+(X_\et,\Lambda)$. Now it is enough to check the values on a basis of \'etale $U\in X_\et$, we can reduce to $X=U$, so we have to prove that
\[
R\Gamma(X,A) = R\Gamma(X^\diamond, c_X^\ast A)\ .
\]
By descent, we can assume that $X$ lives over $\overline{\mathbb F}_p$ and is affine; by taking a closed embedding, we may even assume that $X=\Spec \overline{\mathbb F}_p[X_i^{1/p^\infty},i\in I]$ for some (infinite) set $I$. Also, we can assume that $\Lambda=\mathbb F_\ell$.

Let $C$ be the completed algebraic closure of $\overline{\mathbb F}_p((t))$. By invariance of \'etale cohomology under change of algebraically closed base field, we have $R\Gamma(X,A)=R\Gamma(X_C,A)$, where $X_C=\Spec C[X_i^{1/p^\infty}]$. On the other hand, by Theorem~\ref{thm:changebasefield}, we have $R\Gamma(X^\diamond,c_X^\ast A)=R\Gamma(X_C^\diamond,c_X^\ast A)$, where $X_C^\diamond := X^\diamond\times_{\Spd \overline{\mathbb F}_p} \Spd C$. We can write $X_C^\diamond$ as a filtered colimit of infinite-dimensional balls $X_{C,(n_i)_{i\in I}}$ over $C$, where $n_i\geq 1$ are integers which we can assume to be powers of $p$ and
\[
X_{C,(n_i)_{i\in I}} = \Spd C\langle (t^{n_i}X_i)^{1/p^\infty},i\in I\rangle\ .
\]
Then
\[
R\Gamma(X_C^\diamond,c_X^\ast A) = \varprojlim R\Gamma(X_{C,(n_i)_{i\in I}},c_X^\ast A)\ ,
\]
so it suffices to prove that for any choice of $n_i$, the map
\[
R\Gamma(X,A)\to R\Gamma(X_{C,(n_i)_{i\in I}},c_X^\ast A)
\]
is an isomorphism. Now both sides commute with filtered colimits in $A$, so we can assume that $A$ is constructible, in which case it is pulled back from some finite-dimensional affine space, and by writing $X$ as the inverse limit of finite-dimensional affine spaces and similarly the infinite-dimensional ball as an inverse limit of finite-dimensional balls, we can reduce to the case that $I=\{1,\ldots,m\}$ is finite and $A$ is constructible. As there are only finitely many $n_i$ all of which are powers of $p$, we can assume that all $n_i=1$ by replacing the $X_i$ by $p$-power roots if necessary. Now more generally, for any positive $n\in \mathbb Z[1/p]$, let $X_{C,n} = X_{C,(n_i)_{i\in I}}$ with $n_i=n$ for all $i$.

We would like to claim that
\[
R\Gamma(X_{C,n},c_X^\ast A)
\]
is independent of the choice of $n$ via the inclusion maps. What is easier to see is that this cohomology group is abstractly independent of $n$, as there are automorphisms of $C$ over $\overline{\mathbb F}_p$ sending $t$ to $t^n$, and these induce isomorphisms $X_{C,n}\cong X_{C,1}$ compatible with the map to $X^\diamond$. Moreover, by \cite[Proposition 6.1.1]{Huber}, we know that for each $n$, the cohomology $R\Gamma(X_{C,n},c_X^\ast A)$ is of finite total dimension. On the other hand, as $c_X^\ast A$ is overconvergent, we have
\[
R\Gamma(X_{C,1},c_X^\ast A) = \varinjlim_{n>1} R\Gamma(X_{C,n},c_X^\ast A)\ .
\]
It follows that for $n>1$ sufficiently close to $1$, the map $R\Gamma(X_{C,n},c_X^\ast A)\to R\Gamma(X_{C,1},c_X^\ast A)$ is surjective in each degree, and thus an isomorphism (as both are of the same finite cardinality). Applying automorphisms of $C$ as above, this implies that for all integers $j\geq 1$, the map
\[
R\Gamma(X_{C,n^{j+1}},c_X^\ast A)\to R\Gamma(X_{C,n^j},c_X^\ast A)
\]
is an isomorphism, i.e.~all maps $R\Gamma(X_{C,n^j},c_X^\ast A)\to R\Gamma(X_{C,1},c_X^\ast A)$ are isomorphisms. Thus, we also have
\[
R\Gamma(X_C^\diamond,c_X^\ast A) = \varprojlim_j R\Gamma(X_{C,n^j},c_X^\ast A) = R\Gamma(X_{C,1},c_X^\ast A)\ ,
\]
and we recall that it remained to prove that $R\Gamma(X,A)\to R\Gamma(X_{C,1},c_X^\ast A)$ is an isomorphism. But by the identification of $X_C^\diamond$ with the diamond associated to the rigid space to $C$ associated to $X_C$ and \cite[Theorem 3.8.1]{Huber}, we know that the natural map $R\Gamma(X_C,A)\to R\Gamma(X_C^\diamond,c_X^\ast A)$ is an isomorphism. As $R\Gamma(X,A) = R\Gamma(X_C,A)$, this finishes the proof.
\end{proof}

Passing to right adjoints in Proposition~\ref{prop:schemestensorpullback}, we get the following proposition.

\begin{proposition}\label{prop:schemesRhompushforward} The functor $Rc_{X\ast}$ commutes with the following operations.
\begin{altenumerate}
\item[{\rm (i)}] For all $A\in D_\et(X,\Lambda)$, $B\in D_\et(X^\diamond,\Lambda)$, one has
\[
Rc_{X\ast}R\sHom_\Lambda(c_X^\ast A,B)\cong R\sHom_\Lambda(A,Rc_{X\ast} B)\ .
\]
\item[{\rm (ii)}] For all maps $f: Y\to X$ of schemes of characteristic $p$, one has $Rf_\ast Rc_{Y\ast}\cong Rc_{X\ast} Rf^\diamond_\ast$.
\end{altenumerate}
\end{proposition}

Moreover, we have the following result.

\begin{proposition}\label{prop:schemesshriekpushforward} Let $f: Y\to X$ be a separated map of finite type of qcqs schemes of characteristic $p$, or the perfection thereof. Then $f^\diamond: Y^\diamond\to X^\diamond$ is compactifiable and representable in locally spatial diamonds with $\dimtrg f<\infty$, and $Rf^\diamond_! c_Y^\ast\cong c_X^\ast Rf_!$. Moreover, $Rf^! Rc_{X\ast}\cong Rc_{Y\ast} R(f^\diamond)^!$.
\end{proposition}

\begin{proof} The last statement follows by passage to right adjoints. By Nagata's compactification theorem, one can write $f$ as a composite of an open immersion and a proper map. Moreover, any two such factorizations are dominated by a third map (in fact, the category of factorizations is filtered), which makes the following construction essentially independent of the choice.

If $f$ is an open immersion, the result is easy to check by hand, so assume that $f$ is proper. Then $f^\diamond$ is also proper and representable in spatial diamonds with $\dimtrg f<\infty$, and we have a base change transformation
\[
c_X^\ast Rf_\ast\to Rf^\diamond_\ast c_Y^\ast.
\]
We want to see that this is an equivalence. This can be done after base change to $\overline{\mathbb F}_p$ (as by proper base change, all operations commute with this, and one can check the result pro-\'etale locally), and we can assume that $X$ is affine. Also we may factor $f$ into a closed immersion and a proper map of finite presentation, which reduces us to the case that $f$ is finitely presented. It suffices to check the assertion on $A\in D^{\geq 0}_\et(Y,\Lambda)$ by a Postnikov limit argument. Now the functor $Rf^\diamond_\ast c_Y^\ast$ commutes with all colimits as well as with base change, so we can assume that $A\in D_\et(Y,\Lambda)$ is constructible, and then the whole situation arises as the base change from the case where $X$ and $Y$ are of finite type, so we can assume that we are in that situation.

Taking for $C$ the completed algebraic closure of $\overline{\mathbb F}_p((t))$ and using the fully faithful embedding $D_\et(X^\diamond,\Lambda)\to D_\et(X_C^\diamond,\Lambda)$, it suffices to prove that in the diagram of adic spaces
\[\xymatrix{
Y_C^{\mathrm{ad}}\ar[r]\ar[d]^{f_C} & Y^{\mathrm{ad}}\ar[d]^f\\
X_C^{\mathrm{ad}}\ar[r] & X^{\mathrm{ad}},
}\]
base change holds. Now the result follows from the combination of proper base change for schemes (for $\Spec C\to \Spec \overline{\mathbb F}_p$) and the comparison result \cite[Theorem 3.7.2]{Huber}.
\end{proof}

Finally, we also consider a variant for schemes of mixed characteristic. In that case, we have to impose stronger finiteness hypothesis.

Thus, fix some complete discrete valuation ring $\mathcal O$ with perfect residue field $k$, and consider schemes $X$ locally of finite type over $\Spec \mathcal O$. There is a functor $X\mapsto X^\diamond$ to diamonds over $\Spd \mathcal O$, sending a perfectoid space $S$ of characteristic $p$ to the set of untilts $S^\sharp$ over $\Spa \mathcal O=\Spa(\mathcal O,\mathcal O)$ and a map $S^\sharp\to X$ of locally ringed spaces over $\Spec \mathcal O$. We note that this functor is not quite compatible with the previous functor $X\mapsto X^\diamond$: For example, one sends $\Spec k$ to $\Spd k=\Spd(k,k)$, not to $\Spd(k,k^+)$ where $k^+\subset k$ is the integral closure of $\mathbb F_p$.

As before, we get a functor
\[
c_X^\ast: D_\et(X,\Lambda)\to D_\et(X^\diamond,\Lambda)
\]
satisfying Proposition~\ref{prop:schemestensorpullback}. It commutes with colimits, and so admits a right adjoint $Rc_{X\ast}$, satisfying Proposition~\ref{prop:schemesRhompushforward}. Concerning the compatibility with $Rf_!$, we have the following result.

\begin{proposition}\label{prop:schemesshriekpushforwardmixedchar} Let $f: Y\to X$ be a separated map between schemes of finite type over $\mathcal O$. Then $f^\diamond: Y^\diamond\to X^\diamond$ is compactifiable and representable in locally spatial diamonds with $\dimtrg f<\infty$, and $Rf^\diamond_! c_Y^\ast\cong c_X^\ast Rf_!$. Moreover, $Rf^! Rc_{X\ast}\cong Rc_{Y\ast} R(f^\diamond)^!$.
\end{proposition}

\begin{proof} As in the proof of Proposition~\ref{prop:schemesshriekpushforward}, it is enough to prove that if $f$ is proper, then the base change transformation
\[
c_X^\ast Rf_\ast\to Rf^\diamond_\ast c_Y^\ast
\]
is an equivalence. It suffices to check this on $D^+_\et$ by Postnikov limits and finite cohomological dimension. Let $T\to \Spa \mathcal O$ be a smooth map of adic spaces from an analytic adic space $T$; more precisely, take for $T$ the open subset of $\Spa \mathcal O[[x]]$ where $x\neq 0$. One can form the analytic adic space $X\times_{\Spec \mathcal O} T$; then
\[
(X\times_{\Spec \mathcal O} T)^\diamond = X^\diamond\times_{\Spd \mathcal O} T^\diamond,
\]
and the map $T^\diamond\to \Spd \mathcal O$ is surjective and $\ell$-cohomologically smooth for all $\ell\neq p$, by Proposition~\ref{prop:classicallysmoothisgeomsmooth}. This implies that the pullback functor $D_\et(X^\diamond,\Lambda)\to D_\et((X\times_{\Spec \mathcal O} T)^\diamond,\Lambda)$ is conservative. Thus, it suffices to prove the result after this pullback. Let $f_T: Y\times_{\Spec \mathcal O} T\to X\times_{\Spec \mathcal O} T$ be the induced map. We claim that after pullback to $T$, both sides identify with $Rf_{T\ast} -|_{X\times_{\Spec \mathcal O} T}$. Now for the left-hand side, the claim follows from \cite[Theorem 3.7.2]{Huber}, and for the right-hand side from Proposition~\ref{prop:Rfastsimple} and Lemma~\ref{lem:diamondadicspace}.
\end{proof}

In many cases, the functor $Rf_\ast$ actually commutes with $c_X^\ast$.

\begin{proposition}\label{prop:schemesRfastetale} Let $f: Y\to X$ be a separated map between schemes of finite type over $\mathcal O$, and assume that $\Lambda$ is a finite ring killed by an integer prime to $p$. Then for any constructible complex $A$ of $\Lambda$-modules on $Y_\et$, the map
\[
c_X^\ast Rf_\ast A\to Rf^\diamond_\ast c_Y^\ast A
\]
is an isomorphism.
\end{proposition}

\begin{proof} We may factor $f$ into an open immersion followed by a proper map, and for proper maps the result follows from the previous proposition. Thus, we can assume that $f=j: Y=U\hookrightarrow X$ is an open immersion. By a further d\'evissage, we can reduce to the case that $A=\mathbb F_\ell$ is constant, and that either $U$ is contained in the generic fibre $p\neq 0$, or else is of characteristic $p$. In characteristic $p$, by an application of de Jong's alterations, \cite[Theorem 4.1, Theorem 6.5]{deJongAlteration}, we can, up to replacing $k$ by a finite extension, reduce to the case that the boundary $Z$ is a normal crossing divisor over $k$. As in the previous proof, it suffices to prove the result after pullback along $T\to \Spa \mathcal O$. The right-hand side can then be described as $Rj_{T\ast} \mathbb F_\ell$ for the associated map $j_T: U_T\hookrightarrow X_T$ of analytic adic spaces, and we have to see that the same holds true for the left-hand side. But in the situation of the alteration, both $Rj_\ast \mathbb F_\ell$ and $Rj_{T\ast} \mathbb F_\ell$ can be computed explicitly (as one can compute $Ri^! \mathbb F_\ell$ and $Ri_T^! \mathbb F_\ell$ for all strata $i: Z\hookrightarrow X$, as they are smooth over $k$) and one sees that the comparison map is an isomorphism.

If $U$ lives over the generic fibre $K$ of $\mathcal O$, we assume that $X$ is proper over $\mathcal O$ (by compactifying) and factor $j: U\hookrightarrow X$ as the composite of $U\hookrightarrow X_\eta$ and $X_\eta\hookrightarrow X$. The first part can be handled as before using de Jong's alterations (or resolution of singularities, in characteristic $0$). It remains to handle $X_\eta\hookrightarrow X$. Again, by de Jong's alterations, after replacing $\mathcal O$ by a finite extension, we can assume that $X$ is strictly semistable. Arguing by induction on the dimension, and using smooth pullback, one shows that the map is an isomorphism away from a finite set of $k$-points (the most singular points). But then it suffices to see that the map becomes an isomorphism after proper pushforward along $X\to \Spec \mathcal O$. For proper maps, we already know the commutation; this reduces us to the open immersion $\Spec K\hookrightarrow \Spec \mathcal O$. This can be handled by a direct computation.
\end{proof}

Finally, we have the following fully faithfulness result.

\begin{proposition}\label{prop:schemesconstructiblefullyfaithful} Assume that $\Lambda$ is a finite ring. Then for any scheme $X$ locally of finite type over $\mathcal O$ and any constructible complex $A$ of $\Lambda$-modules on $X_\et$, the adjunction map
\[
A\to Rc_{X\ast} c_X^\ast A
\]
is an isomorphism.
\end{proposition}

\begin{proof} Consider first the case $X=\Spec \mathcal O$. Allowing ourselves to replace $\mathcal O$ by a finite cover, it suffices to treat the case $A=\mathbb Z/n\mathbb Z$ or $A=i_\ast\mathbb Z/n\mathbb Z$ where $i: \Spec k\hookrightarrow \Spec \mathcal O$ is the closed point. The second case is simple, and the first can then be reduced to $Rj_\ast \mathbb F_\ell$ instead. Using Proposition~\ref{prop:schemesRfastetale}, the claim then becomes that $R\Gamma(\Spec K,\mathbb F_\ell)\to R\Gamma(\Spd K,\mathbb F_\ell)$ is an isomorphism, where $K$ is the quotient field of $\mathcal O$, but here both sides are Galois cohomology.

In general, it suffices to check the result on global sections. Writing $R\Gamma(X,-)=R\Gamma(\Spec \mathcal O,Rf_\ast-)$ where $f: X\to \Spec \mathcal O$ is the projection, and using Proposition~\ref{prop:schemesRfastetale}, the claim reduces to the base $\Spec \mathcal O$ that we have already handled.
\end{proof}

\bibliographystyle{amsalpha}
\bibliography{EtCohDiamonds}

\end{document}